\documentclass{mod}
\usepackage{url}
\usepackage{setspace}
\usepackage{scrextend}
\usepackage{cite}
\usepackage{tocloft}

\usepackage{amsmath}
\usepackage{amsthm}
\usepackage{amssymb}

\usepackage{mathtools}
\usepackage{mathrsfs}
\usepackage[all]{xy}
\usepackage[toc,page]{appendix}
\usepackage{titlesec}
\usepackage{textcomp}

\usepackage{subfig}

 \newcounter{mainthm}

 \newtheorem{thm}{Theorem}[section]
 \newtheorem{lem}[thm]{Lemma}
 \newtheorem{prop}[thm]{Proposition}
 \newtheorem{cor}[thm]{Corollary}

 \theoremstyle{definition}
 \newtheorem{defn}[thm]{Definition}
 \newtheorem{ex}[thm]{Example}

\newtheorem{defn-thm}[thm]{Definition-Theorem}
\newtheorem{sublemma}[thm]{Sublemma}
\newtheorem{defn-lem}[thm]{Definition-Lemma}
\newtheorem{defn-prop}[thm]{Definition-Proposition}
\newtheorem{assumption}[thm]{Assumption}
\newtheorem{convention}
[thm]{Convention}

\newtheoremstyle{rmk}
{5pt}
{5pt}
{}
{}
{\itshape}
{}
{.5em}
{}

\theoremstyle{rmk}
\newtheorem{rmk}[thm]{Remark}

\newcommand{\smallsim}{{\text{\texttildelow}}}

\newcommand{\id}{\mathrm{id}}
\newcommand{\mathds}[1]{\text{\usefont{U}{dsrom}{m}{n}#1}}
\newcommand{\one}{\mathds {1}}

\newcommand{\ud}{\mathrm{ud}}
\newcommand{\UD}{\mathscr{UD}}
\newcommand{\simud}{\stackrel{\mathrm{ud}}{\sim}}
\newcommand{\con}{\mathsf{con}}

\newcommand{\pr}{\mathrm{pr}}
\newcommand{\ev}{\mathrm{ev}}

\newcommand{\forget}{{\mathfrak {forget}}}
\newcommand{\uu}{\mathbf u}
\newcommand{\ia}{{\mathfrak a}}
\newcommand{\e}{\mathbf e}
\newcommand{\B}{\mathsf B}
\newcommand{\mi}{\mathfrak i}
\newcommand{\mI}{\mathfrak I}
\newcommand{\mg}{\mathbf g}
\newcommand{\T}{\mathsf T}
\newcommand{\Tr}{\mathscr{T}}
\newcommand{\A}{\mathfrak A}
\newcommand{\Cint}{C^{\mathrm{int}}}
\newcommand{\Cext}{C^{\mathrm{ext}}}

\newcommand{\tri}{{\mathrm{tri}}}

\newcommand{\m}{\mathfrak m}

\newcommand{\M}{\mathfrak M}
\newcommand{\mc}{\mathfrak c}
\newcommand{\mC}{\mathfrak C}
\newcommand{\mP}{\mathfrak P}
\newcommand{\mQ}{\mathfrak Q}
\newcommand{\oi}{{[0,1]}}
\newcommand{\ab}{{[a,b]}}
\newcommand{\f}{\mathfrak f}
\newcommand{\F}{\mathfrak F}
\newcommand{\C}{C}

\newcommand{\mH}{ H}
\newcommand{\h}{\mathfrak h}
\newcommand{\z}{\mathbf z}

\newcommand{\s}{s}
\newcommand{\I}{\mathfrak I}

\newcommand{\K}{\mathbb K}
\newcommand{\GG}{\mathbb G}

\newcommand{\g}{\mathfrak g}
\newcommand{\G}{\mathfrak G}

\newcommand{\an}{\mathrm{an}}
\newcommand{\JJ}{{\pmb{\mathbb J}}}
\newcommand{\mF}{{\mathbf F}}
\newcommand{\mJ}{{\mathbf J}}
\newcommand{\HL}{H^*(L)}
\newcommand{\OL}{\Omega^*(L)}

\DeclareMathOperator{\Log}{Log}
\DeclareMathOperator{\Ob}{Ob}
\DeclareMathOperator{\Gr}{Gr}
\DeclareMathOperator{\vol}{vol}
\DeclareMathOperator{\dist}{dist}
\DeclareMathOperator{\area}{area}
\DeclareMathOperator{\Sp}{Sp}
\DeclareMathOperator{\incl}{Incl}
\DeclareMathOperator{\eval}{Eval}
\DeclareMathOperator{\restr}{Restr}

\DeclareMathOperator{\diff}{Diff}
\DeclareMathOperator{\Spec}{Spec}
\DeclareMathOperator{\val}{val}
\DeclareMathOperator{\Val}{\mathfrak{trop}}
\DeclareMathOperator{\trop}{\mathfrak {trop}}
\DeclareMathOperator{\Hom}{Hom}
\DeclareMathOperator{\DA}{DA}
\DeclareMathOperator{\DI}
{DI}
\DeclareMathOperator{\Obj}{Obj}
\DeclareMathOperator{\Mor}{Mor}
\DeclareMathOperator{\CU}{CU}
\DeclareMathOperator{\Corr}{Corr}
\DeclareMathOperator{\CC}{\mathbf{CC}}

\newcommand*{\Scale}[2][4]{\scalebox{#1}{$#2$}}%

\titleformat{\paragraph}[runin]{\small\sffamily\bfseries
}{}{}{}[]
\titleformat{\subsubsection}[runin]{\itshape\bfseries\normalsize}{\thesubsubsection \ }{0em}{}[\mbox{ . } ]
\titlespacing{\subsubsection}
{0pt}
{3.25ex plus 1ex minus .2ex}
{0pt}

\usepackage{indentfirst}

\usepackage{titletoc}

\titlecontents{section}
[0em]
{\scriptsize\bfseries\vspace{0pt}}%
{\thecontentslabel.\enspace}
{}
{\titlerule*[0.38pc]{.}\contentspage}%

\titlecontents{subsection}
[0em]
{\scriptsize \vspace{0pt}}%
{\qquad \quad \thecontentslabel.\enspace}
{}
{\titlerule*[0.5pc]{.}\contentspage}%

\titlecontents{subsubsection}
[0em]
{\footnotesize \vspace{1pt}}%
{\qquad \qquad \thecontentslabel.\enspace}
{}
{\titlerule*[0.5pc]{.}\contentspage}%

\begin{document}
	
	\setlength{\parindent}{15pt}	\setlength{\parskip}{0em}

\title{Family Floer program and non-archimedean SYZ mirror construction}
\author[Hang Yuan]{Hang Yuan}
\begin{abstract} {\sc Abstract:}
	\footnotesize
Given a Lagrangian fibration, we provide a natural construction of a mirror Landau-Ginzburg model consisting of a non-archimedean analytic space, a superpotential function, and a dual affinoid torus fibration.
The mirror in the B-side is constructed by the counts of holomorphic disks in the A-side together with the non-archimedean analysis and some novel homological algebra of $A_\infty$ structures.
The process builds on family Floer theory and introduces a critical enhancement to the traditional Maurer-Cartan framework.
It fits well with the SYZ dual fibration picture, elucidating both the quantum corrections and the wall-crossing phenomenon.
Instead of a special Lagrangian fibration, we only need to assume a weaker semipositive Lagrangian fibration to carry out the non-archimedean SYZ mirror reconstruction.
\end{abstract}
\maketitle
%
%

\tableofcontents

\section{Introduction}
\label{S_introduction}

The Strominger-Yau-Zaslow (SYZ) conjecture \cite{SYZ} posits that mirror symmetry for Calabi-Yau manifolds can be seen as a duality between torus fibrations. Framed within Kontsevich's Homological Mirror Symmetry (HMS) conjecture \cite{KonICM}, the SYZ perspective suggests that a mirror to a Calabi-Yau manifold could be realized as a moduli space of Lagrangian tori, each endowed with a rank-one local system over $U(1)$ (or actually $U_\Lambda$). Further, the mirror symmetry project has been expanded to include non-Calabi-Yau contexts by taking Landau-Ginzburg models.

The SYZ mirror reconstruction, based on dual Lagrangian fibrations, requires modifications through what is known as "quantum correction" or "instanton correction" to account for non-classical effects.
The quest to define this modification accurately has been ongoing for long time.
In \cite{AuTDual}, Auroux conjectures that for a K\"ahler manifold $X$ with an anticanonical divisor $D$ and a holomorphic volume form $\Omega$ having the pole along $D$, the mirror space $X_{\mathbb C}^\vee$ should be conceptualized as a moduli space of special Lagrangian tori in $X\setminus D$, each equipped with flat $U(1)$ connections. Moreover, there should be a global superpotential $W^\vee_{\mathbb C}: X^\vee_{\mathbb C}\to \mathbb C$, generated by the Fukaya-Oh-Ohta-Ono $\m_0$ obstruction to Floer homology. For a smooth torus fiber $L$, its SYZ dual was anticipated to be $H^1(L; U(1))$, i.e. the space of local $U(1)$-systems.
In this paper, we aim to positively affirm Auroux's conjecture, albeit within a non-archimedean framework.
Our approach builds on Fukaya's family Floer theory proposal \cite{FuFamily,FuAbelian_preprint,FuAbelian}.

\subsection{Main theorem}
Let $(X,\omega)$ be a symplectic manifold of real dimension $2n$ which is closed or convex at infinity. Suppose there is a smooth proper Lagrangian torus fibration $\pi_0:   X_0 \to B_0$ in an open subset $X_0$ of $X$ over an $n$-dimensional base manifold $B_0$.
Let $L_q$ denote the Lagrangian fiber over $q$.
Note that the quantum correction of holomorphic disks bounded by $\pi_0$-fibers is of global nature and sweeps not just in $X_0$ but in $X$. Thus, we should consider the pair $(X,\pi_0)$ as a whole.

The \textit{Novikov field} $\Lambda=\mathbb{C}((T^{\mathbb{R}}))$, consisting of formal power series $x=\sum_{i=0}^\infty a_i T^{\lambda_i}$ where $a_i \in \mathbb{C}$ and $\lambda_i\in\mathbb R$ strictly increases to infinity, is a non-archimedean field. It features a norm defined by $|x|=|x|_\Lambda=e^{-\lambda_0}$ for $a_0 \neq 0$, and this norm adheres to the non-archimedean triangle inequality $|x+y| \leqslant \max\{|x|,|y|\}$. This norm is equivalently characterized by a valuation map, $\val$, defined as $\val(x) = -\log|x|$.
Let $U_\Lambda=\{x\in\Lambda\mid |x|=1\}$ be the \textit{unitary Novikov group}, i.e., the unit circle in $\Lambda$.

\begin{thm}\label{Main_theorem_thm}
	Given Assumption \ref{assumption-mu ge 0} and \ref{assumption_obstruction_ideal} below, we can associate to the pair $(X,\pi_0)$ a triple
	$
	\mathbb X^\vee:=(X_0^\vee,W_0^\vee, \pi_0^\vee)
	$
	consisting of a non-archimedean analytic space $X_0^\vee$ over $\Lambda$, a global analytic function $W_0^\vee$, and an affinoid torus fibration $\pi_0^\vee: X_0^\vee\to B_0$
such that the following properties hold:

	\begin{enumerate}[i)]
		\itemsep 0.15em
		\item The analytic structure of $X_0^\vee$ is unique up to isomorphism.
		\item The integral affine structure on $B_0$ from $\pi_0^\vee$ coincides with the one from the fibration $\pi_0$.
		\item The set of closed points in $X_0^\vee$ coincides with the disjoint union
		\begin{equation}
			\label{union_mirror_intro}
			\bigcup_{q\in B_0} H^1(L_q; U_\Lambda)
		\end{equation}
	of the sets of local $U_\Lambda$-systems on the $\pi_0$-fibers,
	 and the map $\pi_0^\vee$ sends every $H^1(L_q; U_\Lambda)$ to $q$.
	\end{enumerate}
\end{thm}

Arnold-Liouville's theorem allows us to see any smooth Lagrangian torus fibration $\pi_0: X_0 \to B_0$ as a globalization of the complex logarithm map $\Log: (\mathbb C^*)^n \to \mathbb R^n$, where $z_k \mapsto \log|z_k|$. Namely, a smooth Lagrangian torus fibration merges multiple $\Log^{-1}(V_i) \to V_i$ for small open subsets $V_i$'s. The gluing of $\Log^{-1}(V_i)$ occurs within the symplectic manifold category, while the gluing of $V_i$'s endows $B_0$ with an \textit{integral affine structure}.
In non-archimedean analytic geometry, Kontsevich and Soibelman \cite{KSAffine} introduced a parallel concept, \textit{affinoid torus fibration}, acting as a globalization of the tropicalization map $\trop: (\Lambda^*)^n \to \mathbb R^n$, where $z_k \mapsto -\log |z_k|_\Lambda \equiv \val(z_k)$.
To grasp non-archimedean topology intuitively, consider this somewhat imprecise analogy: as sequence convergence in $(\mathbb C^*)^n$ is analyzed by the complex norm on $\mathbb C$, sequence convergence in $(\Lambda^*)^n$ can be approached by the non-archimedean norm on $\Lambda$.
Moreover, the affinoid torus fibration also inherently endows the base with an integral affine structure \cite[4.1]{KSAffine}, with the upstairs gluing executed within the category of non-archimedean analytic spaces.

As a toy model, the $\Log$ and $\trop$ torus fibrations can be considered SYZ duals. From this viewpoint, Theorem \ref{Main_theorem_thm} can be seen as a globalization of this toy SYZ model.
Roughly speaking, the quantum correction from Maslov-0 disks gives rise to a structure sheaf for a non-archimedean analytic space on the "fiberwise dual" set (\ref{union_mirror_intro}); the enumeration of Maslov-2 disks, grouped pearly with Maslov-0 disks, generates a superpotential function $W$ with respect to this analytic structure.
This picture follows the classic SYZ mirror duality of local $U(1)$-systems \cite{Gross_topo_MS, AuTDual}.
However, transitioning from $U(1)$ to $U_\Lambda$ is essential to accommodate the structure of affinoid torus fibrations. This adjustment ultimately leads to a mathematically precise formulation of the SYZ conjecture as exemplified in \cite{Yuan_local_SYZ,Yuan_conifold,Yuan_A_n}.

Let $\mathfrak J(X,\omega)$ be the space of $\omega$-tame almost complex structures. Fix a Lagrangian submanifold $L$. Let $\mu: \pi_2(X,L)\to\mathbb Z$ be the Maslov index. 
An almost complex structure $J$ is called \textit{$L$-semipositive} if $\mu(\beta)\ge 0$ for any $\beta\in \pi_2(X,L)$ that can be represented by a $J$-holomorphic stable map.
Denote by
$\mathfrak J(X,L,\omega)$ the subspace in $\mathfrak J(X,\omega)$ of all the $L$-semipositive $\omega$-tame almost complex structures.

\vspace{-0.4em}
\begin{assumption}\label{assumption-mu ge 0}
	For every compact subset $K\subset B_0$, the intersection
	$
	\mathfrak J_K$
	of $\mathfrak J(X,L_q,\omega)$ for all $q\in K$ has a non-empty interior in $\mathfrak J(X,\omega)$.
	In this case, we call the Lagrangian fibration $\pi$ is \textit{semipositive}.
	We further assume that the base $B_0$ and $\mathfrak J_K$ are connected and open.
\end{assumption}
\vspace{-0.2em}

	We justify the above assumption in two aspects as follows.
	We often know the non-emptiness.
	Assume there is a holomorphic volume form $\Omega$, and we can define the phase function $\theta=\arg \Omega|_L : L\to S^1$ on each Lagrangian fiber $L$.
	A Lagrangian fiber is called \textit{special} if $\theta$ is constant, while it is called \textit{graded} (see \cite{SeidelGraded}) if $\theta$ lifts to a real-valued function. By \cite[Lemma 3.1]{AuTDual}, every graded or special Lagrangian fibration is semipositive.
	Moreover, the openness condition is also inessential. Indeed, we similarly define $\mathfrak J_K^{\le E}$ for any $E>0$ by requiring only the stable maps of energy $\le E$ have non-negative Maslov indices.
	Then, the Gromov compactness exactly implies that every $\mathfrak J_K^{\le E}$ is open. Note that the $\mathfrak J_K^{\le E}$ decreases when $K$ or $E$ increases, and we have $\mathfrak J_K=\lim_{E\to \infty} \mathfrak J_K^{\le E}$.
 
 \vspace{-0.4em}
\begin{assumption}
	\label{assumption_obstruction_ideal}
	Every Lagrangian $\pi_0$-fiber $L$ has vanishing obstruction ideal.
\end{assumption}
\vspace{-0.2em}

The \textit{obstruction ideal} associated to $L$ is a finitely generated ideal in the formal power series ring $\Lambda[[\pi_1(L)]]$ (see (\ref{intro_ia_i}) or \S \ref{sss_local_charts_defn}).
The vanishing of the obstruction ideal remains consistent across different choices. In the recent study \cite{Yuan_unobs}, we call such $L$ as \textit{properly unobstructed} and justify that it is indeed a mild assumption.
It doesn't mean the absence of Maslov-0 holomorphic disks; rather, these disks still play a role in the wall-crossing phenomenon, albeit with a certain cancellation effect among them.
In fact, a sufficient condition is that every Lagrangian $\pi_0$-fiber $L$ is preserved by an anti-symplectic involution due to \cite{Solomon_Involutions}; in this case, an anti-symplectic involution $\varphi$ induces a pairwise cancellation for $\beta\xleftrightarrow{} -\varphi_*\beta$ on $\pi_2(X,L_q)$ to deduce Assumption \ref{assumption_obstruction_ideal}.

Finally, we remark that our proof for Theorem \ref{Main_theorem_thm} actually obtains a more general result at the cost of conciseness, and let us still state it below for completeness:

\vspace{-0.5em}
\begin{thm}
	In the situation of Theorem \ref{Main_theorem_thm}, if we drop Assumption \ref{assumption_obstruction_ideal}, then the same conclusions hold, except that the set of closed points in $X_0^\vee$ is only a subset of (\ref{union_mirror_intro}) locally identified with the zero loci of the obstruction ideals, and that the fibration map $\pi_0^\vee$ is not necessarily an affinoid torus fibration.
\end{thm}

\subsection{Application}
Next, we will discuss several applications, where the non-archimedean framework is always critical.

\vspace{0.1em}
\textbf{\textit{Examples of SYZ conjecture with singular fibers. }}
Recent work \cite{Yuan_local_SYZ,Yuan_conifold,Yuan_A_n} has provided many explicit examples of Theorem \ref{Main_theorem_thm}, where $\pi_0: X_0 \to B_0$ is the smooth part of a singular Lagrangian fibration $\pi: X \to B$. In these examples, the mirror analytic space $X_0^\vee$ embeds into the analytification of an algebraic variety anticipated in both mathematics and physics literature; the dual affinoid torus fibrations $\pi_0^\vee: X_0^\vee \to B_0$ admit explicit formulas that facilitate a natural analytic extension $\pi^\vee: X^\vee \to B$ over the singular region $\Delta=B\setminus B_0$, using the data from the singular part of $\pi$.
Drawing on \cite{hong2018immersed}, we also uncover that the dual singular $\pi^\vee$-fibers over $\Delta$ can be set-theoretically strictly larger than the Maurer-Cartan sets associated to singular Lagrangian $\pi$-fibers over $\Delta$ \cite{Yuan_local_SYZ}.
This supports our claim that the mirror analytic structure of $(X_0^\vee, \pi_0^\vee)$ typically extends beyond the traditional Maurer-Cartan framework.
There are also potential non-explicit examples to uncover. For instance, let $X$ be an elliptic K3 surface and $\pi:X\to B$ is a special Lagrangian fibration with 24 focus-focus singularities (cf. \cite{Gross_Wilson}). By applying Theorem \ref{Main_theorem_thm} to its smooth part $\pi_0$, we abstractly obtain a $\pi_0^\vee:X_0^\vee\to B_0$, and the local singular models in \cite{Yuan_local_SYZ,Yuan_A_n} may facilitate the singular extension $\pi^\vee: X^\vee\to B$. A mirror volume form is also anticipated to exist, but it needs further exploration.

\vspace{0.1em}
\textbf{\textit{Disk counting.}}
 The inclusion of the superpotential in Theorem \ref{Main_theorem_thm} plays a crucial role in performing many concrete calculations. 
 In \cite{Yuan_e.g._FamilyFloer}, classic enumerative geometry results are derived from studying the wall-crossing of $W^\vee$. The Gross's special Lagrangian fibration \cite{Gross_ex} has two chambers containing Clifford and Chekanov tori \cite{Chekanov1996LagrangianTI, eliashberg1997problem}, with the Clifford superpotential well-known \cite{Cho_Oh} and the Chekanov superpotential often unknown. The Chekanov superpotential can be determined by analyzing the gluing maps across the wall. This approach yields all one-pointed open Gromov-Witten invariants for Chekanov-type tori in smooth toric Fano compactifications of $\mathbb C^n$, such as $\mathbb {CP}^n$ or $\mathbb {CP}^r \times \mathbb {CP}^{n-r}$. 
 These retrieve the works of Auroux and Chekanov-Schlenk \cite[5.7 \& 5.12]{AuTDual}\cite{Chekanov_Schlenk}, and Pascaleff-Tonkonog \cite[Th 1.4]{PT_mutation}, showcasing similar results by distinct methods.

 \vspace{0.1em}
 
\textbf{\textit{Folklore conjecture for Landau-Ginzburg models.}}
A celebrated conjecture by Auroux, Seidel, and Kontsevich states that every critical value of the superpotential $W^\vee$ matches an eigenvalue in the quantum cohomology $QH^*(X)$ of $X$. This has been verified for monotone Lagrangians \cite{AuTDual}. In the presence of nontrivial Maslov-0 disks, this has also been recently confirmed in \cite{Yuan_I_FamilyFloer} under Assumption \ref{assumption_obstruction_ideal}.
The new challenge arises from the necessity to engage with minimal model $A_\infty$ algebras or "holomorphic pearly trees".
The usual unobstructedness (i.e. the existence of a weak bounding cochain) is inadequate for this, and we must employ the stronger unobstructedness as Assumption \ref{assumption_obstruction_ideal}.
This conclusion is further supported by explicit examples in \cite{Yuan_local_SYZ}.
The method in Theorem \ref{Main_theorem_thm} doesn't fully exploit ud-homotopy relations (Section \ref{ss_UD}), in contrast to the self-Floer cohomology with affinoid coefficients introduced in \cite{Yuan_c_1}, which extracts more information from these homotopy relations. All these approaches surpass the standard Maurer-Cartan framework and will be further developed in \cite{Yuan_affinoid_coeff}, allowing us to neutralize the impact of Maslov-0 disk obstructions up to affinoid algebra isomorphisms.

 \begin{figure}
 	  \vspace{-1em}
 	\centering
 	\includegraphics[width=3.5cm]{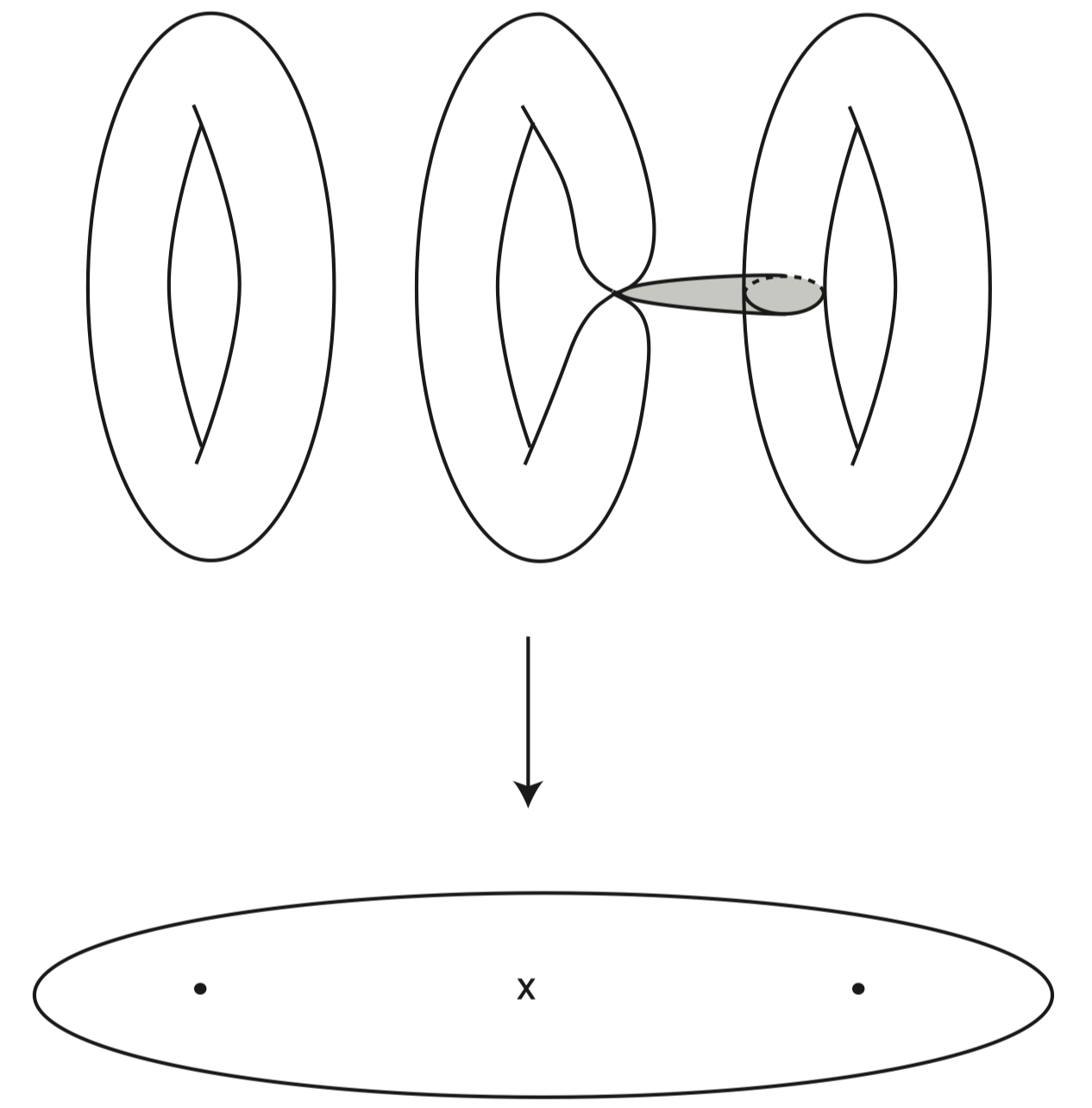}
 	\caption{
 		\footnotesize 
 		The shadowed disk represents the quantum correction (taken from \cite{MSClayII}).
 	}
 	\label{figure_intro_vanishing_cycle}
 	 \vspace{-1em}
 \end{figure}

\subsection{Outline}
\label{ss_sketch_of_proof}

This section offers a brief overview of Theorem \ref{Main_theorem_thm} and explains why the traditional Maurer-Cartan framework is inadequate for the global analytic structure.
A related result was claimed in \cite{Tu} in light of \cite{FuBerkeley,FuCyclic}, but the cocycle condition within the non-archimedean category appears to be unjustified.

\subsubsection{Motivating ideas}
\label{sss_motivating_ideas}
Pick a Lagrangian fiber $L$ and an $\omega$-tame almost complex structure $J$.
Given $\beta\in\pi_2(X,L)$, let $E(\beta)=\omega\cap \beta=\int_\beta \omega$ and $\mu(\beta)$ denote the energy and the Maslov index.
A filtered $A_\infty$ algebra (with topological labels; see \S \ref{ss_Gapped_A_infty}) associated to $L$ is a collection of multilinear maps of degree $2-k-\mu(\beta)$ on the de Rham complex
\[
\check \m_{k,\beta}: \Omega^*(L)^{\otimes k} \to \Omega^*(L)
\]
parameterized by $\beta\in\pi_2(X,L)$ and $k\geqslant 0$ such that
$\check \m_{k,0}=0$ for $k\neq 1,2$; $\check \m_{2,0}(h_1,h_2)= \pm h_1\wedge h_2$; $\check \m_{1,0}(h)=dh$; and the following $A_\infty$ associativity relation
\[
\textstyle \sum_{k_1+k_2=k+1} \sum_{\beta_1+\beta_2=\beta} \sum_{i=1}^{k_1+1} (-1)^\ast \check \m_{k_1,\beta_1}(h_1,\dots, \check \m_{k_2,\beta_2} (h_i, \dots, h_{i+k_2-1}) ,\dots, h_k)=0
\]
By Gromov's compactness, there are at most countably many $\beta$ such that $\m_{k,\beta}$ is nonzero.
Selecting the Hodge decomposition of $\Omega^*(L)$ with respect to a metric $g$ and applying the homological perturbation to the above operators, we can produce a new collection of multilinear maps of degree $2-k-\mu(\beta)$ on the de Rham cohomology
\[
\m_{k,\beta}: H^*(L)^{\otimes k} \to H^*(L)
\]
such that $\m_{1,0}=0$, $\m_{2,0}(h_1,h_2)= \pm h_1\wedge h_2$, and a similar $A_\infty$ relation holds.
Intuitively, these operators not only account for holomorphic disks but also include "holomorphic pearly trees" \cite{Sheridan15,FOOO_2009canonical_Morse}, which is also related to the "clusters" of \cite{cornea2006cluster} (see Figure \ref{figure_intro_homological_pert}).
Given that $\HL$ is finite-dimensional and $\OL$ is infinite-dimensional, it is extremely essential to prioritize the minimal model $A_\infty$ algebras $\m$ over the chain-level $\check\m$.

Define the Maurer-Cartan equation of $\m$ as
\begin{equation}
	\label{introduction_MC_eq}
	\textstyle
	\m_*(b):=
	\sum_\beta \sum_k T^{E(\beta)} \m_{k,\beta}(b,\dots,b)=0
\end{equation}
for $b\in H^1(L)\hat \otimes \Lambda_0$.
By degree reasons and by Assumption \ref{assumption-mu ge 0}, a class $\beta$ can contribute to (\ref{introduction_MC_eq}) only if $\mu(\beta)=0$ or $2$.
Given a basis $\{e_i\}$ of $H^1(L)$, we write $b=x_1e_1+\cdots+x_ne_n$ for $x_i\in \Lambda_0$.
Following \cite[\S 4]{FOOOToricOne}, the equation (\ref{introduction_MC_eq}) is decomposed into
a potential function in $x_i$:
\[
\textstyle
	\mathscr W^0 (x_1,\dots,x_n)
	:=
	\sum_{\mu(\beta)=2}  \sum_k T^{E(\beta)} \m_{k,\beta}(b,\dots,b)=0
\]
together with the weak Maurer-Cartan equation:
\begin{equation}
	\label{introduction_wMC_eq} 
	\textstyle
	\m_{*w}(b):=
	\sum_{\mu(\beta)=0}  \sum_k T^{E(\beta)} \m_{k,\beta}(b,\dots,b)=0
\end{equation}
A key property of $\m$ we use is the (open-string) \textbf{\textit{divisor axiom}}:
\begin{equation}
	\label{intro_DA_compute_eq}
	\textstyle
	\m_{k,\beta}(b,\dots,b)=\frac{(\partial\beta\cap b)^k}{k!} \ \m_{0,\beta} \ \in H^{2-\mu(\beta)}(L) \hat\otimes \Lambda_0
\end{equation}
Thus, the above series $\mathscr {W}^0 $ in $x_i$ can be transformed into a new series in $y_i=e^{x_i}$:
\begin{equation}
	\label{introduction_PO_y_eq} 
	\textstyle
	\mathscr {W}(y_1,\dots,y_n):= \sum_{\mu(\beta)=2} T^{E(\beta)} e^{\partial\beta \cap b}\m_{0,\beta} = \sum_{\mu(\beta)=2} T^{E(\beta)} y_1^{\partial_1\beta}\cdots y_n^{\partial_n \beta} \m_{0,\beta}
\end{equation}
We will forget $\mathscr W^0$ (the conventional Maurer-Cartan equation) and only consider $\mathscr W$ (the induced formal power series).
This is crucial for the non-archimedean analytic structure, without losing any information and ultimately enriching our insights. Note that a formal power series $f$ in $y_i$ is identically zero if and only if $f(y_1,\dots,y_n)=0$ whenever $y_i\in U_\Lambda$ (Lemma \ref{val=0-lem}). 
Besides, the Novikov field $\Lambda$ has the property that every element $y$ in $U_\Lambda$ can be written in the form $y=e^x$ for some $x\in \Lambda_0$ (Lemma \ref{exp-log-lem}). 
Thus, by the divisor axiom of $\m$, we get $\mathscr W^0(x_1,\dots,x_n)=\mathscr W(e^{x_1},\dots,e^{x_n})$, so we recover $\mathscr W^0$ from $\mathscr W$. This principle is applicable to (\ref{introduction_wMC_eq}) as well.

\subsubsection{Mirror local charts}
\label{sss_intro_local_chart}
As mentioned before, the tropicalization map $\trop:(\Lambda^*)^n\to\mathbb R^n$ resembles the logarithm map $\Log: (\mathbb C^*)^n\to\mathbb R^n$.
For a rational convex polyhedron $\Delta\subseteq \mathbb R^n$, the preimage $\Val^{-1}(\Delta)$ is an affinoid space $\Sp\Lambda\langle \Delta\rangle$ that is the spectrum of the \textit{polyhedral affinoid algebra}:
\[
\Lambda\langle \Delta\rangle = \left\{ \textstyle \ f=\sum_{\nu\in\mathbb Z^n} a_\nu \mathbf z^\nu \in \Lambda[[z_1^\pm,\dots, z_n^\pm ]] : f(y_1,\dots, y_n) \ \text{converges} \ \text{if} \  \trop(y_1,\dots, y_n) \in \Delta  \right\}
\]
Roughly, it serves as a basic building block in non-archimedean spaces, analogous to the role of polydisks in complex manifolds.
Alternatively, we note that $f\in\Lambda\langle\Delta\rangle$ if and only if $f(\mathbf y)$ converges at any point $\mathbf y$ in $\trop^{-1}(\Delta)$ for the aforementioned topology in $\Lambda$.
As a set, $\Sp\Lambda\langle \Delta\rangle$ is in bijection with the points in $\trop^{-1}(\Delta)$.
An analytic open domain $\trop^{-1}(U)$ is covered by various affinoid spaces $\trop^{-1}(\Delta_i)$ for $\Delta_i \subseteq U$ with $U=\bigcup \Delta_i$.
By definition, an \textit{{affinoid torus fibration}} is a continuous mapping $f: Y \rightarrow B$ with respect to the non-archimedean analytic topology on $Y$ and the manifold topology on $B$, locally resembling $\trop^{-1}(U) \rightarrow U$ for open subsets $U$ in $\mathbb R^n$ (see \ref{SA_non-archimedean} for more details).

For Theorem \ref{Main_theorem_thm}, the base $B_0$ gains an integral affine structure via the Arnold-Liouville theorem, resembling $\mathbb R^n$ locally, modulo the action of $GL(n,\mathbb Z)\ltimes \mathbb R^n$. A rational convex polyhedron $\Delta$ within an integral affine chart (notation $\Delta$ is reused) remains so in other integral affine charts as well.
Specifically, given a \textit{pointed} integral affine chart
\[
\chi: (U,q)\to (\mathbb R^n,c)
\]
centered at $q\in B_0$ with $\chi(q)=c$, suppose $\Delta\subseteq U$ such that $\chi(\Delta)$ is a rational convex polyhedron in $\mathbb R^n$. This also induces an isomorphism $\Lambda[[\pi_1(L_q)]]\cong \Lambda[[Y_1^\pm,\dots, Y_n^\pm]]$ where
\[
\Lambda[[\pi_1(L_q)]] =\left\{
\textstyle
\sum_{i=0}^\infty s_i Y^{\alpha_i}
\mid s_i\in\Lambda, \alpha_i\in \pi_1(L_q)
\right\}
\]
Define $\Lambda\langle \Delta, q\rangle$ to be the subalgebra in $\Lambda[[\pi_1(L_q)]]$ that is identified with the polyhedral affinoid algebra $\Lambda\langle \chi(\Delta) -c \rangle$ in $ \Lambda[[Y_1^\pm,\dots, Y_n^\pm]]$. The induced $\trop_q: \Sp\Lambda \langle \Delta, q\rangle \to \Delta$ is also identified with the restriction of $\trop$ over the subset $\chi(\Delta)-c $ in $\mathbb R^n$.

If there is no confusion, we often abuse the notations and write $\Delta=\chi(\Delta)-c$, $\trop=\trop_q$, etc.


\begin{figure}
	\centering
	\includegraphics[width=5.8cm]{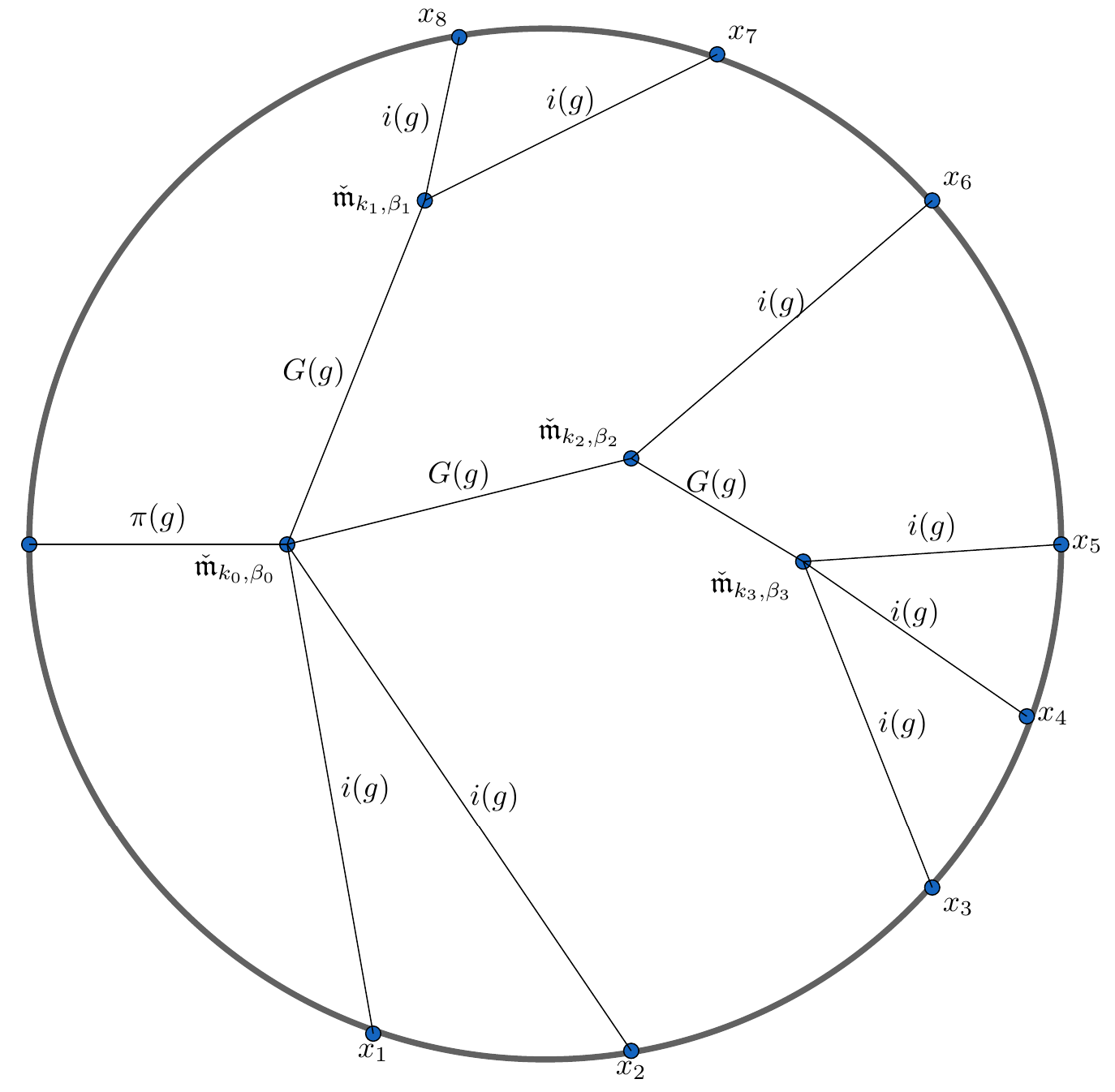}
	\caption{ \footnotesize
		Homological perturbation
	}
	\label{figure_intro_homological_pert}
\end{figure}

Let $J$ be an $\omega$-tame almost complex structure.
Fix a sufficiently fine rational convex polyhedron covering $\{\Delta_i\mid i\in \I\}$ of $B_0$. Pick points $q_i\in\Delta_i$, and take the Lagrangian fibers $L:= L_{q_i}= \pi^{-1}(q_i)$. For each $i\in\mathfrak I$, the moduli spaces of holomorphic disks can produce an $A_\infty$ algebra $\check \m:=\check\m^{J,L}$ on the de Rham complex of $L=L_{q_i}$.
To obtain a minimal\footnote{Here `minimal' means the energy-zero part of $\m_1=\sum T^{E(\beta)}\m_{1,\beta}$ is zero, i.e. $\m_{1,0}=0$, rather than $\m_1=0$.} $A_\infty$ algebra on the de Rham cohomology $H^*(L)$, we need to choose a \textit{contraction} to perform the \textit{homological perturbation} (\S \ref{S_homological_perturbation}).
Fix a metric $g$, and we always stick to the \textit{$g$-harmonic contraction} $\con(g)=(i(g),\pi(g), G(g))$.
It is roughly the data of a Hodge decomposition (see \S \ref{ss_harmonic_contraction}):
	\vspace{-0.7em}
\begin{equation}
	\label{introduction_harmonic_eq}
	\xymatrix{
		\con(g): &
		H^*(L) \ar@<.5ex>[rr]^{i(g)}	& & \Omega^*(L) 
		\ar@<.5ex>[ll]^{\pi(g)}
		\ar@(ul,ur)^{G(g)}
		& &
	}
\end{equation}

Denote by
$\m:=\m^{g,J,L}$ the minimal model $A_\infty$ algebra on $H^*(L)$ obtained by applying homological perturbation to $\check \m$ (Figure \ref{figure_intro_homological_pert}).
Each single $\m_{0,\beta}$ may involve various different Feynman-diagram-type trees with total class $\beta$.
The $\m$ is called a \textit{canonical model} of $\check \m$ (with respect to $g$).
In view of (\ref{union_mirror_intro}), we need to use $(\HL, \m)$ rather than $(\OL, \check \m)$ to carry out the mirror construction.
Consider
\begin{equation}
	\label{introduction_P_series_eq}
	P^i=\sum T^{E(\beta)} Y^{\partial\beta}\m_{0,\beta} \qquad \in \Lambda[[\pi_1(L)]] \hat\otimes H^*(L)
\end{equation}
One can interpret $P^i$ as a vector-valued formal power series in $\Lambda[[Y_1^\pm,\dots, Y_n^
\pm]]$ by choosing a basis. Restricting $P^i$ to $U_\Lambda^n\equiv \trop^{-1}(0)\cong \trop_{q_i}^{-1}(q_i)$ formally retrieves the Maurer-Cartan equation of $\m$, so $P^i$ actually contains more information.
It also enables the recovery of the Maurer-Cartan equation for the $A_\infty$ algebra associated to an adjacent Lagrangian fiber over any point in $\Delta_i$ by Fukaya's trick.
Since $\Delta_i$ is small, the $P^i$ converges on the polytopal domain $\trop^{-1}(\Delta_i)$ due to the \textit{reverse isoperimetric inequalities} \cite{ReverseI}. All the components of $P^i$ live in $\Lambda\langle \Delta_i,q_i\rangle$.

Let $\one$ be the constant-one function, and let $\{\Theta_1,\dots,\Theta_\ell\}$ be a basis of $H^2(L)$. As (\ref{introduction_MC_eq}), we write
\begin{equation}
	\label{W_i_Q_i_eq}
	P^i=W^i  \one +Q^i=W^i  \one+Q^i_1  \Theta_1+\cdots + Q^i_\ell \Theta_\ell
\end{equation}
$W^i$ and $Q^i_k$ refer to the terms with $\mu(\beta)=2$ and $\mu(\beta)=0$ respectively.
Let $\ia_i\subseteq \Lambda\langle \Delta_i, q_i\rangle$ be the ideal generated by the components $Q_k^i$ of $Q^i$, and we call it the \textit{\textbf{obstruction ideal}} (associated to $\m$).
Define:
\begin{equation}
	\label{intro_ia_i}
	X_i:=
	V(\ia_i) = \Sp  \big(\Lambda\langle\Delta_i, q_i\rangle / \ia_i \big)
\end{equation}
For Theorem \ref{Main_theorem_thm}, $X_i$ is a \textit{\textbf{local chart}} for the mirror space $X^\vee$, with $W^i$ representing a local piece of the mirror superpotential $W^\vee$; also, the dual fibration $\pi^\vee$ locally correspond to $\trop=\trop_{q_i}$ on $X_i$.

\subsubsection{Mirror transition maps}
\label{ss_Tu_formula}
We next develop \textit{transition maps} between these local charts. Roughly, a transition map is defined using the geometric input of a Lagrangian isotopy across fibers over a sufficiently small path. This isotopy can translate into a slight isotopy of the almost complex structures via {Fukaya's trick} (\S \ref{S_Fuk_trick}). Then, the algebraic output is a \textit{pseudo-isotopy}, leading to an \textit{$A_\infty$ homotopy equivalence}. We can also convert it to the level of minimal model $A_\infty$ algebras using a family version of harmonic contractions.
The story is further described as follows.

Let $L$ and $\tilde L$ be two adjacent fibers over $q, \tilde q\in B_0$. Choose a small isotopy $F$ so that $F(L)=\tilde L$. 
Denote by $\mathcal M(J,L)$ and $\mathcal M(J, \tilde L)$ the moduli spaces of $J$-holomorphic disks bounding $L$ and $\tilde L$.
Recall that they define the filtered $A_\infty$ algebras, denoted by
$\check \m:= \check \m^{J,L}$ and $\check \m^{\smallsim} := \check \m^{J,\tilde L}$, on $\Omega^*(L)$ and $\Omega^*(\tilde L)$ respectively. Denote the corresponding minimal model $A_\infty$ algebras on $\HL$ and $H^*(\tilde L)$ by
$\m:=\m^{g,J,L}$ and $\m^\smallsim:=\m^{g,J,\tilde L}$ respectively. They define local charts as (\ref{intro_ia_i}).

The \textit{Fukaya's trick} refers to the following.
A map $u:(\mathbb D,\partial\mathbb D)\to (X, L)$ is $J$-holomorphic if and only if
$F \circ u:(\mathbb D,\partial\mathbb D)\to (X, \tilde L)$ is $F_*J$-holomorphic.
So, for $F_*J:=dF\circ J\circ dF^{-1}$, there is a natural identification
$\mathcal M(F_*J,\tilde L) \cong \mathcal M(J,L)$;
the left side gives a new $A_\infty$ algebra, denoted by $\check \m^F:= \check \m^{F_*J, \tilde L}$.
It can be interpreted as a `pushforward' of $\check \m^{J,L}$ for the following Fukaya's trick relation
\begin{equation}
	\label{introduction_Fuk_chain_eq}
	\check \m^F\equiv (F^{-1})^* \diamond \check \m \diamond F^*
\end{equation}
This means $\check \m^F_{k,\beta}(x_1,\dots, x_k)= (F^{-1})^* \check \m_{k,\beta}(F^*x_1,\dots, F^* x_k)$ for $x_i\in\Omega^*(\tilde L)$.
The $A_\infty$ algebras $\check \m^F$ and $\check \m$ are nearly the same, with the distinction in energy given by $E(F_*\beta) = E(\beta) + \langle\partial\beta, \tilde q-q\rangle$ for $\beta = [u] \in \pi_2(X,L)$.
The purpose of using $\check \m^F$ in place of $\check \m$ is to unify the underlying Lagrangians.
Next, we take a path $\mJ=(J_s)$ from $F_*J$ to $J$. Then, the \textit{parameterized moduli space}
$\bigsqcup_s\mathcal M(J_s,\tilde L)$
generates the so-called \textit{pseudo-isotopy} $\check \M$. It is roughly a family of $A_\infty$ algebras, moving from $\check \m^{F}$ to $\check \m^\smallsim$, enriched with additional `derivative' data (Definition \ref{Classical-Pseudo-isotopy-defn}).

Further, we aim to obtain a relationship between the two minimal model $A_\infty$ algebras $\m$ and $\m^\smallsim$.
Take the two distinct harmonic contractions $\con(g)$ and $\con(F_*g)$ in the meantime.
By homological perturbation, we use $\con(g)$ and $\check \m$ (resp. $\check \m^\smallsim$) to obtain the minimal model $\m$ (resp. $\m^\smallsim$). Also, we can similarly use $\con(F_*g)$ and $\check  \m^F$ to obtain another minimal model
$
\m^F:=  \m^{F_*(g,J), \tilde L}
$.
It also satisfies
\begin{equation}
	\label{introduction_Fuk_cohomology_eq}
	\m^F \equiv (F^{-1})^* \diamond \m \diamond F^*
\end{equation}
but the underlying vector space changes from $\OL$ to $\HL$.
This holds due to the $F$-relatedness conditions, including both (\ref{introduction_Fuk_chain_eq}) and $\con(F_*g)=F^{-1*} \circ \con(g)\circ F^*$ (see also Figure \ref{figure_intro_homological_pert}).
Unlike the previous relation (\ref{introduction_Fuk_chain_eq}), only the homotopy class of $F$ matters in (\ref{introduction_Fuk_cohomology_eq}). This sort of flexibility will become crucial for the well-definedness of transition maps and the cocycle conditions.

Finally, our goal is to develop, from the earlier $\check \M$, a pseudo-isotopy $\M$ at the level of minimal model $A_\infty$ algebras, bridging $\m^F$ and $\m^\smallsim$.
To construct it, we carry out the homological perturbation again, but we need to further find a path of metrics $\mg=(g_s)$ between $F_*g$ and $g$.
The path always exists since the metric is also a contractible choice. Utilizing a family version of harmonic contraction, we then derive the pseudo-isotopy $\M$ (\S \ref{ss_pseudo_isotopy_canonical_model}).
This method for generating pseudo-isotopies of minimal model $A_\infty$ algebras is not applicable to 2-pseudo-isotopies when 1-pseudo-isotopy boundary conditions are specified in advance.
This presents a substantial challenge in demonstrating the cocycle condition. We contend that prior research has not sufficiently addressed this issue.
Intuitively, the roles of the data $\mJ$ and $\mg$ can be described as follows:
\vspace{-0.35em}
\[
\xymatrix{
	\boxed{\mathcal M(J,L)\cong \mathcal M(F_*J, \tilde L) \xrightarrow{\mathbf J}  \mathcal M(J,\tilde L)}
	\ar@{~>}[r]
	&
	\boxed{  \check \m \approx \check \m^{F} \xleftarrow{\check \M} \check \m^{\smallsim} }
	\ar@{~>}[r]^{\mg \quad }
	&
	\boxed{   \m \approx  \m^F \xleftarrow{\M}  \m^{\smallsim} }
}
\]


A pseudo-isotopy naturally generates an $A_\infty$ homotopy equivalence in a canonical manner (Theorem \ref{from-pseudo-isotopies-to-A-infty-homo-thm}).
From the above $\M$, we denote the outcome by
$\mC^F:  \m^\smallsim \to \m^F$.
It consists of a collection of operators $\mC_{k,\beta}^F: H^*(\tilde L)^{\otimes k} \to H^*(\tilde L)$ for $k\ge 0$, $\beta\in\pi_2(X,\tilde L)$ with $\mC^F_{0,0}=0$, $\mC^F_{1,0}=\id$, and the equations of $A_\infty$ homomorphisms.
As the category of affinoid spaces is equivalent to the opposite category of affinoid algebras, we first consider
an affinoid algebra homomorphism:
\begin{equation}
	\label{Tu_formula_eq}
	\textstyle
	\phi^F: \Lambda\langle   \Delta ,   q \rangle\to\Lambda\langle \Delta , \tilde q \rangle 
	\qquad
	Y^{\alpha}\mapsto T^{\langle \alpha, \tilde q-q\rangle} Y^{\tilde  \alpha} \exp 
	\Big\langle
	\tilde \alpha,  \sum_{\beta} T^{E(\beta)}
	\mC_{0,\beta}^F  Y^{\partial \beta}
	\Big\rangle 
\end{equation}
where $\alpha \in \pi_1(L)\cong \mathbb Z^n$, $\tilde\alpha:=F_*\alpha\in \pi_1(\tilde L)\cong \mathbb Z^n$, and $\langle\cdot,\cdot\rangle$ is the natural pairing $\pi_1(L)\otimes H^1(L)\to\mathbb R$.
In practice, $\Delta$ is the intersection of two rational convex polyhedron $\Delta_i$'s in \S \ref{sss_intro_local_chart}.
This formula is first discovered by J. Tu \cite[(4.11)]{Tu}. Roughly, the idea involves employing the coordinate change $y = e^{x}$, as mentioned in \S \ref{sss_motivating_ideas}, to integrate the result of \cite[4.3.14]{FOOOBookOne}.
However, the coordinate transformation $y = e^x$ is quite inconsistent with the non-archimedean structure.
The key problem should be that the Maurer-Cartan functor \cite[4.3.13]{FOOOBookOne} only addresses set-theoretic bijections; cf. \cite[Lemma 4.6, 4.9]{Tu}.
It appears there was even no justification for whether (\ref{Tu_formula_eq}) is actually an isomorphism of affinoid algebras.

More concretely, we take two \textit{pointed} integral affine charts as before, say $\chi :(U, q)\to (\mathbb R^n, c)$ and $\tilde \chi: (\tilde U, \tilde q)\to (\mathbb R^n, \tilde c)$, where both $U$ and $\tilde U$ contain $\Delta\subset B_0$. The $\phi^F$ is regarded as a non-archimedean analytic map between the polytope affinoid domains $\trop^{-1}(\chi(\Delta)-c)$ and $\trop^{-1}(\tilde \chi (\Delta)-\tilde c)$, both within $(\Lambda^*)^n$.
Nevertheless, we will frequently use the intrinsic description in (\ref{Tu_formula_eq}).
Given $\mC_{0,\beta}^F\in H^{1-\mu(\beta)}(\tilde L)$, under the semipositive assumption, the contribution to the affinoid algebra homomorphism $\phi^F$ comes merely from Maslov-0 disks; cf. \cite{AuTDual}.


\subsubsection{Mirror cocycle condition}
A critical issue is the dependency of the algebra homomorphism $\phi^F$ on various choices. To validate the cocycle conditions, showing that the transition map is well-defined is essential.
Fortunately, this issue can be addressed through the following points, which will be futher explained shortly:
\begin{itemize}
	\itemsep 0.2em
	\item [\textbf{(A)}] What we need is an affinoid algebra morphism $\varphi: \Lambda \langle  \Delta ,  q \rangle /  \ia \to \Lambda\langle \tilde\Delta, \tilde q \rangle/ \tilde \ia  $
	with $\varphi( W)= \tilde W$.
	\item [\textbf{(B)}] The $\varphi$ depends only on the \textit{ud-homotopy class} of $\mC^F$, so finally it does \textit{not}
	depend on the choices.
\end{itemize}

\vspace{0.25em}

Here we define $P,W,Q,\ia$ and $\tilde P,\tilde W,\tilde Q,\tilde \ia$ for the $A_\infty$ algebras $\m$ and $\m^\smallsim$ as before (\S \ref{sss_intro_local_chart}).
The \textit{{ud-homotopy}} refers to an amended homotopy theory of $A_\infty$ algebras with extra properties; the prefix `u' and `d' means `unitality' and `divisor axiom' respectively.
It is important to note that unless we put `u' and `d' together, it would be problematic to make a well-behaved homotopy theory.
Systematically, we will define a category $\UD$ of these $A_\infty$ algebras to work with the ud-homotopy relations. There are some heavy homological algebras, notably \S \ref{S_whitehead}. 
The non-archimedean analytic gluing forces us to study affinoid algebra morphisms.
The analytic structure goes beyond mere "homotopy invariance of Maurer-Cartan sets", and should be approximated as "ud-homotopy relations for $A_\infty$ maps within minimal model $A_\infty$ algebras".

Suppose $L=L_k$ and $\tilde L=L_j$, and we define the local charts as in (\ref{intro_ia_i}). Then, we define
\[
\psi_{jk}: X_j \to X_k
\]
to be the induced map $\varphi^*$
on the maximal ideal spectrums for the quotient morphism $\varphi=[\phi^F]$.
In other words, we set $\psi_{jk}=\varphi^*: \Sp(\Lambda\langle \tilde\Delta, \tilde q \rangle/ \tilde \ia) \to \Sp (\Lambda \langle  \Delta ,  q \rangle /  \ia)$.
We call the $\psi_{jk}$ the \textit{{gluing map}} or the \textit{{transition map}}.
The original homomorphism $\phi^F$ may be called a pre-transition map.

Point (A) ensures the transition map is well-defined, while point (B) resolves ambiguity from various choices, ultimately leading to the cocycle conditions. 
The key innovations for the proofs introduced in this paper are the wall crossing formula and the ud-homotopy theory respectively.

Achieving (A) will be straightforward using the wall crossing formula below, which will be detailed in Theorem \ref{wall_crossing_thm}.
In brief, it states that for any $\eta\in H_*(L)$, there exists $H_2(L)$-valued formal power series $R^\eta$, recalling $\tilde Q$ is $H^2(L)$-valued formal power series (see (\ref{W_i_Q_i_eq})), such that
\begin{equation}
	\label{introduction_wall_crossing_formula_eq}
	\textstyle
	\phi^F (\langle\eta, P\rangle)=
	\langle F_*\eta, \one \rangle \  \tilde W+   \langle R^\eta ,  \tilde Q \rangle
\end{equation}
where $\langle \cdot,\cdot\rangle$ is the natural pairing, and $\one$ denotes the constant-one function.
\begin{enumerate}[(i)]
	\itemsep 0.2em
	\item If $\eta$ is dual to $\one$, then $\langle F_*\eta,\one\rangle=1$ and $\langle \eta, P\rangle =W$. So, $\phi^F (W)$ equals to $\tilde W$ modulo $\tilde\ia$. 
	\item If $\eta$ runs over $H_{n-2}(L)$, then we see $\langle F_*\eta,\one\rangle=0$, and the $\langle\eta,P\rangle$ runs over all components of $Q$, i.e. the generators of the ideal $\ia$. As $\langle R^\eta ,  \tilde Q \rangle \in \tilde \ia$, we see $\phi^F(\ia) \subset \tilde \ia$.
\end{enumerate}
\vspace{0.25em}

Grasping point (B) demands extra effort and further extends the classical Maurer-Cartan principle in \cite[\S 4.3.2]{FOOOBookOne}. Previously, it was established that an $A_\infty$ homotopy equivalence leads to a bijection among gauge equivalence classes of bounding cochains. But, this is purely set-theoretical and requires additional refinement to be applicable in the non-archimedean context.

With a different choice $F'$, we can similarly define $\mC^{F'}$ and $\phi^{F'}$. But, $\phi^{F'}$ typically differs from the previously defined $\phi^F$ in (\ref{Tu_formula_eq}). This discrepancy was not explored in any previous studies.
Specifically, we write
\[
\phi^{F'}(Y^\alpha)=\phi^F(Y^\alpha) \cdot \exp \Big\langle 
\alpha,
\sum
T^{E(\beta)}  \ (\mC_{0,\beta}^{F'}-\mC_{0,\beta}^F) \ Y^{\partial\beta}
\Big\rangle
\]
for any $\alpha\in\pi_1(L)$. Thus, the crucial factor is the variation in power, termed the \textit{error term}:
\[
S:=
\langle \alpha,
\textstyle \sum
T^{E(\beta)}  \ (\mC_{0,\beta}^{F'}-\mC_{0,\beta}^F) \ Y^{\partial\beta} \rangle
\]
To demonstrate that $\phi^F$ and $\phi^{F'}$ yield the same transition maps $\varphi=[\phi^F]=[\phi^{F'}]$, it is essential to verify that $S$ belongs to the obstruction ideal $\tilde \ia$. 
The main idea is that the ud-homotopy theory introduces the \textit{canceling trick}: If two $A_\infty$ homomorphisms $\f^0=(\f^0_{k,\beta})$ and $\f^1=(\f^1_{k,\beta})$ between minimal model $A_\infty$ algebras are ud-homotopic, then the formal power series
$
\textstyle
\sum T^{E(\beta)} Y^{\partial\beta} (\f^1_{0,\beta}-\f^0_{0,\beta})
$
must be contained in the obstruction ideal.
It is essential to focus on the minimal model $A_\infty$ algebras $\m$ instead of the chain-level $\check\m$, because $\HL$ is finite-dimensional, whereas $\OL$ is infinite-dimensional.
But unfortunately, achieving 2-pseudo-isotopies in minimal model $A_\infty$ algebras and simultaneously maintaining 1-pseudo-isotopy boundary conditions is usually impossible; cf. \cite[Th 2.7]{Tu}. 
To resolve this, we observe that pseudo-isotopy is generally a stronger condition than an $A_\infty$ homotopy equivalence.
In light of this, we utilize diagram chasing for ud-homotopy relations as a strategy to avoid the need for 2-pseudo-isotopies in minimal models, and these ud-homotopy relations are sufficient for the above canceling trick.

%
%

\section{$A_\infty$ structures in Lagrangian Floer theory}
\label{S_preparation}

\subsection{$A_\infty$ structures with topological labels}
\label{ss_Gapped_A_infty}

\subsubsection{Definition}
We introduce the various notions for $ A_\infty$ structures. 
A \textit{label group} is defined as a triple $\G=(\G,E,\mu)$ that consists of an abelian group $\G$ and two group homomorphisms $E:\G\to \mathbb R$, $\mu: \G\to 2\mathbb Z$. Denote by $0$ the \textit{unit} of $\G$.
In practice, the label group for a symplectic manifold $(X,\omega)$ and a compact oriented Lagrangian submanifold $L\subset X$ is defined via the Hurewicz homomorphism image $\pi_2(X,L)= \mathrm{im} \left(\pi_2(X,L, \ast)\to H_2(X,L)\right)$, with a base point $\ast$ in $L$.
It includes the energy map $E: \pi_2(X,L) \to \mathbb R$, $\beta\mapsto \int_\beta \omega$, and the Maslov index $\mu: \pi_2(X,L)\to \mathbb Z$, where $L$'s orientation ensures $\mu$ maps into $2\mathbb Z$.
Define
$
\pi_1(L)=\mathrm{im} \ (\pi_1(L, \ast )\to H_1(L))
$
to be the image of the abelianization map of the fundamental group.
The boundary operator is denoted by
$\partial: \pi_2(X,L)\to \pi_1(L)$.

We frequently consider the case $\G=(\pi_2(X,L), E, \mu)$. Let's turn to the abstract definitions.

\begin{defn}
\label{CC_defn}
Let $\C$, $\C'$ be two graded vector spaces over $\mathbb R$. Let $\G$ be a label group. Given $k\in\mathbb N$ and $\beta\in \G$, we define
$
\CC_{k,\beta}(\C',\C):=\Hom({\C'}^{\otimes k}, \C)$, where $\Hom$ refers to the category of vector spaces.
If $k=0$, we define $\Hom(\mathbb R, \C)\cong C$.
Taking a different label $\beta\in \G$ just gives a copy of the same space.
Given $\beta\in\G$, we define
$
\CC_\beta(\C',\C):= \prod_{k\ge 0} \CC_{k,\beta}(\C',\C)$. Finally, we define
\[
\CC(\C',\C):=\CC_\G(\C',\C)
\]
to be the subspace of $\prod_{\beta\in\G} \CC_{\beta}(\C',\C)=\prod_{k,\beta} \CC_{k,\beta}(\C',\C)$ consisting of $\mathfrak t=(\mathfrak t_{k,\beta})$ satisfying the following \textbf{gappedness condition}:
\begin{enumerate}
	\itemsep 0pt
	\item[(a)] $\mathfrak t_{0,0}=0$. 
	\item[(b)] If $E(\beta)<0$ or $E(\beta)=0$ but $\beta\neq 0$, then we require $\mathfrak t_\beta=(\mathfrak t_{k,\beta})_{k\in\mathbb N}$ vanish identically.
	\item[(c)] For any $E_0>0$, there are only finitely many $\beta\in \G$ such that $\mathfrak t_\beta\neq 0$ and $E(\beta)\le E_0$.
\end{enumerate}
For simplicity, we often omit $C$, $C'$, or $\G$ and use notations like $\CC_{k,\beta}$, $\CC_\beta$, $\CC_\G$, $\CC$, etc., when the context is clear.
An element in $\CC=\CC(C',C)$ is called an \textit{operator system}.
Specifically, elements in $\CC_\G$ are already considered `gapped' in the sense of the literature \cite{FOOOBookOne}.
\end{defn}

\begin{rmk}\label{Induction}
	The geometric idea behind (a) and (b) is that a non-constant pseudo-holomorphic curve has positive energy; besides, the condition (c) corresponds to the Gromov compactness.
	Provided the gappedness condition, we can do induction on the pairs $(k,\beta)$ with $\mathfrak t_{k,\beta}\neq 0$. Indeed, we can introduce an order on the set of all such pairs: $(k',\beta')<(k,\beta)$ if either $E(\beta')<E(\beta)$ or $E(\beta')=E(\beta)$, $k'<k$. The gappedness also tells that there are at most countably many $\beta$'s involved.
\end{rmk}

Denote by $\id$ the identity operator $x\mapsto x$. To compute signs, we introduce the `twisted' identity $\id_\#=\id^\#$ defined by $x\mapsto (-1)^{\deg x-1} x$. We also put
$
x^\# =\id_\# (x) =(-1)^{\deg x-1} x
$.
Given a multi $k$-linear operator $\phi$, we have
$
(-1)^{\deg \phi+k-1}\id_\#  \phi =  \phi  (\id_\#)^{\otimes k}
$.
For the \textit{shifted degree}
$
\deg' x= \deg x-1
$,
we have $\deg' \phi= \deg \phi -1+k$.
Given $p\in\mathbb N$, we put
$
	\phi^{\# p}:= \phi  (\id_{\# p})^{\otimes k}
$
	where $\id_{\# p}$ denotes the $p$-iteration $\id_\# \circ \cdots \circ \id_\#$ of $\id_\#$. It is straightforward to obtain that
\begin{equation}
\label{sign-p-twist-eq}
 \phi^\#=(-1)^{\deg'\phi} \id_\# \circ  \phi, \qquad
\phi^{\# p} = (-1)^{\deg' \phi} \circ \id_\# \circ \phi^{\# (p-1)}= \cdots = (-1)^{p \ \deg' \phi} \id_{\# p} \circ \phi
\end{equation}

\begin{defn} \label{Composition-defn}
	Fix operator systems $\f=(\f_{k,\beta})\in \CC(\C'',\C')$, $\g=(\g_{k,\beta})\in \CC(\C',\C)$ and $\h=(\h_{k,\beta})\in \CC(\C)$. The \textbf{composition} $\g \diamond \f \in \CC(\C'',\C)$ of $\f$ and $\g$ is the operator system defined by
\begin{equation*}
(\g\diamond \f)_{k,\beta}=
\sum_{\ell \ge 1}
\sum_{k_1+\dots+k_\ell=k}
\sum_{\beta_0+ \beta_1+\cdots +\beta_\ell=\beta}
\g_{\ell,\beta_0} 
  ( \f_{k_1,\beta_1}\otimes \cdots \otimes \f_{k_\ell,\beta_\ell} )
\end{equation*}
The \textbf{Gerstenhaber product} $\g \{ \h\} \in \CC(\C',\C)$ is defined by the following operators
\begin{equation*}
(\g\{ \h \})_{k,\beta} = \sum_{\lambda+\mu+\nu=k}\sum_{\beta'+\beta''=\beta} \g_{\lambda+\mu+1,\beta'}   (\id_{\#\deg'\h}^\lambda \otimes \h_{\nu,\beta''}\otimes \id^\mu)
\end{equation*}
The gappedness conditions ensure that the summations in the above definition are all finite. The resulting operator systems $\g\diamond \f$ and $\g\{\h\}$ are gapped as well.
Notice also that $\g_3\diamond (\g_2\diamond  \g_1) = (\g_3\diamond \g_2)\diamond \g_1$.
\end{defn}

%
%

\begin{defn}
		\label{Gapped-Hom-defn}
An \textbf{$A_\infty$ algebra} is a graded $\mathbb R$-vector space $\C$ equipped with an operator system $\m=(\m_{k,\beta})\in \CC_\G(\C,\C)$ such that $\deg \m_{k,\beta}=2-k-\mu(\beta)$ and the following $A_\infty$ associativity relation holds: for any $(k,\beta)$,
\[
			\sum_{\beta'+\beta''=\beta} \sum_{\lambda+\mu+\nu=k} \m_{\lambda+\mu+1,\beta'}   (\id_\#^\lambda \otimes \m_{\nu,\beta''}\otimes \id^\mu)=0
			\]
In short, the relation can be written as $\m \{ \m\}=0$. We call $(C,\m)$ \textit{minimal} if $\m_{1,0}=0$. 	Let $(\C,\m)$, $(\C',\m')$ be two $A_\infty$ algebras.
An \textbf{$A_\infty$ homomorphism} from $\m'$ to $\m$ is an operator system $\f=(\f_{k,\beta})\in \CC_\G(\C',\C)$
such that $\deg \f_{k,\beta}=1-k-\mu(\beta)$ and the $A_\infty$ relation $\m \diamond \f=\f \{ \m' \}$ holds: for any $(k,\beta)$,
\begin{align*}
	\sum_{\ell \ge 1} \sum_{\substack{
			0=j_0\le\cdots\le j_\ell=k
			\\
			\beta_0+\beta_1+\cdots +\beta_\ell=\beta
	}}
	\m_{\ell,\beta_0}  (\f_{j_1-j_0,\beta_1}\otimes \cdots \otimes \f_{j_\ell-j_{\ell-1},\beta_\ell}) = \sum_{\substack{\lambda+\mu+\nu=k \\ \beta'+\beta''=\beta}} \f_{\lambda+\mu+1,\beta'}  (\id_\#^\lambda\otimes \m'_{\nu,\beta''} \otimes \id^\mu)
	\end{align*}

\end{defn}

To highlight the role of the label group $\G$, we may describe $\m$ or $\f$ as \textbf{gapped}, \textit{$\G$-gapped}, and so on.
The above summations are all finite by the gappedness conditions.
If $\f:\m'\to\m$ and $\g:\m''\to\m'$ are two $A_\infty$ homomorphisms,
then $\f\diamond \g$ can be defined first as a system of operators by Definition \ref{Composition-defn}, and one can check $\f\diamond \g$ is indeed an $A_\infty$ homomorphism.
Besides, if $\m$ is a gapped $A_\infty$ algebra and $\f$ is a gapped $A_\infty$ homomorphism, then their shifted degrees are $\deg'\m=1-\mu(\cdot)\equiv 1 \ (\mathrm{mod}  \ 2)$ and $\deg' \f=-\mu(\cdot)\equiv 0 \ (\mathrm{mod} \ 2)$. So, $\m^\#=-\id_\# \circ \m$ and $\f^\#=\id_\# \circ \f$.

From now on, we will often omit mentioning the label group $\G$; we write $(\C,\m)$, $\C$ or $\m$ for an $A_\infty$ algebra; we write $\f: \m'\to \m$ or $\f: (C', \m')\to (C,\m)$ for an $A_\infty$ homomorphism.
In literature, monoids are typically used over our label groups (c.f. \cite{FOOOBookOne}); however, we take a different approach. Note that for an $A_\infty$ algebra $\m$ associated to a Lagrangian $L$, the subset $\{\beta \mid \m_\beta\neq 0\}$ creates a monoid within $\pi_2(X,L)$. Given that these monoids may vary based on choices but are all within the same group, we opt to use the label group as above, including all these monoids implicitly.

\begin{defn}
\label{reduction-defn}
An $A_\infty$ algebra $(C,\m)$ gives a cochain complex $(C,\m_{1,0})$.
By forgetting all $\m_{k,\beta}$ with $\beta\neq 0$, it gives an $A_\infty$ algebra $(C,\bar\m)$, called the \textbf{reduction} of $(\C,\m)$.
An $A_\infty$ homomorphism $\f: (C,\m)\to (C',\m')$ also gives a cochain map $\f_{1,0}: (C,\m_{1,0}) \to (C',\m_{1,0}')$. We call $\f$ an \textbf{$A_\infty$ homotopy equivalence} if $\f_{1,0}$ is a quasi-isomorphism. 
Similarly, by forgetting all $\f_{k,\beta}$ with $\beta\neq 0$, it gives an $A_\infty$ homomorphism $\bar\f$ from $\bar \m$ to $\bar \m'$; we call $\bar \f$ the \textbf{reduction} of $\f$.
On the other hand, we define the following \textit{incomplete and partial} $A_\infty$ conditions for later use.
\begin{itemize}
	\itemsep 2pt
\item[(a1)] An operator system $\bar \m=(\m_{j,0})_{1\le j\le k} \in \CC_0(\C,\C)$ is called an \textbf{$A_k$ algebra (modulo $T^{E=0}$)}
 if $\bar\m \{ \bar\m \}|_{\CC_{j,0}}=0$ and $\deg \m_{j,0}=2-j$ for $1\le j\le k$. 
 \item[(a2)]
 An operator system $\m =(\m_{k,\beta})_{k\ge 0, E(\beta)\le E} \in \CC(\C,\C)$ is called an \textbf{$A_\infty$ algebra modulo $T^E$} 
 if $\m\{ \m\}|_{\CC_{k,\beta}}=0$ and $\deg \m_{k,\beta}=2-k-\mu(\beta)$ for $k\in\mathbb N$ and $\beta$ with $E(\beta)< E$.

\item[(b1)] Let $(\C,\m)$ and $(\C',\m')$ be two $A_\infty$ algebras and let $\bar\m$ and $\bar\m'$ be their reductions. An operator system $\bar\f = (\f_{j,0})_{1\le j \le k} \in \CC_0(\C',\C)$ is called an \textbf{$A_k$ homomorphism (modulo $T^{E=0}$)} if $(\bar \m \diamond \bar\f - \bar\f \{ \bar\m' \} )|_{\CC_{j,0}}=0$ and $\deg \f_{j,0}=1-j$ for $1\le j \le k$. 
\item[(b2)]
An operator system $\f\in \CC(\C',\C)$ is called an \textbf{$A_\infty$ homomorphism modulo $T^E$} if $(\m\diamond \f - \f \{ \m' \} )|_{\CC_{k,\beta}}=0$ and $\deg \f_{k,\beta}=2-k-\mu(\beta)$ for all $k\in\mathbb N$ and $\beta$ with $E(\beta)< E$.
\item[(b3)] Given $\B\in\G$, an operator system $\f\in \CC(\C',\C)$ is called an \textbf{$A_{\infty,\B}$ homomorphism}
if it is an $A_\infty$ homomorphism modulo $T^{E(\B)}$ and moreover $(\m\diamond \f-\f \{ \m'\})|_{\CC_{\B}}=0$.
\end{itemize}
\end{defn}

\subsubsection{Pseudo-isotopies}
\label{S_P_pseudo_isotopy}

From now now, we always assume that $C$ is a direct sum of the spaces of differential forms on several smooth closed manifolds.
Let $P$ be a bounded convex polytope inside an Euclidean space.
The definition of smooth functions on a manifold with corner is delicate. For our purpose, we require such a function to be collared\footnote{This implies the function appears linear near corners and stays constant near boundaries away from corners; cf. \cite[Fig. 17.3]{FOOOKuranishiVFC}.} near the boundaries and corners.
For simplicity, we'll keep this point implicit and continue using terms like $C^\infty(P)$, $\Omega^*(P)$, etc. Ultimately, we only deal with simple cases where $P$ is a 1-simplex or 2-simplex. For a 1-simplex $P=[0,1]$, the collared condition merely requires functions to be linear near both endpoints.

Denote by $C^\infty(P,\C)$ the set of smooth maps from $P$ to $\C$. For instance, we may assume $C=\Omega^*(L)$, and an element is a family $x_s$ of differential forms on $L$ that smoothly depends on $s\in P$.
Now, define
$
 \C_P:=\Omega^*(P)\otimes_{C^\infty(P)} C^\infty(P,\C)
$.
For instance, if $C=\Omega^*(L)$, then there is a natural identification
\begin{equation}\label{convention-Omega P L-eq}
	C_P\cong \Omega^*(P\times L) 
\end{equation}
Namely, any element in $\C_P$ is a linear combination of elements in the form $\eta \otimes x_s$ with $\eta=\eta(s)$ in $\Omega^*(P)$ and $x_s$ is a smooth $P$-family of differential forms on $L$. Then, we identify $\eta\otimes x_s$ with a differential form $\eta(s)\wedge x_s$ on $P\times L$.
Besides, a grading on $\C$ induces a one on $C^\infty(P,\C)$; further, there is a natural \textit{bi-grading} on $\C_P$: an element $\eta\otimes x(\cdot)$ is of degree $(p,q)$ if $\eta\in \Omega^p(P)$ and $x(\cdot)\in C^\infty(P,\C^q)$.

The following two kinds of maps interest us:
\begin{align*}
	\incl:  \ & \ \C \to \C_P \\
	\eval^\s:  \ & \  \C_P\to \C
\end{align*}
where the first map sends $x_0$ to $1\otimes x_0$ and the second map is a family of evaluation maps
parameterized by $\s\in P$, sending $1\otimes x(\cdot)$ to $x(\s)$ and $\eta\otimes x(\cdot)$ to zero for $\eta\in \Omega^{>0}(P)$.
In the case of (\ref{convention-Omega P L-eq}), the $\incl$ and $\eval^s$ can be recognized as the pullbacks $\pr^*$ and $\iota_s^*$ respectively for the projection $\pr: P\times L\to L$ and the inclusions $\iota_s: L\to \{s\}\times L\subset P\times L$.
Since $\iota_s^* \circ \pr^*=\id$, one can easily check that
$\eval^\s \circ \incl = \id
$.
By \cite[Remark 21.28]{FOOOKuranishiVFC} or \cite[\S 4.2.1]{Solomon_Diff_survey}, we make the following definition. It roughly says the $\Omega^*(P)$-linearity up to sign.
For $I=(i_1,\dots,i_r)$, we denote $ds_I=ds_{i_1}\wedge \cdots \wedge ds_{i_r}$.

\begin{defn}\label{pointwise-defn}
A multi-linear operator $\M: \C_P^{\otimes k} \to \C'_P$ is said to be \textbf{$P$-pointwise} (or simply \textbf{pointwise}) if we have the following \textit{signed linearity condition}: for any $\sigma\in \Omega^*(P)$ we have
\[
\M(\eta_1 \otimes x_1,\dots, \sigma\wedge \eta_i\otimes x_i,\dots, \eta_k \otimes x_k ) = (-1)^{\dagger} \sigma \wedge \M
(\eta_1\otimes x_1,\dots, \eta_k\otimes x_k)
\]
where $\dagger=\deg \sigma \cdot  ( (\deg \M-1+k)
+\sum_{a=1}^{i-1} (\deg \eta_a+\deg x_a -1)  )$.
Alternatively, this means
\[
\M (y_1,\dots,y_{i-1}, \sigma\wedge y_i ,y_{i+1}, \dots, y_k) = (-1)^{\deg \sigma \cdot \deg'\M} \sigma \wedge \M ( y_1^{\#\deg \sigma},\dots, y_{i-1}^{\#\deg \sigma}, y_i, y_{i+1},\dots, y_k)
\]
\end{defn}

\begin{rmk}
	\label{pointwise_determine_rmk}
If an operator $\M$ is pointwise, it suffices to consider the inputs in the form of $1\otimes x$ in order to determine it.
Namely, we can find a family of operators $\M_s^I: C^{\otimes k}\to C'$ for any point $s\in P$ and any ordered set $I=\{ i_1<\cdots <i_d  \}$ such that
\[
\M=\sum_I ds_I\otimes \M_s^I=1\otimes \M_s^\varnothing+ \sum_{d\ge 1; i_1<\cdots< i_d} ds_{i_1}\wedge \cdots \wedge ds_{i_d} \otimes  \M_s^{i_1\cdots i_d}
\]
and every $\M_s^I$ is smooth in $s$ for a fixed $I$. To spell out fully, the above equation says
\[
\M (1\otimes x_1,\dots, 1\otimes x_k)(s)= \sum_I ds_I\otimes \M_s^I \big(x_1(s),\dots,x_k(s) \big)
\]
Finally, we note that any operator $\m:C^{\otimes k}\to C'$ can be trivially extended to a $P$-pointwise operator $\M:C^{\otimes k}_P \to C_P'$ simply by setting $\M_s^\varnothing=\m$ and $\M_s^I=0$ for non-empty $I$, for all $s$.
\end{rmk}

Further, given a cochain complex structure $(\C, \m_{1,0})$ on $C$, where $\m_{1,0}:C \to C$ satisfies $\m_{1,0}\circ \m_{1,0}=0$, there is a corresponding cochain complex $(\C_P, \M^P_{1,0})$ where
\begin{equation}\label{MP_10 -d -eq}
\textstyle
\M^P_{1,0}(ds_I\otimes x)= (-1)^{|I|} \big( ds_I\otimes\m_{1,0}(x) + \sum_{i=1}^{\dim P} ds_I\wedge ds_i\otimes \partial_{s_i} x
\big)
\end{equation}
or equivalently,
\[
\textstyle
\M_{1,0}^P= 1\otimes \m_{1,0} +\sum_i ds_i\otimes \partial_{s_i} 
\]
with regard to the convention in Remark \ref{pointwise_determine_rmk}.
One can check the sign here agrees exactly with Definition \ref{pointwise-defn}.
Both $\incl$ and $\eval^s$ become the cochain maps with respect to these differential maps.



Following \cite[\S 4.3]{Solomon_Involutions}, the sign convention we choose makes $\m_{1,0}=d$. It is slightly different from \cite[Definition 21.29]{FOOOKuranishiVFC} which makes $\m_{1,0}(x)=(-1)^{n+\deg x} d(x)$ instead.
This is inessential, as we can relate them to each other by the sign transformation $\tilde \m_{k,\beta}(x_1,\dots,x_k)=(-1)^\epsilon\m_{k,\beta}(x_1,\dots,x_k)$ for $\epsilon=\sum_{1\le i\le k} (n+\deg x_i)$.
The following proposition is straightforward from the definition.

\begin{prop}\label{P_1P_2-prop}
We have the following natural identifications of cochain complexes $(\Omega^*(L)_{P_1})_{P_2}\equiv \Omega^*(L)_{P_2\times P_1}$ and $(H^*(L)_{P_1})_{P_2}\equiv H^*(L)_{P_2\times P_1}$.
\end{prop}


\begin{lem}\label{eval^s quasi-isom-lem}
The evaluation maps $\eval^\s: \Omega^*(L)_P\to \Omega^*(L)$ or $\eval^\s: H^*(L)_P \to H^*(L)$ are quasi-isomorphisms of cochain complexes for all $s\in P$.
Therefore,
$\eval^s: (\OL_{P_1})_{P_2} \to \OL_{P_1}$ and $\eval^s: (\HL_{P_1})_{P_2} \to \HL_{P_1}$ are also quasi-isomorphisms.
\end{lem}

\begin{proof}
For the first map, just observe that the evaluation map $\eval^s$ is identified with the pull-back of inclusion $\iota_s: L\to \{\s\}\times L \subset P\times L$. And, since $P$ is contractible, $\eval^\s\equiv \iota_s^*$ is a quasi-isomorphism. 
For the second map, 
picking up a basis we may identify
$H^*(L)\cong \mathbb R^m$ and $H^*(L)_P\cong \Omega^*(P)^{\oplus m}$.
Since we use the zero differential for $H^*(L)$, the induced differential on $H^*(L)_P$ obtained by (\ref{MP_10 -d -eq}) coincides with (up to sign) $(d^P)^{\oplus m}$ where $d^P$ denotes the exterior derivative on $P$; this completes the proof.
\end{proof}

\begin{defn}\label{P-pseudo-isotopy-defn}
A \textbf{$P$-pseudo-isotopy} (of $A_\infty$ algebras) on $\C$ is defined to be a $P$-pointwise $A_\infty$ algebra structure $\M^P=(\M^P_{k,\beta})\in \CC_\G(\C_P,\C_P)$ so that $\M^P_{1,0}$ is given by (\ref{MP_10 -d -eq}).
When $P=\oi$, it's simply called a \textbf{pseudo-isotopy}. 
\end{defn}


\begin{defn}
	\label{Classical-Pseudo-isotopy-defn}
A (classical) \textbf{pseudo-isotopy} of $\C$ is defined to be a family of operators $\m^s_{k,\beta}$ and $\mathfrak c^s_{k,\beta}$, $s\in\oi$ on $\C$ satisfying the following conditions:
\begin{itemize}
	\itemsep 1pt
\item[(a)] The operators $\m^s_{k,\beta}$ and $\mc^s_{k,\beta}$ are all smooth in $s\in \oi$. Their degrees are $2-k-\mu(\beta)$ and $1-k-\mu(\beta)$ respectively.
\item[(b)] The $(\C,\m^s)$ is a gapped $A_\infty$ algebra for any $s$.
\item[(c)] If $E(\beta)<0$, then $\mc^s_{k,\beta}=0$.
For any $E_0>0$, we can only find at most finitely many $\beta\in \G$ with $E(\beta)\le E_0$ so that $\mc_{k,\beta}^s\neq 0$ for some $s\in\oi$ and $k\in\mathbb N$.
%
\item[(d)]  $\m_{1,0}:=\m^s_{1,0}$ is independent of $s$ and $\mc^s_{1,0}=\frac{d}{ds}$.
\item[(e)] The following holds
\end{itemize}
\[
\frac{d}{ds}\m^s_{k,\beta} 
+ 
\sum_{\substack{i+j+\ell=k }} 
\sum_{\substack{ \beta_1+\beta_2=\beta\\
(i+j+1,\beta_1)\neq (1,0)}}
\mc^s_{i+j+1,\beta_1}  (\id_\#^i \otimes \m^s_{\ell,\beta_2} \otimes \id^j) 
\ \ -
\sum_{\substack{i+j+\ell=k }} 
\sum_{\substack{ \beta_1+\beta_2=\beta \\
(\ell,\beta_2)\neq(1,0)}}
\m^s_{i+j+1, \beta_1} 
(\id^i \otimes 
\mc^s_{\ell,\beta_2} \otimes \id^j) =0
\]
\end{defn}


\begin{prop}\hspace{-0.2em}\emph{\cite[Lemma 8.1]{FuCyclic}} \ 
	\label{equivalent-family-pseudoisotopy-prop}
The two definitions of pseudo-isotopy are equivalent.
\end{prop}

\begin{ex}\label{trivial-pseudo-isotopy-ex}
For any gapped $A_\infty$ algebra $(C,\m)$, one can construct the so-called \textbf{trivial pseudo-isotopy} $\M^{\tri}$ about $\m$. Specifically, if we write $\M^\tri_{k,\beta}=1\otimes \m^s_{k,\beta} \otimes +ds\otimes \mc^s_{k,\beta}$, then $\m_{k,\beta}^s=\m_{k,\beta}$ and $\mc^s_{k,\beta}=0$, except $\mc^s_{1,0}=d/ds$.
In other words,
$\M^{\tri}=1\otimes \m+ds\otimes \tfrac{d}{ds}$.
In higher dimensions, the \textit{trivial $P$-pseudo-isotopy} is similarly defined by $\M^{\tri, P}=1\otimes \m+ \sum_i ds_i\otimes \frac{d}{ds_i}$.
\end{ex}

An alternative term for $P$-pseudo-isotopy in \cite[Definition 21.29]{FOOOKuranishiVFC} is a $P$-parametrized family of $\G$-gapped $A_\infty$ algebra structures on $\C$. Essentially, a $P$-pseudo-isotopy represents a family of $A_\infty$ algebras over $P$, enriched with derivative-like data. Heuristically, for $P=\oi$, $\m^s$ forms a family of $A_\infty$ algebras, and $\mc^s$ can be viewed as the derivative-like data.

\begin{defn}
	\label{restriction_pseudo_isotopy_defn}
	Given a $P$-pseudo-isotopy $(C_P,\M^P)$ defined as
	$
	\textstyle \M^P(s)=1\otimes \m^s+ \sum_{I\neq \varnothing} ds_I\otimes \mc^{I,s}.
	$
	we immediately derive the following:
Each $(\C, \m^s), s\in P$ is an $A_\infty$ algebra, called the \textbf{restriction} of $(\C_P, \M^P)$ at $s \in P$. Alternatively, we say $\M^P$ \textbf{restricts} to $\m^s$ at $s \in P$.
\end{defn}

\begin{rmk}
	\label{eval-incl-as-A_infty-rmk}
	The $\eval^s$ can naturally be seen as an $A_\infty$ homomorphism from $(\C_P, \M^P)$ to $(\C, \m^s)$ by setting $\eval^s_{1,0}=\eval^s$ and $\eval^s_{k,\beta}=0$ for any $(k,\beta) \neq (1,0)$. This means
	$ \m^s \diamond \eval^s =\eval^s \{ \M^P\}$.
	However, the $\incl$ \textit{cannot} be considered an $A_\infty$ homomorphism unless $\M^P$ is a trivial pseudo-isotopy.
\end{rmk}

\subsubsection{Unitality and divisor axiom}
\label{ss_Divisor_axiom}
We next introduce a few definitions for $A_\infty$ structures (with labels).

\begin{defn}
	\label{unit-defn}
\textbf{(i)} An $A_\infty$ algebra $(\C, \m)$ is termed (strictly) \textit{unital} if it contains a degree-zero element $\one \in \C$, called a (strict) \textit{unit}, satisfying the following conditions:
(a0) $\m_{1,0}(\one)=0$; (a1) $\m_{2,0}(\one,x)=(-1)^{\deg x} \m_{2,0}(x,\one)=x$; (a2) $\m_{k,\beta}(\dots,\one,\dots)=0$ for $(k,\beta)\neq (1,0), (2,0)$.

\textbf{(ii)} Let $\one_1$ and $\one_2$ be units of $(C_1,\m_1)$ and $(C_2,\m_2)$.
An $A_\infty$ homomorphism $\f: \m_1\to \m_2$ is called \textit{unital} (with respect to $\one_1$ and $\one_2$) if
(b1) $\f_{1,0}(\one_1)=\one_2$; (b2) $\f_{k,\beta}(\dots, \one_1,\dots)=0$ for $(k,\beta)\neq (1,0)$.

\textbf{(iii)} An $A_\infty$ algebra $(C,\m)$ is called \textit{fully unital} if, for any degree-zero $\e\in C$, 
\begin{itemize}
	\item[(a2')] $\m_{k,\beta}(\dots,\e,\dots)=0$ when $(k,\beta)\neq (1,0),(2,0)$.
\end{itemize}
An $A_\infty$ homomorphism $\f:\m_1\to \m_2$ is called \textit{fully unital} if, for any degree-zero $\e\in C_1$,
\begin{itemize}
	\item [(b2')] $\f_{k,\beta}(\dots,\e,\dots)=0$ when $(k,\beta)\neq (1,0)$.
\end{itemize}
\end{defn}

The concept of a $P$-pseudo-isotopy for $\C$ is more {stringent} than that of a general $A_\infty$ algebra on $\C_P$. Thus, we aim to define unitality in a way that aligns with the pointwiseness.

\begin{defn}
	\label{P_unit_defn}
	A $P$-pseudo-isotopy $(\C_P,\M^P)$ is said to be \textit{$P$-unital} if there exists some $\one\in \C^0$ so that $\incl(\one)$ is a unit of $\M^P$. In this case, we call $\one\in C$ or $\incl(\one)\in \C_P$ a \textit{$P$-unit} of $(C_P,\M^P)$.
\end{defn}

\begin{defn}
	\label{cyclical_unitality_defn}
	An {operator system} $\mathfrak t\in\CC_\G$ is called \textbf{cyclically unital} if, for any degree-zero element $\e$ and any $(k,\beta)\neq (0,0)$, we have:
		\[
		\textstyle
		\CU[\mathfrak t]_{k,\beta}(\e; x_1,\dots, x_k):=
	 \sum_{i=1}^{k+1} \mathfrak t_{k+1,\beta}(x_1^\#,\dots, x_{i-1}^\#, \e,x_i,\dots, x_{k})=0
		\]
\end{defn}
Here the `CU' stands for `cyclical unitality'.
Note that it is necessary to require $(k,\beta)\neq (0,0)$, as $\CU [\mathfrak t]_{0,0}(\e)=\mathfrak t_{1,0}(\e)$ can be non-zero in general.
Be cautious that the cyclical unitality applies for an arbitrary operator system $\mathfrak t$, while the full unitality is defined just for some $A_\infty$ algebra $\m$ or some $A_\infty$ homomorphism $\f$.
The following concept may justify our new unitalities.

\begin{defn}
	\label{quantum-correction-defn}
	We say an $A_\infty$ algebra $(\C,\m)$ is a \textbf{quantum correction to de Rham complex}, or in abbreviation, is a \textbf{q.c.dR}, if $\C$ is isomorphic to some de Rham complex $\Omega^*(N)$ for a manifold $N$ and the following properties:
(a) $\m_{k,0}=0$ for $k\ge 3$; (b) $\m_{1,0}(x)=dx$; (c) $\m_{2,0}(x_1,x_2)=(-1)^{\deg x_1} x_1\wedge x_2$.
\end{defn}

\begin{rmk}
	\label{cyclical_unitality_qcdR_rmk}
\textbf{(1)}
To explain why the full unitality is reasonable, the key observation is that an $A_\infty$ algebra $(\OL, \m)$ associated to a Lagrangian submanifold (\S \ref{S_A_infty_associated}) is also proved to be a q.c.dR in the literature\footnote{See \cite[Definition 21.21 \& Theorem 21.35]{FOOOKuranishiVFC} for the latest de Rham model, and \cite[Definition 3.5.6 \& Remark 3.5.8]{FOOOBookOne} for the earlier singular chain model.}. The constant-one $\one\in\Omega^0(L)$ is known to be the unit of $\m$, but the mere q.c.dR condition is actually sufficient to obtain the conditions (a0) (a1) in Definition \ref{unit-defn}. If we examine more carefully the argument to prove (a2) (see e.g. \cite[(7.3)]{FuCyclic}), then we will discover that the same argument can be applied equally to any other degree-zero form $\e$. We will go back to this point in \S \ref{Subsec_Forgetful_map}.

\textbf{(2)} The cyclical unitality naturally comes out by a homological algebra consideration. Indeed, it gives a successful induction hypothesis, when we attempt to prove Whitehead theorem with divisor axiom (Theorem \ref{Whitehead-full-thm}). Moreover, the cyclical unitality fits perfectly with the congruence relations for a divisor-axiom-preserving homotopy theory of $A_\infty$ homomorphisms (see the proof of Lemma \ref{UD-simud-composition-lem}).
\end{rmk}

\begin{lem}
	\label{unitality_composition_lem}
	If two $A_\infty$ homomorphisms $\f,\g$ are unital (resp. fully unital or cyclically unital), then $\g \diamond \f$ is also unital (resp. fully unital or cyclically unital).
\end{lem}

\begin{proof}
	We only show the cyclical unitality, and the other cases are similar. For $(k,\beta)\neq (0,0)$, 
	\begin{align*}
	\CU[\g\diamond \f]_{k,\beta}(\e;\dots)
	=
	\sum_{(\ell,\beta')\neq (0,0)} \g( \f^\# \dots \f^\# , \CU[\f]_{\ell,\beta'}(\e; \dots), \f\dots \f )
	+
	\sum
	\CU[\g]_{m,\beta''}(\f_{1,0}(\e); \f,\dots,\f)
	\end{align*}
	The exceptional terms with $(\ell,\beta')=(0,0)$ in the first sum are all collected in the second sum.
	By Definition \ref{Gapped-Hom-defn}, we have $\deg\f_{1,0}=0$. Then, the second sum vanishes since $\g$ is cyclically unital. Meanwhile, the first sum also vanishes because $\f$ is cyclically unital.
\end{proof}


%
%
%

For the label group $\G=(\pi_2(X,L),E,\mu)$, we have the boundary homomorphism $\partial: \pi_2(X,L)\to \pi_1(L)$.
Given an operator system $\mathfrak t\in \CC_\G(C, C' )$, we define
\[
\DA[\mathfrak t]_{k,\beta,m} (b;x_1,\dots,x_k):=\sum_{m_0+\dots+m_k=m } {\mathfrak t}_{k+m,\beta} ( \overbrace{b,\dots,b}^{m_0},x_1,\overbrace{b,\dots,b}^{m_1}, \dots, \overbrace{b,\dots,b}^{m_{k-1}},x_k,\overbrace{b,\dots,b}^{m_k})
\]
In practice, it actually suffices to assume $m=1$, thus, we make the abbreviation
$
\DA[\mathfrak t]_{k,\beta}:=\DA[\mathfrak t]_{k,\beta,1}
$.
Given any graded cochain complex $(\C,\delta)$, we define
\begin{equation}
\label{Divisor_input-eq}
\DI(\C):=\DI(\C,\delta)=\{
b\in \C\mid b\in \ker \delta, \ \deg b=1
\}
\end{equation}
and define $\DI(C,\Lambda)=\DI(C)\hat\otimes \Lambda$. We call both of them the spaces of \textbf{divisor inputs}. 
Note that a degree-zero cochain map $\f_{1,0}:C\to C'$ preserves the spaces of divisor input
$\f_{1,0}:\DI(C)\to\DI(C')$.

From now on, we primarily take $C=\HL_P$ or $C=\OL_P$. Recall that their differentials are given as in (\ref{MP_10 -d -eq}).
We claim that in either case, there is a well-defined natural \textbf{cap product}
\begin{equation}
\label{cap_product_eq}
\partial\beta\cap :  \ \DI(C)\to \mathbb R, \quad \DI(C,\Lambda)\to  \Lambda
\end{equation}
for any $\beta\in \pi_2(X,L)$, defined as follows:
Let $b$ be a divisor input. Write
$
b=1\otimes \bar b^s + \textstyle\sum_{i=1}^{\dim P} ds_i \otimes b^s_i
$, and consider:
\textbf{(1)} When $C=\HL_P$, the condition that $b$ is a divisor input implies $\partial_{s_j}\bar b^s=0$. Thus, $\bar b^s\in H^1(L)$ is independent of $s$; we just define $\partial \beta\cap b$ to be $\partial\beta\cap \bar b^s$.
\textbf{(2)} When $C=\OL_P$, the divisor input condition tells $d\bar b^s=0$ and $\partial_{s_j}\bar b^s-db_j^s=0$. Applying the natural quotient $q:Z^1(L)\to H^1(L)$ to the second equation yields that $\partial_{s_j} q(\bar b^s)= q(db_j^s)=0$. Hence, the de Rham cohomology class $q(\bar b^s)$ is also independent of $s$. So we can also define $\partial\beta \cap b:  = \partial \beta\cap \bar b^s=\partial\beta\cap q(\bar b^s)$.

\begin{defn}\label{DivisorAxiom-defn}
An operator system $\mathfrak t\in\CC_\G(C, C')$ is said to satisfy the \textbf{divisor axiom} if for any divisor input $b\in \DI(C)$ and $(k,\beta,m)\neq (0,0,1)$, the following \textit{divisor axiom equation} holds:
\begin{equation}
\label{DivisorAxiom}
\textstyle
	\DA[\mathfrak t]_{k,\beta,m} (b;x_1,\dots,x_k)= \frac{(\partial \beta \cap b)^m}{m!}  ~ {\mathfrak t}_{k,\beta}(x_1,\dots,x_k)
\end{equation}
By an induction on $m$, it is equivalent to only require $m=1$ in (\ref{DivisorAxiom}). Namely, for all $(k,\beta)\neq (0,0)$,
$
\DA[\mathfrak t]_{k,\beta}(b;x_1,\dots,x_k)
=
\mathfrak t_{k+1,\beta}(b,x_1,\dots,x_k)+ \cdots +\mathfrak t_{k+1,\beta}(x_1,\dots,x_k,b) =\partial \beta\cap b \cdot \mathfrak t_{k,\beta}(x_1,\dots,x_k)
$.
\end{defn}

\subsubsection{Category $\UD$}
\label{ss_UD}

Abusing the notations, denote by $\one$ the constant-one function in any one of $\Omega^*(L)$, $H^*(L)$, $\Omega^*(L)_P$, $\HL_P$.
We will show in Theorem \ref{UD_prime-Composition-DA-thm} that the following data form a category:
\begin{equation}\label{UD_prime_defn-eq}
\tilde{\UD}:=\tilde{\UD}(L):=\tilde{\UD}(L,X)
\end{equation}

\begin{enumerate}
	\itemsep 2pt
\item[(I)] An object in $\tilde{\UD}$ is a $\G$-gapped $A_\infty$ algebra with the following properties:
		\begin{itemize}
			\itemsep 1pt
		\item[(I-1)] it is a $P$-pseudo-isotopy for some $P$;
		\item[(I-2)] it is $P$-unital, and the $\one$ is a $P$-unit;
		\item[(I-3)] it is cyclically unital;
		\item[(I-4)] it satisfies the divisor axiom.
		\end{itemize}

\item[(II)] A morphism $\f$ in $\tilde{\UD}$ is a $\G$-gapped $A_\infty$ homomorphism with the following properties
	\begin{itemize}
					\itemsep 1pt
	\item[(II-1)] it is unital with respect to the various $\one$ (so, $\f_{1,0}(\one)=\one$)
	\item[(II-2)] it is cyclically unital.
	\item[(II-3)] it satisfies the divisor axiom;
	\item[(II-4)] it satisfies the following identity for any divisor input $b$:
	\end{itemize}
\end{enumerate}
\vspace{-0.5em}
\begin{equation}
\noindent 
\label{DA-f_10-eq}
\partial \beta \cap \f_{1,0} (b) = \partial \beta \cap b
\end{equation}

The above description aims to investigate the \textbf{u}nitalities and the \textbf{d}ivisor axiom in a systematical and categorical manner.
Our aim is to incorporate the divisor axiom into the standard homotopy theory of $A_\infty$ algebras (e.g. \cite{FOOOBookOne}), but the cyclical unitality need to accompany it as explained in Remark \ref{cyclical_unitality_qcdR_rmk} (2).
We then obtain a subcategory of the category of $A_\infty$ algebras (with labels).

In the item (I-1), be cautious of the context of Definition \ref{P-pseudo-isotopy-defn}: if $(C_P, \M)$ is a $P$-pseudo-isotopy, then $C$ is assumed to be a direct sum of the spaces of differential forms on several smooth closed manifolds (see the start of \S \ref{S_P_pseudo_isotopy}). In particular, this includes the situations of our interests: $C=\HL_P$ or $C=\OL_P$ for a closed Lagrangian $L$. The $\CC_{1,0}$-component of an object $A_\infty$ algebra coincide with the natural cochain complex structures on $\HL_P$ or $\OL_P$ by Definition \ref{P-pseudo-isotopy-defn}.
On the other hand, we will need the condition (II-4) to preserve the divisor axiom for compositions.
We mention some examples for which the (II-4) holds: 
(i) $\f_{1,0}=i(g):H^*(L)\to \Omega^*(L)$ is a `harmonic inclusion' which will be discussed in (\ref{harmonic i(g)}); (ii) $\f_{1,0}=\pi(g):\Omega^*(L)\to H^*(L)$ is a `harmonic projection' in (\ref{harmonic pi(g)});
(iii) $\f_{1,0}=\id$; (iv) $\f_{1,0}$ is so that $\eval^s \cdot \f_{1,0}$ agrees with one of the above (i), (ii), or (iii).

\begin{defn}
\label{UD_defn}
Define
$
\UD:=\UD(L):=\UD(L,X)$
to be
 the subcategory of $\tilde{\UD}$ in (\ref{UD_prime_defn-eq}), consisting of objects $\m$ and morphisms $\f$ satisfying the extra conditions as follows:
\begin{itemize}
	\item[]
	\begin{itemize}
		\itemsep 1pt
		\item [(I-5)] every $\beta$ in the set $\mathsf G_\m:=\{\beta\in\G\mid \m_{\beta}\neq 0\}$ satisfies $\mu(\beta)\ge 0$
		\item [(II-5)] every $\beta$ in the set $\mathsf G_\f :=\{\beta\in\G\mid \f_{\beta}\neq 0\}$ satisfies $\mu(\beta)\ge 0$.
	\end{itemize} 
\end{itemize}
where $\m_{\beta}=(\m_{k,\beta})_{k\in\mathbb N}$ and $\f_\beta=(\f_{k,\beta})_{k\in\mathbb N}$ are defined as the collections of all components of $\m$ and $\f$ with $\beta$ fixed; c.f. Definition \ref{CC_defn}.
\end{defn}

By Assumption \ref{assumption-mu ge 0}, we mainly work with the above subcategory $\UD$ instead of $\tilde{\UD}$.
Denote by $\Obj\UD$ (resp. $\Mor\UD$) the collections of all objects (resp. all morphisms) in $\UD$.
As previously stated, we now prove that $\tilde{\UD}$ and $\UD$ indeed form categories in the subsequent two lemmas.

\begin{thm}
	\label{UD_prime-Composition-DA-thm}
	$\tilde{\UD}$ and $\UD$ are categories.
\end{thm}

\begin{proof}
Given $\f,\g\in\Mor\tilde{\UD}$, our purpose is to show $\g\diamond \f\in\Mor\tilde{\UD}$.
Indeed, as $(\g \diamond \f)_{1,0}=\g_{1,0} \diamond \f_{1,0}$, the $\g\diamond\f$ also satisfies (II-4).
	The unitality and cyclical unitality are already proved in Lemma \ref{unitality_composition_lem}.
	It remains to show the divisor axiom. Indeed,
	\begin{align*}
	\DA[\g\diamond \f]_{k,\beta}(b;\dots)
	\textstyle
	&
	=
	\textstyle
	\sum_i \sum_{(k_i,\beta_i)\neq (0,0) } \g_{\ell,\beta_0}
	\big(
	\f_{k_1,\beta_1}\cdots \DA[\f]_{k_i,\beta_i}(b;\dots)\dots \f_{k_\ell,\beta_\ell}
	\big)  \\
	&+
	\textstyle
	\sum \DA[\g]_{\ell,\beta_0}
	(
	\f_{1,0}(b); \f_{k_1,\beta_1}\cdots \f_{k_\ell,\beta_\ell}
	)
	\end{align*}
	holds for $(k,\beta)\neq (0,0)$ by routine computations.
Since $\f$ and $\g$ satisfy the divisor axiom, we deduce
\begin{align*}
\DA[\g \diamond \f]_{k,\beta}(b;\cdots)
&
=\textstyle
\sum_i \sum_{(k_i,\beta_i)\neq (0,0)}
\partial\beta_i\cap b  \cdot \g_{\ell,\beta_0} (\f_{k_1,\beta_1}\cdots \f_{k_i,\beta_i}\cdots \f_{k_\ell,\beta_\ell}) \\
&
+
\textstyle
\sum \partial\beta_0\cap \f_{1,0}(b) \cdot \g_{\ell,\beta_0} (\f_{\beta_1}\cdots \f_{\beta_\ell})
\end{align*}
By Definition \ref{CC_defn} (a), we have $\f_{0,0}=0$ and so we can actually drop $(k_i,\beta_i)\neq (0,0)$ above. Then, using (\ref{DA-f_10-eq}) together with $\beta=\beta_0+\sum_{i=1}^\ell\beta_i$, we conclude the divisor axiom for $\g \diamond \f$.

As for $\UD$, it suffices to check the extra condition (II-5) in Definition \ref{UD_defn} is preserved. In reality, by Definition \ref{Composition-defn}, if $(\g \diamond \f)_\beta\neq 0$ then for at least one tuple $(\beta_0,\beta_1,\dots,\beta_\ell)$ with $\beta_0+\beta_1+\cdots +\beta_\ell=\beta$, we have $\g_{\beta_0}  (\f_{\beta_1}\otimes \cdots \otimes \f_{\beta_\ell})\neq 0$ and thus particularly $\g_{\beta_0}\neq 0, \f_{\beta_1}\neq 0, \dots ,\f_{\beta_\ell}\neq 0$. Since $\f$ and $\g$ satisfy (II-5), we conclude $\mu(\beta)=\mu(\beta_0)+\cdots +\mu(\beta_\ell)\ge 0$.
\end{proof}


\begin{lem}\label{UD-C_oi-lem}
If $(\C,\m) \in\Obj\UD$, then its trivial pseudo-isotopy $(\C_\oi, \M^\tri) \in\Obj\UD$.
\end{lem}
\begin{proof}
By Proposition \ref{P_1P_2-prop}, we know the (I-1) holds for $(\C_\oi, \M^\tri)$. 
Observe that $\M^\tri_{k,\beta}=1\otimes \m_{k,\beta}$ for $(k,\beta)\neq (1,0)$, and it is direct to show (I-2), (I-3) or (I-5).
It remains to check the divisor axiom (I-4).
Fix a divisor input $b=1\otimes b_0+ds\otimes b_1$; without loss of generality, other inputs $x_i$ can be assumed to be in the form $x_i(s)=1\otimes y_i(s)$.
Then, we have
\begin{align*}
 \DA[\M^\tri]_{k,\beta}
 (b;x_1,\dots,x_k) 
=
1\otimes \DA[\m]_{k,\beta}(b_0;y_1,\dots, y_k) (s)
\textstyle
+
ds\otimes \CU[\m]_{k,\beta}(b_1;y_1,\dots,y_k) (s)
\end{align*}
As $\m$ is cyclically unital, the second term vanishes.
By definition (\ref{cap_product_eq}), we have $\partial\beta\cap b=\partial\beta\cap b_0(s)$. In conclusion,
the desired divisor axiom equations of $\M^\tri$ follow from that of $\m$.
\end{proof}

%

\subsubsection{Homotopy theory with divisor axiom}
\label{ss_Divisor_axiom_homotopy}
Our new homotopy theory is based on the category $\mathscr{UD}$. It is critical to note that Definition \ref{ud-homotopic-defn} below implicitly depends on Lemma \ref{UD-C_oi-lem}.

\begin{defn}\label{ud-homotopic-defn}
We call $\f_0, \f_1 \in \Hom_{\UD}((\C',\m'),(\C,\m))$ are \textbf{ud-homotopic} to each other \textbf{(via $\F$)} if there is 
$
\F\in \Hom_\UD((\C',\m'), (\C_\oi, \M^\tri))
$
so that $\eval^0   \F=\f_0$ and $\eval^1   \F=\f_1$. We write
$
\f_0\simud\f_1
$.
\end{defn}

We can unfold the condition in Definition \ref{ud-homotopic-defn} as follows.
An image $\F_{k,\beta}(x_1,\dots,x_k)$ is essentially an element in $\C_\oi$. By taking into account the bi-grading on $\C_\oi$, we can express $\F$ as $1\otimes \f_s + ds\otimes \h_s$, involving linear operators $\f_s$ and $\h_s$. Without loss of generality, we may also assume that these $\f_s$ and $\h_s$ are constant in $s$ near the end points. Specifically, this implies:
\begin{equation}
\label{F0+_roughly_pointwise_eq}
\F_{k,\beta}(x_1,\dots,x_k) (s)=1\otimes (\f_s)_{k,\beta}(x_1,\dots,x_k)+ds\otimes (\h_s)_{k,\beta}(x_1,\dots,x_k)
\end{equation}
%

\begin{lem}
	\label{UD_homotopy_character_A_infty_formula_lem}
	$\F=1\otimes \f_s+ds\otimes \h_s\in \CC_\G(C',C_\oi)$ gives a $\G$-gapped $A_\infty$ homomorphism from $(C',\m')$ to $(C_\oi,\M^\tri)$ if and only if the following conditions hold together:
	
	\begin{itemize}
		\itemsep 1pt
	\item Every $\f_s$ is an $A_\infty$ homomorphism from $(C',\m')$ to $(C,\m)$;
	\item $\deg (\h_s)_{k,\beta}=-k-\mu(\beta)$;
	\item We have the identity
	\end{itemize}
\vspace{-0.5em}
\begin{equation}
\label{F0+_roughly_A_infty_formula_eq}
\textstyle \frac{d}{ds} \circ \f_s = \sum \h_s  \circ (\id^\bullet_\#\otimes \m'\otimes \id^\bullet ) + \sum \m \circ (\f^\#_s\otimes\cdots \otimes  \f^\#_s \otimes \h_s\otimes \f_s \otimes \cdots \otimes\f_s )
\end{equation}
\end{lem}

\begin{proof}
	For the degrees, just observe that $\deg \f_s=\deg\h_s+1=\deg \F$.
Using Example \ref{trivial-pseudo-isotopy-ex} and (\ref{F0+_roughly_pointwise_eq}), we expand the $A_\infty$ formula $\sum \M^\tri   (\F\otimes \cdots \otimes \F)=\sum \F   (\id^\bullet_\#\otimes \m'\otimes \id^\bullet)$, obtaining that
$1\otimes \m   ( \f_s\otimes \cdots \otimes \f_s) \ 
+ds\otimes
\big(
\textstyle \frac{d}{ds} \circ \f_s -  \m (\f_s^\#\otimes\cdots \otimes\f_s^\#\otimes \h_s\otimes \f_s\otimes \cdots \otimes \f_s)
\big) 
=
1\otimes \f_s \circ (\id^\bullet_\#\otimes \m' \otimes \id^\bullet)  + ds\otimes \h_s \cdot (\id_\#^\bullet \otimes \m' \otimes \id^\bullet)$. The proof is now complete by comparing both sides.
\end{proof}

\begin{lem}
	\label{UD_homotopy_DA_unit_lem}
	In the above situation of Lemma \ref{UD_homotopy_character_A_infty_formula_lem}, we have:
	
	\begin{itemize}
		\itemsep 1pt
		\item [(i)] The $\F$ satisfies divisor axiom if and only if all the $\f_s$ and $\h_s$ satisfy divisor axiom.
		\item [(ii)] The $\F$ is cyclically unital if and only if all the $\f_s$ and $\h_s$ are cyclically unital.
		\item [(iii)] The $\F$ is unital with respect to the constant-ones $\one$ if and only if all the $\f_s$ is unital with respect to $\one$'s and meanwhile $(\h_s)_{k,\beta}(\cdots\one\cdots)=0$ for all $s$ and $(k,\beta)$.
		\item [(iv)] The $\F$ satisfies (II-4) if and only if all the $\f_s$ satisfy (II-4) (\ref{DA-f_10-eq}).
		\item [(v)] The $\F$ satisfies (II-5) if and only if all the $\f_s$ and $\h_s$ satisfies (II-5) (Definition \ref{UD_defn}).
	\end{itemize}
\end{lem}

\begin{proof}
By (\ref{F0+_roughly_pointwise_eq}), the divisor axiom equation
$
\DA[\F]_{k,\beta}(b;x_1,\dots,x_k) 
	=
	\partial \beta \cap b \cdot \F_{k,\beta}(x_1,\dots,x_k)$
means:
	\begin{align*}
	\begin{matrix*}
	1\otimes \DA[\f_s]_{k,\beta}(b;x_1,\dots,x_k)\\
	+ds\otimes \DA[\h_s]_{k,\beta}(b;x_1,\dots,x_k)
	\end{matrix*}
	=
\partial\beta\cap b \cdot
	\begin{pmatrix*}
	1\otimes (\f_s)_{k,\beta} (x_1,\dots,x_k) \\
	+ds\otimes (\h_s)_{k,\beta}(x_1,\dots,x_k)
	\end{pmatrix*}
	\end{align*}
Then, the item (i) is proved by comparison. The properties for the cyclical unitality (ii) and the unitality (iii) can be proved by similar comparisons.
Next, by the definition of the cap product (\ref{cap_product_eq}), we have $\partial \beta\cap \F_{1,0}(b)=\partial\beta\cap (\f_s)_{1,0}(b)$ and the item (iv) holds.
The item (v) is obvious.
\end{proof}

\begin{cor}
	\label{UD_homotopy_summary_cor}
	$\f_0$ and $\f_1$ are ud-homotopic to each other if and only if there exist $(\f_s)$ and $(\h_s)$ so that 
	
	\begin{itemize}
			\itemsep 1pt
		\item [(a)] All the $\f_s\in \Mor \UD$.
		\item [(b)] The formula (\ref{F0+_roughly_A_infty_formula_eq}) holds for $(\f_s)$ and $(\h_s)$.
		\item [(c)] All the $\h_s$ satisfy the divisor axiom, the cyclical unitality, and $(\h_s)_{k,\beta}(\cdots\one\cdots)=0$ for all $(k,\beta)$.
		\item [(d)] $\deg (\h_s)_{k,\beta}=-k-\mu(\beta)$ and all the $\h_s$ satisfy (II-5) (Definition \ref{UD_defn}).
	\end{itemize}
\end{cor}




\begin{lem}
\label{UD-simud-composition-lem}
If $\f_0\simud\f_1$ and $\g_0\simud \g_1$, then $\g_0 \diamond \f_0\simud \g_1 \diamond \f_1$. So, $\simud$ is a congruence relation on $\UD$.
\end{lem}

\begin{proof}
	Let $\f_0,\f_1\in \Hom_\UD( (\C'',\m''), (\C',\m'))$ and $\g_0,\g_1\in \Hom_\UD ( (\C',\m'), (\C,\m))$.
Denote by $\M'$ and $\M$ the trivial pseudo-isotopies about $\m'$ and $\m$. They are defined on $\C'_\oi$ and $\C_\oi$ respectively.
By definition, there exists morphisms $\F: \C'' \to \C'_\oi$ and $\G: \C'\to \C_\oi$ in $\UD$ connecting $\f_0$ and $\f_1$, $\g_0$ and $\g_1$ respectively.
By Theorem \ref{UD_prime-Composition-DA-thm}, the composition $\G \diamond \f_1: \C''\to \C_\oi$ is still in $\UD$. Since $\eval^i \diamond \G \diamond \f_1=\g_i \diamond \f_1$ for $i=0,1$, we conclude that $\g_0 \diamond \f_1\simud \g_1 \diamond \f_1$.
It remains to show $\g_0 \diamond \f_0\simud\g_0 \diamond \f_1$.
We first find families $(\f_s)$ and $(\h_s)$ with the conditions (a) (b) (c) (d) by Corollary \ref{UD_homotopy_summary_cor}.
Next, we aim to check the same four conditions for the other two families: $\hat \f_s:= \g_0 \diamond \f_s$ and
\begin{equation}
\label{hat_h_s_eq}
\hat\h_s= \sum \g_0   (\f_s^\# \cdots \f_s^\#, \h_s, \f_s \cdots\f_s)
\end{equation}

\begin{itemize}
	\itemsep 1pt
	\item[(a)] By condition, $\f_s\in\Mor\UD$. Then, by Theorem \ref{UD_prime-Composition-DA-thm}, we have $\hat\f_s\in\Mor\UD$. 
	
	\item[(b)] To prove the condition (b), we compute
\begin{align*}
\textstyle
\frac{d}{ds} \circ \hat \f_s
&
=
\textstyle
\sum \g_0( \f_s\cdots \f_s, \tfrac{d}{ds}\circ  \f_s, \f_s \cdots \f_s)\\
&
=
\textstyle
\sum \g_0(\f_s\cdots \f_s, \h_s(\id^\bullet_\#\otimes \m''\otimes \id^\bullet), \f_s\cdots \f_s)+ \g_0(\f_s\cdots \f_s, \m'(\f_s^\#\cdots\f^\#_s,\h_s,\f_s\cdots\f_s),\f_s\cdots\f_s)\\
&
=
\textstyle
\sum \g_0(\f^\#_s\cdots\f_s^\#,\h_s,\f_s\cdots\f_s) (\id_\#^\bullet\otimes \m'' \otimes  \id^\bullet) + \g_0(\id_\#^\bullet\otimes \m'\otimes \id^\bullet)(\f^\#_s\cdots\f^\#_s,\h_s,\f_s\cdots\f_s)\\
&
=
\textstyle
\sum \hat\h_s(\id^\bullet_\#\otimes \m''\otimes \id^\bullet)+\m(\g_0\cdots\g_0)(\f^\#_s\cdots\f^\#_s,\h_s,\f_s\cdots\f_s)\\
&
=
\textstyle
\sum \hat\h_s(\id^\bullet_\#\otimes \m''\otimes \id^\bullet)+
\m\big( (\g_0 \diamond \f_s)^\#\cdots (\g_0\diamond \f_s)^\#, \hat \h_s,  (\g_0 \diamond \f_s)\cdots(\g_0\diamond \f_s) \big)
\end{align*}
where the second identity uses the formula (\ref{F0+_roughly_A_infty_formula_eq}) for $\f_s$ and $\h_s$; the third and fourth identities hold because the $\f_s$ and the $\g_0$ are $A_\infty$ homomorphisms. Thus, we show the $\hat\f_s$ and $\hat\h_s$ satisfy (\ref{F0+_roughly_A_infty_formula_eq}).

\item[(c)] Concerning the divisor axiom, we compute as follows:
\begin{align*}
\DA[\hat\h_s]_{k,\beta}(b;\cdots)
&
=
\textstyle
\sum \DA[\g_0] \big(
(\f_s)_{1,0}(b); \f_s^\# \cdots \h_s  \cdots \f_s
\big) \\
&
+
\textstyle
\sum \g_0
\big(
\cdots \DA[\f_s^\#](b;\cdots)\cdots \h_s\cdots
\big)+ \sum \g_0 \big(	\cdots \h_s\cdots \DA[\f_s](b;\cdots)		\big)
\\
&
+
\textstyle
\sum \g_0
\big(\f_s^\#\cdots \DA[\h_s](b;\cdots)\cdots \f_s\big)
+
\textstyle
\sum
\CU[\g_0]
\big( (\h_s)_{1,0}(b); \f_s \cdots \cdots \f_s\big)
\end{align*}
Since $\deg (\h_s)_{1,0}(b)=0$, the fifth sum vanishes due to the cyclical unitality of $\g_0$. Similarly as before, applying the divisor axiom of $\f_s,\h_s$ and $\g_0$, the first four sums together exactly equate $\partial\beta\cap b\cdot \hat\h_s$.
Regarding the cyclical unitality, given a degree-zero $\e$, we first notice that $(\h_s)_{1,0}(\e)=0$, since $\deg (\h_s)_{1,0}(\e)= -1$.
Then, there is a similar computation:
\begin{align*}
\CU[\hat \h_s]_{k,\beta}(\e;\cdots)
&
=
\textstyle
\sum \CU[\g_0]\big( 
(\f_s)_{1,0}(\e); \f_s^\#\cdots \h_s \cdots \f_s
\big)\\
&
+
\textstyle
\sum \g_0 \big(
\cdots \CU[\f_s^\#](\e;\cdots)\cdots \h_s\cdots
\big)
+
\sum \g_0 \big(\cdots \h_s\cdots \CU[\f_s](\e;\cdots)\cdots \big) \\
&
+
\textstyle
\sum \g_0 \big(\f_s^\# \cdots \CU[\h_s](\e;\cdots) \cdots \f_s \big)
\end{align*}
The above four sums must all vanish by the cyclical unitalities of $\f_s$, $\h_s$ and $\g_0$.
Compute
\begin{align*}
\hat \h_s(\cdots\one \cdots)
&
=
\textstyle
\sum \g_0\big(\f^\#_s(\cdots)\cdots \f^\#_s(\cdots)
\
\h_s(\cdots\one\cdots) 
\
\f_s(\cdots)\cdots\f_s(\cdots) \big )\\
&
\textstyle
+
\sum \g_0\big(\cdots (\f_s^\#)_{k_1,\beta_1}(\cdots\one\cdots)\cdots \h_s(\cdots) \cdots \f_s(\cdots)\cdots \big)\\
&
+
\textstyle
\sum \g_0\big( \cdots \f_s^\#(\cdots)\cdots \h_s(\cdots)\cdots \f_s(\cdots \one \cdots)_{k_2,\beta_2}\cdots
\big)
\end{align*}
The first sum vanishes by the property of $(\h_s)$. Since the $\f_s$ is unital with respect to $\one$, a summand term in the second or third summation must be zero, unless $(k_1,\beta_1)=(1,0)$ or $(k_2,\beta_2)=(1,0)$. But, if so, one input is $(\f_s)_{1,0}(\one)=\one$, and the resulting summand term is like $(\g_0)_{\ell,\beta}(\cdots\one\cdots)$ for some $\ell \ge 2$. Thus, the unitality of $\g_0$ also enforces this summand to be zero, and the $\hat\h_s(\cdots\one\cdots)$ vanishes as desired. The condition (c) for $\hat\h_s$ is proved.

\item [(d)]
The degree of $(\hat\h_s)_{k,\beta}$ is clearly $-k-\mu(\beta)$.
If $(\hat \h_s)_\beta\neq 0$, by (\ref{hat_h_s_eq}) there exists a decomposition $\beta_0+\beta_1'+\cdots +\beta_{\ell_1}'+\beta_3+\beta_1''+\cdots+\beta_{\ell_2}''=\beta$ so that all of the followings are non-zero: $(\g_0)_{\beta_0}$, $(\f_s^\#)_{\beta_\lambda'}$ ($\lambda=1,\dots,\ell_1$), $(\h_s)_{\beta_3}$, $(\f_s)_{\beta_\lambda''}$ ($\lambda=1,\dots,\ell_2$). Since all of them satisfy (II-5) (Definition \ref{UD_defn}), the additivity of $\mu$ implies $\mu(\beta)\ge 0$. Consequently, we see $\hat\h_s$ also satisfies (II-5).
\end{itemize}
In summary, the families $\hat \f_s$ and $\hat\h_s$ give a ud-homotopy from $\g_0 \diamond \f_0$ to $\g_0 \diamond \f_1$ by Corollary \ref{UD_homotopy_summary_cor}.
\end{proof}

\subsubsection{Novikov field}
\label{ss_weak_MC}

We define the \textit{Novikov field} to be
\begin{equation}
	\label{Novikov_eq}
	\Lambda= \mathbb C((T^{\mathbb R}))=\left\{ \sum_{i=0}^\infty a_i T^{\lambda_i} \mid a_i\in\mathbb C, \lambda_i\in\mathbb R, \lambda_i\to+\infty \right\}
\end{equation}

The Novikov field has a non-archimedean \textit{valuation map}
$
\val: \Lambda \to \mathbb R\cup \{\infty\}
$
defined by sending a nonzero series $\sum_{i=0}^\infty a_i T^{\lambda_i}$ to the smallest $\lambda_i$ with $a_i\neq 0$ and sending the zero series to $\infty$.
The valuation ring $\Lambda_0:=\val^{-1}[0,\infty]$ is called the \textit{Novikov ring}. It has a unique maximal ideal $\Lambda_+:=\val^{-1}(0,\infty]$.
Note that the $\val$ gives rise to a (non-archimedean) norm $|y|=\exp(-\val(y))$. In particular, we have the so-called \textit{adic topology} on $\Lambda_0$ or $\Lambda$, and so it makes sense to talk about convergence.
By \textit{energy filtration} we mean the filtration defined by setting $F^a\Lambda= \val^{-1}{[a,\infty]}$ for various $a$.
The \textit{residual field} 
$\Lambda_0/\Lambda_+$ coincides with $\mathbb C$, and we have $\Lambda_0\equiv \mathbb C\oplus \Lambda_+$.

The \textit{multiplicative group} 
$
U_\Lambda =\{ y \in \Lambda\mid \val(y)=0\} \equiv \{ y\in \Lambda\mid |y|=1\}
$
resembles the $S^1\equiv U(1)$ and
satisfies that $U_\Lambda \equiv \mathbb C^*\oplus \Lambda_+$.
Define $\Lambda^*=\Lambda\setminus\{0\}$, and we call the following map
\[
\trop : (\Lambda^*)^n\to \mathbb R^n, \qquad (y_1,\dots, y_n)\mapsto \big(\val(y_1),\dots, \val(y_n) \big)
\]
a \textbf{\textit{tropicalization map}} (cf. \cite{NA_nonarchimedean_SYZ,KSAffine}).
The total space $(\Lambda^*)^n$ is a non-archimedean analytic space locally covered by the polytopal domains
$\trop^{-1}(\Delta)$ for rational convex polyhedrons $\Delta$ in $\mathbb R^n$.

An important property of $\Lambda$ is that we can similarly define the \textit{exponential} and \textit{logarithm} functions. 

\begin{defn}\label{exp_log_defn}
	For $x\in\Lambda_0$, there is a unique decomposition $x=x_0+x_+$ with $x_0\in\mathbb C$ and $x_+\in\Lambda_+$. 
	For $y\in U_\Lambda$, there is a unique decomposition $y=y_0 (1+y_+)$ with $y_0\in\mathbb C^*$ and $y_+\in\Lambda_+$. Define
	\[
	\textstyle
	\exp(x):= e^{x_0}\sum_{k\ge 0} \frac{x_+^k}{k!} 
	\	\in      U_\Lambda
	\qquad ; \qquad 
	\textstyle
	\log(y):= (\log(y_0)+2\pi i \mathbb Z) +\sum_{k\ge 1} (-1)^{k+1} \frac{y_+^k}{k}
	\ \in   \Lambda_0/2\pi i \mathbb Z
	\]
\end{defn}

\begin{lem}\label{exp-log-lem}
	The standard isomorphism $\mathbb C^*\cong \mathbb C/2\pi i\mathbb Z$ extends to the following:
	\[
	U_\Lambda \overset{\ \ \ \log \ \ \ }{\underset{\ \ \ \exp \ \ \ }\rightleftarrows} \Lambda_0/ 2\pi i\mathbb Z
	\]
	In particular, for any $y\in U_\Lambda$ (i.e. $\val(y)=0$), there exists some $x\in \Lambda_0$ so that $y=\exp(x)$. 
\end{lem}


\begin{lem}\label{val=0-lem}
	If $f=\sum_{\nu\in\mathbb Z^n} c_\nu \z^\nu
	\in \Lambda [[z_1^\pm, \dots, z_n^\pm]]$ vanishes on $U_\Lambda^n$, then $f$ is identically zero.
\end{lem}

\begin{proof}
	The condition actually tells that $f$ converges on $U_\Lambda^n \equiv \Val^{-1}(0)$. Thus, by Proposition \ref{polyhedral-affinoid-algebra-evaluation-convergence-prop} we know that $f\in \Lambda \langle \{0\} \rangle \equiv \Lambda \langle z_1^\pm, \dots, z_n^\pm \rangle$. Accordingly, as $|\nu|\to \infty$, we have $\val(c_\nu)\to\infty$ or equivalently $|c_\nu|\to 0$. 
	Arguing by contraction, we suppose $f$ was not identically zero.
	Without loss of generality we may assume the sequence $|c_\nu|$ attains a maximal value $|c_{\nu_0}|=1$ for some $\nu_0\in\mathbb Z^n$, and we may further assume $c_{\nu_0}=1$. 
	Recall that $\Lambda_0 =\{ c: \val(c) \ge 0 \}=\{c : |c|\le 1\}$ and  $\Lambda_+ =\{c: \val(c)> 0 \}= \{c : |c| < 1\}$. So, $\Lambda_0/\Lambda_+\cong\mathbb C$. 
	Since $|c_\nu|\le 1$ for all $\nu$,
	we know $f\in \Lambda_0[[z^\pm]]$.
	Taking the quotient by the ideal of elements with norm $<1$,
	we get a power series $\bar f=\sum_{\nu} \bar c_\nu \z^\nu$ over the residue field $\mathbb C$.
	Since $c_\nu \to 0$, we have $|c_\nu|<1$ and thus $\bar c_\nu=0$ for all large enough $\nu$.
	Hence, the $\bar f$ is in reality a nonzero Laurent polynomial with $\bar c_{\nu_0}=1$. But by assumption $\bar f(\mathbf {\bar y})$ vanishes for all $\mathbf {\bar y}\in (\mathbb C^*)^n$; hence, $\bar f$ must be identically zero. This is a contradiction.
\end{proof}

The aforementioned lemmas, while simple, are crucial: to prove that a Laurent formal power series $f(y_1,\dots, y_n) \equiv 0$, it is sufficient to demonstrate $f(e^{x_1},\dots, e^{x_n}) \equiv 0$ for all $x_i \in \Lambda_0$.

Let $(\C,\m)\in\Obj\UD$ and $\f\in \Hom_\UD((\C'',\m''), (\C',\m'))$ be an object and a morphism in $\UD$.
Recall that $\DI(C) $ is the space of divisor inputs (\ref{Divisor_input-eq}).
We define
$\m_*: \DI(\C)\hat\otimes \Lambda_0 \to \C\hat\otimes  \Lambda_0$
as
$\m_*(b):=  \sum_{\beta}\sum_k T^{E(\beta)} \m_{k,\beta}(b,\dots,b)=\sum_\beta T^{E(\beta)} e^{\partial\beta\cap b}  \m_{0,\beta}$, where we have used the divisor axiom.
Similarly, we can define
$\f_*: \DI(\C'')\hat\otimes  \Lambda_0 \to \C'\hat\otimes   \Lambda_0$.
In the literature, such as \cite{FOOOBookOne}, the coefficient ring was typically chosen as $\Lambda_+$ instead of $\Lambda_0$. However, the issue of convergence over $\Lambda_0$ is addressed by the divisor axiom.
In the above, we can always sum over $k$ first:
$
\textstyle
\m_*(b)=
\sum_\beta T^{E(\beta)} \big(\sum_k \m_{k,\beta}(b,\dots, b)\big)=: \sum_\beta T^{E(\beta)} M_\beta
$.
If we write $b=b_0+b_+$ for $\Lambda_0 \equiv \mathbb C\oplus \Lambda_+$, the divisor axiom and Definition \ref{exp_log_defn} yields that
\[
\textstyle
M_\beta=\sum_k \m_{k,\beta}(b,\dots,b)=  \sum_k \frac{1}{k!} (\partial \beta \cap b)^k \m_{0,\beta} 
= \big( e^{\partial \beta\cap b_0} \sum_k \frac{(\partial\beta\cap b_+)^k}{k!} \big) \m_{0,\beta} =e^{\partial \beta \cap b} \m_{0,\beta}
\]
Thus, $\m_*(b)$ faces no convergence issues within the field $\Lambda$. The same reasoning holds true for $\f_*(b)$.

Aligned with existing literature, we term $b$ a \textit{weak bounding cochain} of $\m$ if $\m_(b)$ is a multiple of the unit $\one$, namely $\m_(b) = a \cdot \one$ for some $a$ in $\Lambda_0$. The set of weak bounding cochains, termed the Maurer-Cartan set of $\m$, is set-theoretically invariant under $A_\infty$ homotopy equivalences.
However, these terms are not essential for our current discussion, as one of our main objectives in this paper is exactly to extend beyond the traditional Maurer-Cartan framework.

\subsection{Whitehead theorem in category $\UD$}
\label{S_whitehead}

Once and for all, we fix $(\C',\m'), (\C,\m) \in\Obj\UD$ and
$
\f\in \Hom_\UD((\C,\m),(\C',\m'))
$
so that the induced cochain map $\f_{1,0}: (C,\m_{1,0})\to (C',\m'_{1,0})$ is a quasi-isomorphism.

The entire section is dedicated to proving the theorem below, which readers may choose to skip on their first reading. This theorem extends \cite[Theorem 4.2.45]{FOOOBookOne} by incorporating the topological labels and the extra properties such as (cyclical) unitality and the divisor axiom. The proof introduces several new technical elements such as Sublemma \ref{sublemma-u_k+1}.

\begin{thm}[Whitehead Theorem]\label{Whitehead-full-thm}
There exists $\g\in \Hom_\UD((\C',\m'),(\C,\m))$, unique up to ud-homotopy, so that $\g \diamond \f$ and $\f \diamond \g$ are both ud-homotopic to $\id$. We call $\g$ a \textbf{ud-homotopy inverse} of $\f$; it is often denoted by $\f^{-1}$.
\end{thm}

We will carry out the proof by inductions first on the length filtration in $k$ and then on the energy filtration in $\beta$.
Initially, since $\f_{1,0}$ is a quasi-isomorphism, there exist a degree-zero cochain map
\begin{equation}
\label{Whitehead-g_10-eq}
\g_{1,0}: \C' \to \C
\end{equation}
and a cochain homotopy map $h:\C\to \C$ of degree $-1$ such that
\begin{equation}
\label{Whitehead_g_10_flexibility_eq}
\g_{1,0}\f_{1,0} - \id = \m_{1,0} h +h \m_{1,0}
\end{equation}
Consider a cochain map $\h_{1,0}: \C \to \C_\oi$ defined by
\begin{equation}
\label{Whitehead-h_10}
\h_{1,0}(x) = 1\otimes \big( (1-s)x+s\g_{1,0}\f_{1,0} (x) \big)+ ds \otimes h(x)
\end{equation}
Then, it satisfies
$\eval^0  \h_{1,0} =\id$ and $\eval^1 \h_{1,0}= \g_{1,0} \f_{1,0}$.
By definition, we know $\f_{1,0}(\one)=\one$ and $\m_{1,0}(\one)=0$; by degree reason, we also known $h(\one)=0$. By (\ref{Whitehead_g_10_flexibility_eq}) and (\ref{Whitehead-h_10}), we obtain
\begin{equation}
\label{g_10_h_10_one_eq}
\g_{1,0}(\one)=\one, \quad \text{and} \quad \h_{1,0}(\one)=1\otimes \one=\incl(\one)=\one
\end{equation}

\begin{lem}\label{DA-g_10-h_10-lem}
The property (II-4) (\ref{DA-f_10-eq}) holds for $\g_{1,0}$ and $\h_{1,0}$.
\end{lem}

\begin{proof}
First,
we fix a divisor input $b'\in\DI(C',\m'_{1,0})$ (\ref{Divisor_input-eq}). Since the cochain map $\f_{1,0}$ is a quasi-isomorphism, there exists some $b\in \DI(C,\m_{1,0})$ so that $\g_{1,0}(b')=b+\m_{1,0}(c)$ and $\f_{1,0}(b)=b'+\m_{1,0}(c')$ for some $c$ and $c'$. By degree reason, we have $\deg c=\deg c'=0$.
Since $\f_{1,0}$ satisfies (II-4) (\ref{DA-f_10-eq}), we see that
$\partial\beta\cap \g_{1,0}(b')=\partial \beta \cap b=\partial\beta\cap  \f_{1,0}(b)=\partial\beta\cap b'$.
As for $\h_{1,0}$, by the definition of the cap product (\ref{cap_product_eq}), using (\ref{Whitehead-h_10}) infers that $\partial \beta \cap \h_{1,0} (b) =\partial \beta \cap \eval^0\h_{1,0}(b)=\partial\beta\cap b$.
\end{proof}

	There is certain flexibility in choosing $\g_{1,0}$ in (\ref{Whitehead-g_10-eq}), since we only need to require $\g_{1,0}$ satisfies (\ref{Whitehead_g_10_flexibility_eq}). In practice, we are often interested in the following two cases:
\textbf{(i)} When $\f=\eval^s: H^*(L)_P\to H^*(L)$ or $\Omega^*(L)_P\to \Omega^*(L)$ for some $s\in P$. Then a ud-homotopy inverse $\g=(\eval^s)^{-1}$ of $\eval^s$ can be chosen so that $\g_{1,0}=\incl$.
\textbf{(ii)} Construct a morphism $\mi^g\in\Mor \UD$ as Corollary \ref{UD-canonical-model-thm}, with $\mi^g_{1,0}=i(g)$, where $i(g)$ is the harmonic embedding $H^*(L)\to \Omega^*(L)$ as (\ref{harmonic i(g)}). A ud-homotopy inverse $\g=(\mi^g)^{-1}$ will be chosen so that $\g_{1,0}=\pi(g)$, where $\pi(g)$ is the harmonic projection $\OL \to\HL$ (\ref{harmonic pi(g)}).

\subsubsection{Obstruction for the length filtration}
The pair $(\g_{1,0},\h_{1,0})$ gives the initial step of the induction.
Recall we often omit $C,C'$ and $\G$ for the notation $\CC_\G(C',C)$.
In this section, we aim to inductively construct the energy-zero part of the $\g$ in $\CC_0:=\prod_{k=1}^\infty \CC_{k,0}$.
By the \textit{length filtration} on $\CC_0$ we mean the filtration defined by
$
\textstyle
F^k\CC_0:= \prod_{j=1}^{k} \CC_{j,0}$.

To begin with, we define a natural differential on the vector space $\CC_0$:
\begin{equation}
\label{delta_bar-eq}
\bar\delta:=\bar\delta_k: \CC_{k,0} \to \CC_{k,0}
\qquad \phi \mapsto \m_{1,0} \circ \phi  -(-1)^{\deg \phi -1+k} \textstyle \sum_i \phi \circ (\id^i_\# \otimes \m'_{1,0} \otimes \id^{k-i-1})
\end{equation}
The $\bar\delta$ has degree one, i.e. $\deg \bar\delta \phi= \deg \phi+1$.
By routine computation, we derive
$\bar\delta\circ \bar\delta =0$.

To include the unitalities and divisor axiom, we introduce the following symbols: 
Define
$\CC^\ud_0 = \prod_{k\ge 1} \CC^\ud_{k,0}$, where $\CC_{1,0}^\ud=\{\phi \mid \phi(\one)=\one\}$ and $
\CC_{k,0}^\ud$ ($k\geqslant 2$) is the subset of $\phi\in \CC_{k,0}$ with

\begin{itemize}
	\itemsep 1pt
\item[(L1)] 
$\DA[\phi]_{k-1,0}(b;x_1,\dots,x_{k-1})=0$ for any divisor input $b$.
\item[(L2)] $
\CU[\phi]_{k-1,0}(\e;x_1,\dots,x_{k-1})=0$ for any $\e$ with $\deg \e=0$.
\item[(L3)] $\phi(\dots,\one,\dots)=0$ for the constant-one $\one$.
\end{itemize}

Note that this says
$(\CC^\ud_0,\bar\delta) = \prod_{k\ge 1} (\CC^\ud_{k,0}, \bar \delta_k )$.
Here the condition (L1) just tells the divisor axiom for the special case $\beta=0$ (Definition \ref{DivisorAxiom-defn}), and the (L2) describes the cyclical unitality for $\beta=0$ (Definition \ref{cyclical_unitality_defn}).
The length filtration in $\CC_0$ can be inherited in $\CC^\ud_0$ by restriction.

\begin{prop}[Properties of $\bar \delta$ and $\CC^\ud_0$]
	\label{delta_bar-properties-prop}
\leavevmode
	\begin{enumerate}
	\item[(i)] The $\bar\delta$ restricts to a differential on $\CC_{k,0}^\ud$, so there is a cohomology $H(\CC^\ud_{k,0}, \bar \delta)$ for each $k$.
	\item[(ii)] If $f$ is a cochain map of degree zero with $f(\one)=\one$, then for any $k$, it induces a cochain map 
	\[
	f_*: (\CC^\ud_{k,0},\bar\delta) \to (\CC^\ud_{k,0}, \bar\delta)
	\]
	 sending $\phi$ to $f\circ \phi$. Moreover, if $f$ is a quasi-isomorphism, then so is $f_*$.
	 \item[(iii)] If $f$ is a cochain map of degree zero with $f(\one)=\one$, then for any $k$, it induces a cochain map 
	 \[
	 \hat f: (\CC^\ud_{k,0}, \bar\delta) \to (\CC^\ud_{k,0}, \bar\delta)
	 \]
sending $\phi$ to $ \phi \circ f^{\otimes k}$. Moreover, if $f$ is a quasi-isomorphism, then so is $\hat f$.
	\end{enumerate}
\end{prop}

\begin{proof}
\textbf{(i).}
Since $\bar\delta\circ\bar\delta=0$, it suffices to show $\bar\delta$ maps $\CC^\ud_{k,0}$ into $\CC^\ud_{k,0}$. Firstly, it is easy to check the $\bar \delta$ preserves the condition (L1).
Secondly, given $\phi\in\CC^\ud_{k,0}$, we compute 
\begin{align*}
\textstyle \sum_i\bar\delta \phi(x_1^\#,\dots,x_{i-1}^\#,\e,x_i,\dots,x_{k-1})
=
&
\textstyle
 \pm \m_{1,0}  \sum_i \phi(x_1,\dots, x_{i-1},\e, x_i,\dots, x_{k-1}) \\
&
\textstyle
\pm \sum_i \phi(x_1,\dots,x_{i-1},\m'_{1,0}\e,x_i,\dots,x_{k-1}) \\
&  
\textstyle
\pm \sum_{i>j}\phi   (x_1,\dots,x_{j-1}, \m_{1,0}'(x_j^\#), x^\#_{j+1},\dots,x^\#_{i-1},\e,x_i,\dots,x_{k-1})\\
&  
\textstyle
\pm \sum_{i<j} \phi  (x_1,\dots,x_{i-1},\e^\#,x_i^\#,\dots, x_{j-1}^\#,\m'_{1,0}(x_j),x_{j+1},\dots,x_{k-1})
\end{align*}
The first sum vanishes as $\phi$ satisfies (L2), and the second sum vanishes as $\phi$ satisfies (L1). Also, for any fixed $j$, we set $y^{(j)}_\ell=x^\#_\ell$ for $1\le \ell \le  j-1$, $y^{(j)}_\ell=x_\ell$ for $j+1\le \ell \le k-1$ and $y^{(j)}_j=\m'_{1,0}(x_j)$; so, the last two sums together give
$
\pm \sum_j\sum_\ell \phi(y^{(j)}_1,\dots,y^{(j)}_{\ell-1},\e,y^{(j)}_\ell,\dots,y^{(j)}_{k-1})$ which equals zero by the (L2) condition of $\phi$.
Thirdly, the $\bar\delta$ also preserves the unitality condition (L3) because of $\m'_{1,0}(\one)=0$.

\textbf{(ii).}
Note that $\bar\delta f_* =f_*\bar \delta$ and that the $f_*$ restricts to a map from $\CC^\ud_{k,0}$ into $\CC^\ud_{k,0}$. If $k=1$, we also need the condition $f(\one)=\one$.
If $f$ is a quasi-isomorphism, then it admits an inverse $f'$ in the sense that $f' \circ f =\id + \m_{1,0} h + h \m_{1,0}$ and $f \circ f' = \id + \m'_{1,0} k + k \m'_{1,0}$ for some operators $h$ and $k$. Given $\phi$ in $\ker \bar\delta\cap \CC^\ud_0$, $h\circ \phi \in \CC^\ud_0$ and
$
f'_*f_* \phi 
=
\phi+\m_{1,0} h\phi + h\m_{1,0}\phi 
=
\phi + \bar\delta ( h\phi)
$; the similar holds for $f_*f'_*$.
Hence, we get an isomorphism $H(f_*): H(\CC^\ud, \bar\delta) \to H(\CC^\ud, \bar \delta)$ with an inverse $H(f'_*)$.

\textbf{(iii).}
We can check that $\hat f$ is a cochain map, and it maps $\CC^\ud_{k,0}$ into $\CC^\ud_{k,0}$.
Suppose $f$ is a quasi-isomorphism, and let $h, k$ be as above in (ii). If $\bar\delta \phi=0$, then $\widehat{f'f} (\phi) 
=
\phi(\id+\bar\delta h,\dots, \id + \bar\delta h)
=
\phi+ \bar\delta \eta$ for some $\eta$. Notice that $\bar\delta  (\phi(h_1,\dots,h_k) )=(\bar\delta\phi)(h_1,\dots,h_k)\pm \sum_i \phi(h_1,\dots, \bar\delta h_i,\dots, h_k)$; this implies $\phi( \bar\delta h,\bar \delta h, \id,\dots,\id)=\bar\delta \big(\phi(h,\bar\delta h,\id,\dots, \id) \big)$ as $\bar\delta(\id)=0$; we also compute $\phi(\id+\bar\delta h,\dots, \id+\bar\delta h)$.
By similar argument, we complete the proof.
\end{proof}

\begin{cor}\label{deltabar-cohomology-eval_*-incl_*-cor}
The $\eval^s$ and $\incl$ induce quasi-isomorphisms $\eval^s_*, \widehat{\eval^s}$ and $\incl_*, \widehat{\incl}$ on $\CC^\ud_{k,0}$.
\end{cor}


We explore the obstruction to extending an $A_{k-1}$ homomorphism to an $A_k$ homomorphism.

\begin{thm}\label{o_k-length-thm}
Given an $A_{k-1}$ homomorphism $\bar \g=(\g_{j,0})_{1\le j\le k-1} \in F^{k-1} \CC_0^\ud(\C',\C)$ between two objects $(\C',\m'),(\C,\m)\in\Obj\UD$, there exists a $\bar\delta$-closed element $\mathfrak o_k(\bar\g) \in \CC^\ud_{k,0} (\C',\C)$
such that its cohomology $[\mathfrak o_k (\bar\g)]\in H( \CC^\ud_{k,0},\bar\delta)$ vanishes if and only if $\bar\g$ can be extended to an $A_k$ homomorphism $\bar \g^+=(\bar \g , \g_{k,0})$ with $\g_{k,0}\in\CC_{k,0}^\ud$. Moreover, if so, we have $\mathfrak o_k(\bar \g)+ \bar \delta (\g_{k,0})=0$
\end{thm}

\begin{proof} 
We first make the definition: (which \textit{depends} on the chosen $A_\infty$ algebras $\bar\m'$ and $\bar\m$)
\begin{equation}
\label{o_k_eq}
\mathfrak o_k(\bar\g) :=  \sum_{\ell \neq 1}  \m_{\ell,0}   (\g_{j_1-j_0,0} \otimes \cdots \otimes \g_{j_\ell - j_{\ell-1},0}) - \sum_{ \nu\neq 1}  \g_{\lambda+ \mu+1,0} (\id^\lambda_\# \otimes \m'_{\nu,0} \otimes \id^\mu)
\end{equation}
where the conditions $\ell, \nu\neq 1$ guarantee that the $\mathfrak o_k(\bar \g)$ involve only $\g_{i,0}$ for $i<k$.
As $\bar\m,\bar\m',\bar\g\in\CC^\ud_0$, it is routine to check $\mathfrak o_k(\bar \g) \in \CC^\ud_0$. 
By (\ref{delta_bar-eq}), the $\bar \g^+=(\bar \g, \g_{k,0})$ being an $A_k$ homomorphism is exactly equivalent to $\mathfrak o_k(\bar \g) +\bar \delta ( \g_{k,0})=0$. Moreover, every term in the expression (\ref{o_k_eq}) of $\mathfrak o_k(\bar\g)$ has degree $2-k$. From $\deg \bar\delta=1$, it follows that the degree of $\g_{k,0}$ is $1-k$ as desired.

Now, it remains to prove $\mathfrak o_k(\bar \g)$ is $\bar\delta$-closed.  In fact, using (\ref{delta_bar-eq}), we first compute:
\begin{align*}
\bar\delta \mathfrak o_k(\bar \g)
=
&
\textstyle
\sum_{\ell \neq 1} \m_{1,0} \circ \m_{\ell,0} \circ (\g\otimes \cdots \otimes \g)
+
\sum_{\ell\neq 1} \m_{\ell,0} \circ (\g\otimes \cdots \otimes \g) \circ (\id^\bullet_\# \otimes \m'_{1,0}\otimes \id^\bullet) \\
&
\textstyle
-\big(
\sum_{\nu\neq 1} \m_{1,0}\circ \g \circ (\id^\bullet_\# \otimes \m'_{\nu,0} \otimes \id^\bullet ) 
+
\sum_{\nu\neq 1}  \g \circ (\id^\bullet_\# \otimes \m'_{\nu,0} \otimes \id^\bullet )  \circ (\id^\bullet_\#\otimes \m'_{1,0}\otimes \id^\bullet)
\big)
\end{align*}
Using the $A_\infty$ equations for $\m$ and $\m'$ together with the condition $\bar\g$ is $A_{k-1}$ homomorphism, a tedious but routine calculation deduces $\bar\delta \mathfrak o_k(\bar\g)=0$.
\end{proof}


\begin{lem}\label{o_k-composition-lem}
In the above situation, let $\bar \f^0=(\f^0_{k,0})_k$ and $\bar \f^1=(\f^1_{k,0})_k$ be two $A_\infty$ homomoprhisms\footnote{Actually, it is enough to assume that they are $A_k$ homomorphisms (Definition \ref{reduction-defn}).} in $\CC^\ud_0$. Then,
$\bar\f^1\diamond \bar \g \diamond \bar \f^0$ is also an $A_{k-1}$ homomorphism and
$
[\mathfrak o_k (\bar \f^1 \diamond \bar \g \diamond \bar \f^0) ]= 
\big[
\f^1_{1,0}\circ \mathfrak o_k(\bar \g)  \circ (\f^0_{1,0})^{\otimes k}
\big]
$
\end{lem}

\begin{proof}
$\bar\g\in F^{k-1}\CC^\ud_0\subset \CC^\ud_0$ is an $A_{k-1}$ homomorphism.
Abusing the notations, all the involved $A_\infty$ algebras here are denoted by $\m$.
To begin with, it is routine to check $\bar\f^1\diamond\bar\g\diamond\bar\f^0$ is an $A_{k-1}$ homomorphism. Besides, $\bar\f^1\diamond \bar \g \diamond \bar \f^0$ restricts to the subspace $F^{k-1}\CC^\ud_0\subset F^{k-1} \CC_0$ by the proof of Theorem \ref{UD_prime-Composition-DA-thm}.
So, the $\mathfrak o_k(\bar\f^1\diamond \bar\g \diamond \bar\f^0)$ is well-defined by Theorem \ref{o_k-length-thm}. Indeed, if we further regard $\bar \g$ as an element in $\CC_0^\ud$ by trivial extensions (i.e. we set $\bar\g_k=\bar\g_{k+1}=\cdots =0$), then the operator system 
$
\bar\h:= \bar\f^1  \diamond \bar\g \diamond \bar \f^0 =(\h_{j,0})_{j\ge 1}
$
satisfies that $\bar\h\in\CC_0^\ud$ and $\bar\h|_{F^{k-1}\CC}$ is an $A_{k-1}$ homomorphism.
Define
\[
\textstyle
\CC_{k,0}\ni
\mathfrak O:= \sum_r \m_{r,0} \circ (\bar\h \otimes \cdots \otimes \bar\h)
-
\sum_s  \bar\h \circ ( \id^\bullet_\# \otimes \m_{s,0} \otimes \id^\bullet)
\]
It is easy to compute that
$\mathfrak O = \mathfrak o_k(\bar\h) + \bar \delta \h_{k,0}$.
Also, we can compute the $\mathfrak O$ in another way as follows:
\begin{align*}
\mathfrak
O
=
&
\Scale[0.9]{
 \sum \f^1_{p,0} \circ (\id_\# \otimes \m_{r,0}\otimes \id ) \circ ( \bar\g\otimes \cdots \otimes \bar\g) \circ (\bar\f^0\otimes \cdots \otimes \bar\f^0) 
 -
 \sum  \bar\f^1_{q,0} \circ (\bar\g\otimes \cdots \otimes \bar\g) \circ (\id^\bullet_\#\otimes \m_{s,0}\otimes \id^\bullet) \circ (\bar\f^0\otimes \cdots \otimes \bar\f^0)
 } \\
=
&
\Scale[0.9]{
\sum \f^1_{1,0} \circ \bar\m\circ (\bar\g\otimes \cdots\otimes \bar\g) \circ (\f^0_{1,0})^{\otimes k} 
- 
\sum \f^1_{1,0} \circ \bar\g \circ (\id^\bullet_\# \otimes \bar\m \otimes \id^\bullet ) \circ (\f^0_{1,0})^{\otimes k} 
}\\
=
&
\Scale[0.9]{
	\sum_{\ell\neq 1} \f^1_{1,0} \circ \bar\m_{\ell,0}\circ (\bar\g\otimes \cdots\otimes \bar\g) \circ (\f^0_{1,0})^{\otimes k} 
	- 
	\sum_{\nu\neq 1} \f^1_{1,0} \circ \bar\g \circ (\id^\bullet_\# \otimes \bar\m_{\nu,0} \otimes \id^\bullet ) \circ (\f^0_{1,0})^{\otimes k} 
}
\end{align*}
where the first identity holds by the $A_\infty$ homomorphism equations for $\bar\f^0$ and $\bar\f^1$, the second identity holds since $\bar\g$ is an $A_{k-1}$ homomorphism, and the third is because we have trivially extended $\g$ by $\g_{k,0}=0$. The outcome exactly says $\mathfrak O= \f^1_{1,0} \circ \mathfrak o_k(\g) \circ (\f^0_{1,0})^{\otimes k}$. So, the proof is now complete.
\end{proof}

\subsubsection{Extension for the length filtration}

Let $\M$ and $\M'$ be the trivial pseudo-isotopies about the $\m$ and $\m'$. By Definition \ref{reduction-defn}, the reduction $\bar\M$ (resp. $\bar \M'$) coincides with the trivial pseudo-isotopies about $\bar\m$ (resp. $\bar\m'$).
Denote by $\bar \f = ( \f_{j,0})_{j\ge 1}$ 
the restriction of the $\f\in\Hom_\UD(\m,\m')$ in $ \CC^\ud_0( \C,\C')$; then, $\bar\f \in \Hom_{\UD}(\bar\m,\bar\m')$. We aim to show a weaker version of Theorem \ref{Whitehead-full-thm} in the energy-zero part.

\begin{thm}\label{Whitehead-Leangth-filtration -thm}
There exists an $A_\infty$ homomorphism $\bar\g\in \CC^\ud_0(\C',\C)$ so that
$\bar\g \diamond \bar\f$ is ud-homotopic to $\id$ via some $\bar \h \in \CC^\ud_0(\C,\C_\oi)$ and $\bar\f \diamond \bar\g$ is ud-homotopic to $\id$ via some $\bar \h' \in \CC^\ud_0(\C',\C'_\oi)$
\end{thm}

We plan to prove it by induction on $k$. The initial step when $k=1$ has completed before.

\begin{prop}\label{Length-Whitehead-prop}
Let $k\ge 2$. Given $\bar\f\in\Hom_\UD(\bar\m,\bar\m')$ as above, we assume
\[
\bar\g=(\g_{j,0})_{1\le j\le k-1} \in F^{k-1}\CC^\ud_0(\C',\C), \quad \text{and} \quad 
\bar\h=(\h_{j,0})_{1\le j\le k-1}\in F^{k-1} \CC^\ud_0 (\C,\C_\oi)
\]
are two $A_{k-1}$ homomorphisms such that 
\begin{align*}
\eval^0  \bar\h |_{F^{k-1}\CC_0} = \id, \quad \text{and} \quad
 \eval^1  \bar\h|_{F^{k-1}\CC_0}  = \bar\g\diamond  \bar\f
\end{align*}
Then there exists two $A_{k}$ homomorphisms $\bar \g^+=(\bar\g,\g_{k,0})\in F^k \CC^\ud_0$ and $\bar \h^+ = (\bar\h, \h_{k,0}) \in F^k \CC^\ud_0$, extending $\bar\g$ and $\bar\h$ respectively. Moreover, they satisfy that
\begin{equation}
\label{Length-Whitehead_Eval_01_eq}
\eval^0  \bar \h^+|_{CC_{k,0}} = \id, \quad \text{and} \quad
\eval^1  \bar \h^+ |_{CC_{k,0}}  = \bar \g^+ \diamond \bar \f
\end{equation}
\end{prop}

\begin{proof}
By Lemma \ref{o_k-composition-lem} and by condition, we see that $[\eval^0 \circ \mathfrak o_k( \bar\h)]= [\mathfrak o_k( \eval^0\circ \bar \h)] = [ \mathfrak o_k(\id)] =0$. 
Then, by Corollary \ref{deltabar-cohomology-eval_*-incl_*-cor}, $[ \mathfrak o_k( \bar\h)]=0$, which implies that there exists some $\alpha\in \CC_{k,0}^\ud$ so that $\mathfrak o_k(\bar\h) +\bar\delta \alpha=0$. 
Recall that naively setting $\h_{k,0}=\alpha$ gives an $A_k$ homomorphism extension, but if so, we would miss the property (\ref{Length-Whitehead_Eval_01_eq}).
In reality, we need to put
\[
\h_{k,0}= \alpha - \incl \eval^0 (\alpha) \in\CC_{k,0}^\ud
\]
By Proposition \ref{delta_bar-properties-prop} and by the definition (\ref{o_k_eq}) of $\mathfrak o_k$, we compute that $\bar\delta \eval^0 \alpha=\eval^0\bar\delta\alpha
=\eval^0 (-\mathfrak o_k(\bar \h))=
-\mathfrak o_k( \eval^0 \bar\h)=-\mathfrak o_k(\id)=0$. 
Due to (\ref{delta_bar-eq}), the fact $\M_{1,0}^\tri\circ\incl=\incl\circ \m_{1,0}$ implies that the $\bar\delta$ commutes with $\incl$.
Thus, we get
$\bar\delta \h_{k,0}=\bar\delta \alpha- \incl \bar\delta \eval^0(\alpha)=\bar\delta\alpha$ and $\mathfrak o_k(\bar\h)+\bar\delta \h_{k,0}=0$.
Then, the $\bar\h^+=(\bar\h, \h_{k,0})$ gives an $A_k$ homomorphism extension such that $ \eval^0\circ \bar \h^+|_{\CC_{k,0}}=\eval^0\h_{k,0}=0$.
Since $\eval^1$ is trivially an $A_\infty$ homomorphism, the composition $\eval^1\circ \bar \h^+$ must be an $A_k$ homomorphism, and it extends $\bar \g \circ \bar \f \in F^{k-1}\CC^\ud_0$. Now, we have
\begin{equation}
\label{gf-o_k-delta_pre-eq}
\mathfrak o_k(\bar \g \circ \bar \f)= - \bar \delta  (\eval^1\circ \bar\h^+)_{k,0}= - \bar \delta ( \eval^1 \circ \h_{k,0})
\end{equation}
and so $[\mathfrak o_k(\bar\g\circ \bar \f)]=0$.
By the proof of Lemma \ref{o_k-composition-lem}, we have actually shown that
\begin{equation}
\label{gf-o_k-delta-eq}
\textstyle
\mathfrak o_k (\bar \g \circ \bar \f) + \bar \delta \big(\sum_{\ell< k} \g_{\ell,0} \circ (\f\otimes \cdots \otimes \f) \big) 
=
\mathfrak o_k(\bar \g) \circ \f_{1,0}^{\otimes k}
\end{equation}
and $[\mathfrak o_k (\bar \g) \circ \f_{1,0}^{\otimes k} ]=[\mathfrak o_k( \bar \g \circ \bar \f)]=0$.
By Proposition \ref{delta_bar-properties-prop}, $[\mathfrak o_k(\bar \g)]=0$. Theorem \ref{o_k-length-thm}
allows for an $A_k$ homomorphism extension $\bar \g^+=(\bar \g , \g_{k,0})\in F^k\CC^\ud_0$ for some $\g_{k,0}$ so that $\mathfrak o_k (\bar \g)=-\bar \delta \g_{k,0}$.
Consider
\begin{equation}
\label{Xi_length_filtration_eq}
\textstyle
\Xi:=\left(\eval^1\circ \bar\h^+- \bar \g^+ \circ \bar \f \right)_{k,0}=
\eval^1\circ \h_{k,0} - \g_{k,0} \circ \f_{1,0}^{\otimes k} - \sum_{\ell< k} \g_{\ell,0} \circ (\f\otimes \cdots \otimes  \f)
\end{equation}
Generally, this is not necessary zero in contrast to (\ref{Length-Whitehead_Eval_01_eq}). So, we need to slightly modify $\g_{k,0}$ and $\h_{k,0}$ as follows.
First, using (\ref{gf-o_k-delta_pre-eq}) and (\ref{gf-o_k-delta-eq}) implies that $\bar \delta \Xi= -(\mathfrak o_k(\bar \g)+\bar\delta \g_{k,0})\circ \f_{1,0}^{\otimes k}=0$. 
Also, it is easy to check $\Xi\in \CC^\ud_0$, since so do $\bar \f$, $\bar \g^+$ and $\bar \h^+$.
Thus, the $\Xi$ gives rise to a class $[\Xi]$ in $H( \CC^\ud_{k,0},\bar \delta)$.
Since $\bar\f_{1,0}$ is a quasi-isomorphism, we can find a $\bar \delta$-closed $\Delta \g \in \CC_{k,0}^\ud(\C',\C)$ so that
$
\Xi +\bar \delta \eta = \Delta \g\circ \f_{1,0}^{\otimes k} 
$
for some $\eta \in \CC_{k,0}^\ud(\C' ,\C)$. We define $\Delta \h\in \CC^\ud_{k,0} (\C,\C_\oi)$ by setting $\Delta \h (s)= 1\otimes s \eta$.
In special, we have $\eval^0\circ \Delta \h =0$ and $\eval^1 \circ \Delta \h=\eta$.

Ultimately, we claim the modified $\h'_{k,0}:=\h_{k,0} +  \bar \delta \Delta  \h$ and $\g'_{k,0}:= \g_{k,0}+ \Delta \g$ meet our needs. Firstly, both of them lie in $\CC^\ud_{k,0}$ by construction. Secondly, since $\Delta\g$ and $\bar \delta \Delta \h$ are $\bar\delta$-closed, $(\bar\h, \h'_{k,0})$ and $(\bar\g,\g'_{k,0})$ are still the desired $A_k$ homomorphism extensions by Theorem \ref{o_k-length-thm}.
Thirdly, the condition $\eval^0 \circ \bar \h^+=\id$ is not destroyed, and one can show the modified (\ref{Xi_length_filtration_eq}) is given by
$
\Xi'
:=
\eval^1\circ \h'_{k,0} - \g'_{k,0} \circ \f_{1,0}^{\otimes k} - \sum_{\ell< k} \g_{\ell,0} \circ (\f\otimes \cdots \otimes  \f)
=
\Xi+
\eval^1\circ (\bar \delta \Delta \h) - \Delta \g \circ \f_{1,0}^{\otimes k}
=
\Xi+\bar \delta \eta - \Delta \g \circ \f_{1,0}^{\otimes k}=0
$.
\end{proof}

\begin{proof}
[Proof of Theorem \ref{Whitehead-Leangth-filtration -thm}]
The base case for $\g_{1,0}, \h_{1,0}$ has been completed in (\ref{Whitehead-g_10-eq}), (\ref{Whitehead-h_10}), by Lemma
\ref{DA-g_10-h_10-lem}.
Note that $\eval^0\h_{1,0}=\id$ and $\eval^1\h_{1,0}=\g_{1,0}\f_{1,0}$ and both of them indeed live in $\CC_{1,0}^\ud$ by (\ref{g_10_h_10_one_eq}).
By iteratively applying Proposition \ref{Length-Whitehead-prop}, we obtain $\bar\g=(\g_{k,0})_{k\ge 1}$ and $\bar \h=(\h_{k,0})_{k\ge 1}$ in $\CC_0^\ud$ such that
$
\bar\g \diamond \bar \f\simud\id
$
via $\bar\h$.
It remains to show $\bar\f \diamond \bar\g\simud\id$ via some $\bar\h'\in \CC^\ud_0$. In fact, applying the above result to $\bar\g$ in place of $\bar\f$, we get some $\bar\f'\in\CC_0^\ud$ so that $\bar\f' \diamond \bar\g\simud \id$ via some $\bar{\mathfrak k}\in \CC^\ud_0$.
Finally, the argument in the proof of Lemma \ref{UD-simud-composition-lem} can be performed within the energy-zero part $\CC^\ud_0$; thus, we also obtain $\bar\f'\simud\bar\f' \diamond \bar\g \diamond \bar\f\simud\bar\f$ and $\bar\f \diamond \bar\g\simud\bar\f' \diamond \bar\g\simud\id$ via some operator system in $\CC_0^\ud$.
\end{proof}

\subsubsection{Obstruction for the energy filtration}

In Theorem \ref{Whitehead-Leangth-filtration -thm}, we constructed an ud-homotopy inverse $\bar\g$ for the reduction $\bar\f$ \textit{within the zero-energy components}. Our next step involves inductive constructions for the energy filtration. Unlike the length filtration, a notable distinction is that for a given $\beta \neq 0$, the divisor axiom equations (\ref{DivisorAxiom}) engage $\f_{k,\beta}$ across various $k\in\mathbb N$. 

Define $\CC_\beta^\ud \subset \CC_\beta$ as the subspace of $\mathbb N$-labeled operator systems $\varphi=(\varphi_{k,\beta})_{k\ge 0}$ such that:
\begin{itemize}
	\itemsep 1pt
\item[(E1)] 
$
\DA[\varphi]_{k,\beta}(b;x_1,\dots,x_k) = \partial \beta \cap b \cdot \varphi_{k,\beta}(x_1,\dots,x_{k})
$
for every $k\ge 0$ and divisor input $b$.
\item[(E2)] $\CU[\varphi]_{k,\beta}(\e; x_1,\dots,x_k)=0$ for every $k\ge 0$ and degree-zero input $\e$.
\item[(E3)] $\varphi_{k,\beta}(\dots, \one,\dots)=0$ for every $k$ and the unit $\one$.
\end{itemize}

The three conditions align with the divisor axiom, cyclical unitality, and unitality, respectively. It's noteworthy that (E1) and (E2) exclusively connect $\varphi_{k+1,\beta}$ to $\varphi_{k,\beta}$.
Next, we define
\begin{equation}
\label{delta_bar_g-CC_beta-eq}
\delta:=\delta_{\bar \g}: \CC_\beta \to \CC_\beta
\end{equation}
by sending an element $\varphi=(\varphi_{k,\beta})_{k\ge 0}$, with a fixed shifted degree $\deg' \varphi_{k,\beta}=p$ for all $k$, to an element $\delta_{\bar\g}(\varphi)$
so that the component $\big(\delta_{\bar\g}(\varphi)\big)_{k,\beta}$ in $\CC_{k,\beta}$ is given by (see (\ref{sign-p-twist-eq}) for the notations)
\[
\sum_{\ell \ge 1} \sum_{\substack{
		k_1+\cdots +k_\ell =k \\ 1\le i\le \ell
		}}
\m_{\ell,0}   (\g^{\# p}_{k_1,0}\otimes \cdots \otimes \g^{\# p}_{k_{i-1},0}\otimes \varphi_{k_i,\beta}\otimes \g_{k_{i+1}, 0}\otimes \cdots \otimes \g_{k_\ell, 0}) 
- (-1)^{p} \sum_{\substack{\lambda+\mu+\nu=k
}} \varphi_{\lambda+\mu+1,\beta}   (\id_\#^\lambda\otimes \m'_{\nu,0} \otimes \id^\mu)
\]
For simplicity, we may write 
$
\textstyle
\sum \bar\m  (\bar\g^{\# p} \otimes \cdots \bar\g^{\# p} \otimes \varphi \otimes \bar\g\otimes \cdots \otimes \bar\g)- (-1)^p \sum \varphi   (\id^\bullet_\#\otimes \bar\m'\otimes \id^\bullet)
$.

There are some useful observations as follows:
Firstly, the map $\delta_{\bar \g}$ has degree one in the sense that $\deg' \delta_{\bar\g}(\varphi)=p+1$ whenever $\deg' \varphi=p$.
Secondly, the map $\delta=\delta_{\bar\g}$ depends not only on $\bar\g$ but also on the $A_\infty$ algebras $\m$ and $\m'$.
Thirdly, because $\bar\g, \bar\m,\bar\m' \in\CC^\ud_0$, one can directly check that
\begin{equation}
\label{delta_DA_eq}
\DA[\delta \varphi]_{\bullet,\beta}(b;x_1,\dots,x_k)=\delta \Big( \DA[\varphi]_{\bullet,\beta}(b;x_1,\dots,x_k)\Big)
\end{equation}
as elements in $\CC_\beta$. Fourthly, if we can find some $k$ so that $\varphi_{i,\beta}=0$ for $1\le i\le k-1$, then
\begin{equation}
\label{delta_i_le_k-1_property_eq}
(\delta\varphi)_{k,\beta}=\bar\delta \big( \varphi_{k,\beta}\big)
\end{equation}

\begin{lem}
	\label{delta-delta=0-lem}
	The $\delta\equiv \delta_\g$ is a differential, that is,
$
	\delta \circ \delta =0
$.
\end{lem}

\begin{proof}
Pick up an element $\varphi=(\varphi_{k,\beta})_{k\ge 0}$ with fixed degree components, say $\deg' \varphi_{k,\beta}=p$.
Then, $\deg' \delta (\varphi)=\deg' \varphi+1$ by definition of $\delta:=\delta_{\bar\g}$. We compute
	\begin{align*}
	\delta\delta (\varphi)
	&
	=
	\sum \m \big(\g^{\# (p+1)}\cdots \g^{\#(p+1)} (\delta\phi) \g\cdots \g
	\big)
	+(-1)^{p+2}\sum (\delta\varphi) (
	\id_\#\otimes \m\otimes \id
	)\\
	&
	=
	\sum \m \big( \g^{\#(p+1)}\cdots \g^{\#(p+1)} \m(\g^{\#p}\cdots \g^{\#p}\varphi\g\cdots \g) \g \cdots \g \big)\\
	&
	+
	(-1)^{p+1}\sum \m \big (\g^{\#(p+1)}\cdots \g^{\#(p+1)} \varphi(\id_\#\otimes \m\otimes \id)\g\cdots\g \big)\\
	&
	+
	(-1)^{p+2}
	\sum \m( \g^{\# p}\cdots \g^{\# p}\varphi \g\cdots \g) \ (\id_\#\otimes \m\otimes \id)\\
	&
	+
	(-1)^{p+2}(-1)^{p+1}\sum \varphi \  (\id_\#\otimes \m\otimes \id) \ (\id_\#\otimes \m\otimes \id)
	\quad \ \ =
	:I+II+ III +IV
	\end{align*}
	where any $\id_\#$ or $\id$ actually refers to some $\id_\#^k=\id_\# \otimes \dots\otimes \id_\#$ or $\id^\ell =\id\otimes \cdots \otimes \id$. We compute:
	\begin{align*}
	I&
	=
	\sum \m(\id_\#\otimes \m\otimes \id)(\g^{\# p}\cdots \g^{\# p} \varphi \g \cdots \g)\\
	&
	-
	\sum
	\m(\g^{\# (p+1)}\cdots \g^{\# (p+1)}
	\m(\g^{\# p}\cdots\g^{\# p})
	\g^{\# p} \cdots \g^{\# p}\varphi\g\cdots \g)\\
	&
	-
	\sum
	\m(\g^{\# (p+1)}\cdots \g^{\# (p+1)} , (\id_\#\varphi), \g^\#\cdots \g^\# \m(\g\cdots\g)\g\cdots \g)
	\quad \ \
	=:I_1-I_2-I_3
	\end{align*}
	
Since $\g$ is an $A_\infty$ homomorphism, we have $\m(\g^{\#p}\cdots\g^{\#p})
=
(\m(\g\cdots\g))^{\# p}
=
(\g(\id_\#\otimes \m\otimes \id))^{\# p}
=
(-1)^p \id_{\#p}  \  \g \  (\id_\#\otimes \m\otimes \id)=(-1)^p\g^{\# p}(\id_\#\otimes \m\otimes \id)$
	and
	\[
	I_2=
	(-1)^p
	\sum
	\m(\g^{\# (p+1)}\cdots \g^{\# (p+1)}
	\g^{\# p} \  (\id_\#\otimes \m\otimes \id) \  \g^{\# p} \cdots \g^{\# p}\varphi\g\cdots \g)
	\]
	Additionally, notice that $\varphi^\#=\varphi(\id_\#,\dots,\id_\#)=(-1)^p\id_\# \circ  \varphi$ by (\ref{sign-p-twist-eq}), and we conclude
	\[
	I_3=(-1)^p\sum
	\m(\g^{\# (p+1)}\cdots \g^{\# (p+1)}, \varphi^\#, \g^\#\cdots \g^\# \g(\id_\#\otimes \m\otimes \id)\g\cdots \g)
	\]
	Since $I_1=0$, we derive
	$
	I+II=-(I_2+I_3)+II=(-1)^{p+1}\m(\g^{\# p}\cdots \g^{\# p} \varphi\g\cdots \g) (\id_\#\otimes \m\otimes \id)=-III
	$.
	Finally, the $A_\infty$ equation says $IV=0$, and hence we have $\delta\delta=0$.	
\end{proof}

Provided Lemma \ref{delta-delta=0-lem}, we can show an analog to Proposition \ref{delta_bar-properties-prop}. For simplicity, we will use the same symbol $\m$ for the various $A_\infty$ algebras involved.

\begin{prop}
	\label{delta_beta_property-prop}
Let $\beta\neq 0$ and and let $\bar\g$ be an $A_\infty$ homomorphism modulo $T^{E=0}$ in $\CC^\ud_0$.
\leavevmode
\begin{itemize}
\item[(i)]
The $\delta=\delta_{\bar\g}$ restricts to a differential on $\CC_\beta^\ud$, hence, there is a cohomology
\[
H(\CC^\ud_\beta, \delta_{\bar \g})
\]
 \item[(ii)] If $f$ is a cochain map of degree zero satisfying (\ref{DA-f_10-eq}), $f(\one)=\one$, and the property\footnote{Equivalently, this means $f$ can be viewed as an $A_\infty$ homomorphism in the trivial way}
$\m_{k,0} \circ f^{\otimes k} = f\circ \m_{k,0}$, then it induces a cochain map\footnote{The notation here will become clear soon in Proposition \ref{f0f1_cochain-map-beta-prop}}
 \[
 f_*:=(\id,f)_*: (\CC_\beta^\ud,\delta_{\bar\g} ) \to (\CC_\beta^\ud, \delta_{f\bar\g})
 \]
 sending $\varphi$ to $f\circ \varphi$. Moreover, if $f$ is a quasi-isomorphism which admits a right inverse $h$ in the sense that $f\circ h=\id$, then  $f_*$ is also a quasi-isomorphism.
 \item[(iii)] In the same condition of (ii), there is an induced cochain map
 \[
 \hat f:= (f,\id)_*: (\CC^\ud_\beta,\delta_{\bar\g})\to (\CC^\ud_\beta,\delta_{\bar\g \circ f})
 \]
sending $\varphi$ to $(\varphi_k \circ f^{\otimes k})_{k\ge 0}$. Moreover, if $f$ is a quasi-isomorphism with a left inverse $k$ in the sense that $k \circ f=\id$, then $\hat f$ is also a quasi-isomorphism.
 \end{itemize}
\end{prop}


\begin{proof}
\textbf{(i).}
By Lemma \ref{delta-delta=0-lem}, it suffices to show $\delta:=\delta_{\bar \g}$ maps $\CC^\ud_\beta$ into $\CC^\ud_\beta$. Since $\bar\m'$ and $\bar\g$ also have the unitalities, cyclical unitalities, and divisor axiom, it is routine to check $\delta(\varphi)\in \CC^\ud_\beta$.

\textbf{(ii).}
Note that $f\in\CC^\ud_{1,0}\subset \CC^\ud_0$, and one can then easily see that $f\bar\g \in \CC^\ud_0$ is also an $A_\infty$ homomorphism module $T^{E=0}$ by the condition of $f$. This implies that $\delta_{f\bar\g}$ is well-defined. Also, the degree condition on $f$ ensures that $\delta_{f\bar\g} (f\varphi)= f \delta_{\bar\g}(\varphi)$, so $f_*$ is a cochain map.
Next, suppose $f$ is a quasi-isomorphism with a right inverse $h$. Generally, the $h$ does not commute with $\delta$ like $f$.

Take $\varphi =(\varphi_{k,\beta})_{k\ge 0}\in \CC^\ud_\beta$ so that 
$\delta \varphi=0$ and $f_*\varphi=\delta\xi$ for some $\xi\in \CC_\beta^\ud$, and then we aim to show $\varphi$ is actually $\delta$-exact in $\CC^\ud_\beta$. As $f\circ h=\id$, we may assume
$
f_*\varphi=f\varphi=0
$ by replacing $\varphi$ by $\varphi- \delta h \xi$.
For simplicity, we will often omit $\beta$, for instance, $\varphi_k$ will represent $\varphi_{k,\beta}$, and $\m_k$ (resp.$\g_k$) will stand for $\m_{k,0}$ (resp.$\g_{k,0}$). Namely, we will use $(\cdot)_k$ to denote the component in both $\CC_{k,\beta}$ and $\CC_{k,0}$.
Initially, $\m_1\varphi_0=(\delta\varphi)_0=0$ and $f\varphi_0=(f\varphi)_0=0$. So $\varphi_0=\m_1 \eta_0$ for some $\eta_0$ as $f$ is a quasi-isomorphism of cochain complexes. Additionally, there is certain freedom: we may replace $\eta_0$ by any element in $\eta_0+ \ker \m_1$.
Using the facts that $f$ is a quasi-isomorphism and that $f h=\id$, we may choose $\eta_0$ so that $f\eta_0=0$.
(Indeed, we can first find $\zeta\in\ker \m_1$ so that $f\eta_0=f\zeta+\m_1 \zeta'$ for some $\zeta'$, then we simply replace $\eta_0$ by $\eta_0-\zeta- \m_1 h\zeta'$.) In particular, $(\varphi-\delta\eta_0)_0=\varphi_0-\m_1\eta_0=0$.
Suppose we have $\eta=(\eta_0,\eta_1,\dots, \eta_k,0,0,\dots)$ for $\eta_i\in\CC_{i,\beta}$
with the following \textit{induction hypothesis}

\begin{itemize}
	\itemsep 1pt
\item $\psi:=\varphi - \delta \eta$. Moreover, we require $ \psi_i=0$ ($0\le i\le k$);
\item $(f\eta)_i=0$ ($0\le i\le k$);
\item (E1) holds for every pair $(\eta_i,\eta_{i+1})$ ($0\le i\le k-1$)  and every divisor input $b$.
\item (E2) and (E3) hold for all $\eta_i$.
\end{itemize}

Since $\delta\varphi=0$, we have
$
\bar\delta \psi_{k+1}
=
(\delta \psi)_{k+1} 
=
 (\delta (\varphi-\delta\eta))_{k+1} =0
$.
Also, since $f\varphi=0$, we have
$f_* \psi_{k+1} = - f (\delta \eta)_{k+1}= - (\delta (f\eta) )_{k+1} =0$.
Because of $\varphi\in\CC_\beta^\ud$,
the property (\ref{delta_DA_eq}) implies that
$
\DA[\psi]_{\beta}(b)
=
\partial\beta\cap b\cdot \varphi - \DA[\delta \eta]_{\beta}(b)= \delta\big(
\partial\beta\cap b\cdot \eta-\DA[\eta]_{\beta}(b)
\big)
$ within $F^k\CC_\beta$.
By the third item of the induction hypothesis, $\partial\beta\cap b\cdot \eta-\DA[\eta]_{\beta}(b)$ restricts to zero in $F^{k-1}\CC_\beta$. Notice also that $\DA[\eta]_{\beta}(b)$ restricts to zero in $\CC_{k,\beta}$ as $\eta_{k+1}=0$.
In conclusion, it follows from (\ref{delta_i_le_k-1_property_eq}) that
$\DA[\psi]_{k,\beta}(b)
=
\bar\delta ( \partial \beta \cap b\cdot \eta- \DA[\eta]_\beta(b))_k
=
\partial\beta\cap b\cdot \bar\delta \eta_k$, i.e.
\begin{equation}
\label{differential_delta_beta-prop-eq1}
\psi_{k+1}(b,x_1,\dots,x_k)+\cdots+\psi_{k+1}(x_1,\dots,x_k,b)= \partial \beta \cap b\cdot (\bar\delta \eta_k) (x_1,\dots,x_k)
\end{equation}
Before we continue, we prove a crucial technical sublemma.
  \begin{sublemma}\label{sublemma-u_k+1}
  Fix $u=(u_0,\dots,u_k,0,0\dots)$ so that
   (E1) holds for $(u_i,u_{i+1})$ with $0\le i\le k-1$ and that (E2) and (E3) hold for each $u_i$ with $0\le i\le k$. Then there exists a $(k+1)$-multi-linear map $u_{k+1}$ so that (E1) also holds for $(u_k,u_{k+1})$ and that (E2) and (E3) also hold for $u_{k+1}$. Moreover,
  if $fu_i=0$ for all $0\le i \le k$ then $f u_{k+1}=0$; if $u_i f^{\otimes i}=0$ for all $0\le i \le k$ then $u_{k+1} f^{\otimes (k+1)}=0$
  \end{sublemma}
  \begin{proof}
  [Proof of the sublemma]
  We define 
  $(\partial\beta, x)=\partial\beta\cap x$ if $\deg x=1$ and define $(\partial\beta,x)$ to be zero if otherwise. We also introduce the following $N$-multi-linear map
  \[ \textstyle
  u_N^{(m)}(x_1,\dots,x_N)=\sum_{i=0}^{N-1}  
 u_{N-m}(x_{i+1},\dots, x_{i+N-m}) \cdot ( \partial \beta , x_{i+N-m+1} ) \cdots 
 ( \partial \beta , x_{i+N})
  \]
  where $N,m\in \mathbb N$ are so that $u_{N-m}$ is already given, i.e. $0\le N-m \le k$, and where the $x_j$'s are $N$ arbitrary inputs with the subscript is modulo $N$, i.e. $x_{j+N} \equiv x_j$.
  Observe that $u_N^{(0)}=u_N$ and that these maps are cyclical in the sense $u_N^{(m)}(x_1,\dots,x_N)=u_N^{(m)}(x_{i+1},\dots,x_{i+N})$ for any $i$. In particular, $u_N^{(N)}(x_1,\dots,x_N)=N(\partial\beta, x_1)\cdots(\partial\beta,x_N)\cdot u_0$.
We put
$
  \textstyle
  u_{k+1}
  :=
  \sum_{m=1}^{k} a_m u^{(m)}_{k+1}
  +
  a_{k+1} \
  u_{k+1}^{(k+1)}
$
  for undetermined coefficients $a_m$. Note that the $u_{k+1}$ depends only on $u_0, u_1,\dots,u_k$. We compute
  \begin{align*}
  &
  u^{(k+1)}_{k+1}(b,x_1,\dots,x_k)+\cdots + u^{(k+1)}_{k+1}(x_1,\dots,x_k, b) =\partial\beta\cap b \cdot (k+1) u_k^{(k)}(x_1,\dots,x_k)
  \\
  &
  u^{(m)}_{k+1}(b,x_1,\dots,x_k)+\cdots + u^{(m)}_{k+1}(x_1,\dots,x_k, b) 
  =
  \partial\beta\cap b \cdot \big( 
  m \ u_k^{(m-1)} (x_1,\dots,x_k)
  +    u_k^{(m)}(x_1,\dots,x_k) \big)
  \end{align*}
 for every $1\le m \le k$, and then 
 \begin{align*}
 	\textstyle
\sum_{i=1}^{k+1} u_{k+1} (x_1,\dots, x_{i-1}, b,x_i,\dots, x_k)
 =
 \partial\beta\cap b \big(
 a_1 u_k^{(0)}(x_1,\dots,x_k)
 +
 \sum_{m=1}^{k}
 \big(a_m+(m+1)a_{m+1}\big)
  u_k^{(m)} (x_1,\dots,x_k)
 \big)
 \end{align*}
 Choose $a_m=\frac{(-1)^{m-1}}{m!}$, then $a_1=1$ and $a_m+(m+1)a_{m+1}=0$.
 It follows that $(u_k,u_{k+1})$ satisfies (E1).
 Moreover, since this construction is cyclic, the cyclical unitality holds as well. In fact, we compute
 \begin{align*}
 &
 \sum_{j=1}^{k} u_{k+1}^{(m)}(x^\#_1,\dots,x^\#_{j-1}, \e,x_{j+1},\dots,x_{k+1} )\\
 &
 =
 \sum_{i=0}^{k-1}
 (-1)^{\epsilon_i}
 \sum_{i\le j\le i+k+1-m}
 u_{k+1-m} (x_{i+1}^\#,\dots,x_{j-1}^\#,\e,x_{j+1},x_{i+k+1-m})\cdot (\partial\beta, x_{i+k+2-m})\cdots (\partial\beta, x_{i+k+1})
 \end{align*}
where $\epsilon_i=\sum_{a=1}^i(\deg x_a+1)$ and we think $i,j\in\mathbb Z/ (k+1)\mathbb Z$. The condition $i\le j\le i+k+1-m$ required above comes from the fact that $(\partial\beta,\e)=0$. For each fixed $i$, the summation over $j$ gives zero, since $u_{k+1-m}$ satisfies (E2). Hence $u_{k+1}^{(m)}$ and also $u_{k+1}$ satisfy (E2) as we require.
As for (E3), recall that we have assumed $\beta\neq 0$, so $u_j(\dots,\one,\dots)=0$ even if $j=1$; furthermore, $(\partial\beta, \one)=0$ by definition, hence one can easily check that $u_{k+1}$ obeys (E3) as well. Finally, 
 the last statement is straightforward to check.
 Indeed, if $fu_i=0$ for $0\le i\le k$, then $ fu_{k+1}^{(m)}=0$ for $1\le m\le k+1$, thus $fu_{k+1}=0$. Similarly, if $u_if^{\otimes i}=0$ for $0\le i\le k$, then because $f$ satisfies (\ref{DA-f_10-eq}), we conclude $u_{k+1}^{(m)} f^{\otimes(k+1)}= (u\circ f)^{(m)}_{k+1}$ which also vanishes for $1\le m\le k+1$.
  \end{proof}
 
Back to the proof of Proposition \ref{delta_beta_property-prop} (ii). 
Applying Sublemma \ref{sublemma-u_k+1} to $\eta=(\eta_0,\dots,\eta_k,0,0,\dots)$ provides some $\chi_{k+1}$ so that $f\chi_{k+1}=0$, the (E2) and (E3) hold for $\chi_{k+1}$, and 
\begin{equation}
\label{differential_delta_beta-prop-eq2}
\chi_{k+1}(b,x_1,\dots,x_k) + \cdots + \chi_{k+1}(x_1,\dots,x_k,b)=\partial \beta \cap b \cdot \eta_k(x_1,\dots,x_k)
\end{equation}
for any divisor input $b$.
By (\ref{differential_delta_beta-prop-eq1}) and (\ref{differential_delta_beta-prop-eq2}), we see that $\psi_{k+1}-\bar\delta \chi_{k+1}\in \CC^\ud_{k+1,0}$. Observe that $\bar\delta (\psi_{k+1}-\bar\delta \chi_{k+1})=0$ and $f_*(\psi_{k+1}-\bar\delta \chi_{k+1})=0$.
By Proposition \ref{delta_bar-properties-prop} (ii), the $f_*$ is a quasi-isomorphism between the cochain complexes $(\CC^\ud_{k+1,0},\bar\delta)$. It follows that there exists some $\theta \in \CC^\ud_{k+1,0}$ so that
\begin{equation}
\label{theta_in_proof_eq}
\psi_{k+1}-\bar\delta \chi_{k+1} = \bar\delta \theta
\end{equation}
By construction, $f\theta \in \CC^\ud_{k+1,0}$ is $\bar\delta$-closed. Hence, there exists some $\bar\delta$-closed $\theta' \in\CC_0^\ud$ so that $f\theta'=f\theta+\bar\delta \alpha$ for some $\alpha\in\CC_0^\ud$.
Since $f\circ h=\id$, we have $f\bar\delta h\alpha =\bar\delta\alpha$. Then, replacing $\theta'$ by $\theta'-\bar\delta h \alpha$, one can require $\alpha=0$. So, $f\theta'=f\theta$.
Moreover, one can replace $\theta$ by $\theta-\theta'$ (which lives in $\theta+\ker \bar\delta$) without affecting (\ref{theta_in_proof_eq}). In summary, we may assume
$
f\theta=0
$
in addition to (\ref{theta_in_proof_eq}).
The inductive step can be completed by putting $\eta_{k+1}=\theta+\chi_{k+1}$, $\eta^+=(\eta_0,\dots,\eta_k,\eta_{k+1},0,0,\dots)=\eta+\eta_{k+1}$ and $\psi^+ =\varphi - \delta \eta^+$. We check the four conditions in the induction hypothesis for them as follows:
\begin{itemize}
	\itemsep 1pt
\item  $\psi^+_{k+1}= \big(\psi -\delta(\eta^+-\eta) \big)_{k+1}
=
\psi_{k+1} -\bar\delta \eta_{k+1}=0$ holds by (\ref{theta_in_proof_eq}), and $\psi^+_j=\psi_j=0$ for $j\le k$.
\item $f\eta_{k+1}=f\theta +f\chi_{k+1}=0$.
\item It remains to show (E1) holds for the new pair $(\eta_k,\eta_{k+1})$, which is true thanks to (\ref{differential_delta_beta-prop-eq2}).
\item 
The $\eta_{k+1}$ satisfies (E2) and (E3) since so so $\chi_{k+1}$ and $\theta$.
\end{itemize}
By induction, the proof of (ii) is now established. 
Finally, the item (iii) can be proved almost identically and we omit it. We complete the proof of Proposition \ref{delta_beta_property-prop}.
%
\end{proof}

\begin{prop}
	\label{f0f1_cochain-map-beta-prop}
Let $\bar\f^0=(\f^0_{k,0})_k$ and $\bar\f^1=(\f^1_{k,0})_k$ be two $A_\infty$ homomorphisms (modulo $T^{E=0}$) in $\UD$. 
There exists an induced cochain map
\begin{equation}
\label{f0f1_eq}
(\bar\f^0,\bar\f^1)_* :=(\bar\f^0,\bar\f^1)_*^{\bar\g}: (\CC^\ud_\beta, \delta_{\bar\g}) \to  (\CC^\ud_\beta, \delta_{\bar\f^1\circ\bar\g\circ\bar\f^0})
\end{equation}
Moreover, we have the associativity in the sense that 
$
(\bar\f^2,\bar\f^3)_* \circ (\bar\f^0,\bar\f^1)_*
=
(\bar\f^0 \diamond \bar\f^2, \bar\f^3\diamond \bar\f^1)_*$.
\end{prop}

\begin{proof}
We start with the definition: let $\varphi=(\varphi_{k,\beta})_k$ and $\deg' \varphi_{k,\beta}=p$ for a fixed $p$, we define $(\bar\f^0,\bar\f^1)^{\bar\g}_*\varphi$ to be the operator system $
\textstyle
\sum \bar\f^1 \circ (\bar \g^{\# p} \otimes \cdots \otimes \bar \g^{\# p} \otimes  \varphi \otimes \bar \g \otimes \cdots \otimes \bar \g ) \circ (\bar \f^0\otimes \cdots \otimes \bar \f^0)$.
If $\bar\f^0=\id$ or $\bar\f^1=\id$, we respectively have that
\begin{align*}
(\bar\f^0,\id)^{\bar\g}_* \varphi 
=
\varphi\circ(\bar\f^0\otimes \cdots \otimes \bar\f^0) 
\qquad \text{and} \qquad
(\id,\bar\f^1)^{\bar\g}_* \varphi
=
\bar\f^1\circ (\bar\g^{\# p} \otimes \cdots \otimes \bar \g^{\# p} \otimes \varphi \otimes \bar \g\otimes \cdots \otimes \bar \g) 
\end{align*}
Remark that although generally (\ref{f0f1_eq}) depends on $\bar\g$, the special case $(\bar\f^0,\id)_*^{\bar\g}$ does not.
Since $\deg' \bar\f^0=0$, we know $\deg'(\bar\f^0,\id)_* \varphi =\deg' \varphi =p$. A direct computation shows that
\begin{align}
\label{f0f1_cochain-composition-id-beta-eq}
(\id,\bar\f^1)^{\bar\g\circ\bar\f^0}_* \circ (\bar\f^0,\id)_* = (\bar\f^0,\bar\f^1)^{\bar\g}_*, 
\quad 
\text{and}
\quad 
(\bar\f^0,\id)^{\bar \f^1\circ\bar\g^0}_* \circ (\id,\bar\f^1)^{\bar\g}_* = (\bar\f^0,\bar\f^1)_*
\end{align}
Because of this decomposition, it suffices to prove Proposition \ref{f0f1_cochain-map-beta-prop} for $( \id, \bar \f^1)_* $ and $(\bar\f^0, \id)_*$ separately.
Firstly, as $\bar\f^0$, $\bar\f^1$ and $\bar\g$ satisfy the divisor axiom, we have
\[
\textstyle
\DA[(\bar\f^0,\id)_* \varphi]_{k,\beta}(b;x_1,\dots,x_k)
=\sum \DA[\varphi]_{k_1,\beta}\big(\f^0_{1,0}(b); \bar\f^0(x_1,\dots),\dots, \bar\f^0(\dots,x_k) \big)
\]
\[
\textstyle
\DA[(\id,\bar\f^1)_*\varphi]_{k,\beta} (b;x_1,\dots,x_k)
=
\sum  \bar\f^1\circ
\big(\bar\g^{\# p} (x_1,\dots),\dots, \DA[\varphi]_{k_1,\beta}(b; \dots),\dots,\bar\g(\dots,x_k) \big)
\]
Hence, we have the desired divisor axiom.
Similarly, we can show the unitality and the cyclical unitalities.
To sum up, both $(\id,\bar\f^1)_*$ and $(\bar\f^0,\id)_*$ map $\CC^\ud_\beta$ into $\CC^\ud_\beta$.
In the next, we want to check
\begin{equation}\label{f0f1_cochain-map-beta-prop-eq}
\delta (\bar\f^0,\bar\f^1 )_* (\varphi) = (\bar\f^0,\bar\f^1)_* (\delta \varphi) \quad \text{or more precisely} \quad
\delta_{\bar\f^1 \diamond \bar\g \diamond \bar\f^0} (\bar\f^0,\bar\f^1)_*( \varphi) = (\bar\f^0,\bar\f^1)_* (\delta_{\bar\g} \varphi)
\end{equation}
Without loss of generality, we assume $\deg'\varphi =p$. Consider
\begin{align*}
F_1:=
&
 \bar\f^1
 \circ 
 (\id^\bullet_\#\otimes \bar\m \otimes \id^\bullet )
  \circ 
 ( \bar\g^{\# p} \otimes \cdots \otimes \bar\g^{\# p} \otimes \varphi \otimes \bar \g \otimes \cdots \bar \g) 
 \circ 
 (\bar \f^0\otimes \cdots \bar\f^0) \\
F_2:=
&
\bar\f^1
\circ
(\bar\g^{\# p} \otimes \cdots \otimes \bar\g^{\# p} \otimes \varphi \otimes \bar \g \otimes \cdots \bar \g) 
\circ
(\id^\bullet_\#\otimes \bar\m \otimes \id^\bullet ) 
\circ
(\bar \f^0\otimes \cdots \bar\f^0)
\end{align*}
and put $F=F_1-F_2$.
On the one hand, because the $\bar\f^1$ and $\bar\f^0$ are $A_\infty$ homomorphisms, one can show that the $F$ exactly agrees with the left side of (\ref{f0f1_cochain-map-beta-prop-eq}). On the other hand, as $\bar\g$ is an $A_\infty$ homomorphism, one can show the $F$ agrees with the right side of (\ref{f0f1_cochain-map-beta-prop-eq}). The proof of (\ref{f0f1_eq}) is complete.
Finally, it is clear that $(\bar\f^2,\id)_* (\bar\f^0,\id)_*=(\bar\f^2 \diamond \bar\f^0,\id)_*$ and $(\id,\bar\f^3)_*(\id,\bar\f^1)_*=(\id,\bar\f^3 \diamond \bar\f^1)_*$, then the associativity in general, i.e. $
(\bar\f^2,\bar\f^3)_* \circ (\bar\f^0,\bar\f^1)_*
=
(\bar\f^0 \diamond \bar\f^2, \bar\f^3 \diamond \bar\f^1)_*$, can be proved by the decomposition in (\ref{f0f1_cochain-composition-id-beta-eq}).
\end{proof}

By declaring
$
F^E \CC := \prod_{\beta\in \G; E(\beta) < E} \CC_\beta = \prod_{k\in\mathbb N} \prod_{E(\beta)<E} \CC_{k,\beta}
$
we obtain the \textit{energy filtration} on $\CC$.
We remark that since it may happen that $\partial \beta \cap b \neq \partial \beta'\cap b$ even if $E(\beta)=E(\beta')$, we cannot directly follow \cite{FOOOBookOne}. This gives one reason why we have to consider a copy $\CC_{\beta}$ separately, one for each $\beta\in\G$ (Definition \ref{CC_defn}).
For every $\B\in \G$, we define
\begin{equation}
\label{F^B_CC-eq}
\textstyle
F^\B \CC = F^{E(\B)} \CC\times \CC_\B
=
 \prod_{E(\beta)< E(\B)}  \CC_\beta \times \CC_\B
\end{equation}
Similarly, we can define $F^E\CC^\ud$ and $F^\B\CC^\ud$ adding the conditions of unitality, cyclical unitality, and divisor axiom.
Now, in analogy to Theorem \ref{o_k-length-thm}, we can prove the following:

\begin{thm}\label{O_beta-energy-thm}
Fix $\B\in \G$ with $E(\B)>0$, and fix $(C_1,\m'), (C_2,\m)\in\Obj\UD$.
Suppose $\g=(\g_{k,\beta})_{k,\beta}\in F^{E(\B)}\CC^\ud(C_1,C_2)$ is an $A_\infty$ homomorphism modulo $T^{E(\B)}$ so that (\ref{DA-f_10-eq}) holds for $\g_{1,0}$.
Then, there exists a $\delta_{\bar \g}$-closed
$
\mathfrak o_{\B}(\g)
\in \CC^\ud_\B(C_1,C_2)
$
such that its cohomology $[\mathfrak o_\B (\g)]$ vanishes if and only if
$\g$ can be extended to an $A_{\infty,\B}$ homomorphism $\g^+=(\g,\g_\B)$. Moreover, $\mathfrak o_\B(\g)+ \delta_{\bar\g} (\g_\B)=0$.
\end{thm}

\begin{proof}
Define the $k$-th component (or more precisely, the $\CC_{k,\B}$-component) of $\mathfrak o_\B(\g)$ to be
\begin{equation}\label{O_beta-eq}
	 \hspace{1em}
	 \sum_{\substack{
		0=j_0\le\cdots\le j_\ell=k
		\\
		\beta_0+\beta_1+\cdots +\beta_\ell=\B
		\\
		\forall i\neq 0: \beta_i \neq \B
		}}
\m_{\ell,\beta_0} \circ (\g_{j_1-j_0,\beta_1}\otimes \cdots \otimes \g_{j_\ell-j_{\ell-1},\beta_\ell}) 
- \sum_{\substack{\lambda+\mu+\nu=k \\ \beta'+\beta''=\B
\\
\beta'\neq \B
}} \g_{\lambda+\mu+1,\beta'} \circ (\id_\#^\lambda\otimes \m'_{\nu,\beta''} \otimes \id^\mu)
		\end{equation}
Firstly, parallel to the proof of Theorem \ref{o_k-length-thm}, we can show $\mathfrak o_\B(\g)\in \CC^\ud_\B$ using the conditions that $\m'$, $\m$ and $\g$ are all contained in $\CC^\ud$.
Next, being an $A_{\infty,\B}$ homomorphism exactly means that
\begin{equation}
\label{o_B(g)+delta(g)=0-eq}
\mathfrak o_\B(\g)+ \delta_{\bar\g} (\g_\B)=0
\end{equation}
Suppose we could find $\g_\B$ to solve this, and we would like to study the degrees beforehand. In reality, due to (\ref{O_beta-eq}) above, the degree of an arbitrary term in $\mathfrak o_\B(\g)$ can be computed as follows:
\[
\textstyle
\deg \m_{\ell,\beta_0} + \sum \deg \g_{j_i - j_{i-1}, \beta_i}=2-\ell-\mu(\beta_0)+\sum_{i=1}^\ell \big(1-(j_i-j_{i-1})-\mu(\beta_i) \big)
= 2-k-\mu(\B)
\]
\[
\textstyle \deg \g_{\lambda+\mu+1,\beta'}+ \deg m'_{\nu,\beta''}=1-\lambda-\mu-\mu(\beta')+ 2-\nu-\mu(\beta'')=2-k-\mu(\B)
\]
This is exactly what we desire.
Besides, since the $\delta_{\bar\g}$ in (\ref{delta_bar_g-CC_beta-eq}) has degree one, the degree of $\g_{k,B}$ must equal to $1-k-\mu(\B)$ as we desire.
Beware that we use the `un-shifted' degree here.
In terms of shifted degrees, we actually have $\deg' \mathfrak o_\B(\g) =1-\mu(\B)=1 (\text{mod} \ 2)$ and $\deg' \g_\B=-\mu(\B)=0 (\text{mod} \ 2)$. Now, it remains to check $\mathfrak o_\B(\g)$ is $\delta$-closed. By (\ref{delta_bar_g-CC_beta-eq}), we compute
\begin{align*}
\delta (\mathfrak o_\B(\g))
&
=
\textstyle
\sum\m_{r,0} \circ(\bar\g^\#\otimes \cdots \otimes \bar \g^\# \otimes \m_{\ell,\beta_0} (\g_{\beta_1}\otimes \cdots \otimes  \g_{\beta_\ell})\otimes \bar \g \otimes \cdots\otimes \bar\g) \\
&
+
\textstyle
\sum \m_{\ell,\beta_0} \circ (\g_{\beta_1} \otimes \cdots \otimes \g_{\beta_\ell}) \circ (\id_\#^\bullet \otimes \m'_{s,0}\otimes \id^\bullet) \\
&
-
\textstyle
\sum \m_{r,0} \circ (\bar\g^\# \otimes \cdots \otimes \bar\g^\# \otimes  \g_{\beta'}\circ (\id_\#^\bullet \otimes \m'_{\nu,\beta''} \otimes \id^\bullet ) \otimes \bar \g \otimes \cdots \otimes \bar\g) \\
&
-
\textstyle
\sum
\g_{\beta'} \circ (\id_\#^\bullet\otimes \m'_{\nu, \beta''}\otimes \id^\bullet ) \circ (\id_\#^\bullet \otimes \m'_{s,0} \otimes \id^\bullet )
\qquad \quad
=: 
\textstyle
\Delta_1+\Delta_2-\Delta_3-\Delta_4
\end{align*}
As in (\ref{O_beta-eq}), we need the conditions like $\beta_i\neq \B$ ($i\neq 0$) and $\beta'\neq \B$ (or equivalently $\beta''\neq 0$) in the summations, but for simplicity it may be helpful to add a ghost term $\g_\B=0$.
Now, we consider
\[
\textstyle
\Delta:=\sum_{\gamma+\gamma'+\beta_1+\cdots+\beta_r=\B} \m_{r,\gamma} \circ (\g_{\beta_1}\otimes \cdots \otimes \g_{\beta_r} ) \circ ( \id^\bullet_\#\otimes \m'_{s,\gamma'} \otimes \id^\bullet)
\]
First, separating the terms with $\gamma\neq 0$ and $\gamma=0$ and using that $\g$ is an $A_\infty$ homomorphism modulo $T^{E(\B)}$, we obtain $\Delta=\Gamma_1+\Gamma_2$ where
$
\Gamma_1
:=
\textstyle
\sum_{\gamma\neq 0} \m_{r,\gamma}\circ (\id^\bullet_\#\otimes \m\otimes \id^\bullet) \circ (\g\otimes \cdots \otimes \g)$ and
$\Gamma_2:= 
\textstyle
\sum \m_{r,0} \circ (\id^\bullet_\# \otimes \m\otimes \id^\bullet) \circ (\g\otimes \cdots \otimes \g) -\Delta_1+\Delta_3$.
Next, separating the terms with $\gamma'\neq 0$ and those with $\gamma'=0$, we also see that $\Delta=\Gamma'_1+\Gamma'_2$, where
$
\Gamma'_1:=
\textstyle
 \sum \g \circ (\id^\bullet_\# \otimes \m'\otimes \id^\bullet) \circ (\id^\bullet_\# \otimes \m' \otimes \id^\bullet) - \Delta_4$ and $\Gamma'_2:=
\Delta_2$.
In the end, the $A_\infty$ equations for $\m$ and $\m'$ tells respectively that $\Gamma_1+\Gamma_2=-\Delta_1+\Delta_3$ and $\Gamma'_1+\Gamma'_2=\Delta_2-\Delta_4$. So $\delta (\mathfrak o_\B(\g))= (\Delta_1-\Delta_3)+(\Delta_2-\Delta_4)=-\Delta +\Delta=0$.
\end{proof}

\begin{lem}\label{O_beta-composition-lem}
Let $\f^0$ and $\f^1$ be two $A_{\infty,\B}$ homomorphisms in $\UD$. If $\g\in F^{E(\B)}\CC^\ud$ is any  $A_\infty$ homomorphism modulo $T^{E(\B)}$, then so is $\f^1 \diamond \g \diamond \f^0$, and we have
\[
[\mathfrak o_{\B} (\f^1 \diamond \g \diamond \f^0)]
=
\left[
(\bar\f^0,\bar\f^1)_* \mathfrak o_{ \B}(\g)
\right]
=(\bar\f^0,\bar\f^1)_* [\mathfrak o_{ \B}(\g)]
\]
where $\bar\f^0$ and $\bar\f^1$ are the reductions. Moreover, $
\mathfrak o_\B(\f^1 \diamond \g \diamond \f^0)-(\bar\f^0,\bar\f^1)_* \mathfrak o_\B (\g)= \delta( (\f^1\diamond \g\diamond \f^0)_\B)$.
\end{lem}

\begin{proof}
The statement relies on Proposition \ref{f0f1_cochain-map-beta-prop}; it is analogous to Lemma \ref{o_k-composition-lem}.
Clearly, $\f^1\diamond \g\diamond \f^0$ is also an $A_\infty$ homomorphism modulo $T^{E(\B)}$.
As in the proof of Theorem \ref{UD_prime-Composition-DA-thm}, one can see $\f^1\diamond \g\diamond \f^0$ restricts to an element in $ F^{E(\B)}\CC^\ud$ as well. Hence, due to Theorem \ref{O_beta-energy-thm}, the left side can be defined. Let's extend $\g$ to the whole $\CC$ by adding zero ghost terms.
Define $\h:= \f^1\diamond \g\diamond \f^0$, and then $\bar\h =\bar\f^1\diamond \bar\g\diamond \bar \f^0$. We consider
$
\textstyle
\mathfrak P:=\sum \bar \m \circ (\h\otimes \cdots \otimes \h)- \sum \h \circ (\id_\#^\bullet\otimes \bar\m \otimes \id^\bullet ) \in \CC_\B
$.
First, it is easy to see that $\mathfrak P=\mathfrak o_{\B}(\h)+ \delta_{\bar\h} (\h_\B)$, where $\h_\B\in \CC^\ud_\B$. Next, we expand $\mathfrak P$:
\begin{align*}
\mathfrak
P=
&
\Scale[0.93]{
\sum \f^1\circ (\id^\bullet_\#\otimes \m\otimes \id^\bullet) \circ (\g\otimes \cdots\otimes \g) \circ (\f^0\otimes \cdots\otimes \f^0)
-
\sum \f^1\circ (\g\otimes \cdots \otimes \g) \circ (\id^\bullet_\#\otimes \m\otimes \id^\bullet) \circ (\f^0\otimes \cdots\otimes \f^0)
} \\
=
&
\textstyle
\sum \bar\f^1\circ \Big( \bar \g^\#\otimes \cdots \otimes \bar\g^\#\otimes \big[ \m(\g\otimes \cdots\otimes \g)- \sum \g( \id^\bullet_\#\otimes \m\otimes \id^\bullet) \big] \otimes \bar\g\otimes \cdots\otimes \bar\g \Big) \circ (\bar\f^0\otimes \cdots \otimes \bar\f^0)
\end{align*}
Here the terms in the square bracket are elements in $\CC_\B$; they agree with $\mathfrak o_\B(\g)$. So, we also obtain
$
\mathfrak
P=(\bar\f^0,\bar\f^1)_* \mathfrak o_\B(\g)
$.
By comparison, the proof is now complete.
\end{proof}

\subsubsection{Extension for the energy filtration}

We are now at the stage to finally prove the most general Theorem \ref{Whitehead-full-thm}.
Recall that using the reduction $\bar\f$ of the fixed $\f\in \Hom_\UD((\C,\m),(\C',\m'))$, we have constructed $\bar\g$ and $\bar\h$ in Theorem \ref{Whitehead-Leangth-filtration -thm}.
They serve as the base case in the energy-zero level.
Next, we perform the inductive steps for the energy filtration by the following analogy of Proposition \ref{Length-Whitehead-prop}.

\begin{prop}
\label{Energy-Whitehead-prop}
Let $\B\in\G$ is so that $E(\B)>0$. Suppose
\[ \g=(\g_{k,\beta})\in F^{E(\B)} \CC^\ud(\C',\C), \quad \text{and}\quad 
\h=(\h_{k,\beta}) \in F^{E(\B)}\CC^\ud(\C,\C_\oi)
\]
are two $A_\infty$ homomorphisms modulo $T^{E(\B)}$ such that
\[
\eval^0  \h|_{F^{E(\B)}\CC}=\id, \quad \text{and} \quad \eval^1  \h|_{F^{E(\B)}\CC}= \g\diamond \f
\]
Then, there exists two $A_{\infty,\B}$ homomorphisms $\g^+=(\g,\g_\B) \in F^\B\CC^\ud$ and $\h^+=(\h,\h_\B) \in F^\B\CC^\ud$ extending the $\g$ and $\h$ and satisfying the following property:
\begin{equation}
\label{Energy_Whitehead_extension_eq}
\eval^0  \h^+ |_{F^\B\CC}=\id, \quad \text{and} \quad \eval^1   \h^+|_{F^\B\CC} = \g^+ \diamond \f
\end{equation}
\end{prop}

\begin{proof}
We think of $\g$ and $\h$ as elements in $F^\B\CC$ (\ref{F^B_CC-eq}) by setting zeros in the component $\CC_\B$.
By Lemma \ref{O_beta-composition-lem}, we have
$
(\id,\eval^0)_*[\mathfrak o_\B (\h)]=
[\mathfrak o_\B(\eval^0  \h)]=[\mathfrak o_\B(\id)]=0$; by Proposition \ref{delta_beta_property-prop} (ii), we have $[\mathfrak o_\B (\h)]=0$. By Theorem \ref{O_beta-energy-thm}, we have some $\alpha\in \CC_\B^\ud$ so that $\mathfrak o_\B(\h)+ \delta\alpha =0$.

To ensure (\ref{Energy_Whitehead_extension_eq}), we cannot set $\h_\B=\alpha$ directly, and we need a slight modification of $\alpha$ as before.
Firstly, the condition $\eval^0  \h|_{F^E(\B)\CC}=\id$ implies that $\mathfrak o_\B(\eval^0  \h)=\mathfrak o_\B(\id)=0$ by (\ref{O_beta-eq}). Then, due to Lemma \ref{O_beta-composition-lem}, we have $\mathfrak o_\B(\eval^0  \h)-(\id,\eval^0)_* \mathfrak o_\B(\h)=0$, and thus
 $(\id,\eval^0)_*\mathfrak o_\B(\h)=0$. 
Besides, using Proposition \ref{delta_beta_property-prop} yields that $(\id,\eval^0)_*\delta \alpha= \delta (\id,\eval^0)_*\alpha= \delta (\eval^0  \alpha)$. So, we have
$
0=(\id,\eval^0)_*(\mathfrak o_\B(\h)+\delta\alpha) = \delta (\eval^0  \alpha)
$.
Now, we define:
\[
\h_\B=\alpha - \incl \eval^0   \alpha
\]
Then, we have $\eval^0  \h_\B= 0$, thus, the $\h^+=(\h,\h_\B)$ satisfies the first half of (\ref{Energy_Whitehead_extension_eq}).
Besides, due to (\ref{delta_bar_g-CC_beta-eq}), we can show that $\delta \h_\B= \delta \alpha - \incl \delta \eval^0   \alpha=\delta \alpha$, and hence $\mathfrak o_\B(\h)+\delta \h_\B=0$.
In other words, setting $\h^+=(\h,\h_\B)$ supplies an $A_{\infty,\B}$ extension of $\h$.

The composition $\eval^1  \h^+$ is also an $A_{\infty,\B}$ homomorphism which is exactly an extension of $\g\diamond \f$.
It follows that $\mathfrak o_\B(\g \diamond \f)=\mathfrak o_\B(\eval^1  \h)=\eval^1  \mathfrak o_\B(\h)= - \eval^1 \delta (\h_\B) = - \delta \eval^1 \h_\B$. Meanwhile, by Lemma \ref{O_beta-composition-lem} and Proposition \ref{delta_beta_property-prop} (iii), we also conclude
$0=[\mathfrak o_\B (\g\diamond \f)] = (\f,\id)_*[\mathfrak o_\B(\g)]$
and $ [\mathfrak o_\B(\g)]=0$.
By Theorem \ref{O_beta-energy-thm}, the $\g$ can be extended to some $\g^+=(\g,\g_\B)$ so that $\mathfrak o_\B(\g)= -\delta \g_\B$. Unfortunately, we cannot ensure the following vanishes:
\[
\Pi:= (\eval^1 \h^+ - \g^+ \diamond \f)_\B =\eval^1  \h_\B - (\bar\f,\id)_* \g_\B- \textstyle \sum_{\beta\neq \B} \g_{\beta} \circ (\f\otimes \cdots \otimes \f) 
\]
So, we need to further modify $\g_\B$ and $\h_\B$.
To begin with, observe that since $\f,\g^+,\h^+\in \CC^\ud$, we see that $\Pi\in \CC^\ud_\B$.
We compute
$
\delta \Pi= -\eval^1  \mathfrak o_\B(\h) + (\bar\f,\id)_* \mathfrak o_\B(\g) - \delta( (\g\diamond \f)_\B )
=
 -\mathfrak o_\B(\g\diamond \f) + (\bar\f,\id)_* \mathfrak o_\B(\g) - \delta( (\g\diamond \f)_\B )
$
which vanishes precisely owe to Lemma \ref{O_beta-composition-lem}.

Next, we claim that $(\bar\f,\id)_*$ is a quasi-isomorphism. 
In fact, by Proposition \ref{delta_beta_property-prop}(iii), $(\incl,\id)_*$ is a quasi-isomorphism; so is $(\eval^i,\id)_*$ due to the fact $\eval^i\circ\incl=\id$ and Proposition \ref{f0f1_cochain-map-beta-prop}.
Observe that $(\id,\id)_*=(\bar\h,\id)_*(\eval^0,\id)_*$ and that $(\bar\f,\id)_*(\bar\g,\id)_*=(\bar\h,\id)_*(\eval^1,\id)_*$.
By Theorem \ref{Whitehead-Leangth-filtration -thm}, there is another ud-homotopy $\bar\h'$ between $\bar\f \diamond \bar\g$ and $\id$; we similar conclude that $(\id,\id)_*=(\bar\h',\id)_*(\eval^0,\id)_*$ and $(\bar\g,\id)_*(\bar\f,\id)_*=(\bar\h',\id)_* (\eval^1,\id)_*$.
To sum up, $(\bar\f,\id)_*$ is a quasi-isomorphism.

Now, by the above claim, one can find some $\Delta\g\in \CC^\ud_\B(\C',\C)$ so that $\delta \Delta \g=0$ and $\Pi+\delta \eta= (\bar\f,\id)_* \Delta\g$ for some $\eta\in \CC^\ud_\B$. Define $\Delta \h \in \CC^\ud_\B(\C,\C_\oi)$ by $\Delta \h (s) =1\otimes s\eta$; we have $\eval^0   \Delta\h=0$ and $\eval^1  \Delta\h=\eta$.
Finally, we define $\h'_\B:=\h_\B+ \delta \Delta \h$ and $\g'_\B=\g_\B+ \Delta\g$.
Both are contained in $\CC^\ud_\B$ by construction, and these modifications also given the $A_{\infty,\B}$ extensions since $\delta \h'_\B=\delta \h_\B$ and $\delta \g'_\B=\delta \g_\B$. It remains to show the new extensions $\h^{\prime  +}=(\h, \h'_\B)$ and $\g^{\prime +}= (\g,\g'_\B)$ satisfy (\ref{Energy_Whitehead_extension_eq}). In fact, the first relation is clearly preserved, and the second holds, since
$
\Pi':= (\eval^1  \h^{\prime +} -\g^{\prime +} \diamond \f)_\B = \Pi+ \eval^1 (\delta \Delta \h) - (\bar \f, \id)_* \Delta\g = \Pi+ \delta \eta- (\bar\f,\id)_*\Delta\g=0$.
\end{proof}

\begin{proof}
[Proof of Theorem \ref{Whitehead-full-thm}]
Denote by $\f_\beta$ the component of $\f$ in $\CC_\beta$, and we define
$
\mathsf G_\f:=\{ \beta\in \G \mid \f_\beta\neq 0\}
$.
One can similarly define $\mathsf G_\m$ and $\mathsf G_{\m'}$ for the $A_\infty$ algebras $\m$ and $\m'$. Now, we put
\begin{equation}
\label{G_gapped_whitehead_proof}
\mathsf G=\mathbb N\cdot \mathsf G_\f \cup \mathsf G_\m \cup \mathsf G_{\m'}
\end{equation}
By the gappedness, the image $E(\mathsf G)$ under the energy morphism $E(\cdot)$ is a discrete subset of $[0,\infty)$, say,
$ E(\mathsf G) =\{0=\lambda_0<\lambda_1<\lambda_2<\cdots<\lambda_i<\lambda_{i+1}<\cdots\}
$.
Every $E^{-1}(\lambda_i) \cap \mathsf G$ is finite and $E^{-1}(0)\cap \mathsf G=\{0\}$.
We aim to inductively construct $\g$ and $\h$ within the subspace
\[
\CC^{(\infty)}:= \prod_{j=0}^\infty \prod_{E(\beta)=\lambda_j} \CC^\ud_\beta \qquad \subset \CC_\G
\]
Suppose we already have two operator systems $\g^i$ and $\h^i$ in $\CC^{(i)}:=\prod_{j=0}^i \prod_{\beta\in\mathsf G:E(\beta)=\lambda_j} \CC^\ud_\beta$
so that they are $A_{\infty,\beta}$ homomorphisms for every $\beta\in \G$ with $E(\beta)=\lambda_i$ and the following identities hold in the space $ F^{\lambda_i}\CC \times \prod_{\beta\in\mathsf G: E(\beta)=\lambda_i}\CC_\beta$:
\begin{equation}
\label{induction_eval_g_h_energy_filtration_eq}
\eval^0  \h^i =\id, \quad \text{and}\quad \eval^1  \h^i=\g^i \diamond \f
\end{equation}
By Theorem \ref{Whitehead-Leangth-filtration -thm}, we have the initial step for $i=0$.
Suppose the $\g^i$ and $\h^i$ are given as above.

We first claim the two identities in (\ref{induction_eval_g_h_energy_filtration_eq}) actually hold in the larger space $F^{\lambda_{i+1}}\CC$.
In fact, pick an arbitrary $\beta$ with $\lambda_i<E(\beta)<\lambda_{i+1}$; then $\beta\not\in \mathsf G$. Then, since the $\h^i$ lives in $\CC^{(i)}$, we have $(\eval^{0,1}  \h^i)_\beta=\eval^{0,1}  \h^i_\beta=0$.
First, it is clear that $(\id)_\beta=0$, and so $\eval^0  \h^i=\id$ in $F^{\lambda_{i+1}}\CC$.
Second, $(\g^i\diamond \f)_\beta=\sum \g^i_{\ell,\beta_0}\circ(\f_{\beta_1}\otimes \cdots \otimes \f_{\beta_\ell})$ must vanish, otherwise $\beta=\beta_0+\beta_1+\cdots+\beta_\ell\in \G$. So, $\eval^1  \h^i=\g^i \diamond \f$ in $F^{\lambda_{i+1}}\CC$ as well.
We next claim that the $\g^i$ and $\h^i$ can be viewed as $A_\infty$ homomorphisms modulo $T^{\lambda_{i+1}}$.
Indeed,
pick $\beta$ with $\lambda_i<E(\beta)<\lambda_{i+1}$; then $\beta\notin\mathsf G$.
Suppose one can choose $\beta_i$'s so that $\beta=\sum\beta_i$ and $\m_{\ell,\beta_0} \circ (\g^i_{\beta_1}\otimes \cdots\otimes \g^i_{\beta_\ell})$ is nonzero. Then, by induction hypothesis, we have $\beta_i\in \mathsf G$ for each $i$. It follows that $\beta=\sum\beta_i \in \mathsf G$, a contradiction.
Suppose one can choose $\beta_1$ and $\beta_2$ so that $\beta=\beta_1+\beta_2$ and $\g^i_{\beta_1} (\id^\bullet_\# \otimes \m'_{\beta_2} \otimes \id^\bullet)$ is nonzero, then $\beta_1\in \mathsf G$ by induction hypothesis. Since $\beta_2\in \mathsf G_{\m'}$, we deduce $\beta=\beta_1+\beta_2\in \mathsf G$, a contradiction. Thus $\m \diamond \g^i-\g^i \{ \m'\}$ is zero not only in $F^{\lambda_i}\CC$ but also in $\CC_\beta$ for any $\beta$ with $\lambda_i<E(\beta)<\lambda_{i+1}$.
Namely, the $\g^i$ is an $A_\infty$ homomorphism modulo $T^{\lambda_{i+1}}$.
Similarly, one can show so is $\h^i$.

Now, it is legitimate to apply Proposition \ref{Energy-Whitehead-prop} for every class in the finite set $\mathsf G\cap \{\B\mid E(\B)=\lambda_{i+1}\}$, thereby obtaining the extensions $\g^{i+1}$ and $\h^{i+1}$ in $\CC^{(i+1)}$ of $\g^i$ and $\h^i$ respectively. By construction, they also satisfy the induction hypothesis.
Ultimately, we obtain $\g$ and $\h$ in $\CC^{(\infty)}$ so that $\g \diamond \f$ is ud-homotopic to $\id$ via $\h$.
Moreover, this construction exactly tells the set $\{\beta\mid \g_\beta\neq 0 \ \text{or}\ \h_\beta\neq 0\}$ must be contained in $\mathsf G$; thus, the $\g$ and $\h$ satisfy the condition (II-5) (Definition \ref{UD_defn}).
The conditions (II-1) (II-2) (II-3) for $\g$ and $\h$ hold by the definition of $\CC^\ud$, and the (II-4) holds by Lemma \ref{DA-g_10-h_10-lem}.

In summary, we can find $\g,\h\in\Mor\UD$ so that $\g\diamond \f$ is ud-homotopic to $\id$ via $\h$.
It remains to further prove $\f\diamond \g$ is also ud-homotopic to $\id$ as well. Applying the whole argument again to the $\g$ (instead of $\f$), there exist some $\f' , \h' \in\Mor\UD$ so that $\f'\diamond \g \simud \id$ via $\h'$.
Then, using Lemma \ref{UD-simud-composition-lem} we conclude $\f'\simud\f'\diamond(\g\diamond \f)= (\f'\diamond \g)\diamond \f \simud \f$, and thus $\f\diamond \g\simud \f'\diamond \g\simud\id$.
\end{proof}

\begin{rmk}
	\label{Gapped_Whitehead_preserve_rmk}
	There is a useful observation in the above proof: The ud-homotopy inverse $\g$ of $\f$ satisfies that the set
	$
	\mathsf G_{\g} := \{ \beta\mid \g_\beta\neq 0\} $
	is contained in the set $\mathsf G=\mathbb N\cdot \mathsf G_\f\cup \mathsf G_\m\cup \mathsf G_{\m'}$ (\ref{G_gapped_whitehead_proof}).
\end{rmk}

\subsection{Homological perturbation}
\label{S_homological_perturbation}

A \textit{ribbon tree} is a tree $\T$ with an embedding $\T \xhookrightarrow{} \mathbb D^2\subset \mathbb C$ such that a vertex $v$ has only one edge if and only if $v$ lies in the unit circle $\partial \mathbb D^2$. Such a vertex is called an exterior vertex, and any other vertex is called an interior vertex. The set of all exterior (resp. interior) vertices is denoted by $\Cext_0(\T)$ (resp. $\Cint_0(\T)$). Then, $C_0(\T)=\Cext_0(\T)\cup \Cint_0(\T)$ is the set of all vertices.
Besides, an edge of $\T$ is called exterior if it contains an exterior vertex and is called interior otherwise. The set of all exterior edges is denoted by $\Cext_1(\T)$ and that of all interior edges is denoted by $\Cint_1(\T)$.

A \textit{rooted ribbon tree} is a pair $(\T,v_0)$ of a ribbon tree $\T$ and an exterior vertex $v_0$ therein. We call $v_0$ the \textit{root}; it gives a natural partial order on the set of vertices $C_0(\T)$ by setting
$
v<v'
$
if $v\neq v'$ and there is a path in $\T$ from $v$ to $v_0$ which passes through $v'$. Particularly, the root $v_0$ is the largest vertex with respect to this partial order.


\begin{defn}
	\label{tree-stable-gapped-defn}
\textbf{(i)} For a label group $\G=(\G,E,\mu)$, a \textit{$\G$-decoration} on $(\T,v_0)$ is a map $\B: \Cint_0(\T)\to \G$. It is called \textit{stable} if any $v\in \Cint_0(\T)$ satisfies $E(\B(v))\ge 0$ and has at least three edges whenever $\B(v)=0$.
Given $k\in\mathbb N$ and $\beta\in \G$, we denote by $\Tr(k,\beta):=\Tr(k,\beta;\G)$ the set of $\G$-decorated \textit{stable rooted ribbon trees} $(\T,v_0,\B)$ so that $\#\Cext_0(\T)=k+1$ and $\sum_{v\in \Cint_0(\T)} \B(v)=\beta$.
Note that $\Tr(0,0)=\varnothing$; $\Tr(1,0)$ contains only one element, i.e. the tree $\T_{1,0}$ that consists of exactly two vertices and one edge.
For simplicity, we often omit $v_0, \B$ and write $\T=(\T,v_0,\B)$.

\textbf{(ii)} A \textit{time allocation} of $\T=(\T,v_0,\B)\in \Tr(k.\beta)$ is a map $\tau:\Cint_0(\T)\to \mathbb R$ such that $\tau(v)\le \tau(v')$ whenever $v<v'$.
The set of all time allocations $\tau$ for $\T$ such that all $\tau (v)$ are contained in $[a,b]$ is denoted by $\A_a^b(\T)$. If $\ab=\oi$, we often just write $\A(\T)=\A^1_0(\T)$.
\end{defn}

Any time allocation $\tau$ can be viewed as a point $(x_v=\tau(v))_{v\in \Cint_0(\T)}$ in $[a,b]^{\#\Cint_0(\T)}$. Thus, the set $\A_a^b(\T)$ can be identified with a bounded polyhedron cut out by the inequalities $x_v\le x_{v'}$ for $v<v'$ and $a\le x_v\le b$. Hence, it has a natural measure induced from the Lebesgue measure.

We introduce the notion of canonical models (also known as minimal models).

\begin{defn}
	\label{contraction-defn}
	Let $(\C, \check \m_{1,0})$ and $(\mH,\delta)$ be two graded cochain complexes.
	A triple $(i,\pi,G)$, consisting of two maps $i: H \to \C$, $\pi: C \to  H$ of degree zero and a map $G: C \to C$ of degree $-1$, is called a \textbf{contraction} (for $H$ and $C$) if the following equations hold
	\begin{align}
		\label{cochain-maps-i-pi}
	\check \m_{1,0}\circ i = i \circ \delta  \qquad
	\pi \circ\check  \m_{1,0} = \delta \circ \pi \\
\label{green-operator}
	i\circ \pi-\id_{C} = \check \m_{1,0} \circ G + G\circ \check \m_{1,0}
	\end{align}
	Further, the $(i,\pi,G)$ is called a \textbf{strong contraction}, or say it is \textbf{strong}, if we have the extra conditions:
	\begin{align}
		\label{pi-i-id-eq}
	\pi\circ i - \id_{\mH}=0 \\
\label{side-conditions GG}
	G\circ G=0   \\
\label{side-conditions Gi}
	G\circ i =0   \\
	\label{side-conditions piG}
	\pi\circ G=0  
	\end{align}
\end{defn}

We will give examples in \S \ref{S_harmonic_contraction}.
Here (\ref{side-conditions GG}, \ref{side-conditions Gi}, \ref{side-conditions piG}) are called \textit{side conditions}; c.f. \cite{Markl06}.

For our purpose, we need to slightly generalize \cite[Theorem 5.4.2]{FOOOBookOne} (see also \cite{Markl06}). Since the generality we need is not available in the literature, we give a sketch.

\begin{thm} \label{canonical_model_general_tree-thm}
Fix a $\G$-gapped $A_\infty$ algebra $(C,\check \m)$ and a graded cochain complex $(H,\delta)$.
From a contraction $(i,\pi, G)$,
there is a canonical way to construct a $\G$-gapped $A_\infty$ algebra $\m$ on $H$ together with a $\G$-gapped $A_\infty$ homomorphism
$
\mi: (H,\m)\to (C,\check  \m)
$
so that 
$
\mi_{1,0}=i$ and $\m_{1,0}=\delta$.
\end{thm}

\begin{proof}[Sketch of proof]
By induction on $\#\Cint_0(\T)$,
we aim to construct two sequences of operators
$
\mi_{\T}: H^{\otimes k} \to  C$ and $
\m_{\T}: H^{\otimes k} \to  H$
for all $\T=(\T,v_0,\B)\in \Tr(k,\beta)$ and all $k\in\mathbb N$, $\beta\in \G$ as follows.

When $\#\Cint_0(\T)=0$, the only possibility is that $\T=\T_{1,0}$, the unique tree in $\Tr(1,0)$; then, we define $\mi_{\T_{1,0}}=i$ and $\m_{\T_{1,0}}=\delta
$.
When $\#\Cint_0(\T)=1$, we must have $(k,\beta)\neq(0,0), (1,0)$; there is only one such tree $\T$ in every $\Tr(k,\beta)$.
Then, we define $\mi_{\T}=G\circ \check  \m_{k,\beta} \circ i^{\otimes k}$ and $\m_{\T}=\pi\circ \check  \m_{k,\beta} \circ i^{\otimes k}$.

Suppose now that $\mi_{\T'}$ and $\m_{\T'}$ have been constructed for $\#\Cint_0(\T')\le n$ and assume $\T=(\T,v_0,\B)\in \Tr(k,\beta)$ is a decorated stable rooted ribbon tree (Definition \ref{tree-stable-gapped-defn}) such that $\#\Cint_0(\T)=n+1$.
Consider the edge $e$ of the root $v_0$. Let $v$ be the other vertex of $e$. Then, $v\in \Cint_0(\T)$. We cut all edges of $v$ except $e$ and add a pair of vertices to all the resulting half line segments.
Then we get the tree $\T_{\ell,\B(v)}\in \Tr(\ell,\B(v))$ with one interior vertex $v$ together with the other $\ell$ trees
\begin{equation}
\label{tree-ell}
\T_i:= (\T_i, v_0^i, \B^i)\in\Tr(k_i,\beta_i) \qquad  i=1,\dots \ell
\end{equation}
where $\B(v)+\sum\beta_i =\beta$ and $\sum k_i =k$; the roots $v_0^i$ are newly-added vertices which are ordered counterclockwise. Define
$\mi_{\T} =G\circ \check \m_{\ell,\B(v)}\circ (\mi_{\T_1}\otimes\cdots\otimes \mi_{\T_\ell}) $ and
$\m_{\T} =\pi\circ \check \m_{\ell,\B(v)}\circ (\mi_{\T_1}\otimes\cdots\otimes \mi_{\T_\ell})$.
It is well-defined since $\# \Cint_0(\T_i) < \#\Cint_0(\T)$.
Summing over all the trees in $\Tr(k,\beta)$, we define:
\begin{equation}
\label{i_k,beta_can-eq}
\mi_{k,\beta}
= 
\textstyle \sum_{\T\in\Tr(k,\beta)} \mi_{\T} \quad ; \qquad
\m_{k,\beta}
=
\textstyle \sum_{\T\in\Tr(k,\beta)} \m_{\T}
\end{equation}
For the exceptional cases $(k,\beta)=(0,0), (1,0)$, we define $\mi_{0,0}=0$ and $\m_{0,0}=0$ and $\mi_{1,0}=i, \m_{1,0}=\delta$. 

We claim that there are only finitely many nonzero terms in (\ref{i_k,beta_can-eq}).
Indeed, even though the set $\Tr(k,\beta)$ may be infinite, a tree in $\Tr(k,\beta)$ with nonzero contribution to (\ref{i_k,beta_can-eq}) must satisfies that the image of $\B:\Cint_0(\T)\to \G$ is contained in the set $\mathsf G_\m:=\{\beta\in \G\mid \m_\beta\neq 0\}$. Hence, the gappedness of $\m$ infers that the trees with nonzero contributions form a finite subset of $\Tr(k,\beta)$.

Putting things together yields the inductive formulas:
for all $(k,\beta)\neq (0,0), (1,0)$,
\begin{equation}
\label{induction-tree-formula-eq}
\begin{aligned}
\mi_{k,\beta} &= 
\sum_{\ell\ge 0} \sum_{k_1+\cdots+k_\ell = k} \sum_{\substack{\beta_1+\cdots+\beta_\ell =\beta \\ (\ell,\beta_0)\neq (1,0)}}
 G \circ \check \m_{\ell,\beta_0} \circ (\mi_{k_1,\beta_1}\otimes \cdots \otimes \mi_{k_\ell,\beta_\ell}) \\
 \m_{k,\beta} &= 
\sum_{\ell\ge 0} \sum_{k_1+\cdots+k_\ell = k} \sum_{\substack{\beta_1+\cdots+\beta_\ell =\beta \\ (\ell,\beta_0)\neq (1,0)}}
\pi \circ \check  \m_{\ell,\beta_0} \circ (\mi_{k_1,\beta_1}\otimes \cdots \otimes \mi_{k_\ell,\beta_\ell})
\end{aligned}
\end{equation}
Remark that the condition $(\ell,\beta_0)\neq (1,0)$ in (\ref{induction-tree-formula-eq}) is due to the tree stable condition (Definition \ref{tree-stable-gapped-defn}).
The degrees are as we expect:
$\deg \mi_{k,\beta} = -1+ (2-\ell -\mu(\beta_0)) + \textstyle \sum_i (1-k_i-\mu(\beta_i) )=1-k-\mu(\beta)$
and also $\deg \m_{k,\beta}=2-k-\mu(\beta)$.

\begin{rmk}
	\label{Gapped_canonical_precise-rmk}
	To quote later, the argument of proving the gappedness of $\mi$ and $\m$ is separately presented below:
	Denote by $\mathsf G_\m$ the set of $\beta$ with $\m_\beta\neq 0$.
	Then by (\ref{induction-tree-formula-eq}) we see
	both the set $\{\beta\mid \m_\beta\neq 0\}$ and $\{\beta\mid \mi_\beta\neq 0 \}$ are contained in $\mathbb N\cdot \mathsf G_\m$. Particularly, if $\m$ only involves $\beta$ with $\mu(\beta)\ge 0$, then so do the $\m$ and $\mi$.
	Hence, the conditions (I-5) (II-5) (Definition \ref{UD_defn}) hold for $\m$ and $\mi$.
\end{rmk}

Finally, it is straightforward to prove the $A_\infty$ relations
$
(\check\m \diamond \mi - \mi \{ \m \} )|_{\CC_{k,\beta}}=0
$
and $\m\{\m\}|_{\CC_{k,\beta}}=0$
by inductions on the set of pairs $(k,\beta)$ with the partial order `$<$'.
\end{proof}

\begin{defn}
	\label{canonical_model_defn}
	In Theorem \ref{canonical_model_general_tree-thm}, we call the triple $(\mH, \m, \mi)$ or just $(\mH,\m)$ the \textbf{canonical model} (or \textbf{minimal model}) of $(\C,\check  \m)$ (with respect to the contraction $(i,\pi,G)$).
\end{defn}


\begin{prop}\label{canonical-model_preserve_properties-prop}
	The tree construction in Theorem \ref{canonical_model_general_tree-thm} has the following properties:
\begin{itemize}
	\itemsep 1pt
	\item[(i)] Assume $\partial\beta\cap i(b)=\partial \beta\cap b$. If $\m$ satisfies the divisor axiom, then so do both $\m$ and $\mi$.
	\item[(ii)] If $\m$ is cyclically unital, then so are $\m$ and $\mi$.
	\item [(iii)]
	Assume $(i,\pi,G)$ is a strong contraction.
	If $(\C,\check \m)$ has a unit $\one$ such that $i(\pi(\one))=\one$,
	then the $\m$ has a unit $\pi(\one)$ and the $\mi$ is unital with respect to $\pi(\one)$ and $\one$.
\end{itemize}
\end{prop}

\begin{proof}
	\leavevmode
\textit{\textbf{(i) Divisor axiom.}}
We only show the divisor axiom for $\mi$ and a similar argument applies for $\m$.
By Definition \ref{DivisorAxiom-defn}, it suffices to show
$
\DA[\mi]_{k,\beta}(b;x_1,\dots,x_k)= \partial \beta\cap b  \cdot \mi_{k,\beta}(x_1,\dots,x_k)
$
for all $(k,\beta)\neq (0,0)$.
We can easily check the initial case when $(k,\beta)=(1,0)$.
Suppose this holds for $(k',\beta')<(k,\beta)$, and we aim to show it also holds for $(k,\beta)$.
By the inductive formula (\ref{induction-tree-formula-eq}), we get
\begin{align*}
\DA[\mi]_{k,\beta}(b;x_1,\dots,x_k)
=
&
\textstyle 
\sum_{\substack{(k_i,\beta_i)\neq (0,0)\\ (\ell,\beta_0)\neq (1,0)}} G\circ \check \m_{\ell,\beta_0}\circ \big(
\mi_{k_1,\beta_1}\cdots \DA[\mi]_{k_i,\beta_i}(b;\cdots)\cdots \mi_{k_\ell,\beta_\ell}
\big)
\\
+
&
\textstyle 
\sum_{\beta_0\neq 0} G\circ \DA[\check \m]_{\ell,\beta_0} \big(i(b);
\mi_{k_1,\beta_1}\cdots \mi_{k_\ell,\beta_\ell}\big)
\end{align*}
We may require $\beta_0\neq 0$ in the second sum by using the divisor axiom of $\m$ with $\beta=0\in\G$.
Note that $\DA[\mi]_{0,0}(b)=\mi_{1,0}(b)=i(b)$; the second sum just consists of the excluded terms with $(k_i,\beta_i)=(0,0)$ in the first sum. 
We can directly apply the induction hypothesis to the first sum, and we can apply the divisor axiom of $\check \m$ to the second sum. In consequence, we obtain
\begin{align*}
\DA[\mi]_{k,\beta}(b;x_1,\dots,x_k)
=
&
\textstyle 
\sum_{\substack{(k_i,\beta_i)\neq (0,0)\\ (\ell,\beta_0)\neq (1,0)}} \partial \beta_i\cap b\cdot G\circ\check  \m_{\ell,\beta_0}\circ \big(
\mi_{k_1,\beta_1} \cdots \mi_{k_i,\beta_i}\cdots \mi_{k_\ell,\beta_\ell}
\big)
\\
+
&
\textstyle 
\sum_{\beta_0\neq 0}
\partial\beta_0\cap b \cdot G\circ \check \m_{\ell,\beta_0} \big(\mi_{k_1,\beta_1}\cdots \mi_{k_\ell,\beta_\ell}\big)
\end{align*}
Now, the above conditions $(k_i,\beta_i)\neq (0,0), (\ell,\beta_0)\neq (1,0)$ and $\beta_0\neq 0$ can be removed. Finally, as $\beta_0+\sum_i\beta_i=\beta$, we have $\DA[\mi]_{k,\beta}(b;x_1,\dots,x_k)=\partial\beta \cap b\cdot \mi_{k,\beta}(x_1,\dots,x_k)$.

\textit{\textbf{(ii) Cyclical unitality.}}
Let $\e\in H$ be degree-zero.
When $(k,\beta)=(0,0)$, we recall that $\mi_{1,0}(\e)=i(\e)$ is exceptional in Definition \ref{cyclical_unitality_defn}.
We perform an induction for the pairs $(k,\beta)$ again.
The initial case $(k,\beta)=(1,0)$ is trivial.
For $(k,\beta)\neq (0,0), (1,0)$, suppose it is true for $(k',\beta')<(k,\beta)$. By (\ref{induction-tree-formula-eq}),
\begin{align*}
\CU[\mi]_{k,\beta}(\e;x_1,\dots,x_k)
=&
\textstyle
\sum_{\substack{(k_i,\beta_i)\neq (0,0)\\ (\ell,\beta_0)\neq (1,0)}} 
G
\circ \check \m_{\ell,\beta_0} \circ \big( \mi^{\#}_{k_1,\beta_1} \cdots \CU[\mi]_{k_i,\beta_i}(\e;\cdots)\cdots \mi_{k_\ell,\beta_\ell} \big) \\
+&
\textstyle  \sum_{(\ell,\beta_0)\neq (0,0)} G\circ \CU[\check \m]_{\ell,\beta_0} \big(i(\e);\mi_{k_1,\beta_1}\cdots \mi_{k_\ell,\beta_\ell}(\cdots x_k)\big)
\end{align*}
Since $\check \m$ is cyclical unital, the second sum is zero.
As above, applying the induction hypothesis and the cyclical unitality of $\check \m$ completes the induction step.
The same argument also works for $\m$.

\textit{\textbf{(iii) Unitality.}}
We first prove the unitality of $\mi$.
By condition, we have $\mi_{1,0}(\pi(\one))=i(\pi(\one)=\one$. So, it remains to show $\mi_{k,\beta}(\dots,\pi(\one),\dots)=0$ for $(k,\beta)\neq(1,0)$. Arguing by contradiction, suppose $(k,\beta)\neq (1,0)$ is the smallest pair so that
$\mathfrak A:=\mi_{k,\beta}(\dots,\pi(\one),\dots)\neq 0$ happens for some fixed inputs. Then, we can find some non-zero term in the expansion (\ref{induction-tree-formula-eq}) of $\mathfrak A$, say:
\begin{align*}
\mathfrak a:= G \circ \check \m_{\ell,\beta_0}  (\mi_{k_1,\beta_1} (\dots), \dots, \mi_{k_i,\beta_i}(\dots, \pi(\one) \dots), \dots, \mi_{k_\ell,\beta_\ell}(\dots) )\neq 0
\end{align*}
Since $(\ell,\beta_0)\neq (1,0)$, we must have $(k_i,\beta_i)\neq (k,\beta)$.
By the smallest choice of $(k,\beta)$, we have $(k_i,\beta_i)=(1,0)$. Besides, when $(\ell,\beta_0)\neq (2,0)$, we get $\mi_{k_i,\beta_i}(\dots \pi(\one)\dots)=i(\pi(\one))=\one$ which is the unit of $\check \m$, hence the term $\mathfrak a$ vanishes.
When $(\ell,\beta_0)=(2,0)$, the above term $\mathfrak a$ becomes either
$G\circ\check \m_{2,0} (\mi_{k-1,\beta}(\cdots), \one )$
or
$G\circ\check \m_{2,0} (\one, \mi_{k-1,\beta}(\cdots))$.
By Definition \ref{unit-defn} (a1), the term turns out to be $\pm G\circ \mi_{k-1,\beta}(\cdots)$ in either case.
Now, if $(k-1,\beta)\neq(1,0)$, then by the side condition $G\circ G=0$ (\ref{side-conditions GG}), the term vanishes; if $(k-1,\beta)=(1,0)$, then it is still zero due to the other side condition $G\circ i=0$ (\ref{side-conditions Gi}). This is a contradiction to $\mathfrak a\neq 0$.
As for $\m$, observe that $\m_{1,0}(\pi(\one))=\delta\pi (\one)=\pi\check \m_{1,0}(\one)=0$ by (\ref{cochain-maps-i-pi}). Notice also that $\m_{2,0}(\pi(\one), x)=\pi\circ\check \m_{2,0}\circ( i\pi(\one), i(x))=\pi i(x)=x$ and similarly $(-1)^{\deg x} \m_{2,0}(x,\pi(\one))=x$.
Suppose $(k,\beta)\neq (1,0),(2,0)$ is the smallest pair so that $\m_{k,\beta}(\dots,\pi(\one),\dots)\neq 0$ happens. Then as before we may find a non-zero term in the expansion:
$
\pi \circ\check \m_{\ell,\beta_0}  (\mi_{k_1,\beta_1} (\dots), \dots, \mi_{k_i,\beta_i}(\dots, 
\pi(\one)
\dots), \dots, \mi_{k_\ell,\beta_\ell}(\dots) )
$.
We may require $(k_i,\beta_i)=(1,0)$. Since $\one\equiv i(\pi(\one))$ is a unit of $\check \m$, we can further require $(\ell,\beta_0)=(2,0)$. Then we arrive at $\pm \pi \circ \mi_{k-1,\beta}(\dots)$. Recall that $(k-1,\beta)\neq (1,0)$; so, it is zero again due to the last side-condition $\pi\circ G=0$ (\ref{side-conditions piG}).
\end{proof}

\subsection{Feynman-diagram-type integration}
\begin{thm}
	\emph{\hspace{-0.2em}\cite{FuCyclic}}
	\label{from-pseudo-isotopies-to-A-infty-homo-thm}
There is a canonical way to associate to
a pseudo-isotopy $(C_\oi,\M)$ a $\G$-gapped $A_\infty$ homomorphism 
$
\mC: (C,\m^0) \to (C,\m^1)
$
so that $\mC_{1,0}=\id$, where $\M$ restricts to $\m^0, \m^1$ at $s=0,1$.
\end{thm}

\begin{proof}[Sketch of proof]
Recall that we may write $\M=1\otimes \m^s+ds\otimes \mc^s$.
We aim to construct a sequence of operators
$
\mC_{\T,\tau}: \C^{\otimes k} \to \C
$
for a tree $\T:=(\T,v_0,\B)\in \Tr(k,\beta)$ equipped with a time allocation $\tau\in \A(\T)\equiv \A_0^1(\T)$. 
Recall that the set $\A(\T)$ of all time allocations embeds into $\oi^{\#\Cint_0(\T)}$ and naturally inherits a measure.
We construct these $\mC_{\T,\tau}$ inductively on $\#\Cint_0(\T)$ as well. When $\#\Cint_0(\T)=0$, $\A(\T)=\varnothing$ and the only possibility is $\T=\T_{1,0}$. We define $\mc_{\T_{1,0},\varnothing} =\id$. When $\#\Cint_0(\T) =1$, we must have $(k,\beta)\neq (0,0), (1,0)$ and $\T=\T_{k,\beta}$, the tree with one interior vertex $v$ in the set $\Tr(k,\beta)$. Then, we define
$
\mC_{\T,\tau} = - \mc^{\tau(v)}_{k,\beta}$.

Inductively, suppose the constructions have been established for a tree with $\#\Cint_0(\T') \le n$. Let $\T$ be a tree so that $\#\Cint_0(\T) = n+1$.
Denote by $v$ the closest vertex to the root $v_0$.
We do a similar surgery as in (\ref{tree-ell}) cutting all the incoming edges of $v$; then, we get $\T_{\ell,\B(v)}$ and $\T_i$ ($1\le i\le \ell)$ like there. 
Restricting $\tau$ on interior vertices of $\T_i$ produces a time allocation $\tau_i$ in $\A^{\tau(v)}_0(\T_i)$. Now, we define
$
\mC_{\T,\tau} = - \mc^{\tau(v)}_{\ell,\B(v)} \circ (\mC_{\T_1,\tau_1}\otimes\cdots \otimes \mC_{\T_\ell, \tau_\ell})$.
By the stability, we have $(\ell,\B(v))\neq(1,0)$.
Integrating all possible $\tau\in\A(\T)$, we define
$\mC_\T = \int_{\A(\T)} \mC_{\T,\tau} d\tau$.
Eventually, we define $\mC_{0,0}=0$ and $\mC_{1,0}=\id$; if $(k,\beta)\neq (0,0), (1,0)$, we define
$
\mC_{k,\beta} := \sum_{\T\in \Tr(k,\beta)} \mC_\T
$.

We can also take a pseudo-isotopy on $\C_\ab$ for an arbitrary interval $\ab$. Replacing $\A(\T)\equiv \A_0^1(\T)$ by $\A_a^b(\T)$, one can also define $\mC_{\T}^{[a,b]}$ and $\mC_{k,\beta}^{[a,b]}$ in the same way; we still define $\mC_{0,0}^\ab=0$ and
$\mC^\ab_{1,0}=\id$.

In the special case $a=b$,
$
\mC^{[a,a]}=\id
$ is the trivial $A_\infty$ homomorphisms.
By the Fubini's theorem,
\begin{align*}
\mC_{k,\beta} 
&
=
\sum_{\T\in \Tr(k,\beta)}\int_{\A(\T)} - \mc^{\tau(v)}_{\ell,\B(v)} \circ (\mC_{\T_1,\tau_1}\otimes\cdots \otimes \mC_{\T_\ell, \tau_\ell}) d\tau \\
&
=
\sum_{ \substack{ \ell \ge 1 \\
			\beta_0+\beta_1+\cdots+ \beta_\ell = \beta\\
			k_1+\cdots+k_\ell = k\\
			(\ell,\beta_0)\neq(1,0)}
			}
\sum_{\substack{ \T_1\in \Tr(k_1,\beta_1) \\
				\cdots\cdots \\
				\T_\ell\in \Tr(k_\ell, \beta_\ell)}
				}
\int_0^1 du \int_{\A^u_0(\T_1)\times \cdots \times \A^u_0(\T_\ell)}		
-\mc_{\ell,\beta_0}^u \circ 
\big(
\mC_{\T_1,\tau_1}\otimes\cdots \otimes \mC_{\T_\ell, \tau_\ell}
\big)\cdot d\tau_1\cdots d\tau_\ell 
\end{align*}
For a general interval $\ab$, we have an \textit{inductive formula}:
\begin{equation}\label{inductive_formu-Fubini-yield-eq}
\mC_{k,\beta}^\ab =
\sum_{\ell\ge 1}
\sum_{ \substack{ 
			\beta_0+\beta_1+\cdots+ \beta_\ell = \beta\\
			k_1+\cdots+k_\ell = k \\
			(\ell,\beta_0)\neq(1,0)}
			}
-\int_a^b du \cdot 
\mc^u_{\ell,\beta_0}
\circ
\big(
\mC_{k_1,\beta_1}^{[a,u]} \otimes \cdots \otimes \mC_{k_\ell,\beta_\ell}^{[a,u]}
\big)
\end{equation}
where $(k,\beta)\neq(0,0),(1,0)$.
If we replace $\ab$ by $[0,u]$ in (\ref{inductive_formu-Fubini-yield-eq}), then taking the derivative $\frac{d}{du}$ yields:
\begin{equation}\label{inductive_formu_derivative-C k beta -eq}
\frac{d}{du}\mC^{[0,u]}_{k,\beta} =
\sum_{\ell\ge 1}
\sum_{ \substack{ 
			\beta_0+\beta_1+\cdots+ \beta_\ell = \beta\\
			k_1+\cdots+k_\ell = k \\
			(\ell,\beta_0)\neq(1,0)}
			}
- 
\mc^u_{\ell,\beta_0}
\circ
\big(
\mC_{k_1,\beta_1}^{[0,u]} \otimes \cdots \otimes \mC_{k_\ell,\beta_\ell}^{[0,u]}
\big)
\end{equation}
Beware that we no longer need to assume $(k,\beta)\neq(0,0),(1,0)$ in (\ref{inductive_formu_derivative-C k beta -eq}), since the two excluded cases can be checked by hand.
For instance, when $(k,\beta)=(1,0)$, we have $\frac{d}{du}\mC^{[0,u]}_{1,0}=\frac{d}{du}\id=0$.

Now, we claim that the above-defined operator system $\mC=(\mC_{k,\beta})$ gives an $A_\infty$ homomorphism.
Let us check the degrees first: by Definition \ref{Classical-Pseudo-isotopy-defn} we know $\deg\mc^s_{k,\beta}=1-k-\mu(\beta)$.
Inductively, the equations in (\ref{inductive_formu-Fubini-yield-eq}) imply that
\[
\deg \mC^\ab_{k,\beta} = \deg \mc^u_{\ell,\beta_0} + \sum_{i=1}^\ell \deg \mC^{[a,u]}_{k_i,\beta_i}
= 1-\ell-\mu(\beta_0) + \sum_{i=1}^\ell (1-k_i-\mu(\beta_i) )=1-k-\mu(\beta)
\]
To prove the $A_\infty$ relation, we do induction on the pairs $(k,\beta)$ again (Remark \ref{Induction}), but we strengthen the induction statement by allowing any arbitrary intervals $[a,b]\subset \oi$. Consider
\begin{align*}
\mP_{k,\beta}^\ab
&
:=\sum_{\ell \ge 1}
\sum_{\substack{
\beta_0+\beta_1+\cdots+\beta_\ell=\beta
}} \sum_{0=j_0\le\cdots\le j_\ell=k} 
\m^b_{\ell,\beta_0} \circ 
		(\mC^\ab_{j_1-j_0,\beta_1} \otimes \cdots \otimes \mC^\ab_{j_\ell-j_{\ell-1},\beta_\ell} ) \\
\mQ^\ab_{k,\beta}
&
:= \sum_{\beta'+\beta''=\beta}
\sum_{\substack{i+j+r=k}} 			
			\mC^\ab_{i+j+1,\beta'}\circ ( \id_\#^ i \otimes \m^a_{r,\beta''}\otimes \id^j)
\end{align*}
and the desired $A_\infty$ relation is $\mP^\ab_{k,\beta}=\mQ^\ab_{k,\beta}$. 
When $(k,\beta)=(0,0)$, it is trivial. When $(k,\beta)=(1,0)$, it reduces to $\m^b_{1,0} \circ \mC_{1,0}^\ab = \mC^\ab_{1,0}\circ \m^a_{1,0}$, which is obvious.
Suppose now $\mP^\ab_{k',\beta'}=\mQ^\ab_{k',\beta'}$ holds for all $\ab$ and $(k',\beta')<(k,\beta)$. Without loss of generality, let's take $\ab=\oi$. We can then demonstrate that $\mP^\oi_{k,\beta}=\mQ^\oi_{k,\beta}$ by verifying $\frac{d}{du}\mathfrak P_{k,\beta}^{[0,u]}=\frac{d}{du}\mathfrak Q_{k,\beta}^{[0,u]}$.
\end{proof}

\begin{rmk}
\label{Gapped_mC_precise-rmk}
Similar to Remark \ref{Gapped_canonical_precise-rmk} we have the following observation. Due to (\ref{inductive_formu-Fubini-yield-eq}) or (\ref{inductive_formu_derivative-C k beta -eq}), if we write $\mathsf G=\{\beta \mid \mc^s_{\beta}\neq 0\}$ which is a subset of $\{\beta\mid \M_{\beta}\neq 0\}$, then the set $\{\beta \mid \mC^{[a,b]}_{\beta}\neq 0\}$ is contained in $\mathbb N\cdot \mathsf G$. Particularly, if $\M$ only involved non-negative index $\mu$, then so does $\mC^\ab$.
\end{rmk}

\begin{cor}\label{from-pseudo-isotopies-to-A-infty-homo-cor_trivial_pseudo_isotopy}
	If $\M$ is a trivial pseudo-isotopy, then the $\mC$ is the trivial $A_\infty$ homomorphism $\id$.
\end{cor}


Finally, we present a technical formula for later use, which appears to be unavailable in the literature.

\begin{lem}\label{d/du C [u,1] -lem}
	\[
	\frac{d}{ds} \mC_{k,\beta}^{[s,1]} = 
	\sum_{\lambda+\mu+\nu=k} 
	\sum_{\substack{\beta'+\beta''=\beta \\ (\nu,\beta'')\neq(1,0)}} 
	\mC_{\lambda+\mu+1,\beta'}^{[s,1]}
	\circ( \id^\lambda \otimes \mc^s_{\nu,\beta''}\otimes \id^\mu)
	\]
\end{lem}

\begin{proof}
	When $(k,\beta)=(0,0), (1,0)$ it is trivial.
	We may include $\mC^{[s,b]}$ for an arbitrary upper bound $b\in\mathbb R$. Suppose this enhanced statement holds for $(k',\beta')<(k,\beta)$.
	By virtue of (\ref{inductive_formu-Fubini-yield-eq}), the induction hypothesis deduces that
	\begin{align*}
	\frac{d}{ds} \mC^{[s,1]}_{k,\beta} 
	&
	=
	\sum_{\ell\ge 1} \sum_{(\ell,\beta_0)\neq (1,0)} \sum_{1\le a\le \ell}
	-
	\int_s^1 du \cdot \mc^u_{\ell,\beta_0}\circ
	\big(
	\mC^{[s,u]}_{k_1,\beta_1}
	\otimes
	\cdots \otimes \frac{d}{ds}\mC^{[s,u]}_{k_a,\beta_a} \otimes
	\cdots
	\otimes
	\mC^{[s,u]}_{k_\ell,\beta_\ell}
	\big)		\\
	&
	=
	\sum_{\ell \ge 1}
	\sum_{\substack{\lambda+\mu+\nu=k \\ \beta'+\beta''=\beta \\ (\nu,\beta'')\neq(1,0)}} 
	\sum_{ \substack{ 
			\beta'_0+\beta'_1+\cdots+ \beta'_\ell = \beta'\\
			k'_1+\cdots+k'_\ell = \lambda+\mu+1 \\
			(\ell,\beta'_0)\neq(1,0)}
	}
	\int^1_s du \cdot \mc^u_{\ell,\beta'_0} \circ
	\big(\mC^{[s,u]}_{k'_1,\beta'_1}\otimes\cdots \otimes \mC^{[s,u]}_{k'_\ell,\beta'_\ell} \big)
	\circ (\id^\lambda\otimes \mc^s_{\nu,\beta''}\otimes \id^\mu )
	\end{align*}
	Finally, applying (\ref{inductive_formu-Fubini-yield-eq}) again completes the induction.
\end{proof}

Since the construction in Theorem \ref{from-pseudo-isotopies-to-A-infty-homo-thm} is quite canonical, we have an analog of Proposition \ref{canonical-model_preserve_properties-prop}:

\begin{prop}
	\label{from-pseudo-isotopies-to-A-infty_property_prop}
	The canonical construction in Theorem \ref{from-pseudo-isotopies-to-A-infty-homo-thm} has the following properties:
	
	\begin{itemize}
		\itemsep 2pt
		\item[(a)] If $\M$ satisfies the divisor axiom, then so does $\mC$.
		\item [(b)] If $\M$ is cyclically unital, then so is $\mC$.
		\item [(c)] If $\M$ has a $\oi$-unit $\one\in \C$, then $\mC$ is unital with respect to $\one$'s.
\end{itemize}
\end{prop}

\begin{proof}
\textbf{\textit{(a) Divisor axiom.}}
Write $\M=1\otimes \m^s+ds\otimes \mc^s$ as before.
Using the divisor input in the form $\incl(b)$ in the divisor axiom equations of $\M$, we see that both $\m^s$ and $\mc^s$ satisfy the divisor axiom. 
By the inductive formula (\ref{inductive_formu-Fubini-yield-eq}), we get
\begin{align*}
\DA[\mC]_{k,\beta}(b;x_1,\dots,x_k)
&
=
\textstyle
\sum_{\substack{(k_i,\beta_i)\neq (0,0)
		\\
		(\ell,\beta_0)\neq (1,0)}} 
-\int_0^1 du\cdot  \mc^u_{\ell,\beta_0} \big(
\mC_{k_1,\beta_1}^{[0,u]}\cdots \DA[\mC^{[0,u]}]_{k_i,\beta_i} (b;\cdots)\cdots \mC_{k_\ell,\beta_\ell}^{[0,u]}
\big)
\\
&
+
\textstyle
\sum_{\beta_0\neq 0}
-\int_0^1 du\cdot \DA[\mc^u]_{\ell,\beta_0}
\big(
b;\mC_{k_1,\beta_1}^{[0,u]}\cdots \mC_{k_\ell,\beta_\ell}^{[0,u]}
\big)
\end{align*}
Recall that $\mC_{1,0}^{[0,u]}=\id$.
In the second sum, we can use the divisor axiom of $\mc^u$.
In the first sum, $(k_i,\beta_i)<(k,\beta)$; so we can perform the induction.

\textit{\textbf{(b) Cyclical unitality.}}
Let $\e\in C$ be a degree-zero element.
Notice that the cyclical unitality (Definition \ref{cyclical_unitality_defn}) of $\M$ implies the $\m^s$ and $\mc^s$ are also cyclically unital.
Using the formula (\ref{inductive_formu-Fubini-yield-eq}),
we get
\begin{align*}
\CU[\mC]_{k,\beta}(\e;x_1,\dots, x_k)
&
=
\textstyle
\sum_{\substack {(k_i,\beta_i)\neq (0,0)	\\	(\ell,\beta_0)\neq (1,0)}}
-
\int_0^1 du\cdot \mc^u_{\ell,\beta_0} 
\big(
\mC_{k_1,\beta_1}^{[0,u]\#}\cdots
\CU[\mC^{[0,u]}]_{k_i,\beta_i} (\e;\cdots)\cdots \mC_{k_\ell,\beta_\ell}^{[0,u]}
\big)\\
&
+
\textstyle
\sum_{\beta_0\neq 0}
-\int_0^1 du\cdot \CU[\mc^u]_{\ell,\beta_0}
\big(
\e;\mC_{k_1,\beta_1}^{[0,u]}\cdots \mC_{k_\ell,\beta_\ell}^{[0,u]}
\big)
\end{align*}
By a similar induction, one can show this vanishes.

\textbf{\textit{(c) Unitality.}}
Suppose $\one$ is a $\oi$-unit (Definition \ref{P_unit_defn}) of $\M$, and then $\mc^s_{k,\beta}(\dots \one \dots)=0$ for $(k,\beta)\neq (1,0)$.
As $\mC_{1,0}=\id$, it suffices to show $\mC_{k,\beta}(\dots \one \dots)=0$ for $(k,\beta)\neq (1,0)$. Once again, exploiting the formula (\ref{inductive_formu-Fubini-yield-eq}) we obtain that
$
\textstyle
\mC_{k,\beta}(\dots \one \dots)=-\sum_{(\ell,\beta_0)\neq (1,0)} \int_0^1 \mc^u_{\ell,\beta_0}
(\dots \mC^{[0,u]}_{k_i,\beta_i} (\dots\one\dots)  \dots)$.
Since $\mC_{1,0}^{[0,u]}(\one)=\one$, a term with $(k_i,\beta_i)= (1,0)$ vanishes due to the property of $\mc^s$ we just mentioned. So, we may assume all $(k_i,\beta_i)\neq (1,0)$; then, we can use the induction to conclude that $\mC$ is unital.
\end{proof}

In application, we work with the category $\UD$ (Definition \ref{UD_defn}). By Theorem \ref{Whitehead-full-thm}, the ud-homotopy inverse always exists within $\UD$. The $A_\infty$ homotopy equivalence $\eval^1 \diamond (\eval^0)^{-1}$ shares the same source and target with $\mC$. By Lemma \ref{DA-g_10-h_10-lem}, we can arrange so that $\eval^1\diamond (\eval^0)^{-1}|_{\CC_{1,0}}=\id=\mC_{1,0}$, and they are indeed ud-homotopic to each other. The following theorem is original.

\begin{thm}
	\label{UD-mC_thm}
	If $(\C_\oi, \M) \in\Obj\UD$, then the induced $\mC\in \Mor\UD$. Moreover, $\mC \diamond \eval^0 $ is ud-homotopic to $\eval^1$.
\end{thm}

\begin{proof}
	The first half is basically known already.
	The conditions
	(II-3) (II-2) (II-1) for $\mC$ are just consequences of
	Proposition \ref{from-pseudo-isotopies-to-A-infty_property_prop} (a) (b) (c) respectively. Besides, the condition (II-4) holds as $\mC_{1,0}=\id$. Finally, the (II-5) holds by Remark \ref{Gapped_mC_precise-rmk}.
	Hence, $\mC\in \Mor \UD$ and also $\mC^\ab\in\Mor\UD$.

	Now, we aim to show $\mC \diamond \eval^0$ is ud-homotopic to $\eval^1$.	
	Write $\M=1\otimes \m^s +ds\otimes \mc^s$ and denote by $\M^\tri=1\otimes \m^1+ds\otimes \frac{d}{ds}$ the trivial pseudo-isotopy about $\m^1$ (Example \ref{trivial-pseudo-isotopy-ex}).
By Definition \ref{ud-homotopic-defn}, our goal is to find a morphism $\F$ in $\UD$ from $(\C_\oi, \M)$ to $(\C_\oi, \M^\tri)$ so that $\eval^{\tri,0} \F= \mC\diamond \eval^0$ and $\eval^{\tri,1}  \F=\eval^1$. (For clarity, the symbols $\eval^{\tri,i}$ are used for the evaluation maps about the trivial pseudo-isotopy $\M^\tri$).
Now, we define
\begin{equation}
\label{mC_eval0_F_eq}
\F=1\otimes \mC^{[s,1]}  \ \
\in \CC(\C_\oi,\C_\oi)
\end{equation}

The notation $1\otimes \mC^{[s,1]}$ means the pointwise extension in Remark \ref{pointwise_determine_rmk}.
Explicitly, for $x_i\in C^\infty(\oi, C)$, we have $\F_{k,\beta}(1\otimes x_1,\dots, 1\otimes x_k) (s) 
=
1\otimes \mC^{[s,1]}_{k,\beta}(x_1(s),\dots,x_k(s))$ and
$
\F_{k,\beta}(1\otimes x_1,\dots, ds\otimes x_i,\dots, 1\otimes x_k)(s)
=
(-1)^{\sum_{a=1}^{i-1} (\deg x_a+1)} ds\otimes \mC^{[s,1]}_{k,\beta}(x_1(s),\dots,x_k(s))
$. By the pointwiseness, we also require $\F( \cdots ds\otimes y' \cdots ds\otimes y''\cdots)=0$. As the $\mu$ maps into $2\mathbb Z$, one can check the signs actually agrees with Definition \ref{pointwise-defn}. 
Next, we want to show the $A_\infty$ relation:
$
\textstyle
\sum \M^{\tri} \circ (\F\otimes \cdots \otimes \F) = \sum \F\circ (\id^\bullet_\# \otimes \M \otimes \id^\bullet)
$.
Comparing $1\otimes  -$ and $ds\otimes -$ parts separately, it is equivalent to the following two identities:
\begin{align*}
	 \m^1\circ (\mC^{[s,1]}\otimes \cdots \otimes \mC^{[s,1]})
	 &
	 =\mC^{[s,1]}\circ (\id^\bullet_\#\otimes \m^s\otimes \id^\bullet)\\
	\tfrac{d}{ds} \circ \mC^{[s,1]}
	&
	=\mC^{[s,1]}(\id^\bullet \otimes \mc^s \otimes \id^\bullet)
\end{align*}
The second one holds exactly because of Lemma \ref{d/du C [u,1] -lem}; meanwhile, the first one holds since $\mC^{[s,1]}$ is an $A_\infty$ homomorphism from $\m^s$ to $\m^1$.
Hence, the $\F$ constructed above is indeed a gapped $A_\infty$ homomorphism; by construction, it is clear that $\eval^{\tri,0}\circ \F= \mC\circ \eval^0$ and $\eval^{\tri,1}\circ \F=\eval^1$.

It suffices to check the above $\F$ is a morphism in $\UD$ (Definition \ref{UD_defn}).
It is clear that the condition (II-5) holds for $\F$.
We show (II-1) and (II-4) as follows: 
Denote by $\one$ the $\oi$-unit of the given $\M$. Then, $\one$ is a unit of any $\m^s$; particularly, $\one:=\incl(\one)$ is also a $\oi$-unit of $\M^\tri$.
When $(k,\beta)\neq (1,0)$, we have $\F_{k,\beta}(\dots,\incl(\one),\dots)(s)=1\otimes \mC_{k,\beta}^{[s,1]}(\dots,\one,\dots) \pm ds\otimes  \mC_{k,\beta}^{[s,1]}(\dots,\one,\dots)=0$, since $\mC^{[s,1]}$ is unital. Finally, since $\F_{1,0}=\id_{C_\oi}$, we have $\F_{1,0}(\one)=\one$ and the condition (\ref{DA-f_10-eq}) for $\F_{1,0}$.

Next, we prove the cyclical unitality (II-2) for $\F$.
A degree-zero element of $\C_\oi$ is like $\e:=1\otimes \e(s)$ for some $\e\in C^\infty(\oi, C^0)$, where $C^0$ denotes the degree-zero part of $C$.
By the pointwiseness, we may assume the other inputs are in the form of $1\otimes x_i(s)$ for some $x_i\in C^\infty(\oi, C)$. Then, we compute
\begin{align*}
\CU[\F]_{k,\beta}(\e;1\otimes x_1,\dots,1\otimes x_k) (s)
=
1\otimes \CU[\mC^{[s,1]}]_{k,\beta}
(\e(s); x_1(s),\dots,x_k(s))
\end{align*}
It vanishes due to the cyclical unitality of $\mC^{[s,1]}$.

Finally, we prove the divisor axiom (II-3) for $\F$.
Take a divisor input $b=1\otimes b_0 +ds\otimes b_1$ and assume other inputs are in the form of $y_i=1\otimes x_i$ as above. Then, by (\ref{mC_eval0_F_eq}), we have
\begin{align*}
\DA[\F]_{k,\beta}(b;y_1,\dots,y_k)=1\otimes \DA[\mC^{[s,1]}]_{k,\beta}(b_0;x_1,\dots, x_k) +ds\otimes \CU[\mC^{[s,1]}]_{k,\beta} (b_1;x_1,\dots,x_k)
\end{align*}
Using both the divisor axiom and cyclical unitality of $\mC^{[s,1]}$, we complete the proof.
\end{proof}

\subsection{$A_\infty$ algebras associated to Lagrangians}
\label{S_A_infty_associated}

Let $(X,\omega)$ be a symplectic manifold; let $L$ be a connected compact oriented Lagrangian submanifold equipped with a relative spin structure. 
Denote by $\mathfrak J(X,\omega)$ the space of $\omega$-tame almost complex structures and denote by $\mathfrak J(X,L,\omega)$ (Assumption \ref{assumption-mu ge 0}) its subset consisting of those almost complex structures $J$ which does not allow negative Maslov index $J$-holomorphic disks with boundary on $L$.

\subsubsection{Moduli spaces}
Let $\beta\in \pi_2(X,L)$ and fix $J\in\mathfrak J(X,L,\omega)$.
For $(k,\beta)\neq (0,0), (1,0)$, we denote by 
$
\mathcal M_{k+1,\beta}(J,L)
$
the moduli space of equivalence classes $[\mathbf u, \mathbf z]$ of \textit{$(k+1)$-boundary-marked $J$-holomorphic stable maps $(\uu,\z)$ of genus zero with one disk component bounded by $L$ in the class $\beta$}. 
We require $(k,\beta)\neq (0,0), (1,0)$ for the sake of stability; the $\mathcal M_{1,0}(J,L)$ and $\mathcal M_{2,0}(J,L)$ are just \textit{not defined}.
Here $\z=(z_0,z_1,\dots,z_k)$ represents the boundary marked points ordered counter-clockwisely.

The conception of \textit{stable maps} is standard; see e.g. \cite{MS} and \cite{Frau08}.
The equivalence relation roughly refers to a biholomorphism on the domains of two stable maps which identifies the nodal points, the marked points and the boundaries.
The moduli space $\mathcal M_{k+1,\beta}(J,L)$ admits a natural Hausdorff topology for which it is compact\cite[Theorem 7.1.43]{FOOOBookTwo}.
To get some intuition, we note that a `point' in the interior $\mathcal M^\circ_{k+1,\beta}(J,L)$ of the moduli space is (the equivalence class of) a $J$-holomorphic map $u: (\mathbb D, \partial \mathbb D) \to (X,L)$ with the marked points $\z=(z_0, z_1,\dots, z_k)$ in the boundary $\partial\mathbb D$.

Let $P$ be a smooth compact contractible oriented manifold with corners, and let
$\bigsqcup_i \partial_i P$
be its normalized boundary.
Provided a smooth family $\mathbf J=\{J_t\mid t\in P\}$ in $\mathfrak J(X,L,\omega)$, we consider the following union of moduli spaces:
\begin{equation}
\label{moduli_system}
\mathcal M_{k+1,\beta}(\mathbf J,L):=\bigsqcup_{t\in P} \ \{t\} \times \mathcal M_{k+1,\beta}(J_t,L)
\end{equation}
It is also called a parameterized moduli space.
Denote by $\overline{\ev}_P: \mathcal M_{k+1,\beta}(\mathbf J,L)\to P$ the natural projection onto $P$.
Given $0\le i\le k$ and $t\in P$, there is a natural evaluation map $\mathcal M_{k+1,\beta}(J_t, L)\to L$ sending a stable disk $\uu$ to $\uu(z_i)$. Taking the union yields a natural map $\overline{\ev}_i: \mathcal M_{k+1,\beta}(\mathbf J,L)\to L $ for each $ i=0,1,\dots,k$.
Now, we call the map
\begin{equation}
\label{evaluation_maps_eq}
\ev_i=( \overline{\ev}_P, \overline{\ev}_i): \mathcal M_{k+1,\beta}(\mathbf J,L) \to P\times L
\end{equation}
the $i$-th \textit{evaluation map} of the parameterized moduli space $\mathcal M_{k+1,\beta}(\mJ, L)$, sending $(t, [\uu,\z])$ to $(t,\uu(z_i))$. Occasionally, we also call $\overline{\ev}_i$ a evaluation map.
Meanwhile, we also consider the collection 
\begin{equation}
\label{moduli_space_system_eq}
\mathbb M(\mJ) =\Big\{ \mathcal M_{k+1,\beta}(\mJ,L) \mid (k,\beta)\in \mathbb N\times \G	\Big\}
\end{equation}
of all the moduli spaces simultaneously, and we call $\mathbb M(\mJ)$ a \textit{moduli space system} or a \textit{moduli system}.

The moduli space system $\mathbb M(\mJ)$ admits a system of Kuranishi space structures satisfying a list of \textit{axioms} (see \cite[Condition 21.6]{FOOOKuranishiVFC}, \cite[Condition 21.11]{FOOOKuranishiVFC} or \cite[Theorem 2.16]{FOOODiskTwo} when $P$ is a point).
Following \cite[Definition 21.9 \& Definition 21.13]{FOOOKuranishiVFC}, the data is called a \textit{$P$-parametrized tree-like K-system} (or simply a \textit{tree-like K-system}) on $\mathbb M(\mJ)$.
Note that the letter `K' stands for `Kuranishi'.
Instead of writing down the whole list of axioms, we just indicate they include the following aspects about the moduli space system: (A0) Kuranishi structures; (A1) Evaluation maps; (A2) Orientation; (A3) Compatibilities of boundary decompositions; (A4) Dimension; (A5) Energy and Gromov compactness.

A $P$-parameterized tree-like K-system on $\mathbb M(\mJ)$ will produce an $A_\infty$ algebra structure $\check\M$ on the space $\OL_P\equiv \Omega^*(P\times L)$. In this case, we say that \textit{``the $\M$ is obtained by a tree-like K-system on $\mathbb M(\mJ)$''} or simply say that \textit{``the $\check\M$ is obtained by $\mathbb M(\mJ)$''}.
We add the checks in the notations (e.g. $\check \m$ or $\check \M$) to emphasize the $A_\infty$ algebras are in the  chain-levels (i.e. defined on $\OL$ or $\OL_P$).

Now, we state the following Theorem \ref{A-theorem-main-body} due to Fukaya-Oh-Ohta-Ono.

\begin{thm}\label{A-theorem-main-body}
Given a smooth family $\mathbf J=\{J_t\mid t\in P\}$ in $\mathfrak J(X,\omega)$, there exists a $P$-pseudo-isotopy $(\OL_P, \check \M^P)\in \Obj \tilde{\UD}$, obtained by a $P$-parameterized tree-like K-system on $\mathbb M(\mJ)$, such that
(i) it is a q.c.dR; (Definition \ref{quantum-correction-defn}); (ii) it is fully unital and strictly unital with the constant-one $\one$ as a unit (Definition \ref{unit-defn}); (iii) it satisfies the divisor axiom (Definition \ref{DivisorAxiom-defn}); (iv) if a $\partial_i P$-pseudo-isotopy $\check \M^{\partial_i P}$ is already obtained like this for all $i$, then we may further assume $\check\M^P$ restricts to $\check \M^{\partial_i P}$ for each $i$ (Definition \ref{restriction_pseudo_isotopy_defn}).
If the family $\mJ$ is contained in $\mathfrak J(X,L,\omega)$, we can require $(\OL_P, \check \M^P)\in\Obj\UD$.
Moreover, if $\mJ=(J_t)$ is a constant family, we can require $\check \M^P$ to be a trivial $P$-pseudo-isotopy\footnote{We cannot say it \textit{must} be a trivial pseudo-isotopy, because it also depends on a choice of `virtual fundamental chain' on the moduli spaces and one can wildly make the choice.}.
\end{thm}

A version of Theorem \ref{A-theorem-main-body} has been proved in \cite[\S 21.2]{FOOOKuranishiVFC}, \cite[Theorem 21.35]{FOOOKuranishiVFC} and \cite[Theorem 2.16]{FOOODiskTwo}. However, we want to further use the techniques in \cite{FuCyclic} to further establish the divisor axiom and the full unitality.
Specifically, as explained in Remark \ref{cyclical_unitality_qcdR_rmk}, one can show the full unitality by the method of showing the strict unitality as in \cite[(7.3)]{FuCyclic}.
Moreover, although it is not explicitly stated in the latest literature in \cite{FOOOKuranishiVFC}, \cite{FOOOKuranishiVFC} or \cite{FOOODiskTwo}, the divisor axiom is proved in \cite[Lemma 13.1]{FuCyclic}.

We will give an explanation of the proof of Theorem \ref{A-theorem-main-body} later in \S \ref{ss_A_infty_rel_moduli}.
In practice, we only focus on the cases when $P$ is a $d$-simplex for $d=0,1,2$; the three cases respectively correspond to
the local charts, the transition maps, and the cocycle conditions for the mirror in Theorem \ref{Main_theorem_thm}. 
In particular,

\begin{thm}\label{Vir-m^J-thm}
There is an $A_\infty$ algebra $(\OL, \check\m^{J,L})$ in $\UD$ given by a tree-like K-system on $\mathbb M(J)$.
%
%
%
\end{thm}

Fix $J_0,J_1\in\mathfrak J(X, L,\omega)$. Suppose $\check \m^{J_0,L}$ and $\check \m^{J_1,L}$ are obtained as in Theorem \ref{Vir-m^J-thm}.
Let $\mathbf J=(J_t)$ be a path in $\mathfrak J(X,L,\omega)$ between $J_0$ and $J_1$. Then, we have:

\begin{thm}\label{Vir-M-path-thm}
There exists a pseudo-isotopy $(\OL_\oi, \check\M^{\mJ,L})$ in $\UD$, obtained by a tree-like K-system on $\mathbb M(\mJ)$, such that it restricts to $\check \m^{J_0,L}$ and $\check \m^{J_1,L}$ at $s=0,1$.
Moreover, if $\mathbf J$ is a constant family, then we can require $\check \M^{\mJ,L}$ to be a trivial pseudo-isotopy.
\end{thm}

\begin{convention}
	\label{convention_Xi_J}
	To be more precise, we should write $\check \m^{J,L, \Xi_J}$, further specifying a datum $\Xi=\Xi_J$ of Kuranishi-theory-related choices (Kuranishi structures, CF-perturbations, and so on). We often call $\Xi$ a virtual fundamental chain. Since it is also a contractible choice, we often omit $\Xi$ in the notations.
	Indeed, applying Theorem \ref{Vir-M-path-thm} to the special case $J_0=J_1$ implies that the $\check\m^{J,L,\Xi}$ is independent of the data $\Xi$ up to pseudo-isotopy; see e.g. \cite[Theorem 14.2]{FuCyclic}.
\end{convention}

Fix $J_0, J_1, J_2\in \mathfrak J(X,L,\omega)$;
suppose $\check \m^{J_0,L}$, $\check \m^{J_1,L}$, and $\check \m^{J_2,L}$ are obtained as in Theorem \ref{Vir-m^J-thm}.
Fix $\mJ_{i,i+1}$ to be a path in $\mathfrak J(X,L,\omega)$ from $J_i$ to $J_{i+1}$ for $i=0,1,2 \ (\mathrm{mod} \ 3)$; suppose $\check \M^{\mJ_{i,i+1}}$ are obtained as in Theorem \ref{Vir-M-path-thm}.
Let $\JJ=(J_t)_{t\in\Delta^2}$ be a family in $\mathfrak J(X,L,\omega)$ parameterized by the standard 2-simplex
$
\Delta^2=[v_0,v_1,v_2]
$
such that $\JJ|_{[v_i,v_{i+1}]}=\mJ_{i,i+1}$ for $i=0,1,2 \ (\mathrm{mod} \ 3)$. Then, we have:

\begin{thm}\label{Vir-M-triangle-thm}
There exists a $\Delta^2$-pseudo-isotopy $(\OL_{\Delta^2}, \check\M^{\JJ})$ in $\UD$, obtained by a tree-like K-system on $\mathbb M(\JJ)$, such that it restricts to $\check \M^{\mathbf J_{i,i+1}}$ on the edge $[v_i,v_{i+1}]\subset \Delta^2$.
\end{thm}

\subsubsection{$A_\infty$ relations}
\label{ss_A_infty_rel_moduli}

In this section, we explain how the $P$-parameterized tree-like K-system on the moduli system gives rise to a $P$-pseudo-isotopy in Theorem \ref{A-theorem-main-body} based on the foundational works of Fukaya-Oh-Ohta-Ono in \cite{FOOODiskOne,FOOODiskTwo,FOOOKuranishiVFC}.
We cannot present the full details, but we will make as precise as possible citations, indicating which statements in the references are used.

\textit{\textsf{Virtual fundamental chains .}}
The axiom item (A0) for the tree-like K-system means the existence of a system $\widehat {\mathcal U}=(\widehat{\mathcal U}_{k,\beta})$ of Kuranishi structures on the moduli system $\mathbb M(\mJ)$.
There is a general strategy of constructing `virtual fundamental chains' from a given Kuranishi structure; cf. \cite[\S 6.5]{FOOOKuranishiVFC} or \cite[\S 6.4]{FOOOKuOne}.
Roughly, we first take a good coordinate system 
and find a system $\widehat {\mathcal U^+}$ of collared Kuranishi structures which can be viewed as a `thickening' of $\widehat{\mathcal U}$ indicating the information of `gluing'.
Next, with regard to $\widehat{\mathcal U^+}$ , we can find a system $\widehat{\mathfrak S}$ of CF-perturbations which plays the role of `virtual fundamental chain' for the moduli system.

To be extremely careful, the $\widehat{\mathcal U^+}$ and $\widehat{\mathfrak S}$ depend on a parameter $\epsilon>0$ and are only applied to the system of moduli spaces with an \textit{energy cut} \cite[\S 22.2]{FOOOKuranishiVFC}. In the first place, we only produce $A_\infty$ algebras with energy cuts, but these $A_\infty$ algebras are pseudo-isotopic to each other with an energy cut $E>0$. By algebraic promotion, we pass it to a limit as $E\to\infty$ and obtain a true $A_\infty$ algebra over the Novikov field $\Lambda$. See the proof of \cite[Theorem 21.35]{FOOOKuranishiVFC} and \cite[3.36-3.39]{FuUnobstructed}.
Note that the Kuranishi theory generalizes the manifold theory, and many properties for manifolds still hold.
The two most important properties we need in the Kuranishi theory are the {Stokes' formula} \cite[Theorem 9.28]{FOOOKuranishiVFC}
and {composition formula} \cite[Theorem 10.21]{FOOOKuranishiVFC}.

\textit{\textsf{Composition formulas and Stokes' formulas.}}
By \cite[Definition 7.1]{FOOOKuranishiVFC}, a \textit{smooth correspondence} means a tuple $\mathfrak X=(X, \widehat{\mathcal U} M,M_0, f, f_0)$ consisting of a compact metrizable space $X$ equipped with a Kuranishi structure $\widehat {\mathcal U}$, two smooth manifolds $M_0$ and $M$, a strongly smooth map $f:(X,\widehat {\mathcal U})\to M$ and a weakly submersive strongly smooth map $f_0:(X,\widehat {\mathcal U} ) \to M_0$ in the sense of \cite[Definition 3.40 \& 3.43]{FOOOKuranishiVFC}.
A \textit{perturbed smooth correspondence} means $\mathfrak {\tilde X}=(\mathfrak X,\widehat{\mathfrak S})$ where $\widehat{\mathfrak S}$ is a CF-perturbation for $\widehat{\mathcal U}$ so that $f$ is strongly submersive with respect to $\widehat{\mathfrak S}$. We can always choose such $\widehat{\mathfrak S}$ up to thickenings of Kuranishi structures \cite[Lemma 9.26]{FOOOKuranishiVFC}.
Given this, we may define a map on the de Rham complexes
\begin{equation}
\label{Corr_eq}
\Corr_{\mathfrak {\tilde X}} : \Omega^*(M_0) \to \Omega^{\ell+*}(M)
\end{equation}
for $\ell= \dim (M) -\mathrm{vdim}(X)$ and we call it a \textit{correspondence map}. Now, suppose we have two perturbed smooth correspondences $\mathfrak {\tilde X}_{12}=(X_{12},f_1,f_2, \widehat{\mathfrak S_{12}})$ and $\mathfrak {\tilde X}_{23}=(X_{23},g_2,g_3, \widehat{\mathfrak S_{23}})$. It induces the \textit{fiber product} $\mathfrak {\tilde X}_{13}=(\mathfrak X_{13},h_1,h_3,\widehat{\mathfrak S_{13}})$ \cite[Definition 10.19]{FOOOKuranishiVFC}, illustrated as follows:
\begin{equation}
\label{composition_formula_diagram}
\Scale[0.75]{\xymatrix{
		& & \mathfrak X_{13} \ar@{-->}[dl] \ar@{-->}[dr] 
		\ar@/^1pc/[ddrr]^{\widehat{h_3}}
		\ar@/_1pc/[ddll]_{\widehat{h_1}} \\
		& \mathfrak X_{12} \ar[dl]^{\widehat f_1} \ar[dr]_{\widehat f_2} & & \mathfrak X_{23} \ar[dl]^{\widehat g_2} \ar[dr]_{\widehat g_3} \\
		M_1 & & M_2 & & M_3}
}
\end{equation}
The composition formula \cite[Theorem 10.21]{FOOOKuranishiVFC} means
$
\Corr_{\mathfrak {\tilde X}_{13}}=\Corr_{\mathfrak {\tilde X}_{23}}\circ \Corr_{\mathfrak {\tilde X}_{12}}
$.
We may also write
$
\Corr_{\mathfrak {\tilde X}_{23}} \big( \Corr_{\mathfrak {\tilde X}_{12}} (h_1) \times h_2 \big) =\Corr_{\mathfrak {\tilde X}_{13}} (h_1 \times  h_2)
$ following \cite[Proposition 4.3]{FuCyclic}.

The other important formula inherited from manifold theory is the \textit{Stokes' formula}. A perturbed smooth correspondence $ \tilde {\mathfrak X}$ induces a perturbed smooth correspondence $\partial  \tilde {\mathfrak X}$ on the normalized boundary \cite[9.27]{FOOOKuranishiVFC}. The Stokes' formula for Kuranishi structures states that
\begin{equation}
\label{Stokes_eq}
d_{M_0}\circ  \Corr_{  \mathfrak {\tilde X}} - \Corr_{\mathfrak {\tilde X}}\circ d_M=\Corr_{\partial \mathfrak {\tilde X}}
\end{equation}
where $d_{M_0}, d_M$ are the exterior derivatives on the spaces of differential forms on the manifolds $M_0, M$.

\textsf{\textit{Definition of the operator system.}}
The axiom item (A1) includes the strong smoothness of $\ev_\ell$ for $\ell \ge 0$ and weakly submersiveness of $\ev_0$ ,
and we have a smooth correspondence $\mathfrak X$ where $M_0=P\times L$, $M=(P\times L)^{\times k}$, $f_0=\ev_0$, $f=\ev_1\times\cdots \times \ev_k$, and $X=\mathcal M_{k+1,\beta}(\mathbf J,L)$.
As said above, we can make it a perturbed smooth correspondence $\tilde {\mathfrak X}$.
Then, we define
\begin{equation}
\label{defining_A_infty_algebra_eq}
\check \M_{k,\beta}(x_1,\dots,x_k) = \Corr_{\mathfrak {\tilde X}}(x_1\times \cdots \times x_k)= \Corr ( \mathcal M_{k+1,\beta}(\mathbf J, L) ; (\ev_1,\dots,\ev_k), \ev_0) (x_1\times \cdots \times x_k)
\end{equation}
where $x_1\times\cdots \times x_k$ denotes $\pi_1^*x_1\wedge \cdots \wedge \pi_k^*x_k$ for the projection maps $\pi_i$'s; cf. \cite[(7.1)]{FuCyclic}.
Exceptionally, we \textit{define} $\check \M_{0,0}=0$ and $\check \M_{1,0}=d$.
Remark that if $X$ happens to be a smooth manifold and $f_0$ is submersive, the correspondence map (\ref{Corr_eq}) just gives $\Corr(h)= \pm f_{0!} f^* (h)$ for a form $h$ on $M$. Accordingly, instead of (\ref{defining_A_infty_algebra_eq}), it is often convenient to use the following notation:
\begin{equation}
\label{defining_A_infty_II_eq}
\check \M_{k,\beta}(x_1,\dots,x_k)
=
(-1)^\epsilon \ev_{0!} (\ev_1^*x_1\wedge \cdots \wedge \ev_k^*x_k) \qquad \textstyle \epsilon=\sum_{j=1}^k j(\deg x_j +1) +1
\end{equation}
We adhere to the sign convention in \cite[\S 4.3]{Solomon_Involutions} diverging from that in \cite[(3.40)]{FuUnobstructed}.
Intuitively, the coefficient of $y$ in $\check \M_{k,\beta}(x_1,\dots,x_k)$ is given by the following `virtual count':
\begin{equation}
\label{Kura_vir_count_eq}
\pm\int^{\mathrm{vir}}_{\mathcal M_{k+1,\beta}(\mathbf J,L)} \ev_0^* y
\wedge \ev_1^* x_1 \wedge \cdots \wedge \ev_k^* x_k
\end{equation}
Now, we have defined an operator system $\check \M=(\check \M_{k,\beta})$ in $\CC_\G$.
We also know the degree $\deg \check \M_{k,\beta}=2-k-\mu(\beta)$ is as expected due to the dimension axiom (A4):
$
\mathrm{vdim} \mathcal M_{k+1,\beta}(\mathbf J,L)= \mu(\beta) +(k+1)-3+\dim(P\times L)
$.
The condition (A5) for the energy and Gromov compactness can infer that $\check \M$ is $\G$-gapped.
Now, it remains to show that the $A_\infty$ formula $\check \M \{ \check \M\}=0$.

\textsf{\textit{Proving the $A_\infty$ formula.}}
By the axioms (A2) and (A3), we have
\[
\partial \mathcal M_{k+1,\beta}(\mathbf J,L) \cong \bigsqcup_{\substack{\beta_1+\beta_2=\beta; \
		k_1+k_2=k; \
		1\le i \le k_2
}} (-1)^* \mathcal M_{k_1+1,\beta_1}(\mathbf J,L) \times_{(\ev_i,\ev_0)} \mathcal M_{k_2+1,\beta_2}(\mathbf J,L)
\]
where the left side uses the boundary smooth correspondence and the right side uses the composition of smooth correspondences (both are briefly described before).
Here $(k_1,\beta_1)$ and $ (k_2,\beta_2)$ are not equal to $(0,0),(1,0)$.
Let's temporarily abbreviate $\mathcal M_{k+1,\beta}:=\mathcal M_{k+1,\beta}(\mathbf J,L)$. Then, the diagram (\ref{composition_formula_diagram}) refers to
\[
\Scale[0.7]{
\xymatrix{
		& & 
		\mathcal M_{k_1+1,\beta_1}\times_{(\ev_i,\ev_0)} \mathcal M_{k_2+1,\beta_2} \ar[dl] \ar[dr]  \\
	& \mathcal M_{k_2+1,\beta_2} \ar[dl]^{(\ev_1,\dots,\ev_{k_2})} \ar[dr]_{\iota_i \circ  \ev_0} & & \mathcal M_{k_1+1,\beta_1} \ar[dl]^{(\ev_1,\dots,\ev_{k_1})} \ar[dr]_{\ev_0} \\
	(P\times L)^{\times k_2} & & (P\times L)^{\times k_1} & & P\times L
}
}
\]
Using (\ref{Stokes_eq}) and (\ref{defining_A_infty_algebra_eq}), we obtain
\begin{equation}
\label{Kura_A_infty_1_eq}
\Scale[0.85]{
	d\circ \check \M_{k,\beta} (x_1,\dots, x_k)+ \sum_j (-1)^* \check \M_{k,\beta} (x_1,\dots, dx_j,\dots, x_k)
	=
	\Corr(\partial \mathcal M_{k+1,\beta}; (\ev_1,\dots, \ev_k), \ev_0)
	(x_1\times \cdots \times x_k)
}
\end{equation}
Hence, applying the composition formula, the right side of (\ref{Kura_A_infty_1_eq}) is equal to
\begin{align*}
\Scale[0.9]{
\sum \pm \Corr ( \partial \mathcal M_{k_1+1,\beta_1}) (x_1\times \cdots \times \Corr(\partial \mathcal M_{k_2+1,\beta_2})(x_{i+1}\times \cdots ) \times \cdots \times x_k) 
= \sum \pm \check \M_{k_1,\beta_1}(x_1,\dots, \check \M_{k_2,\beta_2}(\dots)\dots, x_k)
}
\end{align*}
Recall we assume $(k_1,\beta_1)\neq (1,0)\neq (k_2,\beta_2)$ here; but the missing terms with $(k_i,\beta_i)=(1,0)$ can be replenished precisely by the left side of (\ref{Kura_A_infty_1_eq}) as $\check \M_{1,0}=d$.
So, we obtain the desired $A_\infty$ formula.
The signs are decided as in \cite[Condition 21.11]{FOOOKuranishiVFC} or \cite[Theorem 2.16]{FOOODiskTwo}; see also \cite{Solomon_Involutions}.

\textsf{\textit{Pointwiseness.}}
We check the $P$-pointwiseness (Definition \ref{pointwise-defn}) for a $P$-pseudo-isotopy (Definition \ref{P-pseudo-isotopy-defn}). This is basically because of the bifurcation property, i.e. the parameterized moduli spaces being the form of (\ref{moduli_system}) and the evaluation maps $\ev_i=(\overline{\ev}_P,\overline{\ev}_i)$ decomposing like (\ref{evaluation_maps_eq}).
See \cite[\S 22.3]{FOOOKuranishiVFC}. Briefly, we aim to show
$\check \M_{k,\beta} (x_1,\dots, x_i, \sigma \wedge x_{i+1}, x_{i+1} \dots, x_k)= (-1)^\dagger d\sigma \wedge \check \M_{k,\beta}(x_1,\dots, x_k)$
for the sign $\dagger=\sum_{j=1}^i (\deg x_j-1)\deg \sigma$.
By (\ref{Kura_vir_count_eq}), this amounts to count $(-1)^\epsilon \ev^* y
\wedge \ev^* x_1 \wedge \cdots \wedge \ev^* x_{i}\wedge \ev^*(\sigma  \wedge x_{i+1}) \cdots \wedge \ev^* x_k$. The naive switching $\ev^* x_i\wedge \ev^*\sigma =(-1)^{\epsilon'_i}\ev^*\sigma \wedge \ev^* x_i$ gives the sign $\epsilon'_i=\deg x_i \deg \sigma$; however, since the attaching marked point of $\sigma$ is changed from the $(i+1)$-th to the $i$-th marked point, there is an extra sign-change $(-1)^{\deg \sigma}$ due to (\ref{defining_A_infty_II_eq}). So, we set $\epsilon_i:=\epsilon'_i+\deg \sigma=(\deg x_i-1)\deg \sigma$, and the total sign change is exactly given by $\sum \epsilon_i=\dagger$.

\textsf{\textit{Trivial pseudo-isotopies.}}
Suppose the family of almost complex structures happens to be constant: $\mathbf J_0=\{J_t=J\mid t\in P \}$. Then, the moduli space (\ref{moduli_system}) is
$
\mathcal M_{k+1,\beta}(\mathbf J_0,L)\equiv P\times \mathcal M_{k+1,\beta}(J,L)
$
and the evaluation maps (\ref{evaluation_maps_eq}) will reduce to
$
\ev_i=( \overline{\ev}_P, \overline{\ev}_i)\equiv \id_P\times \overline{\ev}_i: P\times \mathcal M_{k+1,\beta}(J,L)  \to P\times L$.
Accordingly, if the tree-like K-system and `virtual fundamental chains' on the moduli space system $\mathbb M(\mJ_0)$ are the trivial extensions from the ones on $\mathbb M(J)$, then
the $A_\infty$ algebra $\check \M$ defined as in (\ref{defining_A_infty_algebra_eq})
will give rise to a \textit{trivial $P$-pseudo-isotopy} (Example \ref{trivial-pseudo-isotopy-ex}) about the $A_\infty$ algebra $\check \m$ for $\mathbb M(J)$.

\textsf{\textit{Contractible choices.}}
We explain the last item in Theorem \ref{A-theorem-main-body}, namely, we aim to obtain $\M^P$ with the specified restrictions $\M^{\partial_i P}$.
The reference is \cite[\S 17]{FOOOKuranishiVFC}.
First of all, we note that the stable map topology on any moduli space is well-understood.
If there exists a Kuranishi structure defined on a neighborhood of a compact space, then we can extend it to a Kuranishi structure globally without any change in that neighborhood \cite[\S 17.9]{FOOOKuranishiVFC}. By construction, the Kuranishi structure on the boundary of a moduli space has a collared extension to its small neighborhood \cite[Remark 17.1]{FOOOKuranishiVFC}.
Accordingly, given $\partial_iP$-parameterized tree-like K-systems on the moduli system $\mathbb M(\mJ|_{\partial_i P})$ for all $i$, one can find an extended $P$-parameterized tree-like K-system on the moduli system $\mathbb M(\mJ)$.
Moreover, the CF-perturbations are actually \textit{contractible} choices, i.e., there is also no obstruction to find an extension $\widehat{\mathfrak S}$ from the given $\widehat{\mathfrak S}_{\partial_i P}$ \cite[\S 17.8]{FOOOKuranishiVFC}. See also around \cite[(7.0.1)]{Pardon-vir}.
To complete Theorem \ref{A-theorem-main-body}, it remains to check the following three properties for the $\check \M$: (1) q.c.dR, (2) the divisor axiom, and (3) unitalties. Each of them will require mild restrictions on these contractible choices but there will be still plenty of them.

\subsubsection{Forgetful maps}
\label{Subsec_Forgetful_map}

We first explain that the $\check \M$ defined as above is indeed a q.c.dR in the sense of Definition \ref{quantum-correction-defn}.
It is known in the literature: \cite[Definition 21.21 \& Theorem 21.35]{FOOOKuranishiVFC} and \cite[Definition 3.5.6 \& Remark 3.5.8]{FOOOBookOne}. We give a brief review as follows.

If $\beta=0$, the moduli space has a simple form: $
\mathcal M_{k+1,0}(\mathbf J,L) \cong P\times L\times \mathbb R^{k-2}
$ for every $k\ge 2$; see \cite[Condition 21.11 (V)]{FOOOKuranishiVFC}. Geometrically, this is because the moduli space only contains the constant maps into $L$; the third component $\mathbb R^{k-2}$ corresponds to the ordered marked points in $\partial \mathbb D$; c.f. \cite[Lemma 1.3]{FuOh}.
In this case, the evaluation maps (\ref{evaluation_maps_eq}) reduce to the projection $\pr = (\id, \overline{\pr}): P\times L \times \mathbb R^{k-2} \to P\times L$. Thus, $\check \M_{k,0}(x_1,\dots,x_k) = \pm\pr_{!}\pr^* x_1 \wedge\cdots \wedge \pr^* x_k$.
When $k\ge 3$, there is a non-trivial component $\mathbb R^{k-2}$ for the fibers of $\pr$, so $\pr_! \pr^*=0$ and $\check \M_{k,0}=0$.
When $k=2$, we have $\pr=\id$ and so $\check \M_{2,0}(x_1,x_2)=(-1)^{\deg x_1} x_1\wedge x_2$. Here the sign is exactly obtained by using (\ref{defining_A_infty_II_eq}); c.f. \cite[Lemma 4.2]{Solomon_Involutions} or \cite[Proposition 3.7]{Solomon_Diff_survey}.

Notice that we also obtain the divisor axiom for $\beta=0$ as a byproduct.
Indeed, fix $b\in \Omega^1(P\times L)$ with $db=0$. When $k\ge 2$, since $\check \M_{k+1,0}=0$, we have $\check \M_{k+1,0}(b,x_1,\dots,x_k)+\cdots+ \check \M_{k+1,0}(x_1,\dots,x_k,b)=0$.
When $k=1$, we have $\check \M_{2,0}(b,x)+\check \M_{2,0}(x,b)= (-1)^{(\deg b-1)\deg x} b\wedge x + (-1)^{(\deg x-1)\deg b} x\wedge b=0$.

Next, we explain the divisor axiom based on the forgetful maps as used in \cite[Lemma 13.1]{FuCyclic}.
Roughly, the system of Kuranishi structures on the moduli spaces and the CF-perturbation can be chosen to be \textit{forgetful-map-compatible}
in the sense of \cite[\S 3, \S 5]{FuCyclic} \cite[Lemma 2.6.16]{FOOO_bookblue}.
This basically means the `local triviality' of these forgetful maps so that they behave like a `fibration', which ensures one can also define the pushforward $\forget_!$.
The divisor axiom for $\beta=0$ is proved above. So, we may assume $\beta\neq 0$, and we aim to show
\[
\DA[\check \M]_{k,\beta}(b;x_1,\dots,x_k)
=
\partial \beta \cap b \cdot \check \M_{k,\beta}(x_1,\dots,x_k)
\]
where $b$ is a divisor input (\ref{Divisor_input-eq}), i.e., a closed 1-form.
For $\ell=1,\dots, k+1$, the $\ell$-th \textit{forgetful map}
\begin{equation}
\label{forgetful_eq}
\forget_\ell:  \mathcal M_{k+2,\beta}(\mathbf J,L)\to \mathcal M_{k+1,\beta}(\mathbf J,L)
\end{equation}
sends $[\uu,\tilde \z]$ to the stablization of $[\uu,\z]$ where $\z$ is obtained by forgetting the $\ell$-th marked point in $\tilde \z$.
Concretely, if $\tilde \z =(z_0,z_1,\dots, z_{k+1})$, we define $\mathbf z=(z_0,\dots, z_{\ell-1},z_{\ell+1},\dots, z_{k+1})$.
Note that the sign of $\forget_\ell$ is equal to $(-1)^{\ell-1}$.
Denote by $\ev_i$ ($0\le i\le k+1$) and $\ev_i'$ ($0\le i\le k$) the evaluation maps (\ref{evaluation_maps_eq}) for the left and right sides of (\ref{forgetful_eq}) respectively. Then
\begin{equation}
\label{forgetful_ev_eq}
\ev_i' \circ \forget_\ell =
\begin{cases}
 \ev_i & \text{if} \quad  0\le i \le \ell -1	 \\
 \ev_{i+1} & \text{if} \quad	 \ell \le i \le k	
\end{cases}
\end{equation}
\[
\Scale[0.9]{
\xymatrix{
	& L & \\
	\mathcal M_{k+2,\beta}(\mathbf J,L) \ar[rr]^{\forget_\ell} \ar[ur]^{\ev_i \text{ or } \ev_{i+1}}
	& & \mathcal M_{k+1,\beta}(\mathbf J,L) \ar[ul]_{\ev'_i}
}
}
\]
The $\ev_\ell$ is missed in (\ref{forgetful_ev_eq}) and will correspond to the divisor input $b$.
By (\ref{defining_A_infty_II_eq}), we have
\begin{align*}
\check \M_{k+1,\beta}(x_1,\dots, x_{\ell-1}, b,x_{\ell},\dots, x_k)
=
& 
(-1)^\epsilon
\ev_{0!} \big( \ev_1^* x_1 \wedge \cdots \wedge \ev_{\ell-1}^* x_{\ell-1} \wedge \ev_\ell^* b\wedge \ev_{\ell+1}^* x_{\ell}\wedge \cdots \wedge \ev_{k+1}^* x_k \big)\\
=
&
(-1)^{\epsilon+\delta}
\ev'_{0!} \forget_{\ell !} \Big( \forget_\ell^* \big(\ev_1^{\prime *} x_1 \wedge \cdots \wedge \ev_k^{\prime *} x_k\big) \wedge \ev_\ell^* b \Big)
\end{align*}
where the sign $\epsilon$ is given by (\ref{defining_A_infty_II_eq}):
\[
\textstyle
\epsilon=1+\sum_{j=1}^{\ell-1} j(\deg x_j+1) +\ell(\deg b+1)+ \sum_{j=\ell}^k (j+1) (\deg x_j+1)
\]
and
$
\delta=\sum_{j=\ell}^{k} \deg x_j \deg b
$
is an extra sign due to the graded antisymmetricity of the wedge products.
In manifold theory, the integration along fiber $\pi_!$ satisfies that $\pi_! (\alpha \wedge \pi^*\beta)=\pi_!\alpha \wedge \beta$. One can generalize it to the Kuranishi theory, which implies that
\begin{align*}
\check \M_{k+1,\beta}(x_1,\dots, x_{\ell-1}, b,x_{\ell},\dots, x_k)
=
(-1)^{\epsilon+\delta+\ell-1}
\ev'_{0!} \big(  \ev^{\prime *}(x_1\times \cdots \times x_k) \cdot \forget_{\ell !} \ev_\ell^*b 
\big)
\end{align*}
Now, we put
$
\F_\ell:=\forget_{\ell!} \ev_\ell^* b
$.
Since the fibers of $\forget_\ell$ are one-dimensional, the degree of $\F_\ell$
is zero, and it can be thought of as a zero form on the moduli space like a `charge'. Fix a `point' $\mathbf p=(t, [\uu,\z])$ in $\mathcal M_{k+1,\beta}(\mathbf J,L)$. Then,
$
\F_\ell (\mathbf p) = \int_{\forget_\ell^{-1}(\mathbf p)} \ev_\ell^* b=\int_{\ev_\ell\big(\forget_\ell^{-1}(\mathbf p)\big)} b
$.
We may view $\ev_\ell \big( \forget_\ell^{-1}( \mathbf p )\big)$ as a singular chain (or current) in $P\times L$ that is roughly described by the set of $(t,\uu(z))$ for some marked point $z\in\partial\mathbb D$ lies between the $(\ell-1)$-th and $(\ell+1)$-th marked points.
But, the order constrains on the marked point $z$ can be eliminated by considering all possible $\ell$. In other words, by taking the summation over $\ell$, we obtain
$
\sum_{\ell} \ev_{\ell} \big( \forget_{\ell}^{-1}( \mathbf p )\big) 
=\{ (t, \uu(z)) \mid z\in\partial \mathbb D\}
=
(\iota_t\times \partial \uu)_*  \partial \mathbb D
$
where we denote by $\iota_t$ the inclusion $\{t\}\to P$. Hence
\[
\textstyle
\sum_\ell \F_\ell(\mathbf p)=
\int_{(\iota_t\times \partial \uu)_*  \partial \mathbb D} b
=
 \int_{\partial \mathbb D} (\iota_t\times \partial \uu)^* b= \partial \beta \cap b
\]
which actually does not depend on the choice of the moduli point $\mathbf p$. In summary, we have
\begin{align*}
\textstyle
\sum_\ell \check \M_{k+1,\beta}(x_1,\dots, x_{\ell-1}, b,x_\ell,\dots, x_k)
=
(-1)^{\epsilon+\delta+\ell-1}
\ev'_{0!} \big( \ev^{\prime *} (x_1\times \cdots \times x_k)
\cdot
\sum_\ell \F_\ell  \big)\\
=
\partial \beta \cap b \cdot (-1)^{\epsilon+\delta+\ell-1}\cdot \ev_{0!}' \ev^{\prime *} (x_1\times \cdots \times x_k)
=
\partial \beta\cap b \cdot \check \M_{k,\beta}(x_1,\dots, x_k)
\end{align*}
and thus the $\check \M$ satisfies the divisor axiom; see \cite[Lemma 13.1]{FuCyclic} for more details.

Next, we aim to show the unitality (Definition \ref{unit-defn}) and the cyclical unitality (Definition \ref{cyclical_unitality_defn}).
As explained in Remark \ref{cyclical_unitality_qcdR_rmk}, because of the q.c.dR property, we may assume $\beta\neq 0$, and it remains to show $\check \M_{k+1,\beta}(\dots,\e,\dots)=0$ for any $\e\in \Omega^0(P\times L)$.
To see this, we consider the forgetful map
in (\ref{forgetful_eq}) and we also have (\ref{forgetful_ev_eq}).
Performing a similar argument as above yields
\[
\check \M_{k+1,\beta}(x_1,\dots, x_{\ell-1}, \e,x_\ell,\dots, x_k)
=
\ev'_{0!} \big(	
\forget_{\ell !}\ev_\ell^*\e\cdot \ev^{\prime *} (x_1\times \cdots \times x_k)
\big)
\]
But this time $\forget_{\ell !}\ev^{*}_\ell\e$
has degree $-1$ and need to vanish. To be specific, one can focus on its local expression in a single Kuranishi chart by a partition of unity. See \cite[(7.3)-(7.4)]{FuCyclic} for more details.

\section{Harmonic contractions and Fukaya's trick}
\label{S_harmonic_contraction}

Note that implementing the homological perturbation (Theorem \ref{canonical_model_general_tree-thm}) requires choosing a contraction (Definition \ref{contraction-defn}), which typically isn't unique. In practice, we choose to employ the harmonic contraction relative to a specific metric. This approach is beneficial because of its explicit nature and its compatibility with Fukaya's trick in the context of minimal model $A_\infty$ algebras.

\subsection{$g$-harmonic contractions}
\label{ss_harmonic_contraction}

Fix a closed manifold $L$ and a Riemannian metric $g$. Put $C=\OL$ and $H=\HL$.
Let $\mathcal H^*_g(L)$ be the space of $g$-harmonic forms.
Firstly, we define a cochain map from $H$ to $C$:
\begin{equation}
\label{harmonic i(g)}
i(g):H^*(L) \xrightarrow{\cong} \mathcal H^*_g(L) \subset \Omega^*(L)
\end{equation}
by the inverse of Hodge isomorphism composited with the inclusion.
Secondly, we also define a cochain map from $\C$ to $\mH$ by the $g$-orthogonal projection $\mathcal H_g: \OL\to \mathcal H_g^*(L)$ composited with $i(g)^{-1}$:
\begin{equation}
\label{harmonic pi(g)}
\pi(g): \Omega^*(L) \xrightarrow{\mathcal H_g} \mathcal H^*_g(L) \xrightarrow{i(g)^{-1}} H^*(L)
\end{equation}
The Hodge decomposition theorem (see e.g. \cite{Warner}) tells that
$
\Omega^*(L)=\mathcal H_g^*(L) \oplus d \Omega^*(L) \oplus \delta_g \Omega^*(L)
$
where the direct sum is $g$-orthogonal and $\delta_g$ is the adjoint of $d$ with respect to the metric $g$. Moreover there is the so-called \textit{Green operator} $ \Gr_g: \Omega^*(L)\to \mathcal H^*_g(L)^{\perp_g} :=d\OL \oplus \delta_g\OL$. It is obtained by declaring $ \Gr_g(\alpha)$ to be the unique solution $\eta$ in $\mathcal H_g^*(L)^{\perp_g}$ of the differential equation $\Delta_g \eta = \alpha - \mathcal H_g(\alpha)$, where $\Delta_g=d \delta_g +\delta_g d$. Hence,
$
\Delta_g\circ \Gr_g = \id - \mathcal H_g$.
It is known that the Green operator $\Gr_g$ commutes with both $d$ and $\delta_g$ \cite{Warner}. Next, we define
\begin{equation}
\label{harmonic G(g)}
G(g) := -\Gr_g \circ \  \delta_g = - \delta_g  \circ \Gr_g
\end{equation}
Since $i(g)\circ \pi(g) = \mathcal H_g$,
it follows that 
\begin{equation}
\label{i(g)pi(g)-id_eq}
d \circ G(g) + G(g) \circ d = \mathcal H_g - \id = i(g) \circ \pi(g) - \id
\end{equation}
Given that harmonic forms are all $d$-closed, it follows that
$
d\circ i(g)=0$.
Since the $d$-exact forms are $g$-orthogonal to the harmonic forms, it also holds that
$
\pi(g) \circ d =0$.
By their definitions, the degrees of $i(g)$ and $\pi(g)$ are $0$, and the degree of $G(g)$ is $-1$.

We now seek to demonstrate that the triple $(i(g),\pi(g),G(g))$ constitutes a strong contraction in the sense of Definition \ref{contraction-defn}. To do so, we need to verify the four side conditions:
As $\pi(g)=i(g)^{-1} \circ \mathcal H_g$ and $\mathcal H_g\circ i(g)=i(g)$, we get
\begin{equation}
\label{pi(g) i(g)=id-eq}
 \pi(g)\circ i(g)=\id
\end{equation}
Secondly, since $\Gr_g$ commutes with $\delta_g$ and $\delta_g\circ\delta_g=0$, we immediately see that
\begin{equation}
\label{G(g) G(g)=0-eq}
G(g)\circ G(g)=0
\end{equation}
Thirdly, by definition, we know $\mathcal H_g\circ G(g)=0$ and so 
\begin{equation}
\label{pi(g) G(g)=0-eq}
\pi(g) \circ G(g)=0
\end{equation}
Fourthly, the image of $i(g)$ is a harmonic form which must be $\delta_g$-closed, thus
\begin{equation}
\label{G(g) i(g)=0-eq}
G(g)\circ i(g)=0
\end{equation}
Note that (\ref{pi(g) i(g)=id-eq}, \ref{G(g) G(g)=0-eq}, \ref{pi(g) G(g)=0-eq}, \ref{G(g) i(g)=0-eq}) correspond to (\ref{pi-i-id-eq}, \ref{side-conditions GG}, \ref{side-conditions piG}, \ref{side-conditions Gi}) separately.
In summary, we have proved:

\begin{lem}\label{strong-contraction-metric g-lem} 
$
\con(g):=(i(g),\pi(g),G(g))
$
 is a strong contraction, called the \textbf{$g$-harmonic contraction}.
\end{lem}

\begin{rmk}\label{harmonic constant-one-rmk}
Obviously, the constant-one $\one\in\Omega^*(L)$ is harmonic with regard to any metric $g$. Thus, for its cohomology class which we still denote by $\one=[\one]$, we have $i(g)(\one)=\one$, $\pi(g)(\one)=\one$, and so $i(g) \circ \pi(g) (\one) = \one$. 
Particularly, the condition $i(\pi(\one))=\one$ in Proposition \ref{canonical-model_preserve_properties-prop} always hold for $\con(g)$
\end{rmk}

Given an $A_\infty$ algebra $(\OL,\check\m)$, we can construct its canonical model 
$
(\HL,\m^g,\mi^g)
$
from $\con(g)$ by Theorem \ref{canonical_model_general_tree-thm}. Note that it is also implied that
\begin{equation}
\label{mi_g_10_eq}
\m^g_{1,0}=0, \qquad \mi^g_{1,0}=i(g)
\end{equation}

\begin{thm}
\label{UD-canonical-model-thm}
If $(\OL,\check \m)\in\Obj\UD$, then
$(\HL, \m^g)\in\Obj\UD$ and $\mi^g\in\Mor\UD$. 
\end{thm}
\begin{proof}
By Remark \ref{Gapped_canonical_precise-rmk}, the (I-5) (II-5) in Definition \ref{UD_defn} hold for $\m^g$ and $\mi^g$.
Because of $\partial\beta\cap i(g)(b)=\partial\beta\cap b$ and $i(g)\pi(g)(\one)=\one$, the theorem then just follows from Lemma \ref{strong-contraction-metric g-lem}
and Proposition \ref{canonical-model_preserve_properties-prop}.
\end{proof}

We need to further exploit the Hodge theorem. 
Denote by $Z^*(L):= \ker d \subset \Omega^*(L)$ the space of closed forms; then, $Z^*(L)=\mathcal H^*_{g}(L)\oplus d\Omega^*(L)$.
Note that we have an isomorphism:
\begin{equation}\label{d_g-eq}
d_g:=d: \delta_{g} \Omega^*(L) \xrightarrow{\cong} d\Omega^*(L)
\end{equation}
Note that $\ker \delta_{g}=\mathcal H^*_{g}(L)\oplus \delta_{g}\Omega^*(L)$.
The restriction of $\delta_g$ also gives an isomorphism:
\begin{equation}
\label{delta_g-eq}
\delta'_{g}:=\delta_{g}|_{d\OL}: d\Omega^*(L)\xrightarrow{\cong} \delta_{g_t}\Omega^*(L)
\end{equation}

We aim to generalize Theorem \ref{UD-canonical-model-thm} to the case of pseudo-isotopies.
Take a smooth family of metrics $\mathbf g:=(g_t)_{t\in\oi}$.
We first abbreviate the above-mentioned operators: $(i_t, \pi_t, G_t):=(i(g_t), \pi(g_t), G(g_t))$, $d_t:=d_{g_t}$, and $\delta_t:=\delta'_{g_t}$. Then, we have the following two isomorphisms:
\[
d_t\circ \delta_t: d\Omega^*(L) \to \delta_{g_t}\Omega^*(L) \to d\Omega^*(L); \ \ \ \ \
\delta_t\circ d_t: \delta_{g_t} \Omega^*(L) \to d\Omega^*(L) \to \delta_{g_t}\Omega^*(L)
\]
Denote by $\Delta_t$ the restriction of the Laplacian $\Delta_{g_t}=d\delta_{g_t}+\delta_{g_t}d$ on the space $\mathcal H_{g_t}^*(L)^\perp \equiv d\Omega^*(L)\oplus \delta_{g_t}\Omega^*(L)$. Then, $\Delta_t=d_t\delta_t\oplus \delta_t d_t$ and
\begin{equation}
\label{Delta_g_inverse-eq}
\Delta_t^{-1} =\delta_t^{-1}d_t^{-1}\oplus d_t^{-1}\delta_t^{-1} : d\Omega^*(L) \oplus \delta_{g_t}\Omega^*(L) \to  d\Omega^*(L)\oplus \delta_{g_t}\Omega^*(L) 
\end{equation}
Accordingly, the Green operator $\Gr_{g_t}$ can be explicitly expressed by
\begin{equation}
\label{Gr_t-eq}
\Gr_{g_t}= \Delta_t^{-1} \circ (\id - \mathcal H_{g_t})
\end{equation}
where $\id-\mathcal H_{g_t}$ is just the projection to $\mathcal H_{g_t}^*(L)^\perp=d\Omega^*(L)\oplus \delta_{g_t}\Omega^*(L)$. By (\ref{harmonic G(g)}), we get
$
G_t \circ \delta_{g_t}=0$.
Further, it follows from (\ref{Delta_g_inverse-eq}) and (\ref{Gr_t-eq}) that
\begin{equation}
\label{G_t|Z*(L)-eq}
G_t|_{Z^*(L)}= -\delta_{g_t} \circ \Delta_t^{-1} \circ \pr_{d\Omega^*(L)}=-d_t^{-1} \circ \pr_{d\Omega^*(L)}
\end{equation}
where the `$\pr$' stands for the orthogonal projection. In particular, this tells $G_t|_{d\Omega^*(L)}=-d_t^{-1}$.

\begin{lem}\label{harmonic h_t k_t-lem}
There exists two smooth families of operators $h_t: H^*(L) \to \Omega^*(L)$ and $k_t: \Omega^*(L) \to H^*(L)$ of degree $-1$ with the following properties: 
\begin{align}
\label{di_t dt =d h_t -eq}
\textstyle \frac{di_t}{dt}= d \circ h_t \\
\label{dpi_t dt=k_t d -eq}
\textstyle \frac{d\pi_t}{dt}= k_t \circ d \\
\label{pi_t h_t=0-eq}
\pi_t\circ h_t=0  \\
\label{k_t i_t=0-eq}
k_t\circ i_t =0   \\
\label{G_t h_t=0_eq}
G_t\circ h_t=0  \\
\label{k_t G_t=0_eq}
k_t\circ G_t=0
\end{align}
\end{lem}

\begin{proof}
On the one hand, recall that $d\circ i_t=0$, hence we may think $i_t: H^*(L) \to Z^*(L)$. 
For the natural projection $q: Z^*(L)\to H^*(L)$, we have $q\circ i_t =\id$ and so $q\circ \frac{d i_t}{dt}=0$.
Thus, the image of $\frac{di_t}{dt}: H^*(L) \to \Omega^*(L)$ is contained in the kernel of $q$, i.e. the space $d\Omega^*(L)$ of $d$-exact forms.
We define
\begin{equation}
\label{h_t-eq}
h_t:=d_t^{-1} \circ \frac{di_t}{dt}: H^*(L) \to d\Omega^*(L) \to \delta_{g_t} \Omega^*(L)\subset  \Omega^*(L)
\end{equation}
In special, we have $\frac{di_t}{dt}=d\circ h_t$. 
On the other hand, since $\pi_t\circ d= 0$, we have $\pi_t|_{Z^*(L)}=q$. So, $\frac{d\pi_t}{dt}|_{Z^*(L)}=0$, and one can view $\frac{d\pi_t }{dt}$ as an operator on $\delta_{g_t}\OL$. Then, we define
\begin{equation}\label{k_t-eq}
k_t:=
\begin{cases}
 \frac{d\pi_t}{dt}\circ d_t^{-1}   &\text{on \ } d\OL \\
0 &\text{on \ } \mathcal H_{g_t}(L)\oplus \delta_{g_t}\Omega^*(L)
\end{cases}
\end{equation}
Thus, $k_t\circ \mathcal H_{g_t}=0$ and $k_t\circ \delta_{g_t}=0$.
As an operator on $\delta_{g_t} \Omega^*(L) $, we have $\frac{d\pi_t}{dt}= k_t \circ d$.
The complement of $\delta_{g_t}\OL$ is $Z^*(L)$; both $d$ and $\frac{d\pi_t}{dt}$ vanish on $Z^*(L)$. The relation $\frac{d\pi_t}{dt}=k_t\circ d$ still holds in $Z^*(L)$.

Now, we show the last four properties:
First, since the image of $h_t$ is contained in $\delta_{g_t}\Omega^*(L)$, we have $\pi_t\circ h_t=0$. 
Second, since the image of $i_t$ is in $\mathcal H_{g_t}^*(L)$, we also have $k_t\circ i_t=0$.
Third, since the image of $h_t$ is contained in $\delta_{g_t}\OL$, it follows that $G_t\circ h_t=0$.
Fourth, using (\ref{harmonic G(g)}) the image of $G_t$ is contained in $\delta_{g_t}\OL$ on which $k_t$ is defined to be zero; hence, $k_t\circ G_t=0$.
\end{proof}

Next, we want to study the relations to $G_t$. Consider the following degree-$(-1)$ operator
\begin{equation}
\label{K_t___eq}
\textstyle
\Gamma_t:=\frac{dG_t}{dt}-i_t \circ k_t -  h_t \circ \pi_t: \OL \to \OL
\end{equation}
Using $d\circ i_t=\pi_t\circ d=0$ and Lemma \ref{harmonic h_t k_t-lem}, we compute
\begin{align*}
d\circ \Gamma_t+ \Gamma_t\circ d
=
d\circ \tfrac{dG_t}{dt}+ \tfrac{dG_t}{dt}\circ d -d \circ h_t\circ \pi_t-i_t\circ k_t\circ d
=
\tfrac{d}{dt} \big(	
d\circ G_t+G_t\circ d-i_t\circ \pi_t
\big)
\end{align*}
Then, applying (\ref{i(g)pi(g)-id_eq}) yields that
\begin{equation}
\label{Gamma_td+dGamma_t-eq}
d\circ \Gamma_t +\Gamma_t\circ d=0
\end{equation}
Therefore, up to a sign, the $\Gamma_t$ is a cochain map from $\OL$ to $\OL$.
The induced map $\HL\to\HL$ on the cohomologies will be zero, and actually we can prove the following stronger result:
\begin{equation}
\label{Gamma_t-i_t=0}
\Gamma_t\circ i_t=0
\end{equation}
In fact, using $\pi_t\circ i_t=\id$, (\ref{k_t i_t=0-eq}), (\ref{G(g) i(g)=0-eq}), and (\ref{G_t|Z*(L)-eq}), we obtain
$
\Gamma_t\circ i_t= \tfrac{dG_t}{dt}\circ i_t -h_t= -G_t \circ \tfrac{di_t}{dt}-d_t^{-1}\circ \tfrac{di_t}{dt}=0$.

\begin{lem}\label{harmonic-sigma-lem}
In the situation of Lemma \ref{harmonic h_t k_t-lem}, there exists a smooth family of operators $\sigma_t: \Omega^*(L) \to \Omega^*(L)$ of degree $-2$, satisfying the following properties:
\begin{align}
\label{dG_t-dt-ik-hpi-eq}
\tfrac{dG_t}{dt}- i_t\circ k_t - h_t\circ \pi_t &= d \circ \sigma_t -\sigma_t \circ d \\
\label{sigma-G=G-sigma=0-eq}
\sigma_t\circ G_t&= G_t \circ \sigma_t=0  \\
\label{pi_t-sigma_t=0}
\pi_t \circ \sigma_t &=0  \\
\label{sigma_t-i_t=0}
\sigma_t\circ i_t  &=0
\end{align}
\end{lem}

\begin{proof}
By (\ref{Gamma_td+dGamma_t-eq}) and (\ref{Gamma_t-i_t=0}), the operator $\Gamma_t:=\frac{dG_t}{dt}-i_t\circ k_t-h_t\circ \pi_t$ maps $Z^*(L)$ into $d\OL$. Define:
\begin{equation}\label{sigma_t-(1)-eq}
\sigma_t:=
\begin{cases}
d^{-1}_t \circ \Gamma_t   &\text{on \ } Z^*(L) \\
0 &\text{on \ } \delta_{g_t}\Omega^*(L)
\end{cases}
\end{equation}
It remains to check the four properties.
The first one can be checked as follows: $(d\circ \sigma_t-\sigma_t\circ d)|_{Z^*(L)}= d\circ \sigma_t|_{Z^*(L)}=\Gamma_t|_{Z^*(L)}$ and $(d\circ \sigma_t-\sigma_t\circ d)|_{\delta_{g_t}\Omega^*(L)}=-\sigma_t\circ d|_{\delta_{g_t}\OL}=-d_t^{-1}\circ \Gamma_t\circ d|_{\delta_{g_t}^*\OL}=d_t^{-1}\circ d\circ \Gamma_t|_{\delta_{g_t}^*\OL}=\Gamma_t|_{\delta_{g_t}^*\OL}$ by (\ref{Gamma_td+dGamma_t-eq}).
Next, note that the images of both $\sigma_t$ and $G_t$ are contained in $\delta_{g_t}\OL$ and that $\sigma_t\circ \delta_{g_t}=G_t\circ \delta_{g_t}=\pi_t\circ \delta_{g_t}=0$; so, we have $\sigma_t\circ G_t=G_t\circ \sigma_t=0$ and $\pi_t\circ \sigma_t=0$. Finally $\sigma_t\circ i_t=0$ holds just because of (\ref{Gamma_t-i_t=0}).
\end{proof}

The purpose to find the above operators is to obtain a \textit{strong} contraction for the pair $(H^*(L)_\oi, d^\oi)$ and $(\OL_\oi, \M_{1,0})$. Recall that the two differentials $d^\oi$ and $\M_{1,0}$ are defined as in (\ref{MP_10 -d -eq}):
\[
d^\oi: \HL_\oi \to \HL_\oi, \ \ \ \ 
 \begin{cases}
 1\otimes  \bar x & \mapsto ds \otimes \partial_s( \bar x) \\
 ds\otimes  \bar x & \mapsto 0
 \end{cases}
\]
\[
\M_{1,0}: \OL_\oi \to \OL_\oi, \ \ \ \ 
 \begin{cases}
 1\otimes   x & \mapsto 1\otimes \m_{1,0}(x) + ds \otimes \partial_s(  x) \\
 ds\otimes  x & \mapsto -ds\otimes \m_{1,0} (x)
 \end{cases}
\]
where we denote by $\m_{1,0}=d$ the exterior differential on $\OL$.
The notations $\m_{1,0}$ or $\M_{1,0}$ hint but are not related to any $A_\infty$ algebra temporarily.
In addition, we define:
\begin{equation}
\label{i_pi_G_[0,1]-eq}
\Scale[0.9]{
\begin{aligned}
& i(\mathbf g): \HL_\oi \to \OL_\oi, \ \ \ \ &&
 \begin{cases}
 1\otimes \bar x & \mapsto 1\otimes i_s(\bar x) + ds \otimes h_s(\bar x) \\
 ds\otimes \bar x & \mapsto ds\otimes i_s(\bar x)
 \end{cases} \\
& \pi(\mathbf g): \OL_\oi \to \HL_\oi, \ \ \ \ &&
 \begin{cases}
 1\otimes  x & \mapsto 1\otimes \pi_s( x) + ds \otimes k_s( x) \\
 ds\otimes  x & \mapsto ds\otimes \pi_s( x)
 \end{cases} \\
&
G(\mg): \OL_\oi \to \OL_\oi, \ \ \ \ &&
 \begin{cases}
 1\otimes  x & \mapsto 1\otimes G_s( x) + ds \otimes \sigma_s( x) \\
 ds\otimes  x & \mapsto -ds\otimes G_s( x)
 \end{cases}
 \end{aligned}
}
\end{equation}

The signs above respect the pointwiseness in Definition \ref{pointwise-defn}. Hence, in view of Remark \ref{pointwise_determine_rmk}, we may more concisely write
$
d^\oi=ds\otimes \partial_s$ and $
\M_{1,0}=1\otimes d+ds\otimes \partial_s$ for the differentials and write
$i(\mg)=1\otimes i_s  \ +ds\otimes h_s$, $\pi(\mg)=1\otimes \pi_s +ds\otimes k_s$
and $G(\mg)=1\otimes G_s+ds\otimes \sigma_s$.

The following statement is basically a consequence of Lemma \ref{harmonic h_t k_t-lem} and \ref{harmonic-sigma-lem}:

\begin{lem}\label{strong-contraction [0,1]-lem}
$
\con(\mg):=(i(\mathbf g),\pi(\mathbf g),G(\mg))
$
is a strong contraction for $H^*(L)_\oi$ and $\Omega^*(L)_\oi$.
\end{lem}

\begin{proof}
Our goal now is to show all the properties listed in Definition \ref{contraction-defn}. To begin with, the degrees are clearly as expected: 
$\deg i(\mg)=\deg \pi(\mg)=0$ and $\deg G(\mg)=-1$.
We note that $\partial_si_s=\partial_s\circ i_s - i_s\circ \partial_s$, $\partial_s \pi_s=\partial_s\circ \pi_s - \pi_s\circ \partial_s$, and $\partial_sG_s=\partial_s\circ G_s-G_s\circ \partial_s$.
Firstly, we aim to prove $i(\mg)$, $\pi(\mg)$ are cochain maps.
By (\ref{di_t dt =d h_t -eq}), 
$
\M_{1,0}\circ i(\mg)
=
1\otimes d\circ i_s +ds\otimes \left( \partial_s\circ i_s-d\circ h_s \right) 
=
ds\otimes i_s\circ \partial_s = i(\mg)\circ d^\oi
$. Then, by (\ref{dpi_t dt=k_t d -eq}), we have
$
\pi(\mg)\circ \M_{1,0}
=
1\otimes \pi_s\circ d +ds\otimes (k_s\circ d+ \pi_s\circ \partial_s)
=
ds\otimes \partial_s\circ \pi_s
=
d^\oi \circ \pi(\mg)
$.
Secondly, we compute
\begin{align*}
i(\mg)\circ \pi(\mg) -\id
&
=1\otimes (i_s\circ \pi_s-\id)+ds\otimes 
(
i_s\circ k_s+h_s\circ \pi_s
)
\\
\M_{1,0}\circ G(\mg)
&
=
1\otimes d\circ G_s +ds\otimes (\partial_s\circ G_s-d\circ \sigma_s)\\
G(\mg)\circ \M_{1,0}
&
=
1\otimes G_s\circ d+ds\otimes (\sigma_s\circ d-G_s\circ \partial_s)
\end{align*}
Applying Lemma \ref{strong-contraction-metric g-lem} to $g_s$ infers that $i_s\circ \pi_s-\id=d\circ G_s+G_s\circ d$. Then, by (\ref{dG_t-dt-ik-hpi-eq}), we get $i(\mg)\circ \pi(\mg)-\id =\M_{1,0}\circ G(\mg) +G(\mg)\circ \M_{1,0}$.
Thirdly, since $\pi_s\circ i_s=\id$, by (\ref{k_t i_t=0-eq}) and (\ref{pi_t h_t=0-eq}) we obtain
\begin{align*}
\pi(\mg)\circ i(\mg)
=
1\otimes \pi_s\circ i_s +ds\otimes ( \pi_s\circ h_s+k_s\circ i_s) =1\otimes \id =\id
\end{align*}
Fourthly, note that $G_s\circ G_s=0$; by using (\ref{sigma-G=G-sigma=0-eq}) we obtain
\begin{align*}
G(\mg)\circ G(\mg)
=
1\otimes G_s\circ G_s +ds\otimes (\sigma_s\circ G_s-G_s\circ \sigma_s)
=0
\end{align*}
Fifthly, note that $G_s\circ i_s=0$; by using (\ref{sigma_t-i_t=0}) and (\ref{G_t h_t=0_eq}) we obtain
\begin{align*}
G(\mg)\circ i(\mg)
=
1\otimes G_s\circ i_s +ds\otimes (\sigma_s\circ i_s - G_s\circ h_s)=0
\end{align*}
Sixthly, note that $\pi_s\circ G_s=0$; by using (\ref{k_t G_t=0_eq}) and (\ref{pi_t-sigma_t=0}) we obtain
\begin{align*}
\pi(\mg)\circ G(\mg)=1\otimes \pi_s\circ G_s +ds\otimes (k_s\circ G_s+\pi_s\circ \sigma_s)=0
\end{align*}
In conclusion, we have checked all the conditions in Definition \ref{contraction-defn}, and the lemma is now proved.
\end{proof}

\begin{lem}
\label{strong-contraction [0,1]-property-lem}
The operators $i(\mathbf g)$, $\pi(\mathbf g)$ and $G(\mg)$ are compatible with $\eval^s$ in the sense that
$\eval^s\circ i(\mathbf g) = i_s\circ \eval^s$, $\eval^s\circ \pi(\mathbf g) = \pi_s\circ \eval^s$, and
$\eval^s\circ G(\mg) = G_s\circ \eval^s$.
\end{lem}

\begin{proof}
The proof is direct.
$
\eval^s\circ i(\mathbf g)( 1\otimes \bar x)=\eval^s \big( 1\otimes i_s(\bar x(s)) +ds \otimes h_s (\bar x(s)) \big) =i_s(\bar x(s))= i_s \circ \eval^s (1\otimes \bar x)
$
and 
$
\eval^s\circ i(\mathbf g)(ds\otimes \bar x)=
\eval^s \big(ds\otimes i_s\bar x(s) \big)=0=i_s\circ \eval^s (ds\otimes \bar x)
$.
Similarly, one can show the lemma for $\pi(\mg)$ and $G(\mg)$.
\end{proof}

\subsection{Pseudo-isotopies of canonical models}
\label{ss_pseudo_isotopy_canonical_model}

Fix a gapped pseudo-isotopy
$
(\OL_\oi,\check \M)
$, and we write
\begin{equation}
\label{check_M_harmonic_section_eq}
\check \M=1\otimes \check \m^s+ds\otimes \check \mc^s
\end{equation}
Take a smooth path $\mg=(g_s)_{0\le s\le 1}$ of metrics on $L$ for which we have 
a strong contraction $\con(\mg)=(i(\mg),\pi(\mg),G(\mg))$ by Lemma \ref{strong-contraction [0,1]-lem}.
Moreover, for each single $g_s$, we also have a harmonic contraction
$
\con(g_s)=
(i(g_s),\pi(g_s),G(g_s))
$
by Lemma \ref{strong-contraction-metric g-lem}.
Abusing the notations, the various constant-one functions are all denoted by $\one$.
By Remark \ref{harmonic constant-one-rmk}, we have
$
i(g_s)(\one)=\one$, $\pi(g_s)(\one)=\one$, and $i(g_s)\circ \pi(g_s)(\one)=\one$.

Every $\check \m^s$ given in (\ref{check_M_harmonic_section_eq}) is an $A_\infty$ algebra on $\OL$; by Theorem \ref{UD-canonical-model-thm}, we get its canonical model with respect to $\con(g_s)$ (Definition \ref{canonical_model_defn}):
\begin{equation*}
\label{m_g_s_eq}
(\HL,\m^{g_s}, \mi^{g_s}) \qquad \text{with  } 
\m^{g_s}\in\Obj\UD, \ \ \mi^{g_s}\in\Mor\UD
\end{equation*}
Note that Theorem \ref{UD-canonical-model-thm} is just a special situation of Theorem \ref{canonical_model_general_tree-thm}; our purpose is to show an analog of Theorem \ref{UD-canonical-model-thm} for a pseudo-isotopy of canonical models.
To be specific, we first apply Theorem \ref{canonical_model_general_tree-thm} to the chain-level pseudo-isotopy $\check \M$ and the $\con(\mg)$, thereby obtaining a canonical model
\begin{equation}\label{M_mg_HL_oi}
(\HL_\oi, \M^\mg, \mI^\mg)
\end{equation}
of $(\OL_\oi, \check \M)$ such that
\begin{equation}
\label{mI_mg_10_eq}
\mI^\mg_{1,0}=i(\mg) \quad \text{and}\quad \M^\mg_{1,0} =ds\otimes \tfrac{d}{ds}
\end{equation}

\begin{lem}
	\label{canonical-model-[01]Pointwise-mi-lem}
Both $\M^\mg$ and $\mI^\mg$ are $\oi$-pointwise. Moreover, we have:
\begin{equation}
\label{eval_MImg_eq}
\eval^s   \M^\mg
=\m^{g_s} \diamond \eval^s, \qquad \text{and}\quad
\eval^s   \mI^\mg
=\mi^{g_s}\diamond \eval^s
\end{equation}
\end{lem}

\begin{proof}
By Lemma \ref{strong-contraction [0,1]-property-lem}, the $i(\mg)$, $\pi(\mg)$, and $G(\mg)$ are $\oi$-pointwise.
Besides, an iterated composition of pointwise operators is still pointwise. So, by construction, the $\M^\mg$ and $\mI^\mg$ are also $\oi$-pointwise.
By Remark \ref{pointwise_determine_rmk}, we write $\mI^\mg=1\otimes  \mi^s+ds\otimes \mathfrak j^s$ and $\M^\mg=1\otimes \m^s+ds\otimes \mc^s$.
Note that $\mi^s=\eval^s\diamond \mI^\mg\diamond \incl$.

Now, it remains to show $\mi^s$ and $\m^s$ coincide with $\mi^{g_s}$ and $\m^{g_s}$ respectively.
To see this, we look at the inductive formula (\ref{induction-tree-formula-eq}). By Lemma \ref{strong-contraction [0,1]-property-lem} and Remark \ref{eval-incl-as-A_infty-rmk}, we have 
\begin{equation} \label{precompose_before_eq}
\begin{aligned}
\eval^s   \mI^\mg_{k,\beta}
&
=
\textstyle \sum \eval^s \circ  G(\mg) \circ \check \M_{\ell,\beta_0} \circ (\mI^\mg_{k_1,\beta_1}\otimes \cdots\otimes \mI^\mg_{k_\ell,\beta_\ell})\\
&
=
\textstyle \sum  G(g_s) \circ \eval^s \circ \check \M_{\ell,\beta_0} \circ (\mI^\mg_{k_1,\beta_1}\otimes \cdots\otimes \mI^\mg_{k_\ell,\beta_\ell}) \\
&
=
\textstyle \sum G(g_s) \circ \check \m^s_{\ell,\beta_0} \circ (\eval^s \ \mI^\mg_{k_1,\beta_1}\otimes \cdots\otimes \eval^s \  \mI^\mg_{k_\ell,\beta_\ell})
\end{aligned}
\end{equation}
On the one hand, we can further pre-compose with the $\incl$ map in the above, thereby obtaining that
\begin{equation*}
\mi^s_{k,\beta}
\equiv
\eval^s\diamond \mI^\mg_{k,\beta}\diamond \incl
=
\textstyle \sum G(g_s)\circ \check \m^s_{\ell,\beta_0}\circ (\mi^s_{k_1,\beta_1} \otimes \cdots \otimes \mi^s_{k_\ell,\beta_\ell})
\end{equation*}
In other words, the $\mi^s$ and $\mi^{g_s}$ share the same inductive formulas (\ref{induction-tree-formula-eq}) with respect to the same $\check \m^s$.
Moreover, by (\ref{i_pi_G_[0,1]-eq}) (\ref{mI_mg_10_eq}) (\ref{mi_g_10_eq}), the initial cases agree:
$\mi^s_{1,0}=\eval^s\circ \mI^\mg_{1,0}\circ \incl= \eval^s\circ i(\mg)\circ \incl = i(g_s)=\mi^{g_s}_{1,0}$.
Hence, by induction, we conclude that $\mi^s=\mi^{g_s}$. Arguing in the same way but replacing $G(g_s)$ by $\pi(g_s)$ everywhere, we can also inductively prove $\m^s=\m^{g_s}$.
On the other hand, if we do not pre-compose with $\incl$'s as above,
then performing the induction arguments with (\ref{precompose_before_eq}) directly yields that $\eval^s  \mI^\mg=i^{g_s}\diamond \eval^s$ and $\eval^s  \M^\mg =\m^{g_s}\diamond \eval^s$.
\end{proof}


\begin{thm}
	\label{UD-canonical-model-[0,1]-thm}
If $(\OL_\oi, \check\M) \in\Obj\UD$, then $(\HL_\oi, \M^\mg) \in\Obj\UD$ and $\mI^\mg\in\Mor \UD$. 
\end{thm}

\begin{proof}
	This an analog of Theorem \ref{UD-canonical-model-thm}.
First, by Lemma \ref{canonical-model-[01]Pointwise-mi-lem} we know $\M^{\mg}$ is a pseudo-isotopy.
By degree reasons, we see $\pi(\mg)(\one)=1\otimes \pi(g_s)(\one)=1\otimes \one=\one$ and $i(\mg)(\one)=1\otimes i(g_s)(\one)=1\otimes \one=\one$. In particular, $i(\mg) (\pi(\mg)(\one))=\one$ holds.
For a divisor input $b\in \HL_\oi$, the cap product (\ref{cap_product_eq}) by definition satisfies that
$\partial\beta\cap i(\mg)(b)=
\partial\beta\cap b$.
Finally, one just use Proposition \ref{canonical-model_preserve_properties-prop} (i)(ii)(iii) and Remark \ref{Gapped_canonical_precise-rmk} to complete the proof.
\end{proof}

%

\subsubsection{An upshot for pseudo-isotopy-induced $A_\infty$ homomorphisms}
Suppose we are in the situation of Theorem \ref{UD-canonical-model-[0,1]-thm}. From the two pseudo-isotopies $(\OL_\oi,\check \M)\in\Obj\UD$ and $(\HL_\oi,\M^\mg)\in\Obj\UD$, the natural construction in Theorem \ref{from-pseudo-isotopies-to-A-infty-homo-thm}
produces two $A_\infty$ homomorphisms $\check \mC$ and $\mC^\mg$ respectively. Moreover, by Theorem \ref{UD-mC_thm}, we also know that $\check \mC\in \Mor\UD$ and $\mC^\mg\in\Mor\UD$.

\begin{lem}\label{mC-diamond-diagram-lem}
In the above situation, we have $\mi^{g_1}\diamond \mC^\mg \simud \check \mC\diamond \mi^{g_0}$
\end{lem}
\begin{proof}
According to Theorem \ref{UD-mC_thm}, we know $\check\mC \diamond  \eval^0\simud \eval^1$ and $\mC^\mg \diamond \eval^0\simud \eval^1$. Also, by Lemma \ref{canonical-model-[01]Pointwise-mi-lem}, we have $\eval^s \diamond \mI^{\mg}= \mi^{g_s} \diamond  \eval^s$ for $s=0,1$. Consequently,
\begin{align*}
\check\mC\diamond \mi^{g_0}\diamond \eval^0
=
\check \mC \diamond \eval^0 \diamond \mI^\mg 
\simud
\eval^1\diamond \mI^\mg
=
\mi^{g_1}\diamond \eval^1 
\simud \mi^{g_1}\diamond \mC^{\mg}\diamond \eval^0
\end{align*}
Take a ud-homotopy inverse of $\eval^0$ which exists by Theorem \ref{Whitehead-full-thm}, and we obtain $\check \mC\diamond \mi^{g_0}\simud \mi^{g_1}\diamond \mC^\mg$.
\end{proof}

The Lemma \ref{mC-diamond-diagram-lem} above holds only for the harmonic contractions and may fail in general. Indeed, the pointwiseness of $\con(\mg)$ ensures Lemma \ref{canonical-model-[01]Pointwise-mi-lem} which is essential to prove Lemma \ref{mC-diamond-diagram-lem}.

\subsection{Fukaya's trick}
\label{S_Fuk_trick}

Let $L$ and $\tilde L$ be two Lagrangian submanifolds in $X$ that are sufficiently adjacent, and suppose also there exists a sufficiently small diffeomorphism $F\in \diff_0(X)$ so that $F(L)=\tilde L$.
Since the $\omega$-tameness is an open condition, we may assume both $J$ and $F_*J$ are in $\mathfrak J(X,\omega)$, where $F_*J:=dF\circ J\circ dF^{-1}$.
Note that whenever $J\in\mathfrak J(X,L,\omega)$ (Assumption \ref{assumption-mu ge 0}), we have $F_*J\in\mathfrak J(X,\tilde L,\omega)$.

Note that the $F$ induces a natural isomorphism 
\begin{equation}
\label{G(X,L)_F_Fuk_trick_eq}
F_*: \pi_2(X,L) \cong \pi_2(X,\tilde L)
\end{equation}
which preserves the Maslov indices and only depends on the homotopy class of $F$.
We denote $\tilde \beta=F_*\beta$.

However, the energy is not preserved by the above $F_*$. If we take a Weinstein neighborhood of $L$ and $\tilde L$, then one can find a tautological 1-form $\lambda$ such that it vanishes exactly on $\tilde L$ and $\omega=d\lambda$.
Then, for an isotopy from $F$ to $\id$, one can use Stokes' formula to show that $E(F_*\beta)-E(\beta)=-\partial \beta\cap \lambda|_L$.

\begin{rmk}
	\label{G(M,L)_energy_vary_rmk}
In our situation,
the $L$ and $\tilde L$ are usually two adjacent Lagrangian torus fibers of $\pi$. We may take some action-angle coordinates $\alpha_i\in \mathbb R/2\pi \mathbb Z$, $x_i\in\mathbb R$ such that $L=\{x_i=0\}$ and $\tilde L=\{x_i=c_i\}$. Then, we have $\lambda= \sum_i (x_i-c_i)d\alpha_i$, so $\lambda|_{\tilde L}=0$ and $\lambda|_L= -\sum c_i d\alpha_i$. 
Assume $L=L_q$ and $\tilde L=L_{q'}$ for some base points $q$ and $q'$ in $B_0$; then $q'-q=(c_1,\dots, c_n)$ can be viewed as a vector in $H^1(L_q)\cong T_qB_0\cong \mathbb R^n$ (c.f. \S \ref{sss_integral_affine_str}). Now, the energy change is given by
$E(\tilde \beta)-E(\beta)=\partial \beta \cap (q'-q)$.
\end{rmk}

As explained in \cite{FuCyclic}, the diffeomorphism $ F$ gives a homeomorphism of moduli spaces:
\begin{equation}
\label{moduli_isom-eq}
F_{\mathcal M}: \mathcal M_{k, \beta}(J,L) \xrightarrow{\cong} \mathcal M_{k,  \tilde \beta} (  F_*J, \tilde L)
\end{equation}
In fact, a $J$-holomorphic disk $u$ on the left side corresponds to $F\circ u$ on the right side, and vice versa.
Furthermore, a Kuranishi-theory choice $\Xi$ (see e.g. Convention \ref{convention_Xi_J}) for the left side can induce via $F$ the other choice $ F_* \Xi$ for the right side moduli.
From these data, using Theorem \ref{Vir-m^J-thm} obtains two chain-level $A_\infty$ algebras $\check \m^{J,\Xi, L}$ on $\Omega^*(L)$ and $\check \m^{ F_*J,F_*\Xi, \tilde L}$ on $\Omega^*(\tilde L)$ which are closely related to each other.
To emphasize that the data are $F$-related, we introduce the following notation:
\begin{equation}
\label{check m_F J L'-eq}
\check \m^{ F_*(J, \Xi), \tilde L }:=\check \m^{ F_*J,  F_*\Xi, \tilde L}
\end{equation}
Heuristically, we may think of it as a `pushforward' of $\check\m^{J,L}$ where the data $J$, $L$, and $\Xi$ are `pushed' to $ F_*J$, $\tilde L= F(L)$ and $ F_*\Xi$ respectively. 
For simplicity, we will omit $\Xi$ and $F_*\Xi$ in the notations since now.
Recall that
$F^*: \Omega^*(\tilde L) \to \Omega^*(L)$ and
$F^*: H^*(\tilde L) \to H^*(L)$
are naturally the cochain maps.
To begin with, we have the following \textit{chain-level} Fukaya trick (see \cite{FuCyclic, FuCounting}):

\begin{lem}[Chain-level Fukaya's trick]
	\label{Fukaya's trick-cochain-level-lem}
	\begin{equation}
	\label{Trick_check_m_eq}
	\check\m^{ F_*J, \tilde L}_{k, \tilde \beta} (x_1,\dots,x_k) = ( F^{-1})^* \check\m^{J, L}_{k,  \beta}( F^*x_1,\dots,  F^*x_k)
	\end{equation}
\end{lem}

\begin{proof}
	[Sketch of proof]
	Since the virtual foundation for the proof goes beyond our scope, we only provide an intuitive description as follows. The argument can be made precise using the virtual techniques; see e.g. \cite[Lemma 13.4]{FuCyclic}. First, geometrically, we observe that the isomorphism (\ref{moduli_isom-eq}) of moduli spaces is compatible with the evaluation maps as illustrated in the following diagram
	\begin{equation*}
	\xymatrix{
		\mathcal M_{k,\beta}(J,L) \ar[rr]^{ F_{\mathcal M}} \ar[d]^{\ev_i} & & \mathcal M_{k, \tilde \beta}(  F_*J, \tilde L) \ar[d]^{\tilde\ev_i} \\
		L \ar[rr]^{ F} & & \tilde L
	}
	\end{equation*}
	In fact, the top horizontal map sends $[\uu,\z]$ to $[ F\circ \uu, \z]$, and it follows $\tilde\ev_i([ F\circ \uu,\z])=  F(\uu(z_i))= F\circ \ev_i([\uu,\z])$. Next, we denote $\check \m^{\smallsim}=\check \m^{F_*J, \tilde L}$ and $\check \m=\check \m^{J,L}$. Then, by (\ref{defining_A_infty_II_eq}) or (\ref{Kura_vir_count_eq}), we obtain:
	\begin{align*}
	\langle \check \m_{k,\beta}( F^*x_1,\dots, F^*x_k),  F^*x_0\rangle
	&
	=
	\textstyle
	\int_{\mathcal M_{k,\beta}(J,L)} {\ev_0}^* F^* x_0 \wedge {\ev_1}^* F^* x_1 \wedge \cdots \wedge {\ev_k}^* F^* x_k  \\
	&
	=
	\textstyle
	\int_{\mathcal M_{k,\beta}(J,L)}  F_{\mathcal M}^* \left(
	\tilde\ev_0^*x_0\wedge \tilde\ev_1^*x_1\wedge \cdots \wedge \tilde\ev_k^* x_k 
	\right)\\
	&
	=
	\textstyle
	\int_{\mathcal M_{k,  \tilde \beta}( F_* J,\tilde L)} \tilde\ev_0^*x_0\wedge \tilde\ev_1^*x_1\wedge \cdots \wedge \tilde\ev_k^* x_k \\
	&
	=\langle \check \m^{\smallsim}_{k, \tilde \beta}(x_1,\dots,x_k), x_0\rangle
	=\langle F^*  \check  \m^{\smallsim}_{k, \tilde \beta}(x_1,\dots,x_k),   F^* x_0\rangle
	\end{align*}
	Hence, we obtain $\check \m_{k,\beta}( F^*x_1,\dots, F^*x_k)=F^*  \check  \m^{\smallsim}_{k, \tilde \beta}(x_1,\dots,x_k)$.
\end{proof}

There is an induced Fukaya's trick in the cohomology-level.
The basic idea is that for the induced metric $F_*g=(F^{-1})^*g$, the two distinct harmonic contractions $\con(g)$ and $\con(F_*g)$ (Lemma \ref{strong-contraction-metric g-lem}) are also $F$-related, which can be incorporated into the homological perturbation.

\begin{lem}\label{Fukaya's trick-contraction-lem}
	The harmonic contraction $\con(F_*g)=\big(i( F_*g), \pi(  F_*g), G( F_*g) \big)$ satisfies:
	\begin{equation}
	\label{contraction-FukayaTrick-eq}
	\setlength{\jot}{0pt} 
	\begin{aligned}
	i(F_*g)= & ( F^{-1})^* \circ i(g) \circ  F^* 
	\\
	\pi(F_*g)  = & ( F^{-1})^* \circ \pi(g) \circ  F^*
	\\
	G(F_*g)  = & ( F^{-1})^* \circ G(g) \circ  F^* 
	\end{aligned}
	\end{equation}
\end{lem}

\begin{proof}
	The three relations are easy to prove from the definitions in (\ref{harmonic i(g)}, \ref{harmonic pi(g)}, \ref{harmonic G(g)}).
\end{proof}

In view of Theorem \ref{UD-canonical-model-thm}, we denote by
\begin{equation}
\label{m_gJL-eq}
\left( H^*(L), \m^{g,J,L}, \mi^{g,J,L} \right)
\end{equation}
the canonical model of $\check \m^{J,L}$ with regard to $\con(g)$.
Meanwhile, analogously to (\ref{check m_F J L'-eq}), we denote by 
\begin{equation}
\label{m_F g J L' -eq}
( H^*(\tilde L), \m^{ F_*(g,J),\tilde L}, \mi^{ F_*(g,J), \tilde L} )
\end{equation}
the canonical model of $ \check \m^{ F_*J,\tilde L}$ with respect to $\con(F_*g)$.
The Fukaya's trick relates $\check \m^{F_*J,\tilde L}$ to $\check \m^{J,L}$ (Lemma \ref{Fukaya's trick-cochain-level-lem}). The use of harmonic contractions ensures an analog of Fukaya's trick for (\ref{m_gJL-eq}, \ref{m_F g J L' -eq}).

\begin{lem}[Cohomology-level Fukaya's trick]
	\label{Fukaya's trick-canonical-level-lem}
	\begin{align}
	\label{Trick_m_can_eq}
	\m^{ F_*(g,J), \tilde L}_{k, \tilde  \beta} (x_1,\dots,x_k) = ( F^{-1})^* \m^{g,J,L}_{k, \beta}( F^*x_1,\dots,  F^*x_k) \\
	\label{Trick_mi_can_eq}
	\mi^{ F_*(g,J),\tilde L}_{k,  \tilde \beta} (x_1,\dots,x_k) = ( F^{-1})^* \mi^{g,J, L}_{k,   \beta}( F^*x_1,\dots,  F^*x_k)
	\end{align}
\end{lem}

\begin{proof}
	Write $h=F_*g$.
	Denote $\m:=\m^{g,J,L}$, $\m^ F:=\m^{ F_*(g,J), \tilde L}$ and $\check \m:=\check \m^{J,L}$, $\check \m^F:=\check \m^{F_*J,\tilde L}$; also, denote $\mi:=\mi^{g,J,L}$,
	$\mi^ F:=\mi^{ F_*(g,J),\tilde L}$. 
	 We perform an induction on the pairs $(k,  \beta)$ as before. We prove the lemma for $\mi^ F$ first. The initial case for $(k,  \beta)=(1,0)$ is true by (\ref{contraction-FukayaTrick-eq}) and (\ref{mi_g_10_eq}).
	 Suppose the formula (\ref{Trick_mi_can_eq}) for $\mi^F$ is correct whenever $(k',  \beta')<(k,  \beta)$.
	Then, the inductive formula (\ref{induction-tree-formula-eq}) implies that
	\begin{align*}
	\mi^ F_{k,    \beta}
	=
	\textstyle
	\sum_{(\ell,  \beta_0)\neq(1,0)} G (h) \circ \check \m^F_{\ell,   \beta_0} \circ (\mi^ F_{k_1,   \beta_1}\otimes \cdots \otimes \mi^ F_{k_\ell,   \beta_\ell})
	\end{align*}
	Furthermore, applying Lemma \ref{Fukaya's trick-cochain-level-lem}, Lemma \ref{Fukaya's trick-contraction-lem}, and the induction hypothesis, we conclude that
	\[
	\mi^ F_{k,  \beta}
	=
	\textstyle
	( F^{-1})^* 
	\sum_{(\ell,\beta_0)\neq(1,0)} G(g) \circ
	\check \m_{\ell,  \beta_0} 
	\circ
	(
	\mi_{k_1,  \beta_1}
	\otimes
	\cdots
	\otimes
	\mi_{k_\ell,  \beta_\ell}
	)\circ ( F^*)^{\otimes k}
	=
	(F^{-1})^* \mi_{k,\beta}\circ (F^*)^{\otimes k}
	\] 
	Replacing $G(g)$, $G(h)$ by $\pi(g)$, $\pi(h)$, the statement for $\m^F$ can be proved by a similar induction.
\end{proof}

%
%
%

\begin{cor}\label{Fukaya's trick Independence of F-cor}
	If $F_0$ is isotopic to $F_1$ with $F_0(L)=F_1(L)=\tilde L$, then $\m^{ F_{0*}(g,J),\tilde L}=\m^{ F_{1*}(g,J),\tilde L}$.
\end{cor}

\begin{proof}
	Use Lemma \ref{Fukaya's trick-canonical-level-lem} and the homotopy invariance of $F^*: H^*(\tilde L)\to H^*(L)$.
\end{proof}

Corollary \ref{Fukaya's trick Independence of F-cor} basically suggests that the application of Fukaya's trick in the cohomology-level $A_\infty$ algebras only depends on the homotopy class of the diffeomorphism $F$. In contrast, since $ F_0^*\neq  F_1^*: \Omega^*(\tilde L) \to \Omega^*(L)$, it is \textit{not} necessary that  $\mi^{ F_*(g,J),\tilde L}$ agrees with $\mi^{ F_{1*}(g,J),\tilde L}$.


In practice, we need an upgrade of Fukaya's trick that can be applied to certain pseudo-isotopies.
Fix a smooth path $\mathbf J=(J_s)_{s\in\oi}$ in $\mathfrak J(X,\omega)$; fix a smooth family $\mF=(F_s)_{s\in\oi}$
near the identity in $\diff_0(X)$ such that $F_s(L)=\tilde L$ for each $s$.
Define a diffeomorphism on $X\times \oi$, still denoted by
$
\mF : X \times \oi \to X \times \oi$, which sends $(x,s)$ to $(F_s(x),s)$.
Assume $L$ and $\tilde L$ are sufficiently close to each other.
Every $F_s$ induces an isomorphism $F_{s*}: \pi_2(X,L)\to \pi_2(X, \tilde L)$ by (\ref{G(X,L)_F_Fuk_trick_eq}), but these isomorphisms are the same by the homotopy invariance. Denote any of them by the same notation:
\begin{equation}
	\label{G(M,L)_F_Fuk_trick_[01]family_eq}
	 \mF_*:\pi_2(X,L) \cong \pi_2(X,\tilde L)
\end{equation}
As previously, we will write $\tilde \beta=\mF_*\beta$.
Moreover, we have the following two natural maps
\begin{equation}
\label{mF^*_[0,1]-eq}
\mF^*: \Omega^*(\tilde L)_\oi \xrightarrow{} \Omega^*(L)_\oi, \quad \text{and} \quad
\mF^*: H^*(\tilde L)_\oi \to H^*(L)_\oi
\end{equation}
Explicitly, a differential form $\alpha_1(s)+ds\wedge \alpha_2(s)$ is sent to $F_s^*\alpha_1(s)+ds\wedge  (F_s^*\alpha_2(s)+ O(\partial_s F_s)  )$ where $O(\partial_s F_s)$ represents some terms in $\partial_s F_s$. 
Recall that $\eval^s$ is just the pullback of the $s$-inclusion $\iota_s: L \xhookrightarrow{} L\times \oi$ (see around (\ref{convention-Omega P L-eq})); then, since $\mF\circ \iota_s=\iota_s\circ F_s$, we conclude that
\begin{equation}
\label{eval_s_mF_F_s_eq}
\eval^s\circ \mF^*=F_s^*\circ \eval^s: \qquad \Omega^*(\tilde L)_\oi\to \Omega^*( L)
\end{equation}

Every $F_s$ induces an identification $\mathcal M_{k,\beta}(J_s,L)\cong \mathcal M_{k,\tilde\beta} (F_{s*}J_s, \tilde L)$ as in (\ref{moduli_isom-eq}), and the union of them gives an identification of the parameterized moduli spaces (\ref{moduli_system}):
\begin{equation}
\label{moduli_isom-pathJ-eq}
\textstyle
\mF_{\mathcal M}:= \bigsqcup_{s\in\oi} (F_s)_{\mathcal M}: \ \mathcal M_{k,  \beta}(\mathbf J, L) \xrightarrow{\cong} \mathcal M_{k, \tilde \beta} ( \mF_*\mathbf J, \tilde L)
\end{equation}
Taking all the pairs $(k,\beta)$ simultaneously yields an identification of the moduli space systems (\ref{moduli_space_system_eq}): $\mathbb M(\mJ)\cong \mathbb M(\mF_*\mJ)$.
Note that for every $s$, we have an $A_\infty$ algebra $\check \m^{J_s, L}$ in $\UD$ obtained by the moduli system $\mathbb M(J_s)$ (Theorem \ref{Vir-m^J-thm}).
Also, by the moduli space system $\mathbb M(\mJ)$, we obtain a pseudo-isotopy 
$
(\Omega^*(L)_\oi, \check \M^{\mJ, L})
$
in $\UD$
such that it restricts to $\check \m^{J_s, L}$ at $s=0,1$ (Theorem \ref{Vir-M-path-thm}).
Just like (\ref{moduli_isom-eq}, \ref{check m_F J L'-eq}) before, the above identification (\ref{moduli_isom-pathJ-eq}) yields a new pseudo-isotopy 
$
(\Omega^*(\tilde L)_\oi, \check \M^{\mF_*\mJ, \tilde L} )
$
that comes from the moduli space system $\mathbb M(\mF_*\mJ)$.
Similar to Lemma \ref{Fukaya's trick-cochain-level-lem}, we can prove the following:

\begin{lem}[Chain-level Fukaya's trick for pseudo-isotopies]
	\label{Fukaya_II-cochain-level-lem}
	For $x_1,\dots,x_k\in\Omega^*(L)_\oi$, we have
	\begin{equation}
	\label{Trick_check_M_family_eq}
	\check \M_{k, \tilde \beta}^{ \mF_*\mJ, \tilde L} (x_1,\dots,x_k)
	=
	( \mF^{-1})^* \check \M_{k, \beta}^{\mJ, L} ( \mF^*x_1,\dots, \mF^*x_k)
	\end{equation}
\end{lem}

%
%
%

The restriction of $\check \M^{\mF_* \mJ, \tilde L}$ at $s$ gives rise to an $A_\infty$ algebra that comes from the moduli space system $\mathbb M(F_{s*}J_s)$, and we denote it by $\check \m^{F_{s*}J_s, \tilde L}$ by our previous convention.
Then, the $\check \M^{\mathbf J, L}$ (resp.$\check \M^{\mF_*\mathbf J,\tilde L}$) restricts at $s=0,1$ to the $\check \m^{J_s, L}$ (resp. $\check\m^{F_{s*}J_s, \tilde L}$), i.e.
\begin{equation}
\label{Eval_Trick_check_M^F-eq}
\eval^s \diamond  \check \M^{\mathbf J, L} = \check \m^{J_s, L}\diamond \eval^s,
\qquad
\eval^s  \diamond \check \M^{\mF_*\mathbf J, \tilde L}= \check \m^{F_{s*}J_s, \tilde L}\diamond \eval^s
\end{equation}
Note that the pair $\check \m^{J_s,L}$ and $\check \m^{F_{s*}J_s, \tilde L}$ also satisfies Fukaya's trick, i.e. Lemma \ref{Fukaya's trick-cochain-level-lem}.

Let $\mg=(g_s)_{s\in\oi}$ be a family of metrics; then $\mF_*\mg =( F_{s*}g_s)_{s\in\oi}$.
Due to Lemma \ref{strong-contraction [0,1]-lem}, we have the two strong contractions
$\con(\mg)$ and
$\con(\mF_*\mg)$
for the two pairs of cochain complexes: $(H^*(L)_\oi, \Omega^*(L)_\oi )$ and $(H^*(\tilde L)_\oi, \Omega^*(\tilde L)_\oi )$ respectively.
Then, we denote by
\begin{equation}
\label{can_[01]_Fuk_trick_eq}
\big(H^*(L)_\oi,  \M^{\mg,\mathbf J, L},  \mI^{\mg,\mathbf J,L}
\big)\qquad \text{and}\quad \big(H^*(\tilde L)_\oi,  \M^{\mF_*(\mg,\mathbf J), \tilde L},  \mI^{\mF_*(\mg,\mathbf J), \tilde L}
\big)
\end{equation}
the two canonical models of $\check \M^{\mJ, L}$ and $\check \M^{\mF_*\mathbf J, \tilde L}$ with regard to the two contractions $\con(\mg)$ and $\con(\mF_*\mg)$ respectively.
Due to Theorem \ref{UD-canonical-model-[0,1]-thm}, they are all in the category $\UD$.
Furthermore, one can use (\ref{Eval_Trick_check_M^F-eq}) and Lemma \ref{canonical-model-[01]Pointwise-mi-lem} to show that
the restriction of $\M^{\mg,\mJ,L}$ (resp. $\M^{\mF_*(\mg,\mJ),\tilde L}$) at $s=0,1$ agrees with $\m^{g_s,J_s,L}$ (resp. $\m^{F_{s*}(g_s,J_s),\tilde L}$), i.e.
\begin{equation}
	\label{Eval_Trick_cano-eq}
	\eval^s \diamond \M^{\mg,\mJ,L}=\m^{g_s,J_s,L}\diamond \eval^s,
	\qquad
	\eval^s \diamond \M^{\mF_*(\mg,\mJ),\tilde L}
	=
	\m^{F_{s*}(g_s,J_s),\tilde L}\diamond \eval^s
\end{equation}

\begin{lem}
	\label{Fukaya_II_constant_contraction_[0,1]_lem}
	Suppose $\partial_sF_s=0$. Then we have
	$
	\con(\mF_*\mg)=\mF^{-1*} \circ  \con (\mg)\circ \mF^*
	$. Consequently,
	\begin{equation}
	\label{Trick_M_can_eq}
\begin{aligned}
	\M_{k,\tilde\beta}^{\mF_*(\mg,\mJ), \tilde L}(x_1,\dots,x_k) &= \mF^{-1*} \M^{\mg,\mJ,L}_{k,\beta} (\mF^*x_1,\dots,\mF^*x_k) \\
	\mI_{k,\tilde \beta}^{\mF_*(\mg,\mJ),\tilde L}(x_1,\dots,x_k) &= \mF^{-1*} \mI^{\mg,\mJ,L}_{k,\beta} (\mF^*x_1,\dots,\mF^*x_k) 
\end{aligned}
	\end{equation}
\end{lem}

\begin{proof}
	By (\ref{i_pi_G_[0,1]-eq}), we write $i(\mg)=1\otimes i_s+ds\otimes h_s$, $\pi(\mg)=1\otimes \pi_s+ ds\otimes k_s$, and $G(\mg)=1\otimes G_s+ds\otimes \sigma_s$; meanwhile we write $i(\mF_*\mg)=1\otimes \tilde i_s+ds\otimes \tilde h_s$, $\pi(\mF_*\mg)=1\otimes \tilde \pi_s +ds\otimes \tilde k_s$, and $G(\mF_*\mg)=1\otimes \tilde G_s+ds\otimes \tilde \sigma_s$.
	
	As $\partial_sF_s=0$, there is no harm to set $F=F_s$ and $F=\mF$. Then, the $\mF^*$ simply sends $\alpha_1(s)+ds\wedge \alpha_2(s)$ to $F^*\alpha_1(s)+ds\wedge F^* \alpha_2(s)$.
	Let $d_s=d_{g_s}$ and $\tilde d_s=d_{F_*g_s}$ be the isomorphisms as in (\ref{d_g-eq}); then, $\tilde d_s=F^{-1*}\circ d_s\circ F^*$.
	By (\ref{contraction-FukayaTrick-eq}), we have
	$F^*\circ \tilde i_s=i_s\circ F^*$, $F^* \circ \tilde \pi_s=\pi_s\circ F^*$, and $F^* \circ \tilde G_s=G_s\circ F^*$.
	Moreover, the same properties for the pairs $(h_s, \tilde h_s)$, $(k_s,\tilde k_s)$, and $(\sigma_s, \tilde \sigma_s)$ can be proved by direct computations with their definition formulas (\ref{h_t-eq}), (\ref{k_t-eq}), (\ref{sigma_t-(1)-eq}). In summary, $\con(\mF_*\mg)=\mF^{-1*}\circ \con(\mg)\circ \mF^*$. 
	
	Ultimately, by the same argument in the proof of Lemma \ref{Fukaya's trick-canonical-level-lem}, one can also show (\ref{Trick_M_can_eq}).
\end{proof}

In the above, the second half is analogous to Lemma \ref{Fukaya's trick-canonical-level-lem}. 
The first half resembles Lemma \ref{Fukaya's trick-contraction-lem}, but the condition $\partial_sF_s=0$ is necessary, because
$\con(\mF_*\mg) \neq \mF^{-1*} \circ \con(\mg) \circ \mF^*$ in general.
The main issue is about the terms involving non-zero $\partial_sF_s$;
but fortunately, these terms only live in $ds\otimes -$ parts in $\HL_\oi$ or $\OL_\oi$ and are always killed by $\eval^s$.
From this observation, one can imagine Theorem \ref{UD-mC_thm} will be useful.
Actually, even in the case of $\partial_sF_s\neq 0$, we can still somehow compare the two pseudo-isotopies $\M^{\mg,\mJ,L}$ and $\M^{\mF_*(\mg,\mJ), \tilde L}$ `up to ud-homotopy'. See
Lemma \ref{Fukaya_II_key_lemma} later.

There is an elegant way to reinterpret the Fukaya's trick equations: one can regard the various $F^*$ as trivial $A_\infty$ homomorphisms (in $\UD$) which concentrate on $\CC_{1,0}\subset \CC$.
Denote by $\Omega^F$ the operator system in $\CC( \Omega^*(\tilde L), \Omega^*(\tilde L))$ defined by setting $\Omega^F_{1,0}=F^*$ and all other $\Omega^F_{k,\beta}=0$.
Similarly, we define $H^F\in \CC(H^*(\tilde L), H^*(\tilde L))$, $\Omega^\mF\in \CC(\Omega^*(\tilde L)_\oi, \Omega^*(\tilde L)_\oi)$ and $H^\mF\in \CC(H^*(\tilde L)_\oi, H^*(\tilde L)_\oi)$ by setting $H^F_{1,0}$, $\Omega^\mF_{1,0}$, and $H^\mF_{1,0}$ to be $F^*$ or $\mF^*$.
Remark that there are two different label groups in the story: $\G=\pi_2(X,L)$ and $\tilde \G=\pi_2(X,\tilde L)$. But this causes no serious troubles, since we actually have a natural identification $\G\cong \tilde \G$ from (\ref{G(X,L)_F_Fuk_trick_eq}) and (\ref{G(M,L)_F_Fuk_trick_[01]family_eq}).
For simplicity, we will keep this point implicit.

On the one hand, due to the Fukaya's trick equations in (\ref{Trick_check_m_eq}, \ref{Trick_m_can_eq}, \ref{Trick_check_M_family_eq}), all of them except $H^\mF$ can be viewed as $A_\infty$ homomorphisms in $\Mor\UD$:
\begin{equation*}
\Omega^{F}:  \check \m^{F_{*}J, \tilde L} \to \check \m^{J,L}, \qquad
H^{F}: \m^{F_{*}(g,J),\tilde L}  \to \m^{g,J,L}, \qquad
\Omega^{\mF}: \check \M^{\mF_*\mathbf J,\tilde L}  \to \check \M^{\mathbf J,  L}
\end{equation*}
But in the special case $\partial_s F_s=0$, it follows from Lemma \ref{Fukaya_II_constant_contraction_[0,1]_lem} that the exceptional $H^\mF$ also gives a morphism $H^\mF: \M^{\mF_*(\mg,\mJ),\tilde L}\to \M^{\mg,\mJ,L}$ in $\UD$.

On the other hand, the equation (\ref{Trick_mi_can_eq}) describes a relation between the two $A_\infty$ homomorphisms $\mi^{g,J,L}$ and $\mi^{F_*(g,J),\tilde L}$; from the above perspective, it exactly means that the following two compositions of morphisms in $\UD$ agree with each other:
\begin{align}
\label{TrickA_mi_can_eq}
\Omega^{F}\diamond \mi^{F_{*}(g,J), \tilde L} = \mi^{g,J, L} \diamond H^{F}
\end{align}
Moreover, since every $\eval^s$ can be trivially viewed as a morphism in $\UD$, one can also translate (\ref{eval_s_mF_F_s_eq}) into the following relation of compositions of morphisms in $\UD$:
\begin{equation}
\label{TrickA_eval}
\Omega^{F_s}\diamond \eval^s = \eval^s \diamond \Omega^\mF
\end{equation}

Eventually, according to Theorem \ref{from-pseudo-isotopies-to-A-infty-homo-thm} and Theorem \ref{UD-mC_thm}, from the four above-mentioned pseudo-isotopies $\check \M^{\mJ, L}$, $\check \M^{\mF_*\mJ, \tilde L}$ and $\M^{\mg, \mJ, L}$, $\M^{\mF_*(\mg,\mJ),\tilde L}$, we have naturally constructed four morphisms in $\UD$ which we denote by $\check \mC^\mJ$, $\check \mC^{\mF_*\mJ}$, $\mC^{\mg,\mJ}$ and $\mC^{\mF_*(\mg,\mJ)}$ respectively.

The sources and targets of the four morphisms can be determined by (\ref{Eval_Trick_check_M^F-eq}) and (\ref{Eval_Trick_cano-eq}). For instance, the pseudo-isotopy $\check \M^\mJ$ restricts to $\check \m^{J_0, L}$ and $\check \m^{J_1,L}$ at $s=0$ and $s=1$ respectively, so $\check \mC^\mJ: \check \m^{J_0,L}\to \check \m^{J_1,L}$.

\begin{lem}
	\label{Fukaya_II_key_lemma}
	In the above situation,
	$H^{F_1}\diamond \mC^{\mF_*(\mg,\mJ)}$ is ud-homotopic to $\mC^{\mg,\mJ}\diamond H^{F_0}$.
\end{lem}
\vspace{-1em}
\[
\Scale[0.9]
{
	\xymatrix
	{
		\check \m^{F_{0*}J_0} \ar@/_6pc/[dddddd]_{\Omega^{F_0}} \ar[rrrr]^{\check \mC^{\mF_*\mJ} } &  &  &  & \check \m^{F_{1*}J_1} \ar@/^6pc/[dddddd]^{\Omega^{F_1}}\\
		&  & \check \M^{\mF_*\mJ} \ar[ull]^{\eval^{0+}} \ar@/^2.5pc/[dddd]^{\Omega^\mF} \ar[urr]_{\eval^{1+}} &  &  \\
		\m^{F_{0*}(g_0,J_0)}\ar[uu]^{\mi^{F_{0*}(g_0,J_0)}} \ar@/_1pc/[dd]_{H^{F_0}} \ar[rrrr]^{\mC^{\mF_*(\mg,\mJ)}}  &  &  &  & \m^{F_{1*}(g_1,J_1)} \ar@/^1pc/[dd]^{H^{F_1}} \ar[uu]_{\mi^{F_{1*}(g_1,J_1)}} \\
		&  &  &  &  \\
		\m^{g_0,J_0} \ar[dd]_{\mi^{g_0,J_0}} \ar[rrrr]^{\mC^{\mg,\mJ}} &  &  &  & \m^{g_1,J_1} \ar[dd]^{\mi^{g_1,J_1}} \\
		&  & \check \M^{\mJ} \ar[dll]_{ \eval^0} \ar[drr]^{ \eval^1} &  &  \\
		\check \m^{J_0} \ar[rrrr]^{\check \mC^{\mJ}} &  &  &  & \check \m^{J_1}
	}
}
\]
\vspace{-1em}
\begin{proof}
	By Lemma \ref{mC-diamond-diagram-lem}, the top and bottom rectangle diagrams commute up to ud-homotopy, namely,
	\begin{align}
	\label{key_1_eq}
	\mi^{F_{1*}(g_1,J_1)} \diamond \mC^{\mF_*(\mg,\mJ)} 
	&
	\simud \check \mC^{\mF_*\mJ} \diamond \mi^{F_{0*}(g_0,J_0)} \\
	\label{key_2_eq}
	\mi^{g_1,J_1} \diamond \mC^{\mg,\mJ} 
	&
	\simud \check \mC^\mJ\diamond \mi^{g_0,J_0}
	\end{align}
	and Theorem \ref{UD-mC_thm} concludes that (just to distinguish, we use different notations: $\eval^{i+}$ and $\eval^i$)
	\begin{align}
	\label{key_3_eq}
	\check \mC^{\mF_*\mJ} \diamond \eval^{0+} &\simud \eval^{1+} \\
	\label{key_4_eq}
	\check \mC^{\mJ} \diamond \eval^0 &\simud \eval^1
	\end{align}
	Now, we are going to chase the diagram
	\begin{align*}
	\mi^{g_1,J_1} \diamond H^{F_1} \diamond \mC^{\mF_*(\mg,\mJ)} 
	&
	= \Omega^{F_1} \diamond \mi^{{F_{1*}(g_1,J_1)}} \diamond \mC^{\mF_*(\mg,\mJ)} \tag{use (\ref{TrickA_mi_can_eq})} \\
	&
	\simud
	\Omega^{F_1}\diamond \check \mC^{\mF_*\mJ} \diamond \mi^{F_{0*}(g_0,J_0)} \tag{use (\ref{key_1_eq})}
	\\
	&
	\simud
	\Omega^{F_1}\diamond \eval^{1+}  \diamond (\eval^{0+})^{-1} \diamond \mi^{F_{0*}(g_0,J_0)} \tag{use (\ref{key_3_eq})}
	\\
	&
	\simud
	\eval^1\diamond \Omega^\mF \diamond (\eval^{0+})^{-1} \diamond \mi^{F_{0*}(g_0,J_0)}
	\tag{use (\ref{TrickA_eval})} \\
	&
	\simud \check \mC^{\mJ} \diamond \eval^0 \diamond\Omega^\mF \diamond (\eval^{0+})^{-1} \diamond \mi^{F_{0*}(g_0,J_0)}
	\tag{use (\ref{key_4_eq})}
	\\
	&
	=
	\check \mC^{\mJ} \diamond \Omega^{F_0} \diamond \eval^{0+} \diamond (\eval^{0+})^{-1} \diamond \mi^{F_{0*}(g_0,J_0)} 
	\tag{use (\ref{TrickA_eval}) again}
	\\
	&
	\simud \check \mC^{\mJ}\diamond \Omega^{F_0} \diamond \mi^{F_{0*}(g_0,J_0)} \\
	&
	\simud \check \mC^\mJ \diamond \mi^{g_0,J_0}\diamond H^{F_0}
	\tag{use (\ref{TrickA_mi_can_eq}) again} \\
	&
	\simud \mi^{g_1,J_1}\diamond \mC^{\mg,\mJ} \diamond H^{F_0} \tag{use (\ref{key_2_eq})}
	\end{align*}
	Finally, finding a homotopy inverse of $\mi^{g_1,J_1}$ by Theorem \ref{Whitehead-full-thm} completes the proof.
\end{proof}

The main reason why we go around in such a cumbersome way is that, as mentioned previously, the $H^\mF$ is generally \textit{not} an $A_\infty$ homomorphism (unless $\partial_sF_s=0$ like Lemma \ref{Fukaya_II_constant_contraction_[0,1]_lem} (\ref{Trick_M_can_eq})). Hence, it cannot be included in the above diagram chasing.
But, since what we need is not an identity relation but just a ud-homotopy relation, we have some flexibility and can overcome this issue as above.

As a corollary, the following special case for the constant families will be used later.

\begin{cor}\label{Fukaya_II_key_cor}
	Suppose $\mg=\hat g$ and $\mJ=\hat J$ are the constant paths at $g$ and $J$ but $\mF=(F_s)$ is arbitrary. If $\check \M^{\hat J}$ is the trivial pseudo-isotopy about $\check \m^J$, then we have $\mC^{\mF_*(\hat g,\hat J)}\simud\id$.
\end{cor}

\begin{proof}
	Denote by $\check \M^{\hat g,\hat J}$ the canonical model of $\check \M^{\hat J}$ with respect to $\con(\hat g)$, and then $\M^{\hat g,\hat J}$ is a trivial pseudo-isotopy about $\check \m^{g,J}$.
	Denote by $\mC^{\hat g,\hat J}$ the $A_\infty$ homomorphism constructed from $\M^{\hat g,\hat J}$ (using Theorem \ref{from-pseudo-isotopies-to-A-infty-homo-thm}); it follows from Corollary \ref{from-pseudo-isotopies-to-A-infty-homo-cor_trivial_pseudo_isotopy} that $\mC^{\hat g,\hat J}=\id$.
	Moreover, because $H^{F_0}=H^{F_1}$, Lemma \ref{Fukaya_II_key_lemma} infers that
	$H^{F_1}\diamond \mC^{\mF_*(\hat g,\hat J)} \simud \mC^{\hat g,\hat J}\diamond H^{F_0}=\id\diamond H^{F_0}=H^{F_1}$ and so
	$\mC^{\mF_*(\hat g,\hat J)}\simud  \id$.
\end{proof}

We note that by Corollary \ref{Fukaya's trick Independence of F-cor} the source and target of $\mC^{\mF_*(\hat g,\hat J)}$ are actually the same $\m^{F_{0*}(g,J)}=\m^{F_{1*}(g,J)}$. Hence, Corollary \ref{Fukaya_II_key_cor} just tells that choosing different $F$ in the application of Fukaya's trick does not cause any ambiguities up to ud-homotopy.

On the other hand, if we keep $F_s$ constant but allow $g_s$ and $J_s$ to vary, then it is easy to see that the ud-homotopic relation in Lemma \ref{Fukaya_II_key_lemma} can be actually strengthened to an identity relation:
 
\begin{prop}
	\label{Fukaya_II_F_s=F_mC_prop}
	Suppose $\partial_s F_s=0$. Then $H^F\diamond \mC^{\mF_*(\mg,\mJ)}=\mC^{\mg,\mJ}\diamond H^F$. Concretely, we have
	\[
	\mC^{\mF_*(\mg,\mJ)}_{k,\tilde \beta}= F^{-1*} \mC^{\mg,\mJ}_{k,\beta}( F^*x_1,\dots,F^*x_k)
	\]
	for any $x_1,\dots,x_k\in H^*(\tilde L)$.
\end{prop}

\begin{proof}
	We write $\M^{\mF_*(\mg,\mJ)}=1\otimes \tilde \m^s +ds\otimes \tilde \mc^s$ and $\M^{\mg,\mJ}=1\otimes \m^s +ds\otimes \mc^s$. So, $\mC^{\mg,\mJ}: \m^{0}\to \m^{1}$ and $\mC^{\mF_*(\mg,\mJ)}: \tilde\m^{0}\to \tilde\m^{1}$.
	In this case,
	$\tilde \m^s = F^{-1*}\diamond \m^s \diamond F^{*}$ and $
	\tilde \mc^s = F^{-1*}\diamond \mc^s \diamond F^{*}$ by Lemma \ref{Fukaya_II_constant_contraction_[0,1]_lem}.
	Just like the proof of Lemma \ref{Fukaya's trick-canonical-level-lem}, the desired equation can be proved by the inductive formula (\ref{inductive_formu-Fubini-yield-eq}).
\end{proof}

\section{Mirror construction}
\label{S_Mirror_construction}

\subsection{Set-up}
We are going to show the main statement Theorem \ref{Main_theorem_thm}.
Recall that $(X,\omega)$ is a symplectic manifold of dimension $2n$ and we denote by $\mathfrak J(X,\omega)$ the space of $\omega$-tame almost complex structures. Suppose $
\pi: X_0  \to B_0$ is a Lagrangian torus fibration in an open domain $X_0 \subset X$ over a base manifold $B_0$.
By Arnold-Liouville theorem \cite{ArnoldBook}, every Lagrangian fiber $L_q:=\pi^{-1}(q)$ over some point $q\in B_0$ is diffeomorphic to the $n$-torus $T^n$, and there is a canonical integral affine manifold structure on $B_0$.

\subsubsection{Integral affine structure}
\label{sss_integral_affine_str}
In a Weinstein neighborhood of $L_q$, we consider the so-called \textit{action-angle coordinates} $(\alpha_1,\dots,\alpha_n, x_1,\dots,x_n)$ where $\alpha_i\in \mathbb R/ 2\pi \mathbb Z\cong S^1$, $x_i\in \mathbb R$, and $\omega=\sum_{i=1}^n dx_i\wedge d\alpha_i$ (see e.g. \cite[\S 3.1]{KSAffine}).
The coordinates $x=(x_1,\dots,x_n)$ give rise to a local chart near $q\in B_0$ over which the fibration $\pi$ has the local expression $(\alpha,x)\mapsto x$.
Such a chart is called an \textit{integral affine chart}. The atlas of all the integral affine charts gives rise to an \textit{integral affine structure} on $B_0$ (see Appendix \ref{SA_integral_affine}).
Besides, taking an integral affine transformation $x'=Ax+b$ and $\alpha'=(A^T)^{-1}\alpha+c$ for some $A\in GL(n,\mathbb Z)$, $b\in\mathbb R^n$ and $c\in (\mathbb R/2\pi \mathbb Z)^n$, we obtain another action-angle coordinate system.

Inserting the vector field $\partial_{x_k}$ to $\omega$ supplies a {closed} one-form
$\omega(\tfrac{\partial}{\partial x_k},\cdot) = d\alpha_k$ (it is not exact).
Also, the assignment $\frac{\partial}{\partial x_k}\mapsto d\alpha_k$ is invariant for another action-angle coordinate system $(\alpha',x')$, hence, we obtain a canonical operator
$
\varrho^0:  T_qB_0\to Z^1(L_q)$ defined by $\xi \mapsto \omega(\xi,\cdot)$.
Composing it with the natural quotient $Z^1(L_q)\to H^1(L_q)$ yields a linear map $\varrho: T_q B_0 \to H^1(L_q)$ whose source and target vector spaces have the same dimension. Further, it is surjective, since the cohomology classes $[d\alpha_k]$ exactly form a basis of $H^1(L_q)$, therefore, the $\varrho$ is a vector space isomorphism.

In summary, there is a canonical isomorphism $T_qB_0 \cong H^1(L_q)$ for every $q \in B_0$.
	
	Moreover, gluing these isomorphisms for various $q$ yields an isomorphism
$TB_0\cong \bigcup_{q\in B_0} H^1(L_q)$ of vector bundles.
Accordingly, we also have an isomorphism of the $\mathbb Z$ lattices (local systems):
$T^{\mathbb Z} B_0 \cong  \bigcup_{q \in B_0} H^1(L_q, \mathbb Z)$.
The right side possesses the Gauss-Manin connection; it naturally gives rise to a flat connection on $TB_0$ so that $T^{\mathbb Z}B_0$ is parallel.
To sum up, the fibration $\pi$ determines an integral affine structure $\nabla$ on the smooth locus $B_0$.
The integral affine connection $\nabla$ induces an exponential map $
	\exp_q: T_qB_0\cong H^1(L_q) \to B_0
	$ at each $q\in B_0$, and we call it an \textit{affine exponential map}.

\subsubsection{Diffeomorphisms and isotopies among the fibers}
\label{sss_diff_construct}
We want to introduce a concrete construction of a small diffeomorphism that sends one fiber to an adjacent fiber.
For any $q\in B_0$,
there is a sufficiently small neighborhood $W_q$ which admits a smooth map
\begin{equation}
\label{chi_u-eq}
\chi_q:W_q  \to  \diff_0(X)
\end{equation}
such that $\chi_q(q)=\id$; the diffeomorphism $\chi_q(q')$ sends $L_{q}$ to $L_{q'}$ and is supported in $\pi^{-1}(W_q)$ (see e.g. \cite[Theorem B]{PalaisLocalTriviality}).
For $q_1',q_2'\in W_q$, we define
\begin{equation}
\label{mF_u_v1_v2-eq}
F_q^{q'_1,q'_2} := \chi_q(q'_2)\circ \chi_q(q'_1)^{-1}
\end{equation}
It is a diffeomorphism that sends $L_{q'_1}$ to $L_{q'_2}$.
One can choose a neighborhood $\mathcal U$ of the identity in $\diff_0(X)$ such that $\mathcal U\cdot \mathcal U\subset \mathcal U$ and $\mathcal U^{-1}\subset \mathcal U$.
Now, shrinking every $W_q$ if necessary, we may further require $\chi_q(W_q)\subset \mathcal U$. Particularly, for any $q_1',q_2'\in W_q$ we have $F_q^{q'_1,q'_2} \in \mathcal U$.

\subsubsection{Uniform reverse isoperimetric inequalities}
The reverse isoperimetric inequality for pseudo-holomorphic curves is discovered by Groman and Solomon \cite{ReverseI}. Later, DuVal \cite{ReverseII} found a simpler proof of it. It states that the length of the boundary of a pseudo-holomorphic disk $u$ bounding a Lagrangian $L$ is controlled by its energy/area $E(u)$; specifically, there is a constant $c$ so that $E(u)\ge c \ell(u)$. It is used extensively by Abouzaid \cite{AboFamilyWithout} in the study of the family Floer program.

However, we need a technical improvement. Not just a single fixed Lagrangian being considered, we need to allow a \textit{uniform constant} for a compact family among Lagrangian fibers.
To avoid a digression, the proof of Lemma \ref{reverse_ineq-mainbody_lem} below is postponed in the appendix.

\begin{lem}
	[Corollary \ref{reverse-ineq-cor-appendix}]\label{reverse_ineq-mainbody_lem}
	Fix $J\in\mathfrak J(X,\omega)$ and fix a Lagrangian submanifold $L\subset X$. There is a $C^1$-neighborhood $\mathcal V_0$ of $J$, a small Weinstein neighborhood $\nu_X L$ of $L$, and a constant $c_0>0$ such that:
	If $\tilde L\subset \nu_X L$ is an adjacent Lagrangian submanifold given by the graph of a small closed one-form on $L$, then for any $\tilde J\in\mathcal V_0$ and any $\tilde J$-holomorphic disk $u:(\mathbb D,\partial\mathbb D)\to (X, \tilde L)$, we have
	\[
	E(u) \ge c_0 \cdot \ell (\partial u)
	\]
\end{lem}

In our situation, we have the following immediate corollary.

\begin{cor}\label{Reverse_ineq_cor_for_compact_domain}
	Given $J\in \mathfrak J(X,\omega)$ and a compact domain $K\subset B_0$,
	there is a $C^1$-neighborhood 
	$
	\mathcal V=\mathcal V_{J,K}$ in $\mathfrak J(X,\omega)$
	of $J$ together with a constant $c=c_{J,K}>0$ satisfying the following property: For any $\tilde J\in\mathcal V$ and $q\in K$, a $\tilde J$-holomorphic disk $u:(\mathbb D, \partial \mathbb D) \to (X,L_q)$ satisfies that
	$
	E(u) \ge c \cdot \ell (\partial u )
	$.
\end{cor}

\begin{proof}
	For any $q\in K$, one can associate to $L_q$ a small neighborhood $\mathcal V(q)$ of $J$ together with a neighborhood $U(q)$ of $q\in K$ in the base $B_0$ (corresponding to the Weinstein neighborhood $\pi^{-1}(U(q))$ of $L_q$) as in Lemma \ref{reverse_ineq-mainbody_lem}. Since $K$ is compact, one can find finitely many points $q_a$ ($a\in A$) in $K$ such that $\bigcup_{a\in A} U(q_a)= K$. Now, we set $\mathcal V=\bigcap_{a\in A} \mathcal V(q_a)$. The proof then follows.
\end{proof}

\subsection{Local charts}

Fix a compact domain $K\subset B_0$ together with a slightly larger compact domain $K'$ such that
\begin{equation}
\label{K_domain_eq}
K \Subset K' \subset B_0
\end{equation}
By Assumption \ref{assumption-mu ge 0}, we fix an almost complex structure $J\in \mathfrak J_K$.
We also fix a metric $g$ on $X$. 

\subsubsection{Decomposition on the base}
\label{sss_V_U_neighborhood_and_P_complex}
Now, we want to find a sufficiently fine covering of $K$ by \textit{rational convex polyhedrons}.
There are plenty of such coverings; for instance, we may find a \textit{polyhedral complexes} (Definition \ref{Polyhedral decomp-defn-appendix}) such that each cell has sufficiently small diameter.

On the one hand, by Corollary \ref{Reverse_ineq_cor_for_compact_domain}, we have $c=c_{J,K'}>0$ and a neighborhood
$\mathcal V':=\mathcal V_{J,K'}   \subset \mathfrak J_K $
of $J$
so that: for any $\tilde J\in\mathcal V'$ and $q\in K'$, a $\tilde J$-holomorphic disk $u$ bounding $L_q$ satisfies $E(u)\ge c \cdot  \ell(\partial u)$.

On the other hand, recall that we choose $\mathcal U$ and $\{W_q \mid q \in B_0 \}$ in \S \ref{sss_diff_construct} such that $\chi_q(W_q)\subset \mathcal U$.
Shrinking $\mathcal U$ and $W_q$'s if necessary, there exists a smaller neighborhood $\mathcal V\subset \mathcal V'$ such that
the map $\diff_0(X)\times \mathfrak J(X,\omega)\to \mathfrak J(X,\omega)$ defined by $(F,J)\mapsto F_*J\equiv dF\circ J\circ dF^{-1}$ sends $\mathcal U\times \mathcal V$ into $\mathcal V'$.

Let $\delta>0$ be the Lebesgue number of the covering $\{W_q \mid q\in K\}$ of the compact domain $K$. Take
\begin{equation}
\label{epsilon_polyhedral_diameter-eq}
0<\epsilon \le \min \{ \tfrac{c}{2}, \tfrac{\delta}{3}\}
\end{equation}
By Lemma \ref{Polyhedral decomposition- integral-lem-appendix}, we can find a rational polyhedral complex
\begin{equation}
\label{P_polyhedral_cx-eq}
\mathscr P
:= \mathscr P_K :=
\{ \Delta_i\mid i\in \I\}
\end{equation}
where $\I$ is a finite index set, such that $K\subset |\mathscr P| \subset K'$ and each cell $\Delta_i$ has diameter less that $\epsilon$.

\begin{lem}
	\label{Delta_Lebesgue_lem}
	Given $i\in \I$, we denote by $\mathscr P_i$ the subcomplex consisting of all $\Delta_j$ in $\mathscr P$ with $\Delta_j\cap \Delta_i\neq \varnothing$.
	Then, there exists some $q\in K$ such that $|\mathscr P_i|\subset W_q$.
\end{lem}

\begin{proof}
	It suffices to show the diameter of $|\mathscr P_i|$ is less than the Lebesgue number $\delta$. In fact, given any two points $x,y \in |\mathscr P_i|$, we may assume $x\in \Delta_{j_1}$ and $y\in \Delta_{j_2}$ for some $j_1, j_2\in\I$. By definition, there exist some points $x'\in \Delta_{j_1}\cap \Delta_i$ and $y'\in \Delta_{j_2}\cap \Delta_i$. Then,
	$d(x,y)\le d(x,x')+d(x',y')+d(y',y)\le \mathrm{diam}(\Delta_{j_1}) +\mathrm{diam}(\Delta_i)+\mathrm{diam}(\Delta_{j_2})\le 3\epsilon < \delta$.
\end{proof}

\subsubsection{Polytopal domains}
\label{sss_polytopal_domains_local_chart}
Recall that $B_0$ is an integral affine manifold. Given $q\in B_0$, we define
\[
\Lambda[[\pi_1(L_q)]] =\left\{
\sum_{i=0}^\infty s_i Y^{\alpha_i}
\mid s_i\in\Lambda, \alpha_i\in \pi_1(L_q)
\right\}
\]
where $Y$ is a formal symbol.
Note that $L_q\cong T^n$; given a basis, there is an isomorphism $\Lambda[[\pi_1(L_q)]]\cong \Lambda[[Y_1^\pm,\dots, Y_n^\pm]]$ via the identification $Y^\alpha\leftrightarrow Y^{\alpha_1}\cdots Y^{\alpha_n}$ for $\alpha =(\alpha_1,\dots,\alpha_n) \in  \pi_1(L_q)\cong \mathbb Z^n$.

Suppose $\Delta\subset B_0$ is a rational convex polyhedron (Proposition \ref{Polyhedron-transition-prop-appendix}) such that there is an integral affine chart containing both $q$ and $\Delta$.
We define
\begin{equation*}
\Lambda\langle \Delta; q \rangle =\left\{
\sum s_\alpha Y^\alpha \in \Lambda[[\pi_1(L_q)]] \mid  \val(s_\alpha) + \langle \alpha , \gamma \rangle \to \infty \ \text {for all $\gamma\in H^1(L_q)$ with $\exp_q(\gamma)\in \Delta$}
\right\}
\end{equation*}
where the bracket denotes the natural pairing $\pi_1(L_q) \times H^1(L_q) \to \mathbb R$ and the $\exp_q$ denotes the affine exponential map (\S \ref{sss_integral_affine_str}).
Abusing the terminologies, we call it a polyhedral affinoid $\Lambda$-algebra; cf. Proposition \ref{polyhedral-affinoid-algebra-evaluation-convergence-prop}.
In reality, given a basis, the $\Delta':=\exp_q^{-1}(\Delta)$ can be naturally identified with a rational convex polyhedron in $H^1(L_q)\cong \mathbb R^n$, thereby inducing a ring isomorphism
$
\eta: \Lambda\langle \Delta ,q\rangle \to \Lambda\langle \Delta' \rangle$
defined by $Y^\alpha\mapsto Y_1^{\alpha_1}\cdots Y_n^{\alpha_n}$.
It also gives an affinoid space isomorphism
$\eta^*: \Sp \Lambda\langle \Delta' \rangle \to \Sp \Lambda\langle\Delta;q\rangle$.

By Proposition \ref{polyhedral-affinoid-algebra-prop-appendix}, the affinoid space $\Sp\Lambda\langle \Delta, q\rangle$ can be identified with the polytopal domain $\Sp\Lambda\langle \Delta'\rangle\equiv  \trop^{-1}(\Delta')$ in the non-archimedean torus fibration $\trop: (\Lambda^*)^n \to \mathbb R^n$.
We define
\begin{equation}
\label{Val_q-eq}
\Val_q:=\Val^\Delta_q:=\exp_q\circ \Val \circ (\eta^*)^{-1}: \Sp \Lambda \langle \Delta;q \rangle \to \Delta \subset B
\end{equation}

From now on, for simplicity, we will often not distinguish them if there is no confusion.
\[
\xymatrix{
	(\Lambda^*)^n \ar[d]^{\Val} 
	\ar@{}[r]|-*[@]{\supset} & \Sp \Lambda \langle \exp_q^{-1}(\Delta)\rangle \ar[d]^{\Val}  \ar@{->}[r]_{\quad \eta^*}^{ \quad \cong} & \Sp \Lambda \langle\Delta;q\rangle \ar[d]^{\Val_q} \\
	*+[l]{H^1(L_q)\cong \mathbb R^n} \ar@{}[r]|-*[@]{\supset}& \exp_q^{-1}(\Delta) \ar[r]^{\exp_q} & *+[r]{\Delta\subset B}
}
\]

\subsubsection{Maurer-Cartan formal power series}
\label{sss_MC_formal_series}

For any $i$, we fix a point $q_i\in \Delta_i$.
We set $L_i:=L_{q_i}$.
Since $J\in\mathfrak J_K$ and $q_i\in K$, we know $J\in \mathfrak J(X,L_i,\omega)$ (Assumption \ref{assumption-mu ge 0}).
By Theorem \ref{Vir-m^J-thm}, we have an $A_\infty$ algebra $\check \m^{J , L_i}$ in $\UD$. By Theorem \ref{UD-canonical-model-thm}, it has a canonical model $(\m^{g,J,L_i}, \mi^{g,J,L_i})$ in $\UD$ for the $g$-harmonic contraction using the convention in (\ref{m_gJL-eq}).
For simplicity, we simplify their notations:
\begin{equation}
\label{check m_Ji-eq}
\check\m^{J, i} := \check \m^{J,L_i}
\end{equation}
\begin{equation}
\label{m_gJi-eq}
\m^{ g, J, i} := \m^{g,J,L_i}
\end{equation}

We call the following formal power series
\[
P^{g,J,i}= \sum_{\beta\in \G(X,L_i)} T^{E(\beta)} Y^{\partial \beta} \m_{0,\beta}^{g,J, i}  \quad \in \Lambda[[\pi_1(L_i)]]\hat\otimes H^*(L_i)
\]
the \textit{Maurer-Cartan formal power series} associated to the $\m^{g,J,i}$.
The gappedness tells that it only involves at most countable many $\beta\in \pi_2(X,L_i)\equiv \G(X,L_i)$.
Note that the coefficient ring of the natural pairing
$\langle \cdot ,\cdot\rangle:  H_{*}(L_{i}) \otimes H^*(L_{i}) \to \mathbb R$ can be extended to $\Lambda$ or $\Lambda[[\pi_1(L_i)]]$.
Given $\eta\in H_*(L_i)$, we call the formal power series
$
\langle \eta, P^{g,J,i} \rangle \equiv
\textstyle
 \sum_\beta T^{E(\beta)}Y^{\partial\beta} \langle \eta, \m_{0,\beta}^{g,J,i} \rangle
$
the $\eta$-\textit{component} of $P^{g,J,i}$.

By degree reason, we have
$
\m_{0,\beta}^{g,J,i}\in H^{2-\mu(\beta)}(L_{i})$; it can be nonzero only if $\mu(\beta)=0$ or $2$.
Hence, we decompose the $P^{g,J,i}$ as follows:
\[
W^{g,J,i} =\sum_{\mu(\beta)=2} T^{E(\beta)} Y^{\partial \beta} {\m_{0,\beta}^{g,J,i}} \big/ {\one_i}
\]
and
\[
Q^{g,J,i} =  \sum_{\mu(\beta) = 0} T^{E(\beta)} Y^{\partial \beta} \m_{0,\beta}^{g,J, i} 
\]
Then, we have
$
P^{g,J,i}=W^{g,J,i} \cdot \one_i  + Q^{g,J,i}$. Here we denote by $\one=\one_i$ the constant-one function on $L_i$.

\begin{lem}\label{convergent-P^J-lem}
Every component of $P^{g,J,i}$
is contained in $\Lambda\langle \Delta_i; q_i\rangle$.
Namely, $W^{g,J,i}\in \Lambda\langle \Delta_i; q_i\rangle$ and each component of $Q^{g,J,i}$ lies in $\Lambda \langle \Delta_i; q_i\rangle$.
\end{lem}

\begin{proof}
Suppose $u:(\mathbb D,\partial\mathbb D)\to (X,L_i)$ is a nontrivial $J$-holomorphic disk.
Note that the base point $q$ lives in the compact domain $K'$ (\ref{K_domain_eq});
thus, for the constant $c>0$ and the neighborhood $\mathcal V$ of $J$ (introduced in \S \ref{sss_V_U_neighborhood_and_P_complex}), we have $E(u)\ge c\ell(u)$.
By construction, we have $\mathrm{diam}(\Delta_i) \le \epsilon$ for the $\epsilon$ in (\ref{epsilon_polyhedral_diameter-eq}).
Hence, $\frac{1}{2} E(u)\ge \frac{c}{2} \ell(\partial u) \ge \epsilon \ell (\partial u)$.
Denote the boundary of $u$ by $\sigma(t)=u(e^{2\pi i t})$, and then
\[
\textstyle
\langle \partial u, \gamma\rangle 
=\int_{\partial u} \gamma =\int_0^1 \gamma(\sigma'(t)) dt \le \int_0^1 |\gamma|\cdot |\sigma'(t)| dt\le \epsilon \int_0^1 |\sigma'(t)|dt \le \epsilon \cdot \ell(\partial u)
\]
holds for any $\gamma \in \Delta_i\subset H^1(L_i)$. It follows that $\frac{1}{2}E(u) \ge \langle \partial u, \gamma\rangle$.

In general, suppose $\beta$ is represented by a stable map $\uu$. First, discarding a sphere bubble only decreases the energy $E(\beta)$ but does not effect $\partial\beta$. So, applying the above argument to each disk component of $\uu$ and taking the summation, we obtain $\frac{1}{2} E(\beta) \ge \langle \partial\beta, \gamma\rangle$ for any $\gamma\in \Delta_i$. It follows that
$|\val (T^{E(\beta)}) + \langle \partial \beta, \gamma\rangle | \ge E(\beta) - |\langle \partial \beta, \gamma \rangle| \ge \tfrac{1}{2} E(\beta) \to \infty$. The proof is now complete.
\end{proof}

\subsubsection{Definition of local charts}
\label{sss_local_charts_defn}
Denote by
\begin{equation}\label{ideal_a_i-eq}
\ia_i:=\ia(g,J,i)
\end{equation}
the ideal in $\Lambda\langle \Delta; q_i \rangle$ generated by the $\eta$-components $\langle \eta, Q^{g,J,i}\rangle$ for all $\eta\in H_*(L_i)$, and we call it \textit{the obstruction ideal} associated to the $\m^{g,J,i}$.
In practice, this obstruction ideal often vanishes.
Note that the $\ia_i$ is a finitely generated ideal. 
For any $i\in \I$, we consider the quotient affinoid algebra $
A_i:=\Lambda\langle \Delta_i, q_i \rangle / \sqrt \ia_i$.
Now, we define
\begin{equation}\label{X_i-eq}
X_i:=\Sp A_i
\end{equation}
It is supposed to be a \textit{local chart} of the mirror space $X^\vee$ in Theorem \ref{Main_theorem_thm}; besides, a local piece of the superpotential $W^\vee$ will be given by $W^{g,J,i}$.
Next, we aim to define the gluing maps (i.e. transition maps) among these local charts $X_i$.

\subsection{Transition maps}

\subsubsection{Fukaya's trick}
\label{sss_Fuk_trick_transition_map}
Let $\Delta_j$ and $\Delta_k$ be two adjacent rational convex polyhedrons in $\mathscr P$.
By Lemma \ref{Delta_Lebesgue_lem}, there is some neighborhood $W_q\supset \Delta_j\cup \Delta_k$ (it is not unique in general).
Then, there exists a diffeomorphism $F\in \diff_0(X)$
such that $F(L_k)=L_j$ and $F\in\mathcal U$ (\S \ref{sss_diff_construct}).
Recall that we have $F_*J\in \mathcal V'$ (\S \ref{sss_V_U_neighborhood_and_P_complex}), and so the reverse isoperimetric inequality can be applied.

Recall the $A_\infty$ algebras $\check \m^{J,j}$, $\check \m^{J,k}$, $\m^{g,J,j}$, and $\m^{g,J,k}$ are defined in (\ref{check m_Ji-eq}) and (\ref{m_gJi-eq}).
By Fukaya's trick, we can obtain the pushforward $A_\infty$ algebras $\check \m^{F_*J, L_j}$ and $\m^{F_*(g,J),L_j}$ as in (\ref{check m_F J L'-eq}) and (\ref{m_F g J L' -eq}).
Again, for clarity, we slightly change the notations as follows:
\begin{equation}
\label{check_m_F(Ji)-eq}
\check \m^{ F_*(J, k)}:= \check \m^{ F_*J, L_j}, \quad   
\m^{ F_*(g,J,k)}:=  \m^{ F_*(g,J), L_j}
\end{equation}

\subsubsection{$A_\infty$ homotopy equivalence}
\label{sss_A_homotopy_equivalence_local_chart}
Now, we choose a path
$\mathbf J:=\mathbf J_F=(J_s)_{s\in\oi}$
of almost complex structures in $\mathcal V$ between $J_0=J$ and $J_1=F_*J$.
Note that the reverse isoperimetric inequality applies for every $J_s$-holomorphic disks by the definition of $\mathcal V$ in \S \ref{sss_V_U_neighborhood_and_P_complex}.
We also choose a path $\mg:=\mg_F=(g_s)_{s\in\oi}$
of metrics between $g_0=g$ and $g_1=F_*g$.

First, applying Theorem \ref{Vir-M-path-thm} to the moduli space system $\mathbb M(\mJ)$ yields a pseudo-isotopy 
\begin{equation}
\label{check_M-mF_j-eq}
\check \M^{ F,j}:= \check \M^{\mathbf J, L_j}\in\Obj\UD(L_j,X)
\end{equation}
on $\Omega^*(L_j)_\oi$ such that it restricts to $\check \m^{J,j}$ and $\check\m^{F_*(J,k)}$ at $s=0$ and $s=1$.
Also, by Theorem \ref{canonical_model_general_tree-thm}, we obtain from the contraction $\con(\mg)$ a canonical model of $\check \M^{F,j}$ which we denote by
\begin{equation}
\label{M^Fj_MI^Fj-eq}
(H^*(L_{j})_\oi, \M^{F,j},\mI^{F,j} ):= (H^*(L_{j})_\oi, \M^{\mg,\mathbf J},\mI^{\mg, \mJ} )
\end{equation}
Due to Theorem \ref{UD-canonical-model-[0,1]-thm}, we know $\M^{F,j}\in\Ob\UD$ and $\mI^{F,j}\in\Mor\UD$.
Moreover, applying Theorem \ref{from-pseudo-isotopies-to-A-infty-homo-thm} to the two pseudo-isotopies $\check \M^{F,j}$ and $\M^{F,j}$ produces the following two $A_\infty$ homomorphisms
\begin{align}
\label{check_mC_mF_j-eq}
\check \mC^{ F}:=\check \mC^{ F,j}:= \check \mC^{\mJ,L_{j}} \ \ \
&:
 \check \m^{J,j}  \ \ \to \check \m^{ F_*(J,k)} 
\\
\label{mC_mF,j-eq}
\mC^{ F} :=\mC^{ F,j}:= \mC^{\mg, \mJ,L_j }
&:
 \m^{g,J,j} \to \m^{F_*(g,J,k)}
\end{align}
Due to Theorem \ref{UD-mC_thm}, both of them are morphisms in $\UD$.

Define
$\Delta_{jk}=\Delta_{kj}=\Delta_j\cap \Delta_k$,
and we can similarly define $\Lambda \langle \Delta_{jk}; q_j \rangle$ and $\Lambda \langle \Delta_{jk}; q_k \rangle$ as before in \S \ref{sss_polytopal_domains_local_chart}.
For the natural inclusions
$
\Lambda \langle\Delta_j; q_j\rangle
\subset
\Lambda\langle\Delta_{jk}; q_j\rangle$ and $\Lambda \langle\Delta_k; q_k\rangle
\subset
\Lambda\langle\Delta_{jk}; q_k\rangle$, the extension of ideal $\ia_j$ in (\ref{ideal_a_i-eq}) will be still denoted by $\ia_j$.
Then, the quotient algebras
$
A_{jk}:= \Lambda\langle \Delta_{jk} ; q_j \rangle / \sqrt{\ia_{j}}$
and
$A_{kj}:= \Lambda\langle \Delta_{jk} ; q_k \rangle / \sqrt{ \ia_{k}}$ have the natural inclusions
$
A_j \xhookrightarrow{} A_{jk}
$; so, there are the natural embedding maps
$
X_{jk} := \Sp A_{jk} \xhookrightarrow{} X_j $ and
$
X_{kj}:= \Sp A_{kj} \xhookrightarrow{} X_k
$.

\subsubsection{Ring homomorphism}
We aim to use Proposition \ref{gluing X_ij -prop-appendix} to glue the various local charts.
Since $\Lambda \langle \Delta_{jk}; q_j \rangle$ and $\Lambda \langle \Delta_{jk}; q_k \rangle$ are contained in $\Lambda[[\pi_1(L_j)]]$ and $\Lambda[[\pi_1(L_k)]]$ respectively, we start with the following ring homomorphism:
\begin{equation}
\label{phi_jk_defn-eq}
\begin{aligned}
\phi_{jk}^{ F}: 
\Lambda[[\pi_1(L_k)]]
&\to \Lambda[[\pi_1(L_j)]]\\
s Y^\alpha & \mapsto sT^{\langle \alpha, q_j -q_k  \rangle} \cdot 
Y^{  F_* \alpha}\cdot
\exp  \Big\langle  F_*\alpha, 
\sum_{\beta\in \G(X,L_j)} \mC^{ F}_{0,\beta} T^{E(\beta)} Y^{\partial \beta}
\Big\rangle
\end{aligned}
\end{equation}
where $s\in \Lambda$, $\alpha\in \pi_1(L_{k})$ and $q_j-q_k:=\exp_{q_k}^{-1}(q_j)$ represents the preimage of $q_j$ in $H^1(L_{k})\cong T_{q_k}B_0\cong \mathbb R^n$ under the affine exponential map $\exp_{q_k}$ (\S \ref{sss_integral_affine_str}). Recall that $\mC_{0,0}^F=0$ by definition.

More importantly, we observe that the class $\mC^F_{0,\beta}$ lives in $H^{1-\mu(\beta)}(L_j)$. If it is nonzero, then we must have $\mu(\beta)\ge 0$ (Assumption \ref{assumption-mu ge 0}), and so the only possibility is $\mu(\beta)=0$.

The gluing maps among the various local charts are derived from these ring homomorphisms (\ref{phi_jk_defn-eq}) and are therefore only contributed by the counts of Maslov index zero holomorphic disks.
This exactly agrees with the observations for the \textit{wall crossing phenomenon} \cite{AuTDual}.
Remark also that the formulas in \cite[Proposition 3.9]{AuTDual} or \cite[(2.4)]{AAK_blowup_toric} are essentially compatible with (\ref{phi_jk_defn-eq}).

\subsubsection{Affinoid algebra homomorphism}

We need a technical lemma for the convergence.

\begin{lem}
	\label{mC_0beta_E(beta)_lemma}
	If $\mC^{F}_{0,\beta}\neq 0$, then $\frac{1}{2}E(\beta) \ge |\langle\partial\beta, q-q_j\rangle|$ for any $q\in \Delta_{jk}$.
\end{lem}

\begin{proof}
	Without loss of generality, we may assume $\beta$ can be represented by a $J_s$- holomorphic disk $u$ for some $s\in \oi$.\footnote{	In fact, the $\mC^F$ is obtained from the pseudo-isotopy $\check \M^{F,j}$. If we denote by $\mathsf G$ the set of $\beta$ with $\check \M^{F,j}_\beta\neq 0$, then by Remark \ref{Gapped_canonical_precise-rmk} and Remark \ref{Gapped_mC_precise-rmk}, the set of $\beta$ with $\mC^F_{0,\beta}\neq 0$ is contained in $\mathbb N\cdot \mathsf G$.
	So, for a general $\beta$, we have a decomposition $\beta=\sum_{m=1}^N \beta_m$, where every $\beta_m$ can be represented by some pseudo-holomorphic disk.}
	Recall that $J_s\in\mathcal V$, thus, one can apply the reverse isoperimetric inequality to conclude that
	$E(u)\ge c \cdot \ell(\partial u)$, where $c$ and $\mathcal V$ are chosen as in \S \ref{sss_V_U_neighborhood_and_P_complex}.
	 By condition, $|q-q_j|\le \mathrm{diam}(\Delta_j)<\epsilon$. Just like the proof of Lemma \ref{convergent-P^J-lem} one can similarly show $|\langle \partial\beta , q-q_j\rangle | \le \epsilon \cdot \ell(\partial u)$. 
	 Then, from (\ref{epsilon_polyhedral_diameter-eq}), it follows that $|\langle \partial\beta , q-q_j\rangle | \le \frac{1}{2}E(\beta)$ which completes the proof.
\end{proof}

\begin{thm}\label{affinoid_alg_homo_thm}
The $\phi_{jk}^F$ in (\ref{phi_jk_defn-eq}) restricts to an affinoid algebra homomorphism (using the same notation): 
	\begin{equation}
	\label{phi_jk-eq}
	\phi_{jk}^{ F}: \Lambda \langle \Delta_{kj};q_k\rangle \to \Lambda \langle \Delta_{jk};q_j\rangle
	\end{equation}
	Moreover, $\Val_{q_k}\circ (\phi_{jk}^F)^*=\Val_{q_j}$.
\end{thm}

\begin{proof}	
	Fix $f=\sum_{i=0}^\infty s_i Y^{\alpha_i}$ in $\Lambda \langle \Delta_{kj},q_k\rangle$.
	By Proposition \ref{polyhedral-affinoid-algebra-evaluation-convergence-prop}, it suffices to show that $\phi_{jk}^F(f)$ converges on $\trop_{q_j}^{-1}(\Delta_{jk})$.
	As in \S \ref{sss_polytopal_domains_local_chart}, the $\trop_{q_j}$ is identified with the restriction of $\trop$ over $\Delta':=\exp_{q_j}^{-1}(\Delta_{jk})$.
	
	Let
	$\mathbf y$ be a point in $\trop_{q_j}^{-1}(\Delta_{jk})\cong \trop^{-1}(\Delta')$;
we can view it as a point $\mathbf y=(y_1,\dots, y_n)$ in $(\Lambda^*)^n$ such that
	$q_{\mathbf y}-q_j:=\trop(\mathbf y)=(\val(y_1),\dots, \val(y_n))\in\mathbb R^n$
	is contained in $\Delta' \subset H^1(L_j)\cong \mathbb R^n$.
	Recall that we may identify both $Y^\alpha$ and $Y^{F_*\alpha}$ with $Y_1^{\alpha_1}\cdots Y_n^{\alpha_n}$ for $(\alpha_1,\dots,\alpha_n)\in\mathbb Z^n\cong \pi_1(L_j)\cong\pi_1(L_k)$. By the substitution at $\mathbf y$ we mean that one replaces any such monomial by the value $y_1^{\alpha_1}\cdots y_n^{\alpha_n}\in\Lambda$.	
	
	It remains to show that after the substitution at $\mathbf y$, the series $\phi_{jk}^F(f)|_{Y=\mathbf y} = \sum_{\ell=0}^\infty \phi_{jk}^F(s_\ell Y^{\alpha_\ell})|_{Y=\mathbf y}$ converges in the Novikov field $\Lambda$.
	In other words, we aim to show that (see e.g.  \cite[2.1/3]{BoschBook})
	\begin{equation}
	\label{Proving_alg_homo_phi_jk_F_eq}
	\lim_{\ell \to \infty} \val\big( \phi_{jk}^F(s_\ell Y^{\alpha_\ell}) \mid_{Y=\mathbf y} \big)
	=\infty
	\end{equation}
Actually, for arbitrary $\alpha$, the definition formula (\ref{phi_jk_defn-eq}) first tells that
	\[
	\textstyle
	\phi_{jk}(Y^\alpha)|_{Y=\mathbf y}= T^{ \langle \alpha,   q_j-q_k \rangle } \cdot \mathbf y^{\alpha} \cdot
	 \exp
	 \Big( \sum_{\mu(\beta)=0, \beta\neq 0} \langle F_*\alpha, \mC_{0,\beta}^F\rangle T^{E(\beta)} \mathbf y^{\partial\beta} \Big)
	\]
	The valuation of a nonzero monomial in the exponent is given by $\val(T^{E(\beta)} \mathbf y^{\partial\beta})=E(\beta)+ \langle \partial \beta , q_{\mathbf y}-q_j \rangle \ge \frac{1}{2} E(\beta)>0$
	thanks to Lemma \ref{mC_0beta_E(beta)_lemma}.
	It follows that the exponential power is contained in $1+\Lambda_+$. Thus,
	\[
	\val(\phi_{jk}^F(s_\ell Y^{\alpha_\ell})|_{Y={\mathbf y}})
	=
	\val(s_\ell T^{\langle \alpha_\ell , q_j-q_k\rangle }  \mathbf y^{\alpha_\ell}) =\val(s_\ell)+\langle \alpha_\ell, q_j-q_k\rangle + \langle \alpha_\ell, q_{\mathbf y}-q_j\rangle=\val(s_\ell)+ \langle \alpha_\ell , q_{\mathbf y}-q_k\rangle
	\]
	Moreover, the condition that $f\in \Lambda\langle \Delta_{jk},q_k\rangle $ exactly means that for any $q\in \Delta_{jk}$ we have $\val(s_\ell)+ \langle \alpha_\ell, q -q_k\rangle\to \infty$. Hence, we prove (\ref{Proving_alg_homo_phi_jk_F_eq}).
	Finally, it is easy to see that the $\phi_{jk}^F$ preserves the multiplication. This completes the first half of lemma.
	
	We next show the second statement.
	Note that a point $\mathbf y\in \Val_{q_j}^{-1}(\Delta_{jk})$ corresponds to the maximal ideal $\m_{\mathbf y}=\{f  \mid f(\mathbf y)=0\}$ in $\Lambda\langle \Delta_{jk};q_j \rangle$. So, the point $\mathbf z := (\phi_{jk}^F)^*(\mathbf y)$ is defined by its corresponding maximal ideal $(\phi_{jk}^F)^*(\m_{\mathbf y})$. Write $\mathbf y=(y_1,\dots, y_n)$ and $\mathbf z=(z_1,\dots, z_n)$. Then, the above argument shows that $z_r= T^{a_r} y_r e_r$, where $e_r$ is some exponential power in $1+\Lambda_+$ and $(a_1,\dots, a_n)=q_j-q_k$. Thus, $\val(z_r)=a_r+\val(y_r)$ for $1\le r\le n$, and we obtain $\Val_{q_j}(\mathbf z)=\Val_{q_k}(\mathbf y)$.
\end{proof}

\subsubsection{Wall crossing formula}
\label{sss_wall_crossing_formula}

Recall the local charts are given by the quotient algebras $A_i$'s (\S \ref{sss_local_charts_defn}). We need to further show that the $\phi_{jk}^F$ given in Theorem \ref{affinoid_alg_homo_thm} induces a quotient homomorphism modulo the ideals of weak Maurer-Cartan equations.
To achieve this, we need the following result.

Fix a basis $(f_i)$ of $\pi_1(L_k)$; it induces a dual basis $(\theta_i)$ of $H^1(L_k)$. Since $F(L_k)=L_j$, we also choose the bases $(\tilde f_i)$ and $(\tilde\theta_i)$ of $\pi_1(L_j)$ and $H^1(L_j)$ which are $F$-related to the previous ones. For any $\beta\in\G(X,L_k)$, we write $\tilde \beta:= F_*\beta \in \G(X,L_j)$ as before.
Note that $\theta_{pq}=\theta_p\wedge \theta_q$ ($1\le p<q\le n$) form a basis of $H^2(L_k)$. Denote $Q^{g,J,j}=\sum_{1\le p<q\le n} Q_{pq}^{g,J,j}\cdot \theta_{pq}$ (\S \ref{sss_MC_formal_series}).

\begin{thm}[Wall crossing formula]
	\label{wall_crossing_thm}
	For $\eta\in H_*(L_k)$, we have
	\begin{equation}
	\label{wall_crossing_eq}
	\phi_{jk}^F (\langle\eta, P^{g,J,k}\rangle)
	=
	\langle F_*\eta, \one_j\rangle \cdot W^{g,J,j} + \sum_{1\le p<q\le n} R^{F,\eta}_{pq} \cdot Q^{g,J,j}_{pq}
	\end{equation}
	where
	\begin{align*}
	R_{pq}^{F,\eta}=\sum_{ \tilde \beta\in\G(X,L_j)}
	T^{E( \tilde \beta)}Y^{\partial \tilde \beta} 
	\langle
	F_*\eta,
	\mC^{ F}_{1, \tilde \beta} ( \theta_{pq})
	\rangle \in \Lambda \langle \Delta_j; q_j\rangle
	\end{align*}
\end{thm}

\begin{proof}
	By the same arguments as in Lemma \ref{convergent-P^J-lem}, every $R_{pq}^{F,\eta}$ above is also contained in $\Lambda\langle \Delta_j,q_j\rangle$.
	We can think of (\ref{wall_crossing_eq}) as an identity in $\Lambda[[\pi_1(L_j)]] \cong \Lambda[[Y_1^\pm,\dots, Y_n^\pm]]$ for the above bases.
	By Lemma \ref{val=0-lem} again, it suffices to show the identity holds on $U_\Lambda^n$.
	
	Assume $\mathbf y=(y_1,\dots,y_n)\in U_\Lambda^n$, i.e. $\val(y_i)=0$ for all $i$. 
	By Lemma \ref{exp-log-lem}, there exists $x_i\in \Lambda_0$ such that $y_i=e^{x_i}=\sum_{k \ge 0} \frac{1}{k!}x_i^k$.
	We put $b:=x_1\theta_1+\cdots+x_n \theta_n\in H^1(L_{k})\hat \otimes \Lambda_0$
	and
	$\tilde b:= F^{-1*} b=x_1\tilde \theta_1+\cdots+x_n \tilde \theta_n\in H^1(L_{j})\hat\otimes \Lambda_0$. 
Set $\partial_i\beta= \partial\beta \cap \theta_i=\partial\tilde\beta \cap \tilde \theta_i$, and we note that 
\[
\partial \beta\cap b=\partial\tilde \beta\cap \tilde b=\partial_1\beta \cdot x_1+\cdots +\partial_n \beta \cdot x_n
\]
Besides, substituting $y_i=e^{x_i}$ into the monomial $Y^{\partial\beta}\equiv Y^{\partial \tilde \beta}\equiv Y_1^{\partial_1\beta}\cdots Y_n^{\partial_n \beta}$ gives the same value in the Novikov field $\Lambda$:
$e^{\langle\partial \beta,b\rangle}=e^{\langle \partial \tilde\beta ,\tilde b\rangle}=e^{\partial_1\beta \cdot x_1+\cdots +\partial_n \beta \cdot x_n}$.
We put $q_j-q_k:= \exp_{q_k}^{-1}(q_j)$ as before; then by Remark \ref{G(M,L)_energy_vary_rmk}, we have
	(see also \cite{FuCyclic} \cite{AboFamilyFaithful}):
	\begin{equation}
	\label{energy-change-eq}
	E(\tilde \beta)=E(\beta)+ 
	\langle \partial \beta , q_j-q_k \rangle	
	\end{equation}
Recall that $\mC^{ F}_{1,0}=\id$ and $\m_{0, \beta}^{g,J,k}= F^*\m_{0,\tilde \beta}^{ F_*(g,J,k)}$ (Lemma \ref{Fukaya's trick-canonical-level-lem}).
Next, using (\ref{phi_jk_defn-eq}), we compute:
	\begin{align*}
	\phi_{jk}^F(\langle\eta, P^{g,J,k}\rangle )|_{Y=\mathbf y}
	&
	=
	\sum_{\beta }
	\langle \eta, \m_{0,\beta}^{g,J,k}\rangle
	T^{E(\beta)}
	\
	T^{\langle \partial\beta, q_j-q_k \rangle }
	Y^{  \partial \tilde  \beta} 
	\exp 
	\Big( 
	\sum_{\gamma } 
	\langle   \partial \tilde \beta, \mC^{ F}_{0,\gamma} \rangle
	T^{E(\gamma)} Y^{\partial \gamma}
	\Big) 
	\mid_{Y=\mathbf y}
	\\
	&
	=
	\sum_{\tilde \beta }
	\langle \eta,  F^* \m_{0,\tilde \beta}^{ F_*(g,J,k)}\rangle 
	\
	T^{E(\tilde\beta) }
	e^{\langle \partial\tilde \beta, \tilde b\rangle}
	\exp 
	\Big( 
	\sum_{\gamma } 
	\langle \partial\tilde \beta, \mC^{ F}_{0,\gamma} \rangle
	T^{E(\gamma)} e^{\langle \partial \gamma, \tilde b\rangle}
	\Big)
	\\
	&
	=
	\sum_{\tilde \beta}
	\langle  F_*  \eta, \m_{0,\tilde \beta}^{ F_*(g,J,k)}\rangle 
	T^{E(\tilde \beta)}
	\exp 
	\Big\langle \partial \tilde \beta, \tilde b+
	\sum_{\gamma} 
	\mC^{ F}_{0,\gamma}
	T^{E(\gamma)} e^{\langle \partial \gamma, \tilde b\rangle}
	\Big\rangle
	\end{align*}
	Note that $\mC^F\in\Mor\UD$ (\ref{mC_mF,j-eq}). We have $\mu(\gamma)\ge 0$ whenever $\mC^F_{0,\gamma}\neq 0$ (Definition \ref{UD_defn} (II-5)).
	So, the class $\mC_{0,\gamma}^{ F}\in H^{1-\mu(\gamma)}(L_j)$ is nonzero only if $\mu(\gamma)=0$; then, $\mC^F_{0,\gamma}\in H^1(L_j)$.
	By the divisor axiom of $\mC^F$, the exponent is given by bracketing $\partial \tilde \beta$ with the following class:
	\begin{equation}
	\label{mC_*_F_eq} 
	\tilde b+
	\sum_{\gamma} 
	\mC^{ F}_{0,\gamma}
	T^{E(\gamma)} e^{\langle \partial \gamma, \tilde b\rangle}
	=
	\sum_{k\ge 0} \sum_{\gamma} 
	T^{E(\gamma)} \mC^{ F}_{k,\gamma} (\tilde b,\dots,\tilde b)
	\equiv
	\mC_*^{ F} ( \tilde b)
	\end{equation}
	Notice that $\mC_*^F(\tilde b)\in H^1(L_{j})\hat\otimes \Lambda_0$ can be also viewed as a divisor input (\ref{Divisor_input-eq}).
	Using the divisor axiom of $\m^{F_*(g,J,k)}$ to this new divisor input $\mC_*^F(\tilde b)$ yields:
	\begin{align*}
	\phi_{jk}^F (\langle\eta, P^{g,J,k}\rangle)|_{Y=\mathbf y} 
	&
	=
	\Big\langle  F_*\eta,  \sum_{\tilde\beta}
	T^{E(\tilde \beta)}
	\m_{0,\tilde \beta}^{ F_*(g,J,k)}
	\exp \langle \partial \tilde \beta, \mC_*^{ F}(\tilde b) \rangle 
	\Big\rangle\\
	&
	=
	\Big\langle  F_*\eta, 
	\sum_k\sum_{\tilde\beta} T^{E(\tilde \beta)} \m^{ F_*(g,J,k) }_{k,\tilde \beta}( \mC_*^{ F}(\tilde b), \dots, \mC_*^{ F}(\tilde b) )
	\Big\rangle
	\end{align*}
	Further, by the $A_\infty$ equation of $\mC^{ F}$ and the divisor axioms of both $\mC^F$ and $\m^{g,J,j}$, we compute
	\begin{align*}
	\phi_{jk} (\langle\eta, P^{g,J,k}\rangle)|_{Y=\mathbf y}
	&
	=\textstyle
	\left\langle F_*\eta, \sum T^{E(\tilde\beta_1)} \mC_{\tilde\beta_1}^F 
	\big(
	\tilde b,\dots, \tilde b, \sum T^{E(\tilde\beta_2)}
	\m_{\tilde\beta_2}^{g,J,i} (\tilde b,\dots, \tilde b), \tilde b,\dots, \tilde b
	\big)
	\right\rangle\\
	&
	=\textstyle
	\left\langle
	F_*\eta, \sum T^{E(\tilde\beta_1)}
	e^{\langle \partial \tilde\beta_1, \tilde b\rangle}
	\mC_{1,\tilde\beta_1}^F 
	\big( \sum T^{E(\tilde\beta_2)} e^{\langle \partial \tilde\beta_2, \tilde b\rangle}\m_{0,\tilde\beta_2}^{g,J,j} \big)
	\right\rangle \\
	&
	=\textstyle
	\sum T^{E(\tilde \beta)} \mathbf y^{\partial \tilde \beta} 
	\left\langle F_*\eta, \mC^F_{1,\tilde \beta} (P^{g,J,j}(\mathbf y) ) \right\rangle
	\end{align*}
	Since $\mC^F$ is unital, we know $\mC_{1,\tilde \beta}^F(\one_j)$ is zero for $\tilde\beta\neq 0$ and $\mC^F_{1,0}(\one_j)=\one_j$ for $\tilde \beta=0$. It follows that
	\begin{align*}
	\phi_{jk}^F (\langle\eta, P^{g,J,k}\rangle)|_{Y=\mathbf y}
	&
	=
	\sum T^{E(\tilde \beta)} \mathbf y^{\partial \tilde \beta}
	\Big(
	\langle F_*\eta, \mC_{1,\tilde \beta}^F(\one_j) \rangle   W^{g,J,j} (\mathbf y)
	+
	\langle F_*\eta, \mC_{1,\tilde \beta}^F(\tilde\theta_{pq}) \rangle   Q_{pq}^{g,J,j}(\mathbf y)
	\Big)\\
	&
	=
	\langle F_*\eta, \one_j\rangle \cdot W^{g,J,j}(\mathbf y) + \sum_{1\le p<q\le n} R^{F,\eta}_{pq}(\mathbf y) \cdot Q^{g,J,j}_{pq}(\mathbf y)
	\end{align*}
	In summary, the wall crossing formula (\ref{wall_crossing_eq}) holds for any point $\mathbf y=(y_1,\dots, y_n)$ in $U_\Lambda^n$. Finally, due to Lemma \ref{val=0-lem}, this actually holds for all $\mathbf y$. 
	The proof is now complete.
\end{proof}

\subsubsection{Definition of transition maps}
For the quotient algebras $A_{jk}$ and $A_{kj}$
(\S\ref{sss_A_homotopy_equivalence_local_chart}), the wall crossing formula implies the existence of quotient affinoid algebra homomorphisms:
\begin{equation*}
\xymatrix{
	\Lambda\langle\Delta_{jk}; q_k\rangle \ar[rr]^{\phi_{jk}^F}\ar@{->>}[d] & & \Lambda\langle\Delta_{jk};q_j\rangle \ar@{->>}[d] \\
	A_{kj}\ar[rr]^{\varphi_{jk}} & & A_{kj}
}
\end{equation*}
\begin{thm}
	\label{phi_jk-A-thm}
	The affinoid algebra homomorphism 
	$\phi_{jk}^F$ in (\ref{phi_jk-eq})
	induces a quotient homomorphism:
	\begin{equation}
	\label{varphi_jk_eq}
	\varphi_{jk}:=\varphi^F_{jk}: A_{kj}\to A_{jk}
	\end{equation}
	Moreover, we have $\varphi_{jk} (W^{g,J,k})
	=W^{g,J,j}$.
\end{thm}

\begin{proof}
	This is basically a corollary of Theorem \ref{wall_crossing_thm}.
	In the wall crossing formula (\ref{wall_crossing_eq}), if $\eta$ is dual to $\theta_{rs}$ for $1\le r<s\le n$, then $\langle F_*\eta,\one_j\rangle =0$ and $\langle \eta,P^{g,J,k}\rangle =Q_{rs}^{g,J,k}$. Hence,
	\[
	\phi_{jk}^F ( Q^{g,J,k}_{rs} )
	=
	\phi_{jk}^F (\langle\eta, P^{g,J,k}\rangle)
	=
	\sum R^{F,\eta}_{pq}   Q_{pq}^{g,J,j} \in \ia_j
	\]
	Since the ideal $\ia_k$ is generated by these $Q_{rs}^{g,J,k}$, it follows that $\phi^F_{jk}(\ia_k)\subset \ia_j$. So, the quotient map $\varphi_{jk}:A_{kj}\to A_{jk}$ is well-defined.
If $\eta$ is dual to $\one_k$, then $\langle F_*\eta, \one_j\rangle=1$ and $\langle \eta, P^{g,J,k}\rangle =W^{g,J,k}$. Hence,
	\[
	\phi_{jk}^F (W^{g,J,k})
	=
	\phi_{jk}^F (\langle\eta, P^{g,J,k}\rangle)
	=
	W^{g,J,j}+\sum R^{F,\eta}_{pq}  Q_{pq}^{F,\eta} \in W^{g,J,j}+\ia_j
	\]
It follows that $\varphi_{jk} (W^{g,J,k})= W^{g,J,j}$.
\end{proof}

Recall that we have defined $X_{jk}=\Sp A_{jk}$ and $X_{kj}=\Sp A_{kj}$ (\S \ref{sss_A_homotopy_equivalence_local_chart}) which are respectively open domains in the local charts $X_j$ and $X_k$ (\S \ref{sss_local_charts_defn}). Then, a \textit{gluing map} (or called \textit{transition map})
\begin{equation}
\label{psi_jk-eq}
\psi_{jk} :=\varphi^*_{jk}: X_{jk} \to X_{kj}
\end{equation}
is defined to be the map associated to the quotient homomorphism $\varphi_{jk}:A_{kj}\to A_{jk}$ in (\ref{varphi_jk_eq}).
Abusing the terminologies, we also call the algebra homomorphism $\varphi_{jk}$ a gluing map or a transition map.
Moreover, using the second statement of Theorem \ref{affinoid_alg_homo_thm}, we immediately obtain:

\begin{cor}
	\label{Val_psi_transition-compatible-cor}
	The transition map $\psi_{jk}$ satisfies that
	$\Val_{q_k}\circ \psi_{jk}=\Val_{q_j}$.
\end{cor}

\subsection{Choice-independence of transition maps}
\label{ss_Choice_indep}

In this section, we aim to prove the induced homomorphism $\varphi_{jk}$ in (\ref{varphi_jk_eq}) does not depends on the various choices, so the transition map $\psi_{jk}=\varphi_{jk}^*$ in (\ref{psi_jk-eq}) is well-defined.
The philosophy behind this statement comes exactly from the gauge equivalence of (weak) bounding cochains as introduced in \cite[\S 4.3]{FOOOBookOne}, but it is only some inspiration and we never use it directly.

\begin{thm}
\label{phi_jk-indep-thm}
The transition maps $\psi_{jk}$ or $\varphi_{jk}$ are independent of the choices of $F$, $\mathbf J$, $\mg$.
\end{thm}

Suppose we differently choose $F'\in \mathcal U\subset \diff_0(X)$ (\S \ref{sss_Fuk_trick_transition_map}), and we also differently choose $\mJ'=(J'_s)$, and $\mg'=(g'_s)$ (\S \ref{sss_A_homotopy_equivalence_local_chart}).
Similar to (\ref{check_m_F(Ji)-eq}, \ref{check_M-mF_j-eq}, 
\ref{M^Fj_MI^Fj-eq}, 
\ref{check_mC_mF_j-eq}, 
\ref{mC_mF,j-eq}), we can use these data to obtain 

\begin{enumerate}[(i)]
	\itemsep 2pt
	\item a chain-level $A_\infty$ algebra $\check \m^{F'_*(J,k)}$, 
\item the canonical model $(H^*(L_j), \m^{F'_*(g,J,k)}, \mi^{F'_*(g,J,k)})$ of $\check \m^{F'_*(J,k)}$ with respect to $\con(F'_*g)$,
\item a chain-level pseudo-isotopy $\check \M^{F',j}$ between $\check \m^{J,j}$ and $\check \m^{F'_*(J,k)}$, 
\item a cohomology-level pseudo-isotopy $\M^{F',j}$ between $\m^{g,J,j}$ and $\m^{F'_*(g,J,k)}$ (it comes with $\I^{F',j}$),
\item the chain-level $A_\infty$ homomorphisms $\check \mC^{F'}: \check \m^{J,j}\to \check \m^{F_*(J,k)}$ induced by $\check \M^{F',j}$, 
\item the cohomology-level $A_\infty$ homomorphism $\mC^{F'} : \m^{g,J,j}\to \m^{F'_*(g,J,k)}$ induced by $\M^{F',j}$.
\end{enumerate}

Notations are used in the same pattern. All of them are contained in the category $\UD$; the similar Fukaya's trick equations hold.
Furthermore, just as (\ref{phi_jk_defn-eq}, \ref{phi_jk-eq}), we can define the homomorphisms $\phi^{F'}_{jk}$ using the above $\mC^{F'}$ (in place of $\mC^{F}$).
Then, we have two quotient homomorphisms  $\varphi_{jk}^{F'}$ and $\varphi_{jk}^F$ which are induced by $\phi^{F'}_{jk}$ and $\phi^F_{jk}$ respectively (\ref{varphi_jk_eq}).
The goal of Theorem \ref{phi_jk-indep-thm} is to prove $\varphi_{jk}^F = \varphi_{jk}^{F'}$.

Find a path $ \mF=(F_s)_{s\in\oi}$ of diffeomorphisms in $\mathcal U\subset \diff_0(X)$ between $F_0= F$ and $F_1= F'$ such that $F_s(L_{k})=L_{j}$
for all $s\in\oi$.
Let $\hat J$ and $\hat g$ be the constant $\oi$-families at $J$ and $g$. Consider
$\mF_* \hat J=(F_{s*}J)_{s\in\oi}$ and $\mF_* \hat g=(F_{s*}g)_{s\in\oi}$.

\subsubsection{Chain-level}
On the one hand, we have an `air-cored triangle' family of almost complex structures whose vertices are given by $J$, $F_*J$, $F'_*J$ and whose edges are given by $\mJ$, $\mJ'$, $\mF_* \hat J$.
Recall that both $\mJ$ and $\mJ'$ are contained in the neighborhood $\mathcal V$; also, since $F_s\in\mathcal U$ for any $s$, we have $F_s(\mathcal V)\subset \mathcal V'$ and $\mF_*\hat J\subset \mathcal V'$
(\S\ref{sss_V_U_neighborhood_and_P_complex}).
Now, we can fill it to obtain a `solid triangle' family in $\mathcal V'$.
Namely, there exists a smooth family $\JJ=(J_x)_{x\in\Delta^2}$ in $\mathcal V'$ parameterized by the 2-simplex $\Delta^2=[v_0,v_1,v_2]$ such that $\JJ|_{[v_0,v_1]}=\mJ$, $\JJ|_{[v_1,v_2]}=\mF_* \hat J$, and $\JJ|_{[v_0,v_2]}=\mJ'$.

On the other hand, by Theorem \ref{Vir-M-path-thm}, the moduli space system $\mathbb M(\hat J)$ for the constant family $\hat J$ can produce a trivial pseudo-isotopy
 $(\Omega^*(L_k)_\oi, \check \M^{\hat J, k})$ about the $A_\infty$ algebra $(\Omega^*(L_k), \check \m^{J,k})$ (\ref{check m_Ji-eq}).
Due to Fukaya's trick (Lemma \ref{Fukaya_II-cochain-level-lem}), it gives rise to a `pushforward' pseudo-isotopy, denoted by:
\begin{equation}
\label{check_M_mF_hat_J_eq}
\big( \Omega^*(L_j), \check \M^{\mF_* \hat J} \big)
\end{equation}
satisfying the equations of Fukaya's trick (\ref{Trick_check_M_family_eq}).
Moreover, by (\ref{Eval_Trick_check_M^F-eq}), it is a pseudo-isotopy between $\check \m^{F_*(J,k)}$ and $\check \m^{F'_*(J,k)}$ (in the chain-level), and it lives in $\UD$.
Finally, due to Theorem \ref{from-pseudo-isotopies-to-A-infty-homo-thm} and Theorem \ref{UD-mC_thm}, its induced $A_\infty$ homotopy equivalence
\begin{equation}
\label{mC_check_F_hat_J_eq}
\check \mC^{\mF_* \hat J}: \check \m^{F_*(J,k)} \to \check \m^{F'_*(J,k)}
\end{equation}
is a morphism in $\UD$. It is supposed to satisfy the following property:

\[
\Scale[0.9]{
\xymatrix{
	& & F'_*J& & & & & \check \m^{F'_*(J,k)} \\
	J\ar@{-}[drr]_{\mathbf J}
	\ar@{} [rr]|{\JJ} 
	\ar@{-}[urr]^{\mathbf J'} & & & & & \check \m^{J,k} \ar@{} [rr]|{\check \M^{\JJ}} 
	\ar@{-}[drr]_{\check \M^{F,j} } \ar@{-}[urr]^{\check \M^{F',j} } & & \\
	& & F_*J \ar@{-}[uu]_{\mF_*\hat  J} & & & & & \check \m^{F_*(J,k)}\ar@{-}[uu]_{\check \M^{\mF_*\hat J}}
}
}
\]

\begin{lem}
\label{simud_mC_mC-lem}
$\check \mC^{\mF_*\hat J}\diamond \check \mC^F \simud \check \mC^{F'}$
\end{lem}

\begin{proof}
	The pseudo-isotopies $\check \M^{ F,j}$, $\check \M^{ F',j}$ and $\check \M^{\mF_*\hat J}$ coincide at their ends. So, applying Theorem \ref{Vir-M-triangle-thm} to the moduli space system $\mathbb M(\JJ)$,
we obtain a $\Delta^2$-pseudo-isotopy
\begin{equation}
\label{check_M_mJ_mJ_eq}
\check \M^{\JJ}\in\Obj\UD(L_{j},X)
\end{equation}
which restricts to $\check \M^{ F,j}$, $\check \M^{ F',j}$ and $\check \M^{\mF_*\hat J}$ on the edges $[v_0,v_1]$, $[v_0,v_2]$, and $[v_1,v_2]$ respectively.\footnote{Some clarification of the Kuranishi-theory choices may be helpful. These choices are taken in the following order. Firstly, we fix the choices $\Xi_i,i\in\I$ for the $A_\infty$ algebras $\check \m^{J,i}$ (\ref{check m_Ji-eq}). Secondly, the choice that defines $\check \M^{\mF_*(\hat g,\hat J)}$ in (\ref{check_M_mF_hat_J_eq}) is just induced by Fukaya's trick, thus, there is no extra choice at this stage. Thirdly, we make the choices to define the pseudo-isotopies $\check \M^{F,j}$ and $\check \M^{ F',j}$ as in (\ref{check_M-mF_j-eq}). Finally, we make the choices for the definition of $\check \M^\JJ$ by Theorem \ref{Vir-M-triangle-thm}.}

Let $I$ be an edge of $\Delta^2$. 
We set
\begin{equation}
\label{restr_eq}
\restr^{\Delta^2}_{I}: \Omega^*(L_{j})_{\Delta^2} \to \Omega^*(L_{j})_{I}
\end{equation}
to be the natural restriction map.
Recall that one can identify $\Omega^*( I \times L_j)\cong \Omega^*(L_j)_{I}$ and $\Omega^*(\Delta^2 \times L_j)\cong \Omega^*(L_j)_{\Delta^2}$ (cf. (\ref{convention-Omega P L-eq}))
Similar to Remark \ref{eval-incl-as-A_infty-rmk}, one can view $\restr_I^{\Delta^2}$ as a morphism in $\UD$, i.e. an $A_\infty$ homotopy equivalence from $\check \M^\JJ$ to one of $\check \M^{F,j}$, $\check \M^{F',j}$, $\check \M^{\mF_*\hat J}$ according to the choice of $I$.
Then, by Theorem \ref{Whitehead-full-thm}, we get an ud-homotopy inverse $(\restr^{\Delta^2}_I)^{-1}$ which is also a morphism in $\UD$.

Let $v$ be a vertex of $I$, and we can similarly define $\eval_v^I$ and $\eval_v^{\Delta^2}$. Then,
$\eval_v^I \diamond \restr^{\Delta^2}_I = \eval^{\Delta^2}_v$, thereby obtaining an ud-homotopic relation:
\[
\eval^{I}_{v} \simud \eval^{\Delta^2}_{v}\diamond (\restr^{\Delta^2}_I )^{-1}
\]
Further, let $v'$ be the other vertex of $I$, then one can easily show that
\[
\eval_{v'}^I\diamond (\eval_v^I)^{-1} \simud \eval_{v'}^{\Delta^2}\diamond (\eval_v^{\Delta^2})^{-1}
\]
From Theorem \ref{UD-mC_thm}, it follows that
\[
\check \mC^{F}\simud \eval^{[v_0,v_1]}_{v_1} \diamond (\eval^{[v_0,v_1]}_{v_0})^{-1}, 
\quad
\check \mC^{ F'}\simud \eval^{[v_0,v_2]}_{v_2} \diamond (\eval^{[v_0, v_2]}_{v_0})^{-1},
\quad
\check \mC^{\mF_* \hat J}\simud \eval^{[v_1,v_2]}_{v_2} \diamond (\eval^{[v_1,v_2]}_{v_1})^{-1}
\]
and thus
\[
\check \mC^{F}\simud \eval^{\Delta^2}_{v_1} \diamond (\eval^{\Delta^2}_{v_0})^{-1}, 
\quad
\check \mC^{ F'}\simud \eval^{\Delta^2}_{v_2} \diamond (\eval^{\Delta^2}_{v_0})^{-1},
\quad
\check \mC^{\mF_* \hat J}\simud \eval^{\Delta^2}_{v_2} \diamond (\eval^{\Delta^2}_{v_1})^{-1}
\]
Now, it is immediate that $\check \mC^{\mF_*\hat J}\diamond \check \mC^F \simud \check \mC^{F'}$; the proof is complete.
\end{proof}

\subsubsection{Cohomology-level}
Let
$(H^*(L_{j}), \M^{\mF_*(\hat g, \hat J)}, \mI^{\mF_* (\hat g,\hat J)})$ be
the canonical model (Definition \ref{canonical_model_defn}) of the pseudo-isotopy $\check \M^{\mF_* \hat J}$ in (\ref{check_M_mF_hat_J_eq}) with respect to $\con(\mF_* \hat g)$. By Theorem \ref{UD-canonical-model-[0,1]-thm}, we know $\M^{\mF_* (\hat g, \hat J)}\in\Obj\UD$ and $\mI^{\mF_* (\hat g, \hat J)}\in\Mor\UD$.
By (\ref{Eval_Trick_cano-eq}), the $\M^{\mF_*(\hat g,\hat J)}$ is a pseudo-isotopy between $\m^{F_*(g,J,k)}$ and $\m^{F'_*(g,J,k)}$ (in the cohomology-level).
Similar to (\ref{mC_check_F_hat_J_eq}), the $A_\infty$ homotopy equivalence induced by $\M^{\mF_* (\hat g, \hat J)}$ is also a morphism in $\UD$, denoted by
\[
\mC^{\mF_* (\hat g, \hat J)}: \m^{F_*(g,J,k)} \to \m^{F'_*(g,J,k)}
\]
By Corollary \ref{Fukaya's trick Independence of F-cor}, the source and target $A_\infty$ algebras of $\mC^{\mF_*(\hat g, \hat J)}$ are identically the same:
\begin{equation}
\label{m_F_F'_eq}
\m:=
\m^{ F_{*}(g,J,k)}=
\m^{ F'_{*}(g,J,k)}
\end{equation}
Moreover, it follows from Corollary \ref{Fukaya_II_key_cor} that
\begin{equation}
\label{mC_I_simud_id-eq}
\mC^{\mF_*(\hat g,\hat J)}\simud \id
\end{equation}
Ultimately, applying Lemma \ref{mC-diamond-diagram-lem} repeatedly to the pairs $(\check \mC^F, \mC^F)$, $(\check \mC^{F'}, \mC^F)$, and $(\check \mC^{\mF_*\hat J}, \mC^{\mF_*(\hat g,\hat J)})$, we obtain the following ud-homotopy relations in sequence:
\begin{align}
\label{Choice_indep_1-eq}
\mi^{F_*(g,J,k)} \diamond \mC^F
&\simud \check\mC^F\diamond \mi^{g,J,j}  \\
\label{Choice_indep_2-eq}
\mi^{F'_*(g,J,k)} \diamond \mC^{F'}  
&\simud  \check \mC^{F'}\diamond \mi^{g,J,j} \\
\label{Choice_indep_3-eq}
\mi^{F'_*(g,J,k)}\diamond \mC^{\mF_* (\hat g,\hat J)}
&\simud 
\check \mC^{\mF_*\hat J} \diamond \mi^{F_*(g,J,k)} 
\end{align}

\[
\Scale[0.8]
{
	\xymatrix{
		& & & & & & \m^{F'_*(g,J,k)} \ar[ddl]^{\mi^{F'_*(g,J,k)}}& \m\ar@{=}[l] \\
		&\\
		& & & & & \check \m^{F'_*(g,J,k)} \\
		&\\
		\m^{g,J,j}
		\ar[rr]^{\mi^{g,J,j}}
		\ar[rrrrrruuuu]^{\mC^{F'}}\ar[rrrrrrdddd]_{\mC^F} & & \check \m^{J,j} \ar@{}[rrr]|{\check \M^{\JJ}} \ar[ddrrr]_{\check \mC^F}\ar[uurrr]^{\check \mC^{F'}} & & & & \\
		& \\
		& & & & & \check \m^{F_*(J,k)} \ar[uuuu]_{\check \mC^{\mF_*( \hat h,\hat J)}}\\
		&\\
		& & & & & & \m^{F_*(g,J,k)} \ar[uul]_{\mi^{F_*(g,J,k)}} \ar[uuuuuuuu]_{\mC^{\mF_*( \hat g,\hat J)} \ \ \simud \ \id } & \m\ar@{=}[l]
	}
}
\]

\begin{lem}
	\label{choice_indep_main-lem}
	$\mC^{ F}\simud \mC^{ F'}$.
\end{lem}

\begin{proof}
	The proof is given by chasing the diagram:
	\begin{align*}
	\mi^{F'_*(g,J,k)} \diamond \mC^{F'} 
	&
	\simud \check \mC^{F'}\diamond \mi^{g,J,j} 
	\tag{use (\ref{Choice_indep_2-eq})} \\
	&
	\simud \check \mC^{\mF_*(\hat g,\hat J)} \diamond \check \mC^F \diamond \mi^{g,J,j} 
	\tag{use Lemma \ref{simud_mC_mC-lem}} \\
	&
	\simud \check \mC^{\mF_*(\hat g,\hat J)} \diamond  \mi^{F_*(g,J,k)} \diamond \mC^F
	\tag{use (\ref{Choice_indep_1-eq})}\\
	&
	\simud
	\mi^{F'_*(g,J,k)}\diamond \mC^{\mF_*(\hat g,\hat J)}\diamond \mC^F \tag{use (\ref{Choice_indep_3-eq})}
	\end{align*}
	Using Theorem \ref{Whitehead-full-thm}, we conclude $\mC^{F'}\simud\mC^{\mF_*(\hat g,\hat J)}\diamond \mC^F$. By (\ref{mC_I_simud_id-eq}), we finally show that $\mC^F\simud\mC^{F'}$.
\end{proof}

\subsubsection{Choice-independence's proof}

\begin{proof}[Proof of Theorem \ref{phi_jk-indep-thm}]
Let $\phi_{jk}^F$ and $\phi_{jk}^{F'}$ be the homomorphisms defined as in (\ref{phi_jk-eq}) with respect to the different choices.
Note that $F_*\alpha=F'_*\alpha$ for all $\alpha\in\pi_1(L_{k})$ since $F$ is isotopic to $F'$. By the definition formulas (\ref{phi_jk_defn-eq}), we only need to study the difference of the series in the exponents
\[
\sum_{\beta\neq 0} (\mC^{F'}_{0,\beta}-\mC^F_{0,\beta}) \ T^{E(\beta)} Y^{\partial \beta}
\]
called the \textit{error term}.
Due to Corollary \ref{UD_homotopy_summary_cor},
the ud-homotopy condition $\mC^F\simud \mC^{F'}$ in Lemma \ref{choice_indep_main-lem} can be unpacked to 
the existence of operators $(\f_s)_{s\in\oi}$ and $(\h_s)_{s\in\oi}$
with the following conditions:

\begin{enumerate}[(a)]
\item $\f_s\in\Hom_\UD(\m^{g,J,j},\m)$, where $\f_0=\mC^F$ and $\f_1=\mC^{F'}$;
\item $\frac{d}{ds}\circ \f_s=\sum \h_s\circ (\id^\bullet_\#\otimes \m^{g,J,j}\otimes \id^\bullet) +\sum \m\circ (\f_s^\#\cdots\f^\#_s, \h_s, \f_s\cdots \f_s)$;
\item $\h_s$ satisfies the divisor axiom, the cyclical unitality, and the property $\h_s(\cdots\one\cdots)=0$;
\item $\deg (\f_s)_{k,\beta}=1-k-\mu(\beta)$, $\deg (\h_s)_{k,\beta}=-k-\mu(\beta)$, and $(\h_s)_{\beta}\neq 0$ only if $\mu(\beta)\ge 0$
\end{enumerate}
Recall that the target $A_\infty$ algebra $\m$ is given in (\ref{m_F_F'_eq}).
Also, recall that the nonzero terms $\mC^F_{0,\beta},\mC^{F'}_{0,\beta}$ must live in $H^{1}(L_j)$.
So, we fix $\gamma \in \pi_1(L_j)$ and consider the $\gamma$-component of the error term:
\[
S(Y):=S^\gamma(Y):= \sum_{\beta\neq 0} \langle\gamma, \mC_{0,\beta}^{F'}-\mC_{0,\beta}^F\rangle \ T^{E(\beta)} Y^{\partial \beta}
\]
Choose a basis $\{\gamma_k\}$ of $\pi_1(L_j)$. Denote by $\{\theta_k\}$ its dual basis of $H^1(L_j)$.
Then, we have $\Lambda[[\pi_1(L_j)]]\cong \Lambda[[Y_1^\pm,\dots, Y_n^\pm]]$;
we may write
$
S(Y)=S(Y_1,\dots,Y_n)=\sum_{\beta\neq 0} \langle\gamma, \mC_{0,\beta}^{F'}-\mC_{0,\beta}^F\rangle \ T^{E(\beta)} Y_1^{\partial_1\beta}\cdots Y_n^{\partial_n\beta}
$
where $\partial_i\beta\in \mathbb Z$ and $\partial\beta =\partial_1\beta \cdot \gamma_1+\cdots +\partial_n\beta \cdot \gamma_n$.

We will apply Lemma \ref{val=0-lem} again.
Given $\mathbf y=(y_1,\dots,y_n)\in U_\Lambda^n$, there exists $x_i\in \Lambda_0$ so that $y_i=e^{x_i}$ for each $i$ due to Lemma \ref{exp-log-lem}.
Put $b=x_1\theta_1+\cdots +x_n\theta_n\in H^1(L_j)\hat\otimes \Lambda$, and then $e^{\partial\beta \cap b}=e^{\partial_1\beta\cdot x_1+\cdots +\partial_n\beta\cdot x_n}=\mathbf y^{\partial\beta}$.
Thus, the divisor axioms of $\mC^F$ and $\mC^{F'}$ infer that
\begin{align*}
S(\mathbf y)
=
\sum_{k,\beta}
T^{E(\beta)} 
\langle\gamma, \mC^{F'}_{k,\beta}(b,\dots,b)-\mC_{k,\beta}^F(b,\dots,b)\rangle
=
\langle\gamma ,\mC^{F'}_*(b)-\mC^F_*(b) \rangle
\equiv 
\langle \gamma, \f_{1*}(b)-\f_{0*}(b) \rangle	
\end{align*}
However, from a different point of view, using the condition (b) yields the following computation:
\begin{align*}
\f_{1*}(b)-\f_{0*}(b)
&
\textstyle
=
\sum  T^{E(\beta)} \int_0^1 ds\cdot \frac{d}{ds}\circ (\f_s)_{k,\beta}(b,\dots,b)\\
&
\textstyle
=
\sum
 T^{E(\beta_1)} \int_0^1 ds\cdot (\h_s)_{\lambda+\mu+1,\beta_1} (b,\dots,b,  T^{E(\beta_2)}\m^{g,J,j}_{\nu,\beta_2}(b,\dots,b),b,\dots,b) \\
&  
\textstyle 
+ \sum  \int_0^1 ds\cdot \m \big((\f_s)(b,\dots,b),\dots, (\h_s)_{\ell,\beta_0}(b,\dots,b), \dots, (\f_s)(b,\dots,b) \big)
\end{align*}
The condition (d) implies $\deg (\h_s)_{\ell,\beta_0}(b,\dots,b)=-\mu(\beta_0) \le 0$, so we may assume this degree always equals to zero. Then, by the cyclical unitality of $\m$, the second summation above vanishes. Moreover, the operator
$
\textstyle
\mathfrak H:=\int_0^1 ds \cdot \h_s
$
also satisfies the divisor axiom since so does every $\h_s$.
Hence,
\begin{align*}
\f_{1*}(b)-\f_{0*}(b)
=
&
\textstyle
\sum T^{E(\beta_1)} \mathfrak H_{k_1,\beta_1}\big( b,\dots,b, T^{E(\beta_2)}\m_{k_2,\beta_2}^{g,J,j}(b,\dots,b),b,\dots,b
\big)
\\
=
&
\textstyle
\sum T^{E(\beta_1)} e^{\langle \partial\beta_1, b \rangle} \ 
\mathfrak H_{1,\beta_1}
\big(
T^{E(\beta_2)} e^{\langle\partial\beta_2, b \rangle} \m_{0,\beta_2}^{g,J,j}
\big)
=
\sum T^{E(\beta)} \mathbf y^{\partial\beta} \mathfrak H_{1,\beta}(P^{g,J,j}(\mathbf y))
\end{align*}
Recall that the basis $\{\theta_i\}$ of $H^1(L_j)$ induces a basis of $H^2(L_j)$ given by $\theta_{pq}:=\theta_p\wedge \theta_q$ for $1\le p< q\le n$.
Recall also that $P^{g,J,j}=W^{g,J,j}\one_j+\sum_{p<q} Q_{pq}^{g,J,j} \theta_{pq}$. Putting things together, we have
\begin{align*}
\textstyle
S(\mathbf y)=\langle \gamma, \f_{1*}(b)-\f_{0*}(b) \rangle
=
S_0(\mathbf y) \cdot W^{g,J,j}(\mathbf y) + \sum_{p<q} Q_{pq}^{g,J,j}(\mathbf y) \cdot S_{pq}(\mathbf y) 
\end{align*}
where we denote
\begin{align*}
S_0:=S_0^\gamma:=
&
 \textstyle \sum_\beta T^{E(\beta)} Y^{\partial \beta} 
 \langle \gamma, \mathfrak H_{1,\beta}(\one_j)\rangle \\
S_{pq}:=S_{pq}^\gamma:=
&
\textstyle
\sum_\beta T^{E(\beta)} Y^{\partial \beta} \langle \gamma, \mathfrak H_{1,\beta}(\theta_{pq}) \rangle
\end{align*}
But as $\deg \mathfrak H_{1,\beta}=-1-\mu(\beta)<0$, we know $\mathfrak H_{1,\beta}(\one_j)=0$ and so $S_0=0$. Consequently,
\begin{equation}
\label{S(Y)_for_mC_FF'-eq}
\textstyle 
S(\mathbf y) = \sum_{p<q} S_{pq}(\mathbf y) \cdot Q_{pq}^{g,J,j}(\mathbf y)
\end{equation}
holds for any arbitrary point $\mathbf y$ in $U_\Lambda^n$. And, it actually holds everywhere thanks to Lemma \ref{val=0-lem}.
Thus, the $\gamma$-component $S(Y)=S^\gamma(Y)=\sum_{\beta\neq 0} \langle\gamma, \mC_{0,\beta}^{F'}-\mC_{0,\beta}^F\rangle \ T^{E(\beta)} Y^{\partial \beta}$ of the error term is contained in $ \ia_j$ for any $\gamma\in\pi_1(L_j)$.
Finally, due to (\ref{phi_jk_defn-eq}), we have
\begin{align*}
\phi^{F'}_{jk} (sY^\alpha) 
&
=
\textstyle
\phi_{jk}^F(sY^{\alpha})  \cdot \exp \Big( \sum  \langle F_*\alpha, \mC_{0,\beta}^{F'}-\mC_{0,\beta}^F\rangle T^{E(\beta)} Y^{\partial\beta} 
\Big)   \\
&
=
\phi_{jk}^F(sY^{\alpha}) \cdot \exp( S^{F_*\alpha}(Y))   \\
&
=
\phi_{jk}^F(sY^{\alpha}) \cdot 
\Scale[0.9]{
	\big(
	1+S^{F_*\alpha}(Y) + \frac{S^{F_*\alpha}(Y)^2}{2!}+ \cdots
	\big)}
=
\phi^F_{jk}(sY^\alpha)+ (\text{something in $\ia_j$})
\end{align*}
Hence, the $\phi_{jk}^F$ and $\phi_{jk}^{F'}$ actually induce the same quotient homomorphism $\varphi_{jk}: A_{kj}\to A_{jk}$ in (\ref{varphi_jk_eq}).
\end{proof}

\subsection{Cocycle conditions}

The last step for the mirror construction is to show the cocycle conditions among the various transition maps $\psi_{ij}$ (\ref{psi_jk-eq}).
The idea of the proof is very similar to that of Theorem \ref{phi_jk-indep-thm}.
Let $\Delta_i$, $\Delta_j$ and $\Delta_k$ be three adjacent polyhedrons.
We have defined the local charts
$X_i$, $X_j$, $X_k$ (\ref{X_i-eq}) and the transition maps $\psi_{ik}$, $\psi_{jk}$, $\psi_{ij}$ (\ref{psi_jk-eq}) which correspond to the quotient algebra homomorphisms $\varphi_{ik}$, $\varphi_{jk}$, $\varphi_{ij}$ (\ref{varphi_jk_eq}).

\begin{thm}
\label{cocycle-condition-thm}
$
\psi_{ik}=\psi_{jk}\circ \psi_{ij}$ or equivalently
$
\varphi_{ik}=\varphi_{ij}\circ \varphi_{jk}
$.
\end{thm}

We begin with some preparations.
Due to Lemma \ref{Delta_Lebesgue_lem}, 
we can pick up $W_q$ for some $q\in B_0$ such that $W_q\supset \Delta_i\cup \Delta_j\cup \Delta_k$.
Using the construction (\ref{mF_u_v1_v2-eq}) produces several specific diffeomorphisms $F_{ij}:=F_q^{q_j,q_i}$, $F_{jk}:=F_q^{q_k,q_j}$, and $F_{ik}:=F_q^{q_k,q_i}$ in the neighborhood $\mathcal U$.
Observe that for $a,b\in\{i,j,k\}$, we have $F_{ab}(L_b)=L_a$ and $F_{ij}\circ F_{jk}=F_{ik}$.
By Theorem \ref{phi_jk-indep-thm}, the ambiguities caused by taking different choices have been eliminated, and so we can use these specific choices for the computation.

Just as \S \ref{sss_A_homotopy_equivalence_local_chart}, we make choices 
$
(\mJ_{ij},\mg_{ij})$, $(\mJ_{jk},\mg_{jk})$, and
$(\mJ_{ik},\mg_{ik})$: for $a,b\in\{i,j,k\}$, the $\mJ_{ab}$ is a path of almost complex structures in $\mathcal V$ from $J$ to $F_{ab*}J$, and the $\mg_{ab}$ is a path of metrics from $g$ to $F_{ab*}g$.
In the same way as (\ref{mC_mF,j-eq}),
these data produce three $A_\infty$ homomorphisms in $\UD$:
\begin{align*}
\mC^{F_{ij}}: \m^{g,J,i} &\to \m^{F_{ij*}(g,J,j)} \\
\mC^{F_{jk}}: \m^{g,J,j} &\to \m^{F_{jk*}(g,J,k)} \\
\mC^{F_{ik}}: \m^{g,J,i} &\to \m^{F_{ik*}(g,J,k)} 
\end{align*}
Using them, we can similarly construct the algebra homomorphisms $\phi^{F_{ik}}_{ik}$, $\phi_{ij}^{F_{ij}}$, and $\phi_{jk}^{F_{jk}}$ as in (\ref{phi_jk-eq}).

Next, we aim to compare $\phi_{ik}^{F_{ik}}$ with $\phi_{ij}^{F_{ij}}\circ \phi_{jk}^{F_{jk}}$.
Heuristically, let us think of the indexes $i$ and $k$ as the `source' and `target', while the index $j$ is thought of as the `bridge'.
A subtle point is that 
the target of $\mC^{F_{ij}}$ does not match the source of $\mC^{F_{jk}}$,
but this issue is inessential thanks to Fukaya's trick.
In reality, applying Proposition \ref{Fukaya_II_F_s=F_mC_prop} to the constant family $\mF=(F_s)$ where $F_s=F_{ij}$ for each $s$, we obtain a push-forward $A_\infty$ homomorphism
\[
\tilde \mC^{F_{jk}}:= \mC^{\mF_*(\mg_{jk},\mJ_{jk})} : \m^{F_{ij*}(g,J,j)} \to \m^{F_{ij*}F_{jk*}(g,J,k)}\equiv \m^{F_{ik*}(g,J,k)}
\]
such that $H^{F_{ij}}\circ \tilde \mC^{F_{jk}} = \mC^{F_{jk}} \circ H^{F_{ij}}$; see around (\ref{TrickA_mi_can_eq}) for the notation $H^{F_{ij}}$.

\[
\xymatrix{
	&
	& \m^{g,J,j} \ar@{.>}[d]_{H^{F_{ij}}} \ar@{.>}[r]^{\mC^{F_{jk}}} & \m^{F_{jk*}(g,J,k)} \ar@{.>}[d]^{H^{F_{ij}}}
	& & & \text{on} \ H^*(L_j) \\
	\m^{g,J,i}\ar[rr]^{\mC^{F_{ij}}} \ar@/_2pc/[rrr]^{\mC^{F_{ik}}} &
	& \m^{F_{ij*}(g,J,j)} \ar[r]^{\tilde \mC^{F_{jk}}} & *+[r]{\m^{F_{ij*}F_{jk*}(g,J,k)} = \m^{F_{ik*}(g,J,k)}} &  &  &	\text{on} \ H^*(L_i)
}
\]

\begin{lem}\label{mCmC_simud_mC_FFF_lem}
	$\mC^{F_{ik}}$ is ud-homotopic to $\tilde \mC^{F_{jk}}\circ \mC^{F_{ij}}$.
\end{lem}

\begin{proof}[Sketch of proof]
The proof is almost the same as Lemma \ref{choice_indep_main-lem}. Write $\mJ_{ab}=(J_{ab}^s)_{s\in\oi}$ and $\mg_{ab}=(g_{ab}^s)_{s\in\oi}$ for $a,b\in\{i,j,k\}$.
Let $\check \M^{\mJ_{ij}}$, $\check \M^{\mJ_{ik}}$ and $\check \M^{\mF_*\mJ_{jk}}$ be the underlying chain-level pseudo-isotopies (Theorem \ref{Vir-M-path-thm}).
We can extend them to a $\Delta^2$-pseudo-isotopy as in (\ref{check_M_mJ_mJ_eq}).
Just like Lemma \ref{simud_mC_mC-lem}, we can similarly prove $\check {\tilde \mC}^{F_{jk}}\diamond \check \mC^{F_{ij}}\simud \check \mC^{F_{ik}}$ for the $A_\infty$ homomorphisms associated to the three chain-level pseudo-isotopies. 
Next, for the cohomology-level, utilizing Lemma \ref{mC-diamond-diagram-lem} and chasing diagrams like the proof of Lemma \ref{choice_indep_main-lem}, we can similarly show $\tilde \mC^{F_{jk}}\diamond \mC^{F_{ij}} \simud \mC^{F_{ik}}$.
\end{proof}

\begin{proof}[Proof of Theorem \ref{cocycle-condition-thm}]
There are natural isomorphisms induced by $F_{ij}$, $F_{jk}$ and $F_{ik}$ among the label groups $\G(X,L_i)$, $\G(X,L_j)$, and $\G(X,L_k)$ (\ref{G(X,L)_F_Fuk_trick_eq}). Their elements are denoted by $\beta$, $\beta'$ and $\beta''$ in sequence. We will always follow this convention in the below. For instance, given $\alpha''\in \pi_1(L_k)$ we will set $\alpha=F_{ik*}\alpha''\in \pi_1(L_i)$ and $\alpha'=F_{jk*}\alpha'' \in \pi_1(L_j)$. Note that $F_{ij*}\alpha'=\alpha$.

To start with, we specify the bases as before. Let $\{f''_\ell\}_{1\le \ell \le n}$ be a basis of $\pi_1(L_k)$ which can induce a basis on $\pi_1(L_i)$ (resp. $\pi(L_j)$) given by $f_\ell:=F_{ik*}f''_\ell$ (resp. $f_\ell':=F_{jk*}f''_\ell$).
Denote the dual bases by $\{\theta_\ell\}$, $\{\theta'_\ell\}$, and $\{\theta''_\ell\}$ respectively.
Since the bases are related by the diffeomorphisms, the above $\alpha$, $\alpha'$, and $\alpha''$ can be all identified with the same tuple $(\alpha_1,\dots,\alpha_n)\in\mathbb Z^n$ such that $\alpha=\sum \alpha_\ell f_\ell$, $\alpha'=\sum \alpha_\ell f'_\ell$, and $\alpha''=\sum \alpha_\ell f''_\ell$.
Thus, each of $Y^{\alpha}$, $Y^{\alpha'}$ or $Y^{\alpha''}$ can be viewed as $Y_1^{\alpha_1}\cdots Y_n^{\alpha_n}$, thereby inducing the natural identifications
$\Lambda[[\pi_1(L_a)]]\cong\Lambda[[Y_1^\pm,\dots, Y_n^\pm]]$ for any $a\in\{i,j,k\}$.
Particularly, we can regard $\partial\beta$, $\partial \beta'$ or $\partial\beta''$ as the same tuple $(\partial_1 \beta,\dots, \partial_n \beta)\in\mathbb Z^n$ where $\partial_i\beta= \langle \partial\beta, \theta_i\rangle=\langle \partial\beta', \theta_i'\rangle =\langle \partial\beta'', \theta_i''\rangle$; we can also identify $Y^{\partial\beta}$, $Y^{\partial\beta'}$, or $Y^{\partial\beta''}$ with the same monomial $Y_1^{\partial_1\beta}\cdots Y_n^{\partial_n \beta}$.

Using the energy formula (\ref{energy-change-eq}) and the definition formulas (\ref{phi_jk_defn-eq}), we first compute as follows:
\begin{align*}
& \quad \phi_{ij}^{F_{ij}}\circ \phi_{jk}^{F_{jk}} (Y^{\alpha''})\\
&
=
\textstyle
	\phi_{ij}^{F_{ij}} 
\Big(
T^{\langle \alpha'', q_j-q_k\rangle}
Y^{\alpha'}
\exp
 \langle \alpha', 
 \sum \mC_{0,\beta'}^{F_{jk}}  T^{E(\beta')} Y^{\partial \beta'}\rangle
\Big) \\
&
=
\textstyle
T^{\langle \alpha'', q_j-q_k\rangle}
\phi_{ij}^{F_{ij}}(Y^{\alpha'} ) 
\exp
\Big(
\sum  \langle \alpha', \mC_{0,\beta'}^{F_{jk}} \rangle  T^{E(\beta')} \phi_{ij}^{F_{ij}} (Y^{\partial \beta'})
\Big)
\\
&
=
\textstyle 
T^{\langle \alpha'', q_j-q_k\rangle} T^{\langle \alpha', q_i-q_j \rangle} Y^\alpha 
\exp
 \langle \alpha, \sum  \mC_{0,\gamma}^{F_{ij}}  T^{E(\gamma)}Y^{\partial \gamma}
 \rangle 
\exp
\Big(
\sum
\langle\alpha',  \mC_{0,\beta'}^{F_{jk}} \rangle T^{E(\beta')}
T^{\langle \partial\beta', q_i-q_j \rangle}
Y^{\partial\beta} 
\exp
 \langle
\partial\beta, \sum \mC_{0,\eta}^{F_{ij}}
T^{E(\eta)}Y^{\partial \eta}
\rangle
\Big)   \\
&
=
\textstyle
T^{\langle \alpha'', q_i-q_k\rangle} Y^\alpha 
\exp 
 \langle \alpha,
 \sum \mC_{0,\gamma}^{F_{ij}}  T^{E(\gamma)}Y^{\partial \gamma} \rangle   
\exp
\Big(
\sum
\langle\alpha, \tilde \mC_{0,\beta}^{F_{jk}} \rangle T^{E(\beta)}
Y^{\partial\beta}
\exp
\langle
\partial\beta, \sum \mC_{0,\eta}^{F_{ij}}
T^{E(\eta)}Y^{\partial \eta}
\rangle
\Big)
\end{align*}
where in the last step we use Proposition \ref{Fukaya_II_F_s=F_mC_prop} to get $\mC_{0,\beta'}^{F_{jk}}=F_{ij}^*\tilde \mC_{0,\beta}^{F_{jk}}$ and so $\langle \alpha', \mC_{0,\beta'}^{F_{jk}}\rangle =\langle \alpha, \tilde \mC_{0,\beta}^{F_{jk}}\rangle$.

Once again, we will take advantage of Lemma
\ref{val=0-lem} for the computations.
Assume $\mathbf y=(y_1,\dots, y_n)$ is an arbitrary point in $U_\Lambda^n$, and we can find $x_\ell \in\Lambda_0$ such that $y_\ell=e^{x_\ell}$ for each $1\le \ell \le n$ due to Lemma \ref{exp-log-lem}.
We set $b=\sum x_\ell \theta_\ell$, $b'=\sum x_\ell \theta'_\ell$, and $b''=\sum x_\ell  \theta''_\ell$.
For a general monomial $Y^{\alpha}$, the evaluation at $\mathbf y$ gives the value $\mathbf y^\alpha=\exp(\alpha\cap b)=\exp(\alpha'\cap b')=\exp(\alpha''\cap b'')=e^{\alpha_1  x_1+\cdots+ \alpha_n  x_n}$ in $\Lambda$.

After the substitution $Y=\mathbf y$, applying the divisor axiom to the first exponential power, we obtain
\[
\textstyle
\exp 
\langle \alpha,
\sum \mC_{0,\gamma}^{F_{ij}}  T^{E(\gamma)}Y^{\partial \gamma} \rangle  
 |_{Y=\mathbf y}
=
\sum  \langle\alpha, \mC_{0,\gamma}^{F_{ij}}\rangle T^{E(\gamma)}e^{\partial\gamma\cap b}=\langle \alpha, \mC^{F_{ij}}_*(b)-b\rangle
\]
Recall (\S \ref{ss_weak_MC}) for the notation $\mC_*^{F_{ij}}$.
The similar holds for the last power replacing $\alpha$ by $\partial \beta$.
Thus,
\begin{align*}
\phi_{ij}^{F_{ij}}\circ \phi_{jk}^{F_{jk}}(Y^{\alpha''})|_{Y=\mathbf y}
=&
\textstyle
T^{\langle \alpha'', q_i-q_k\rangle}
\cdot
\mathbf y^\alpha
\cdot
\exp  
\langle\alpha,\mC_*^{F_{ij}}(b)-b\rangle
\cdot
\exp
\Big(
\sum \langle\alpha,\tilde \mC_{0,\beta}^{F_{jk}} \rangle T^{E(\beta)}\mathbf y^{\partial\beta} \exp \langle\partial\beta, \mC^{F_{ij}}_*(b)-b
\rangle
\Big)  \\
=&
\textstyle
T^{\langle \alpha'', q_i-q_k\rangle}
\cdot \exp 
\langle\alpha, \mC_*^{F_{ij}}(b)\rangle
\cdot
\exp
\Big(
\sum \langle\alpha, \tilde \mC_{0,\beta}^{F_{jk}}\rangle
T^{E(\beta)}\exp \langle \partial\beta, \mC^{F_{ij}}_*(b)\rangle
\Big)   \\
=
&
\textstyle
T^{\langle \alpha'', q_i-q_k\rangle}
\cdot \exp \langle \alpha, \widehat b\rangle 
\cdot
\exp
\Big(
\sum \langle\alpha, \tilde \mC_{0,\beta}^{F_{jk}}\rangle
T^{E(\beta)}
\exp \langle \partial\beta , \widehat b \rangle 
\Big)
\end{align*}
where we put
$
\widehat b:=\mC_*^{F_{ij}}(b)\in H^1(L_i)\hat\otimes \Lambda_0
$.
Viewing $\widehat b$ as a new divisor input (\ref{Divisor_input-eq}), we similarly gain
\[
\textstyle
\sum \langle\alpha, \tilde \mC_{0,\beta}^{F_{jk}}\rangle
T^{E(\beta)}
\exp \langle \partial\beta , \widehat b \rangle =\langle \alpha, \tilde \mC^{F_{jk}}_*(\widehat b)-\widehat b\rangle
\]
Besides, it is easy to check
$\tilde \mC_*^{F_{jk}}(\widehat b)=\tilde \mC^{F_{jk}}_* \big( \mC_*^{F_{ij}} (b) \big)=
(\tilde \mC^{F_{jk}} \diamond \mC^{F_{ij}})_*(b)
$ from the definition. Therefore,
\begin{align*}
\phi_{ij}^{F_{ij}}\circ \phi_{jk}^{F_{jk}}(Y^{\alpha''})|_{Y=\mathbf y}
=
T^{\langle \alpha'', q_i-q_k\rangle}  
\exp \langle \alpha, \widehat b\rangle \cdot
\exp \langle\alpha, \tilde \mC_*^{F_{jk}}(\widehat b) - \widehat b \rangle  
=
T^{\langle \alpha'', q_i-q_k\rangle}  \exp \langle\alpha,  ( \tilde\mC^{F_{jk}} \circ \mC^{F_{ij}})_*(b) \rangle
\end{align*}
Just like how we use Lemma \ref{choice_indep_main-lem} to show (\ref{S(Y)_for_mC_FF'-eq}) before, we can similarly use Lemma \ref{mCmC_simud_mC_FFF_lem}
to show
\begin{align*}
\textstyle
 \langle\alpha,  ( \tilde\mC^{F_{jk}} \diamond \mC^{F_{ij}})_*(b) \rangle
 -
  \langle\alpha,  \mC^{F_{ik}}_*(b) \rangle
 = \sum_{p<q} S_{pq}^{\alpha}(\mathbf y) \cdot Q_{pq}^{g,J,i}(\mathbf y)
\end{align*}
for some formal series $S_{pq}^\alpha$ depending on $\alpha$. 
Accordingly, given an arbitrary point $\mathbf y$ in $U_\Lambda^n$,
we have
\begin{align*}
\phi_{ij}^{F_{ij}}\circ \phi_{jk}^{F_{jk}}(Y^{\alpha''})|_{Y=\mathbf y}
=&
\Scale[0.9]
{T^{\langle \alpha'', q_i-q_k\rangle}  \exp \langle\alpha, \mC_*^{F_{ik}}(b)\rangle \cdot \exp \Big[ \textstyle \sum_{p<q} S_{pq}^{\alpha}(\mathbf y) \cdot Q_{pq}^{g,J,i}(\mathbf y) \Big]}\\
=&
\Scale[0.9]
{T^{\langle \alpha'', q_i-q_k\rangle}
\mathbf y^\alpha \exp \langle
\alpha, \textstyle \sum_{\beta} \mC_{0,\beta}^{F_{ik}}T^{E(\beta)} 
\mathbf y^{\partial\beta}
\rangle\cdot \exp \Big[ \sum_{p<q} S_{pq}^{\alpha}(\mathbf y) \cdot Q_{pq}^{g,J,i}(\mathbf y) \Big] }
\\
=&
\Scale[0.9]
{
\phi_{ik}^{F_{ik}}( \mathbf y^{\alpha''})\cdot \exp \Big[ \textstyle \sum_{p<q} S_{pq}^{\alpha}( \mathbf y ) \cdot Q_{pq}^{g,J,i}( \mathbf y) \Big] 
}
\end{align*}
where the second equation uses the divisor axiom of $\mC^{F_{ik}}$ in reverse, and the last equation just follows from the definition.
By Lemma \ref{val=0-lem} again, we actually know the above equation holds everywhere, and so $\phi_{ij}^{F_{ij}}\circ \phi_{jk}^{F_{jk}}(Y^{\alpha''})=\phi_{ik}^{F_{ik}}(Y^{\alpha''})+ (\text{something in $\ia_i$})$. Therefore, $\varphi_{ij}\circ \varphi_{jk} =\varphi_{ik}$.
\end{proof}

\begin{cor}\label{cocycle-condition_ij_cor}
$\psi_{ji}\circ \psi_{ij}=\id$ and $\psi_{ii}=\id$.
\end{cor}

\begin{proof}
It is straightforward from Theorem \ref{cocycle-condition-thm}.
In fact, we just need to set $k=i$, $F_{ii}=\id$ and set $\mathbf J_{ii}, \mathbf g_{ii}$ to be the constant families. By construction, we can easily show that $\mC^{F_{ii}}=\mC^\id=\id$. From the definition formulas (\ref{phi_jk_defn-eq}, \ref{phi_jk-eq}), it follows that $\phi_{ii}^{F_{ii}}=\id$ and thus $\psi_{ji}\circ \psi_{ij}=\psi_{ii}=\id$.
\end{proof}

\begin{proof}
[Proof of Main Theorem \ref{Main_theorem_thm}]
Note that the above construction starts from a chosen compact domain $K\subset B_0$ (\ref{K_domain_eq}) together with the related choices, such like $J\in\mathfrak J_K$.
By Theorem \ref{cocycle-condition-thm}, Corollary \ref{cocycle-condition_ij_cor}, and Proposition \ref{gluing X_ij -prop-appendix}, we can glue $X_i$ (\S \ref{sss_local_charts_defn}) along $X_{ij}$ (\S \ref{sss_A_homotopy_equivalence_local_chart}) through $\psi_{ij}$ (\ref{psi_jk-eq}), thereby obtaining a rigid analytic space $X^\vee_{J,K}$ such that the collection $(X_i)$ gives an admissible covering.
Moreover, by Theorem \ref{phi_jk-A-thm}, the various $W^{g,J,i}$ on $X_i$ are compatible with the gluing maps $\psi_{ij}$ and give rise to a global function $W^\vee_{J,K}$ on $X^\vee_{J,K}$.
Finally, by Corollary \ref{Val_psi_transition-compatible-cor}, one can also glue the various $\Val_{q_i}$ (\ref{Val_q-eq}) to obtain a map $\pi^\vee_{J,K} : X^\vee_{J,K} \to K$. 

When $K$ is fixed, we claim that the isomorphism class of triple $\mathbb X^\vee_{J,K}:=(X_{J,K}^\vee, W_{J,K}^\vee, \pi^\vee_{J,K})$ does not depend on $J$. Indeed, suppose $\tilde J\in\mathfrak J_K$ is another choice. First, using a path between $J$ and $\tilde J$ inside $\mathfrak J_K$ and considering the induced pseudo-isotopies, one can construct local isomorphisms between the local pieces just like how we previously construct the transition maps.
Then, by the same method of showing the cocycle conditions, they can be glued together to obtain a global isomorphism $\mathbb X^\vee_{J,K} \cong \mathbb X^\vee_{\tilde J, K}$.
On the other hand, if $\tilde K$ is another compact domain, we may assume $\tilde K\supset K$; the same construction yields a natural open embedding $\mathbb X^\vee_{J,K}\subset \mathbb X^\vee_{J,\tilde K}$.

In general, we choose a sequence of compact domains $K_1\subset \cdots \subset K_n\subset K_{n+1} \subset \cdots$ such that $\bigcup_{n\ge 1} K_n =B_0$.
Then, we also have a sequence $\mathfrak J_{K_1} \supset \cdots \supset \mathfrak J_{K_n}\supset \mathfrak J_{K_{n+1}} \supset \cdots$. By Assumption \ref{assumption-mu ge 0}, each $\mathfrak J_{K_n}$ is open. Fix some $J_n\in\mathfrak J_{K_n}$ and a sufficiently small neighborhood $\mathcal  V_n$ of $J_n$ in $\mathfrak J_{K_n}$, and then we choose some $\tilde J_n\in\mathcal V_n\cap \mathcal V_{n+1}$ to serve as a bridge. By the above arguments, we can obtain an open embedding $\mathbb X^\vee_{J_n,K_n}\xhookrightarrow{} \mathbb X^\vee_{J_{n+1}, K_{n+1}}$.
Ultimately, applying Proposition \ref{gluing X_ij -prop-appendix} again to the increasing sequence $(\mathbb X^\vee_{J_n,K_n})$, we get the mirror $\mathbb X^\vee=(X^\vee, W^\vee, \pi^\vee)$ consisting of an analytic space with a fibration $\pi^\vee: X^\vee \to B_0$ and a global function $W^\vee$.
\end{proof}

\appendix

\section{Appendix}

\subsection{General aspects of non-archimedean geometry}
\label{SA_non-archimedean}

We briefly cover the fundamentals of non-archimedean analytic geometry.
A partial list of references are \cite{BoschBook, BGR, EKL,Gubler}.

Let $\K$ be an algebraically closed field. 
A norm $|\cdot|$ on $\K$ is called \textit{non-archimedean} if $|a+b|\le \max\{|a|,|b|\}$ for any $a,b\in\K$. It is equivalent to a valuation $\val: \K \to \mathbb R \cup \{\infty\}$ satisfying (1) $\val(a)=0$ if and only if $a=0$; (2) $\val(ab)=\val(a) + \val(b)$; (3) $\val(a+b)\ge \min\{\val(a),\val(b)\}$. They are related to each other by $\val(a)=-\log|a|$ and $|a|=e^{-\val(a)}$. 
For the purpose of this paper, we always choose $\K$ to be the Novikov field $\Lambda=\mathbb C((T^{\mathbb R}))$ as (\ref{Novikov_eq}).
Let $T_d= \K \langle z_1,\dots,z_d\rangle \subset \K[[z_1,\dots,z_d]]$ be the set of all formal power series $\sum_{\nu\in\mathbb Z^d_{\ge 0}}a_{\nu} \mathbf{z}^\nu$ so that $|a_{\nu}| \to 0$ as $|\nu| = \sum |\nu_i|\to \infty$. It form a Banach $\K$-algebra, called the $d$-th \textit{Tate algebra}.

An \textit{affinoid algebra} is defined to be a $\K$-Banach algebra $A$ admitting a continuous epimorphism $T_d\to A$ for some $d$. 
The spectrum of maximal ideals, $\Sp A := \mathrm{Max}A$, is called an \textit{affinoid $\K$-space} or an \textit{affinoid space}.

This falls within classical Tate's rigid analytic geometry \cite{Tate_origin}. In the context of modernized Berkovich's analytic geometry \cite{Berkovich_2012spectral}, an affinoid space actually refers to the spectrum of multiplicative seminorms on $A$, adding extra generic points (see e.g. \cite[1.3]{baker2008introduction}).
However, these theories are essentially equivalent when focusing on so-called \textit{good} Berkovich analytic spaces. There exists a category equivalence between the (sub)category of good Berkovich analytic spaces and the category of rigid analytic spaces, as shown in \cite[1.6]{Berkovich1993etale}.
To make our discussion more accessible to a broader audience, we will henceforth focus on rigid geometry.

An \textit{analytic space} over $\Lambda$ is a topological space with an atlas of affinoid domains glued by affinoid isomorphisms.
This resembles how a scheme is covered by an atlas of affine schemes or a complex manifold by an atlas of holomorphic charts. Thus, to construct an analytic space as outlined in Theorem \ref{Main_theorem_thm}, it's necessary to utilize affinoid algebra isomorphisms.
Note that points $x,y,\dots$ in $\Sp A$ correspond to maximal ideals $\m_x,\m_y,\dots$ in the algebra $A$ defined by $\m_x=\{ f\in A\mid f(x)=0\}$. Each $f\in A$ can be regarded as a \textit{function} on the space $\Sp A$ by setting $f(x)$ to be the residue class of $f$ in $A / \m_x\equiv \K$. (Indeed, $A/ \m_x$ is first a finite field over $\K$ due to \cite[2.2/12]{BoschBook} and must equal to $\K$, since $\K$ is algebraically closed.)
If we set $V(F)=\{ x\in \Sp A \mid f(x) =0, \ \forall f\in F\}$ and $\id(E) = \{ f \in A \mid f(x)=0, \  \forall x\in E\}$, then the Hilbert's Nullstellensatz \cite[3.2/4]{BoschBook} also holds in the sense that for an ideal $\ia\subset A$, we have $\id(V(\ia)) =\sqrt{\ia}$. In particular, several functions $f_i$ on $\Sp A$ has no common zeros if and only if the unit ideal is generated by $f_i$.
As a typical example, the affinoid space associated to the $n$-th Tate algebra is exactly the unit ball in $\K^n$
\[
B^n(\K) \xrightarrow{\cong} \Sp T_n, \ \ \ x\mapsto \m_x=\{f\in T_n\mid f(x)=0\}
\]
where $B^n(\K):=\{(x_1,\dots,x_n) \in \K^n \mid |x_i|\le 1\}$. 
The Tate algebra can be viewed as the ring of well-defined functions on the unit ball, meaning that a formal power series $f\in \K[[z_1,\dots, z_n]]$ belongs to $T_n$ if and only if $f$ converges on the unit ball $B^n(\K)$. By \cite[Proposition 3.1.8]{EKL}, if $A=T_d /(f_1,\dots,f_r)$ for $f_1,\dots, f_r\in T_d$, then $\Sp A$ agrees with $V(f_1,\dots,f_r)$ of $B^n(\K)$.

\begin{prop}[{\cite[5.3/5]{BoschBook}} {\cite[\S 4.1.4]{TemkinIntro}}]
\label{gluing X_ij -prop-appendix}
Consider the data (i) analytic spaces $X_i$, $i\in I$; (ii) open subspaces $X_{ij}\subset X_i$; (iii) morphisms $\psi_{ij}:X_{ij}\to X_{ji}$.
Suppose (a) $X_{ii}=X_i$, $\psi_{ij}\circ \psi_{ji} =\id$ and $\psi_{ii}=\id$; (b) $\psi_{ij}$ induces isomorphisms $\psi_{ij}: X_{ij}\cap X_{ik} \to X_{ji}\cap X_{jk}$ with cocycle conditions $\psi_{ik}= \psi_{jk}\circ \psi_{ij}$. Then these $X_i$ can be glued by identifying $X_{ij}$ with $X_{ji}$ to get an analytic space $X$ admitting $(X_i)_{i\in I}$.
\end{prop}

Every $\K$-scheme $X$ of locally finite type admits the \textit{analytification}, which is an analytic $\K$-space together with a morphism $X^{\an}\to X$, satisfying some universal properties. An instructive fact is that \textit{the map on the underlying sets identifies the points of $X^{\an}$ with the closed points of the scheme $X$}. In reality, we have the so-called GAGA-functor from the category of $\K$ schemes of locally finite type to the category of analytic $\K$-space (see \cite[Sec. 5.4]{BoschBook} or \cite[Sec.5.1]{TemkinIntro}).
For example,
fix $d>0$ and $s>1$. Pick the (scaled) Tate algebra $T_n^{(i)}:= \K\langle s^{-i} z_1, \dots,  s^{-i} z_n \rangle$ for some $i\in \mathbb N$ which consists of formal power series $\sum a^\nu \mathbf z^\nu$ with $|a_\nu| s^{i|\nu|} \to 0$.
As above, $\Sp T_n^{(i)}$ agrees with the ball with radius $s^i$.
The analytification of the affine $n$-space $\mathbb A_{\mathbb K}^n$ is given by $\mathbb A_{\mathbb K}^{n,\an}=\bigcup_{i=0}^\infty \Sp T_n^{(i)}$.

We next consider a tropical example. Let $\GG^n_m=\Spec \K[z_1^\pm,\dots, z_n^\pm]$. The points of the analytification $\GG_m^{n,\an}$ are in bijection with the closed points in $\GG_m^n$, i.e. the set $(\K^\times)^n$ where $\K^\times =\K\setminus\{0\}$.
In fact, the analytification $\GG_m^{n,\an}$ admits an admissible covering
$
\textstyle
\GG_m^{n,\an} = \bigcup_{r \ge 1} \Sp \left( \K \langle r^{-1}z_i, r^{-1}z_i^{-1} \mid 1\le i \le n  \rangle \right)
$
where $\Sp  ( \K \langle r^{-1}z_i , r^{-1} z^{-1}_i \rangle  )$ is the subset of $(\mathbb K^\times )^n$ defined by $\frac{1}{r}\le |x_i| \le r$.
In our non-archimedean setting, analogous to the logarithm map $(\mathbb C^*)^n \to \mathbb R^n$, we have the following well-known map:
\[
\Val: ( \K^{\times})^n\cong \GG_m^{n,\an} \to \mathbb R^n, \ \ \
\Val (a_1,\dots,a_n)=(\val(a_1),\dots,\val(a_n))
\]
Let $\Gamma=\val( \K^\times)$ be the valuation group. For instance, when $\K$ is the Novikov field, we have $\Gamma=\mathbb R$. A convex subset $\Delta\subset \mathbb R^n$ is called a \textit{$\Gamma$-rational convex polyhedron} (often omitting `$\Gamma$') if it is defined by finitely many inequalities 
$
\sum_j b_{ij} x_j \ge c_i$
for $c_i\in \Gamma$ and $b_{ij}\in \mathbb Z$.

For a bounded $\Gamma$-rational polyhedron $\Delta\subset \mathbb R^n$, we define the so-called \textbf{\textit{polyhedral affinoid algebras}}
$
\K\langle\Delta\rangle \subset \K[[ z_1^\pm,\dots, z_n^\pm]]
$
to be the set of formal Laurent series 
$f=\sum_{\nu\in \mathbb Z^d} a_{\nu} \mathbf z^{\nu}$
so that 
$
\val(a_{\nu}) + \nu \cdot u \to \infty
$
for all $u\in\Delta$. The multiplication $\cdot$ on $\K\langle \Delta\rangle$ is given by
$\big(\sum_\nu a_\nu \mathbf z^\nu\big) \cdot
\big(\sum_\nu b_\nu \mathbf z^\nu\big)
=
\sum_\nu \big( \sum_{\nu=\nu_1+\nu_2} a_{\nu_1}b_{\nu_2} \big) \cdot \mathbf z^\nu$
which converges as
$
\val(\sum_{\nu=\nu_1+\nu_2}a_{\nu_1}b_{\nu_2} ) + \nu\cdot u \ge \min_{\nu=\nu_1+\nu_2}\{ \val(a_{\nu_1})+\val(b_{\nu_2}) +\nu_1\cdot u+\nu_2\cdot u \} \to \infty$.

\begin{prop}\label{polyhedral-affinoid-algebra-evaluation-convergence-prop}
	Let $f=\sum_{\nu\in\mathbb Z^d} a_\nu \mathbf z^\nu \in \K[[z^\pm]]$. Then
$f \in \K\langle \Delta\rangle$ if and only if $f(\mathbf y)$ converges for any $\mathbf y \in \Val^{-1}(\Delta)$. In this case, $f$ can be recognized as a global function defined on $\trop^{-1}(\Delta)$.
\end{prop}

\begin{proof}
Recall \cite[2.1/3]{BoschBook} that the convergence means exactly that $a_\nu \mathbf y^\nu$ forms a zero sequence, or equivalently $\val (a_\nu) +\nu \cdot \Val(\mathbf y) \to \infty$. This exactly corresponds to the definition of $\K\langle\Delta\rangle$.
\end{proof}

\begin{prop}
 \label{polyhedral-affinoid-algebra-prop-appendix}
$\K\langle \Delta \rangle$ is an affinoid algebra and
$U_\Delta:= \Val^{-1} (\Delta) $
is identified with $\Sp \K \langle\Delta \rangle$ via $\mathbf x\xleftrightarrow{} \m_{\mathbf x}$. Moreover, $\Val^{-1}(\Delta)$ is a Weiestrass domain in $(\K^{\times})^n$, called a \textit{polytopal domain}.
\end{prop}

This is due to \cite[3.1.5/3.18(c)]{EKL}; see also \cite[Prop. 4.1]{Gubler}.

Next, we briefly explain the notion of affinoid torus fibration in Theorem \ref{Main_theorem_thm}.
For every open subset $U$ in $\mathbb R^n$, the pre-image $\trop^{-1}(U)$ is an \textit{analytic open subset}, which is not necessarily an affinoid space yet.
Note that an analytic open domain $\trop^{-1}(U)$ is covered by various affinoid spaces $\trop^{-1}(\Delta)$ for $\Delta\subset U$.
The set $\Sp\Lambda\langle \Delta\rangle$ of all maximal ideals are in bijection with the points in $\trop^{-1}(\Delta)$.

\begin{defn}
	An \textit{affinoid torus fibration} $f:Y\to B$ refers to a continuous map, from an analytic space to a topological manifold, that is locally modeled on $\trop^{-1}(U)\to U$ for open subsets $U$ in $\mathbb R^n$ (see also \cite{KSAffine, NA_nonarchimedean_SYZ} for further details).
\end{defn}

\subsection{Uniform reverse isoperimetric inequalities}
\label{SA_reverse_isoperimetric}

\begin{thm}[Theorem 1.1 \cite{ReverseI}]
	\label{reverse-ineq-original-thm-appendix}
Let $(X,\omega)$ be a symplectic manifold and $L$ be a closed Lagrangian submanifold. For a tame almost complex structure $J$, there exists a constant $c=c(L,J)>0$ so that
$
\area( u; g_J) \ge c \cdot \ell(\partial u; g_J)$ for any $J$-holomorphic curve $u:(\mathbb D, \partial \mathbb D) \to (X,L)$.
\end{thm}

For our purpose, we need to strengthen it to Theorem \ref{reverse-ineq-uniform-thm-appendix} below. As in \cite[Appendix A]{AboFamilyWithout}, we will closely follow DuVal's more flexible argument \cite{ReverseII}.
First, we note that by \cite[Lemma 2.2.1]{MS}, if $u$ is a $J$-holomorphic curve for some $\omega$-tame almost complex structure $J$, then the energy $
E(u)=\frac{1}{2}\int_\Sigma |du|^2 d \mathrm{vol}_\Sigma$ is actually topological:
$E(u)=\area(u)
=\int_\Sigma u^*\omega$.

\begin{thm}\label{reverse-ineq-uniform-thm-appendix}
Fix $(X,\omega)$, $L$ and $J$ as above. There is a $C^1$-neighborhood $\mathcal V_1$ of $J$ and a constant $c_1>0$ such that for any $\tilde J \in \mathcal V_1$ and $\tilde J$-holomorphic disk $u:(\mathbb D,\partial\mathbb D)\to (X,L)$, we have 
$
E(u) \ge c_1 \cdot \ell (\partial u )
$.
\end{thm}

\begin{cor}\label{reverse-ineq-cor-appendix}
There is a $C^1$-neighborhood $\mathcal V_0$ of $J$, a Weinstein neighborhood $\nu_X L$ of $L$ and a constant $c_0>0$ such that:
If $\tilde L\subset \nu_X L$ is an adjacent Lagrangian given by the graph of a small closed one-form on $L$, then for $\tilde J\in\mathcal V_0$ and $\tilde J$-holomorphic disk $u:(\mathbb D,\partial\mathbb D)\to (X, \tilde L)$, we have
$
E(u) \ge c_0 \cdot \ell (\partial u)
$.
\end{cor}

\begin{proof}
[Proof of Theorem \ref{reverse-ineq-uniform-thm-appendix} implying Corollary \ref{reverse-ineq-cor-appendix}]
Suppose a neighborhood $\mathcal V_1$ of $J$ and a constant $c_1>0$ are obtained by Theorem \ref{reverse-ineq-uniform-thm-appendix}.
Choose the $\mathcal V_0\subset \mathcal V_1$ and $\nu_X L$ small enough, and we may require that for any such $\tilde L\subset \nu_XL$ and $\tilde J \in\mathcal V_0$, one can find a small $F\in\diff_0(X)$ such that $F(\tilde L)=L$ and $F_*\tilde J\in\mathcal V_1$.
Then, since $F\circ u$ is a $F_* \tilde J$-holomorphic disk bounding $L$, we have $E(F\circ u)\ge c_1  \ell(\partial (F\circ u) )$.
Finally, we can compare the energy and the boundary as follows:
$
E(u) \gtrsim E(F\circ u) \gtrsim  \ell (\partial (F\circ u)) \gtrsim \ell(\partial u)
$.
\end{proof}

%
%

\begin{defn}
	\label{plurisubharmonic_defn}
	Define $d^{\tilde J} f:= -df\circ \tilde J$.
	Given an almost complex structure $\tilde J$, a function $\rho$ is called \textit{$\tilde J$-plurisubharmonic} (resp. \textit{strict $\tilde J$-plurisubharmonic}) if $dd^{\tilde J} \rho (v, \tilde Jv)\ge 0$ (resp. $>0$) for any $v\neq 0$.
\end{defn}

Now, we aim to show Theorem \ref{reverse-ineq-uniform-thm-appendix}.
We slightly generalize a lemma in \cite{ReverseII} by allowing a small neighborhood of the almost complex structure. 

\begin{lem}\label{plurisubharmonic-lem}
There exists a $C^1$-neighborhood $\mathcal W$ of $J$ and a function $\rho$ of class $C^2$ on a tubular neighborhood of $L$ such that the $\rho$ vanishes exactly on $L$. Moreover, for any $\tilde J \in \mathcal W$,

\begin{itemize}
	\itemsep 2pt
\item[(i)] $\sqrt \rho$ is $\tilde J$-plurisubharmonic outside $L$
\item[(ii)] $\rho$ is strictly $\tilde J$-plurisubharmonic
\end{itemize} 

\end{lem}

\begin{proof}
Let $\nu_X L$ be a small tubular neighborhood of $L$. Fix a nonnegative $C^2$ function $\rho\ge 0$ defined on $\nu_XL$ which vanishes exactly on $L$. And, we are going to modify $\rho$ to meet the requirements.

The question is local. We take a system of local coordinates $z^\alpha = x^\alpha + i y^\alpha$ ($1\le \alpha \le n$) near $L$ in $X$ so that the coordinates $x^\alpha$ in the subspace $\mathbb R^n \cong \{y^\alpha=0\}\subset \mathbb C^n$ gives a local chart of $L$.
Note that since $L$ is compact, one can cover $L$ be finitely many such local charts.
We may require the restriction of $J$ on $L\cong \mathbb R^n$ is the standard complex structure $J_0$ on $\mathbb C^n$.
Note that $J=J_0 + O(|y|)$ and $\rho=O(|y|^2)$; besides, we write
\[
\rho = \sum a_{\beta\gamma}(x) y^\beta y^\gamma +O(|y|^3)=: q+ O(|y|^3)
\]
where $(a_{\beta\gamma}(x))$ is a symmetric strictly positive matrix which only depends on $x=(x^\alpha)$.
Put $x^{\alpha+n}=y^\alpha$ for $1\le \alpha \le n$; we will use the letters $i,j,\dots$ to indicate integers from $1$ to $2n$, while we will use $\alpha,\beta,\dots$ for integers from $1$ to $n$.
We will also use the Einstein summation convention.

Now, we express an arbitrary almost complex structure $\tilde J$ in a small neighborhood $\mathcal W$ of $J$ with respect to these local coordinates as follows:
\[
\textstyle
\tilde J= \sum_{i,j=1}^{2n}\tilde J_j^i \frac{\partial}{\partial x^i} \otimes dx^j
\]
One may similarly write $J=J^i_j \partial_{x^i}\otimes dx^j$ and $J_0=(J_0)^i_j \partial_{x^i}\otimes dx^j$. Then, $(J_0)^{\beta+n}_\alpha = \delta^\beta_\alpha$, $(J_0)^{\beta}_{\alpha+n} = -\delta^\beta_\alpha$ and all other $(J_0)^j_i=0$;
given a function $f$ on $\nu_XL$, we have
\[
d^{\tilde J}f =- \tilde J_i^j \partial_{x^j} f dx^i
\]
By Definition \ref{plurisubharmonic_defn}, $\theta=\theta_{ij} dx^i\wedge dx^j$ with $\theta_{ij}=-\theta_{ji}$ is said to be $\tilde J$-positive (resp. strictly $\tilde J$-positive) if
\[
\theta(v,\tilde J v) =v^i \theta_{ik}  \tilde J^k_j v^j \ge 0  \ (\text{resp.} \ >0)
\]
for any $v=v^i\partial_{x^i}\neq 0$, i.e. the (strict) positivity of the following symmetry matrix 
$(\theta_{ik}\tilde J^k_j+\theta_{jk}\tilde J^k_i)_{1\le i, j\le 2n}
$.

\textit{Step one.} We first study $\rho$.
Compute
$dd^{\tilde{J}}\rho= dd^{\tilde J} q + O(|y|)$.
Beware that this $O(|y|)$ actually depends on $\tilde J$ but there is a uniform constant $C$ so that $O(|y|)\le C|y|$.
Next, we compute
\begin{align*}
dd^{\tilde J} q
 &
 =
 a_{\beta\gamma} dd^{\tilde J} (y^\beta y^\gamma) + O(|y|)
=
2a_{\beta\gamma} \tilde J_i^{n+\beta} dx^i\wedge dy^\gamma  + O(|y|) \\
&
=
 2a_{\beta\gamma}(J_0)_i^{n+\beta} dx^i\wedge dy^\gamma +
 O(\tilde J- J_0) +O(|y|) \\
 &
 =
 2a_{\beta\gamma}dx^\beta\wedge dy^\gamma +
 O(\tilde J- J_0) +O(|y|) \qquad
 =:2\theta+ O(\tilde J- J_0) +O(|y|) 
\end{align*}
where the $O(\tilde J-J_0)$ represents a term bounded by a multiple of the $C^1$-norm of $\tilde J - J_0$. Define a 2-form $\theta=\theta_{ij}dx^i\wedge dx^j =a_{\beta\gamma}dx^\beta\wedge dy^\gamma$, where $\theta_{\beta,\gamma+n}=-\theta_{\gamma+n ,\beta} =\frac{1}{2} a_{\beta\gamma}$ and all other $\theta_{ij}=0$.
Then, 
\begin{align}
\label{dd^J rho_pre}
dd^{\tilde J} \rho ( \cdot, \tilde J \cdot) 
= 2\theta (\cdot, J_0 \cdot) + O(\tilde J - J_0) + O(|y|)
\end{align}
It can be viewed as a matrix identity, and the $\theta(\cdot, J_0\cdot)$ corresponds to the symmetric strictly positive-definite $(2n)\times (2n)$ matrix $A=(A_{ij})$ defined by setting
$
A_{\beta\gamma}=A_{\beta+n, \gamma+n} = a_{\beta\gamma}
$
and all other $A_{ij}=0$.
Remark that the $\theta$ and $A$ only depend on $\rho$.
By shrinking the neighborhood $\nu_XL$, one can make the $O(|y|)$ small; by shrinking $\mathcal W$, we can make the $O(\tilde J - J_0)$ small. So, by (\ref{dd^J rho_pre}), one can ensure that for any $\tilde J\in \mathcal W$, the following equation holds
\begin{equation}
\label{dd^J rho}
dd^{\tilde J} \rho (\cdot, \tilde J \cdot ) \ge \theta( \cdot, J_0\cdot )
\end{equation}

\textit{Step two.}
We next deal with $\sqrt \rho$ and aim to prove the following:
\begin{equation}\label{dd^J sqrt rho}
dd^{\tilde J}\sqrt \rho (\cdot, \tilde J \cdot)\ge O(1)
\end{equation}
if $\tilde J$ is sufficiently close to $J$. A subtle point is that $dd^{\tilde J}\sqrt \rho$ may be unbounded near $L\cong\{y^\alpha=0\}$.
The idea is that one can show the unbounded part corresponds to a positive semi-definite matrix.
To see this, observe first that the highest unbounded terms are asymptotically $O(\frac{1}{|y|})$.
Recall that $q=a_{\beta\gamma}(x) y^\beta y^\gamma$.
In the computation of $dd^{\tilde J}\sqrt q$ modulo $O(1)$, there is no need to differentiate $a_{\beta\gamma}(x)$, since otherwise one cannot produce any $O(\frac{1}{|y|})$ terms.
Hence, modulo $O(1)$, we may assume that $a_{\beta\gamma}$ are all constant; up to a linear transformation, we may further assume $q=|y|^2$.

Now, we have $\rho=q+O(|y|^3)$ and $\sqrt \rho = \sqrt q + O(|y|^2)=|y|+O(|y|^2)$. Then, $dd^{\tilde J} \sqrt \rho = dd^{\tilde J}|y| +O(1)$, and the unbounded part is $dd^{\tilde J} |y|$. Hence, to show (\ref{dd^J sqrt rho}), it suffices to show $dd^{\tilde J}|y|(\cdot, \tilde J\cdot)$ is positive semi-definite.
In reality, we fix $v=v^i\partial_{x^i}\neq 0$ and compute:
\begin{align*}
dd^{\tilde J} |y| (v,\tilde J v)
&
=
\textstyle
\big( \sum_{\alpha,\beta} \tilde J_\ell^{n+\alpha} \frac{\delta^\alpha_\beta |y|^2 - y^\alpha y^\beta}{|y|^3}dx^\ell\wedge dy^\beta
\big)
 \Big( v^i\partial_{x^i}, \tilde J^j_k v^k\partial_{x^j} \Big)\\
&
=
\textstyle
\sum_{\alpha,\beta}  \frac{\delta^\alpha_\beta |y|^2 - y^\alpha y^\beta}{|y|^3} \cdot
\Big(
 (v^\ell\tilde J_\ell^{n+\alpha}) \cdot (v^k \tilde J^{n+\beta}_k) + v^{n+\alpha} v^{n+\beta}
\Big)
\end{align*}
where the last equality holds follows from the identity $\tilde J^m_\ell \tilde J^{\ell}_k = - \delta^m_k$.
By setting $u^\alpha= \sum_\ell v^\ell \tilde J^{n+\alpha}_\ell$ or $u^\alpha=v^{n+\alpha}$, it reduces to the obvious positive semi-definiteness of the quadratic form 
$
\sum_{\alpha,\beta}  ( \delta^\alpha_\beta |y|^2 - y^\alpha y^\beta) \cdot  u^\alpha u^\beta = n\sum_\alpha (u^\alpha y^\alpha)^2 - (\sum_\alpha u^\alpha y^\alpha)^2 $. Now our claim (\ref{dd^J sqrt rho}) is now established. 

\textit{Step three.} 
Finally, we replace $\rho$ by $\rho_1= (\sqrt \rho + B \rho)^2=\rho+2B\rho^{\frac{3}{2}}+B^2\rho^2$ for a sufficiently large constant $B>0$.
Since $\rho=O(|y|^2)$ and $\rho^{\frac{3}{2}}=O(|y^3|)$, the $\rho_1$ is $C^2$ near $L$.

(i) Observe that $\sqrt \rho_1= \sqrt \rho +B \rho$, and so $dd^{\tilde J} \sqrt \rho_1 (\cdot, \tilde J \cdot) =dd^{\tilde J} \sqrt \rho  (\cdot, \tilde J \cdot)+ Bdd^{\tilde J} \rho (\cdot, \tilde J \cdot)$ is positive-definite by (\ref{dd^J rho}) and (\ref{dd^J sqrt rho}) for a sufficiently large $B>0$.
Thus, the $\sqrt {\rho_1}$ is $\tilde J$-plurisubharmonic.

(ii) We compute $dd^{\tilde J} \rho_1 = (\frac{3B}{2\sqrt \rho} + 2B^2 ) d\rho \wedge d^{\tilde J} \rho + (1+3B\sqrt \rho +2B^2 \rho) dd^{\tilde J} \rho $. For any $v\neq 0$, we have $d\rho\wedge d^{\tilde J} \rho (v, \tilde J v) = (d\rho (v))^2 + (d\rho (\tilde J v))^2 \ge 0$, so $d\rho\wedge d^{\tilde J} \rho(\cdot, \tilde J \cdot)$ is positive semi-definite. By (\ref{dd^J rho}), we also show that $dd^{\tilde J} \rho_1 (\cdot, \tilde J \cdot) >0$.
Thus, the $\rho_1$ is strictly $\tilde J$-plurisubharmonic.
\end{proof}

\begin{proof}
[Proof of Theorem \ref{reverse-ineq-uniform-thm-appendix}]
Let $\rho$ and $\mathcal W$ be as in Lemma \ref{plurisubharmonic-lem}, and let $\nu_XL$ be the tubular neighborhood of $L$ therein.
For any function $f$ we denote by $h_{\tilde J}^f$ the symmetric tensor defined as follows: 
\[
h^f_{\tilde J}(v,w)=\tfrac{1}{2}\big( dd^{\tilde J}f (v,\tilde Jw) + dd^{\tilde J} f (w,\tilde Jv) \big)
\]
Then, the Lemma \ref{plurisubharmonic-lem} actually tells that $h^{\sqrt \rho}_{\tilde J}$ is a semi-metric and $h^\rho_{\tilde J}$ is a metric on their domains for any $\tilde J \in \mathcal W$.
Shrinking $\nu_X L$ if necessary, we may assume $\nu_X L = 
\{ p\in X \mid \dist (p, L) \le d_0\}$ for some small $d_0>0$.
Notice that $\rho=O(|y|^2)$ and $|y|$ is comparable to $\dist(p,L)$, thus, there is $c_2>0$ so that
\begin{equation}
\label{rho  y^2 -eq}
\tfrac{1}{c_2}\dist(p,L) \le \sqrt \rho(p) \le c_2 \dist(p,L)
\end{equation}
for any $p\in \nu_X L$.
In particular, there exists a constant $r_0>0$ which is independent of $\tilde J\in\mathcal W$ such that
\begin{equation}\label{subset nu_ML -eq}
\{ p\mid \rho(p) \le  r_0^2 \} \subset \nu_X L
\end{equation}
We choose local coordinates $(x^\alpha, y^\alpha)$ as before. Given $p\in \nu_X L$, the distance $\dist (p, L)$ is comparable to $|y|$.
Take a cube $Q$ centered at $p$ so that $\partial Q$ touches $L=\{y^\alpha=0\}$.
Applying the gradient estimate \cite[Eq.(3.15)]{EllipticPDE} to $\rho$ on $Q$ deduces that 
\[
|\nabla \rho| \le \tfrac{n}{|y|} \textstyle{ \sup_{\partial Q} |\rho|+ \tfrac{|y|}{2}\sup_Q|\Delta\rho|}
\lesssim( \tfrac{1}{|y|} |2y|^2+ \tfrac{|y|}{2}) \lesssim |y|
\]
Particularly, there is a constant $c_3>0$ so that
\begin{equation}
\label{gradient-estimate-eq}
|\nabla \rho| \le c_3 \dist( p, L )
\end{equation}
Suppose we have a $\tilde J$-holomorphic disk $u: (\mathbb D, \partial \mathbb D) \to (X,L)$. Denote by $j$ the complex structure on the unit disk $\mathbb D\subset \mathbb C$.
The equation $\tilde J\circ du = du \circ j$ implies that for any function $f$, we have
$u^* d^{\tilde J} f = d^j u^* f$. Thus,
$
u^*dd^{\tilde J} f =dd^j u^*f
$.
In our case, $\rho$ is $\tilde J$-plurisubharmonic tells that $u^*\rho$ is $j$-plurisubharmonic; then, $u^*h^\rho_{\tilde J}$ is a semi-metric, and we have $u^*dd^{\tilde J} \rho=d \vol_{u^* h^\rho_{\tilde J}}$.
Consider the following function
\[
a(r)= \frac{1}{r} \int_{ \{u^*\rho \le r^2\} } u^* dd^{\tilde J}\rho
=\frac{1}{r} \int_{ \{u^*\rho \le r^2\} } d\vol_{u^* h^\rho_{\tilde J}} 
\]
It is at least well-defined on $[0, r_0 ]$ thanks to (\ref{subset nu_ML -eq}).
Since $d^{\tilde J} \rho = 2\sqrt \rho \cdot  d^{\tilde J} \sqrt \rho$, the Stokes formula implies $a(r) =2\int_{\{u^*\rho =r^2\}} u^* d^{\tilde J} \sqrt \rho$ and also
\[
a(r')-a(r)=
2\int_{\{r^2\le u^*\rho \le {r'}^2\}} du^* d^{\tilde J} \sqrt \rho =
2\int_{\{r^2\le u^*\rho \le {r'}^2\}} d\vol_{u^*h^{\sqrt \rho}_{\tilde J}}
\]
 for $r'\ge r$. By condition, the $h_{\tilde J}^{\sqrt \rho}$ is a semi-metric away $L\cong\{\rho=0\}$, and thus the $a(r)$ is increasing.
Recall that the constant $r_0$ is independent of $\tilde J$. Now, we have the energy estimate
\[
E(u)
\gtrsim r_0  \ a(r_0)\ge
r_0 \lim_{r\to 0} a(r) \gtrsim \lim_{r\to 0} a(r)
\]
Due to (\ref{rho  y^2 -eq}) and (\ref{gradient-estimate-eq}), we get $|\nabla \rho|
\lesssim \dist (\cdot , L) \lesssim \sqrt \rho \lesssim r$.
So, the coarea formula tells
$
\int_0^{r^2} \ell ( \{u^*\rho =t\} ) dt
=
\int_{\{ u^*\rho \le r^2 \}} | \nabla  \rho | \cdot d \vol_{ u^* h^\rho_{\tilde J}} \lesssim r \cdot r \ a(r) =r^2 a(r)
$ and so
\[
E(u)\gtrsim \lim_{r\to 0} a(r)\gtrsim \lim_{r\to 0} \tfrac{1}{r^2} \textstyle \int_0^{r^2} \ell (\rho =t) dt =\ell (\rho=0)=\ell(\partial u)
\]
\end{proof}

\subsection{Integral affine structures}\label{SA_integral_affine}

We state two equivalent definitions of an \textit{integral affine structure} on an $n$-dimensional manifold $Y$ (see \cite{KSAffine}): (i) there is an atlas of charts
such that the transition functions belong to $GL(n,\mathbb Z) \ltimes \mathbb R^n$; (ii) there is a torsion-free flat connection $\nabla$ on $TY$ and a $\nabla$-covariant lattice of maximal rank $TY^{\mathbb Z}$.

Given an integral affine structure on $Y$, a local chart $\phi: U \to \mathbb R^n$ on a small open subset $U\subset Y$ is called an \textit{integral affine chart} if the torsion-free flat connection is given by $\nabla=d$ in $TY|_U$.
For the coordinates $x_i=x_i\circ \phi$, the lattice $TY^{\mathbb Z}_x$ for each $x\in U$ is the free abelian group generated by $\partial / \partial x_i$'s.
The transition maps among them belong to $GL(n,\mathbb Z) \ltimes \mathbb R^n$. If we have two systems of local coordinates $(x_i)$ and $(x'_i)$, then the transition map is of the form
$
x'_j = \sum_k a_{jk} x_k + f_j
$
for some matrix $A=(a_{jk})\in GL(n,\mathbb Z)$ and $f_j\in \mathbb R$.

\begin{prop}\label{Polyhedron-transition-prop-appendix}
Let $U,U'$ be two integral affine charts. A subset $\Delta\subset U\cap U'$ is a rational convex polyhedron in $U\subset \mathbb R^n$ if and only if it is a one in $U'\subset \mathbb R^n$. Such $\Delta$ is also called a \emph{rational convex polyhedron} in $Y$.
\end{prop}

\begin{proof}
Assume $\Delta$ is defined by $\sum_j b_{ij}x'_j \ge c_i$ in $U'$ for some $b_{ij}\in \mathbb Z$ and $c_j\in \mathbb R$, and the transition map is given as above. Then using the coordinated in $U$, $\Delta$ is given by $\sum_k (\sum_j b_{ij} a_{jk}) x_k \ge c_i - \sum_j b_{ij}f_j$. 
\end{proof}

\begin{defn}\label{Polyhedral decomp-defn-appendix}
A \emph{rational polyhedral complex} $\mathscr Q$ in an integral affine manifold $Y$ is a CW complex so that (i) the underlying space of each cell $\Delta\in \mathscr Q$ is a rational convex polyhedron in $Y$; (ii) each face of $\Delta\in\mathscr Q$ is in $\mathscr Q$; (iii) The intersection $\Delta\cap\Delta' \in \mathscr Q$ is a face of both $\Delta, \Delta' \in \mathscr Q$.
\end{defn}

\begin{lem} \label{Polyhedral decomposition- integral-lem-appendix}
Fix $\epsilon>0$ and a metric in $Y$. Suppose $K\Subset K'\subset Y$ are two compact domains in $Y$. Then, there exists a rational polyhedral complex $\mathscr P$ in $B_0$, and all the cells in $\mathscr P$ have diameters less than $\epsilon$, and the underlying topological space satisfies $K \Subset |\mathscr P| \Subset K'$.
\end{lem}

\begin{proof}
In an integral affine chart at some point $p\in K$, there always exists a rational hypersurface $H$ (codimension-one rational polyhedron) passing through $p$. By extension, we may require $H$ is maximal with respect to the inclusions among the compact domains in $Y$.
In special, the $H$ must `escape' $K$ in the sense that $H\cap K$ is closed in $K$.
Now, we cover $K'$ by finitely many open sets $V_1,\dots, V_m$ so that for every $1\le i\le m$, the closure $\bar V_i$ is compact and is contained in some integral affine chart $U_i$.
We may find a sufficiently dense collection of rational hyperplanes in $U_i\cong \phi(U_i)\subset \mathbb R^n$ so that every chamber enclosed by them has diameter less that $\epsilon$.
Extending all these hyperplanes to be maximal in the above sense, we get a collection $\mathscr H$ of closed rational hyperplanes in $K'$. By Proposition \ref{Polyhedron-transition-prop-appendix}, the $\mathscr H$ divides $K$ into rational polyhedrons each of which has diameter less than $\epsilon$. 
\end{proof}

\paragraph{\pmb{Acknowledgment.}}
This work is done during my PhD program at Stony Brook University.
My foremost thanks are to my advisor, Kenji Fukaya, for his encouragement, guidance and support along the way.
I would like to thank my college teacher Jun Sun for useful lessons of differential geometry.
I also benefit from conversations with
Mohammed Abouzaid, Mohamed El Alami, Catherine Cannizzo, Xujia Chen, Andrew Hanlon, 
Enrica Mazzon,
Mark McLean, 
Kevin Sackel, 
Yuhan Sun, Yi Wang,
and Aleksey Zinger.
I am grateful to
Tim Campion,
J\'er\^ome Poineau,
Xavier Xarles, and some anonymous people for discussions via
\textit{mathoverflow.net} from which I learn a lot about non-archimedean geometry; moreover, I am especially grateful to Will Sawin for his answer there which inspires the elegant proof of Lemma \ref{val=0-lem}.
Thanks also to the organizers of \textit{Western Hemisphere Virtual Symplectic Seminar} for giving a chance to have a video talk in April 2020. 
Finally, I would like to thank deeply anonymous referees for helpful comments.

\bibliographystyle{abbrv}
\bibliography{mybib}

\end{document}